\documentclass[aos]{imsart}

\usepackage{latexsym,amsmath,amssymb,amsfonts,amsthm,bbm,mathrsfs,breakcites,dsfont,xcolor,bm}
\usepackage{stmaryrd,epsf}
\usepackage{natbib}
\usepackage{url}
\usepackage{dsfont}
\usepackage{enumerate}
\usepackage{graphicx}
\usepackage{graphics}
\usepackage{psfrag}

\usepackage[colorlinks,linkcolor=red,citecolor=blue,urlcolor=blue,breaklinks]{hyperref}

\usepackage{algorithm, algpseudocode}
\algblock{Input}{EndInput}
\algnotext{EndInput}
\algblock{Output}{EndOutput}
\algnotext{EndOutput}

\newtheorem{theorem}{Theorem}
\newtheorem{lemma}[theorem]{Lemma}
\newtheorem{proposition}[theorem]{Proposition}
\newtheorem{corollary}[theorem]{Corollary}

\newtheorem*{assumption*}{A\!\!}
\newtheorem{definition}{Definition}
\newtheorem*{definition*}{Definition}
\newtheorem{definitionalt}{Definition}[definition]
\newenvironment{definitionp}[1]{
  \renewcommand\thedefinitionalt{#1}
  \definitionalt
}{\enddefinitionalt}

\newtheorem{example}{Example}
\newcommand{\indep}{\rotatebox[origin=c]{90}{$\models$}}

\DeclareMathOperator*{\argmin}{argmin}
\DeclareMathOperator*{\argmax}{argmax}

\newcommand{\IN}{\mbox{$\mathbb{N}$}}
\newcommand{\IR}{\mbox{$\mathbb{R}$}}

\newcommand{\IE}{\mbox{$\mathbb{E}$}}

\newcommand{\IP}{\mbox{$\mathbb{P}$}}

\usepackage{xcolor}
\usepackage[draft,inline,nomargin,index]{fixme}
\fxsetup{theme=color,mode=multiuser}
\FXRegisterAuthor{tc}{atc}{\color{red} TC}
\FXRegisterAuthor{ts}{ats}{\color{red} TS}
\FXRegisterAuthor{tb}{atb}{\color{red} TB}

\usepackage[draft,inline,nomargin,index]{fixme}
\fxsetup{theme=color,mode=multiuser}

\begin{document}

\begin{frontmatter}

\title{Nonparametric classification with missing data}
\runtitle{Classification with missing data}

\begin{aug}
\author[A]{\fnms{Torben} \snm{Sell}\ead[label=e1]{torben.sell@ed.ac.uk}},  
\author[B]{\fnms{Thomas} B. \snm{Berrett}\ead[label=e2]{tom.berrett@warwick.ac.uk}}\\ 
\and 
\author[A]{\fnms{Timothy} I. \snm{Cannings}\ead[label=e3]{timothy.cannings@ed.ac.uk}}

\runauthor{Sell, Berrett and Cannings}

\address[A]{School of Mathematics, University of Edinburgh\printead[presep={,\ }]{e1,e3},} 

  \address[B]{Department of Statistics, University of Warwick\printead[presep={,\ }]{e2}} 
    
\end{aug}

\begin{abstract}
We introduce a new nonparametric framework for classification problems in the presence of missing data. The key aspect of our framework is that the regression function decomposes into an anova-type sum of orthogonal functions, of which some (or even many) may be zero. Working under a general missingness setting, which allows features to be missing not at random, our main goal is to derive the minimax rate for the excess risk in this problem. In addition to the decomposition property, the rate depends on parameters that control the tail behaviour of the marginal feature distributions, the smoothness of the regression function and a margin condition. The ambient data dimension does not appear in the minimax rate, which can therefore be faster than in the classical nonparametric setting. We further propose a new method, called the \emph{Hard-thresholding Anova Missing data (HAM)} classifier, based on a careful combination of a $k$-nearest neighbour algorithm and a thresholding step. The HAM classifier attains the minimax rate up to polylogarithmic factors and numerical experiments further illustrate its utility. 
\end{abstract}



\end{frontmatter}

\date{}

\maketitle

\section{Introduction} 

Missing data are ubiquitous in modern statistics, posing a major challenge in a plethora of applications.  For instance, missingness may arise when our data are collated from sources that have measured different variables, e.g.~in healthcare the data routinely collected on patients may differ across clinics or hospitals. Other reasons include sensor failure, data censoring and privacy concerns, among many others.  The topic has been widely studied in many statistical problems; see \citet{josse2018introductionm} and \citet{little2019statistical} for a thorough overview.  Recent works study missing data in several different contexts, including high-dimensional regression \citep{loh2011high,chandrasekher2020imputation}, covariance and precision matrix estimation \citep{cai2016minimax,loh2018high}, remote sensing \citep{zhang2018missing}, principal component analysis \citep{elsener2019sparse,zhu2019high}, $U$-statistics \citep{cannings2022correlation},  changepoint detection \citep{follain2022high} and testing whether the missingness is independent of the data \citep{berrett2022optimal,bordino2024tests}. 

We focus here on missing data in binary classification problems, where the practitioner is presented with the task of assigning a new observation to one of two classes, based on a training set of labelled data.  There is a large literature on classification and many algorithms are available in situations when the data is fully observed; see, for example, \citet{devroye2013probabilistic} and \citet{boucheron2005theory}. In settings with missing data, however, there are very few specifically designed classification algorithms, and in many applications practitioners often use ad-hoc approaches to deal with missingness, either resorting to complete-case analysis (i.e.~using only the fully observed data points) or removing any covariates for which missing data is present.  An alternative option, sometimes used in combination with those mentioned previously, is to impute any missing values to obtain a `full' data set before employing an existing algorithm. However, the majority of existing classifiers are designed to work in idealised settings, where the training and test data arise as fully-observed, independent realisations from the same underlying population; it is often unclear how they perform in combination with these ad-hoc approaches.

The main aim of this paper is to establish a new nonparametric framework for classification problems with missing data.  To model the missingness, we will view our training data as independent copies of triples $(X,Y,O)$ arising from a distribution $Q$ on $\mathbb{R}^d \times \{0,1\} \times \{0,1\}^d$, where $X$ denotes a $d$-dimensional feature vector, $Y$ is its class label and $O$ is the concomitant \emph{observation indicator}. The interpretation here will be that we only observe the components of $X$ for which the corresponding entry of $O$ is $1$. The label $Y$ is always observed. For $\omega \in \{0,1\}^d$, we say that $(X,Y,O)$ is an \emph{available case} for $\omega$ if the entries of $O$ are $1$ whenever the corresponding entry of $\omega$ is $1$.  We write $\eta(x) := \mathbb{P}(Y=1 | X=x)$ for the \emph{regression function} of $Y$ given $X$ at $x \in \mathbb{R}^d$. 

Our framework facilitates the use of all available cases for each observation pattern $\omega \in \{0,1\}^d$. The key observation is that any regression function $\eta$ on $\mathbb{R}^d$ may be decomposed into an orthogonal sum. That is, we can write 
\begin{equation}
\label{eq:ANOVAdecomp}
\eta(x) = \frac{1}{2} + \sum_{\omega \in \{0,1\}^d} f_{\omega}(x),
\end{equation}
where, for each $\omega \in \{0,1\}^d$, the function $f_{\omega}$ in fact only depends on those components of $x$ for which the corresponding entries of $\omega$ are $1$.  We can then estimate the regression function $\eta$ by recursively estimating each of the $f_{\omega}$ functions starting with $f_{(0,\ldots, 0)}$, based on the available cases for $\omega$.  One interpretation of~\eqref{eq:ANOVAdecomp} is that of an anova decomposition \citep{efron1981jackknife}, and the functions $f_\omega$ may therefore be viewed as capturing the interactions between the variables. Our framework (see Definition~\ref{def:effectivespace}) specifies a subset (which may be empty) of observation patterns in $\{0,1\}^d$ for which the corresponding interaction terms are identically~$0$.  The idea here is that the difficulty (in a sense that we make precise later) of the classification problem is determined by the non-zero $f_{\omega}$, as opposed to the properties of the full regression function in the ambient $d$-dimensional space.

The primary contribution of the paper is to establish the minimax rate (up to polylogarithmic factors) for the excess risk in this problem (Theorem~\ref{thm:minmax_bounds}).  The rate depends on various parameters that index the class of distributions considered in our framework.  Throughout the paper, we will place only mild restrictions on the distribution of the missingness indicator $O$. Indeed, we neither ask that the components of $O$ are independent of each other, nor that $O$ is necessarily independent of $(X,Y)$, and we even allow for some \emph{missing not at random} settings. See Definition~\ref{def:missingness} in Section~\ref{sec:statistical_setting} for a precise statement of our conditions on the missingness mechanism. Then, in addition to those imposed by Definitions~\ref{def:missingness} and~\ref{def:effectivespace}, we place restrictions on the tail properties of the marginal distribution of the feature vectors (Definition~\ref{def:lowerdensity}), make use of smoothness assumptions on the non-zero $f_{\omega}$ (Definition~\ref{def:smoothness}) and employ a standard margin assumption (Definition~\ref{def:margin}). 

Theorem~\ref{thm:minmax_bounds} reveals some striking features of our framework. Most notably, we see that consistent (and rate optimal) classification is possible even in situations where we have no complete cases in our training dataset.  In contrast, in problems where there is in fact no missing data, we may obtain a faster rate than one could typically expect in a $d$-dimensional nonparametric classification problem.  Overall there is a delicate interplay between the number of available cases for each observation pattern $\omega$ and the difficulty of estimating the corresponding function $f_{\omega}$ in our decomposition in~\eqref{eq:ANOVAdecomp}.

Our next contribution is to propose a new classification algorithm for missing data, introduced formally in Algorithm~\ref{alg:nonadaptive} in Section~\ref{sec:algorithmUB}, which we call the \emph{Hard-thresholding Anova Missing data (HAM)} classifier.  Our algorithm first estimates the functions $f_\omega$ for each $\omega \in \{0,1\}^d$ for which data is available using a nearest neighbour type approach.  We then apply hard-thresholding, setting our final estimator of $f_{\omega}$ to be $0$ unless there is sufficient evidence that it is nonzero on the population level.  We prove in Theorem~\ref{thm:nonadaptiveUBnew} that our procedure attains the minimax rate in Theorem~\ref{thm:minmax_bounds}  up to polylogarithmic factors, subject to a mild sample size condition.  We demonstrate the excellent practical performance of our proposal in a numerical study using simulated and real data in Section~\ref{sec:numericalresults}.  The key ideas behind the proofs of the upper and lower bounds of our main results are outlined in Section~\ref{sec:proofs}.  Finally, additional theoretical and numerical results, as well as the full details of the proofs, are presented in the supplementary material.

Existing work on classification with missing values in the feature vectors includes \citet{josse2019consistency}, who show that in certain settings imputing the missing values with a constant can lead to consistent classification. \citet{cai2019high} consider missing data in the setting of high-dimensional linear discriminant analysis; they propose using the available cases to estimate the corresponding class means and variances under a missing completely at random assumption.  In regression, \citet{le2020neumiss} use implicit imputation methods in non-trivial missingness settings for prediction problems, \citet{le2021sa} study a wide class of impute-then-regress procedures, and \cite{ayme2022near} derive minimax rates for prediction using linear models with missing data. 

In the related problem of \emph{semi-supervised learning}, the feature vectors are fully observed but some of the class labels in the training data set are missing. This topic has received far more attention, see \citet{ahfock2022semi} and \citet{sportisse2023labels} for very recent contributions, and the earlier book \citet{chapelle2006semi}. Similarly, \emph{label noise}, which refers to scenarios where some of the observations in the training dataset have been attributed to the wrong class, has also been widely studied; see the review articles \citet{frenay2013classification,frenay2014comprehensive}, as well as the recent papers \citet{cannings2020classification} and \citet{lee2022binary}.  Finally, in \emph{transfer learning} \citep{weiss2016survey} the practitioner has access to a small labelled training data set from the target population of interest alongside a larger, noisy source dataset (which may be unlabelled). Here the main interest lies in exploiting the large source sample to improve inferences on the target population; see, for example, \citet{cai2021transfer} and \citet{reeve2021adaptive} for recent works in the context of nonparametric classification problems.

It is convenient to now fix some notation that we will use throughout the paper. For $m \in \mathbb{N}$, let $[m] = \{1,\ldots, m\}$. Let $\mathbf{0}_d := (0,\ldots,0)^T \in \{0,1\}^d$ and $\mathbf{1}_d := (1,\ldots,1)^T \in \{0,1\}^d$ be the $d$-dimensional vectors of zeros and ones, respectively. For $j\in[d]$, we write $e_j$ for the $j$th canonical Euclidean basis vector in $\mathbb{R}^d$, which contains a $1$ in the $j$th component, and $0$s elsewhere. For $\omega \in \{0,1\}^d$, we write $d_{\omega} := \|\omega\|_1$ for the number of $1$s in $\omega$, which will refer to as the \emph{dimension} of $\omega$.  For $\omega = (\omega_1, \ldots, \omega_d)^T, \omega' = (\omega'_1,\ldots, \omega'_d)^T \in \{0,1\}^d$, we write $\omega \preceq \omega'$ if $\omega_j\leq \omega'_j$ for all $j\in[d]$, $\omega \prec \omega'$ if $\omega_j \leq \omega'_j$ for all $j\in[d]$ and $\omega\neq \omega'$,  $\omega \wedge \omega' := (\omega_1 \wedge \omega'_1, \ldots, \omega_d \wedge \omega'_d)^T$, and $\omega \vee \omega' := (\omega_1 \vee \omega'_1, \ldots, \omega_d \vee \omega'_d)^T$.  For $x \in \mathbb{R}^d$ and $\omega \in \{0,1\}^d$, we write $x^\omega := x \odot \omega + \mathbf{0}_d \odot (\mathbf{1}_d-\omega) \in \mathbb{R}^d$, and $\odot$ denotes the entrywise product.  Further, let $\|\cdot\|$ denote the Euclidean norm and $B_r(x) := \{z \in \mathbb{R}^d : \|z-x\| < r\}$ be the open Euclidean ball of radius $r > 0$ at $x \in \mathbb{R}^d$.  For $s \in [d]$, we write $\mathcal{L}_s$ for the Lebesgue measure on $\mathbb{R}^s$. We write $\mathbb{S}^{d\times d}$ for the set of symmetric positive definite matrices. For random variables $X$, $Y$ and $Z$ we write $X\indep Y$ to denote the independence of $X$ and $Y$, and $X\indep Y\mid Z$ to denote that $X$ and $Y$ are conditionally independent given $Z$.  Finally, for $x\in \mathbb{R}$, we write $\log_{+}(x):= \log(x \vee e )$.

\section{Statistical setting and minimax results\label{sec:statistical_setting}}

Let $\mathcal{Q}$ denote the class of all distributions on $\mathbb{R}^d \times \{0,1\} \times \{0,1\}^d$. Then, for $Q \in \mathcal{Q}$, suppose that $(X,Y,O) \sim Q$, and recall that $X$ denotes the feature vector taking values in $\mathbb{R}^d$, $Y$ denotes its class label (either $0$ or $1$) and $O$ denotes an observation indicator in $\{0,1\}^d$.  We write $P\equiv P_Q$ for the marginal distribution of $(X,Y)$, write $\eta: \mathbb{R}^d \rightarrow [0,1]$ for the regression function given by $\eta(x) := \mathbb{P}_P(Y = 1 | X = x)$, and let $\mu$ denote the marginal distribution of $X$ on $\mathbb{R}^d$.   

For $n \in \mathbb{N}$ and $Q \in \mathcal{Q}$, let $(X_1, Y_1, O_1), \ldots, (X_n, Y_n, O_n)$ be independent and identically distributed triples from $Q$.  In our missing data problem, we only observe those entries of $X_i$ for which the corresponding entry of $O_i$ is $1$. In other words, our training data is $D_n := \bigl((X_1^{O_1}, Y_1, O_1), \ldots, (X_n^{O_n}, Y_n, O_n)\bigr)$ and we are presented with the task of assigning a fully observed\footnote{If the test point is not fully observed, then the problem essentially reduces to a lower dimensional classification problem. There will be distributions in our model class where data for the unobserved features in the test point provides no information on the distribution of the observed features.} test point $X$ to either class $0$ or class $1$.  Despite replacing the unobserved entries in $X_1^{O_1}, \ldots, X_n^{O_n}$ with zeros, we are still able to distinguish between true (observed) zeros and the zeros that have arisen due to missing values, since $O_1, \ldots, O_n$ are also given in the training data $D_n$.  Recall that, for a pattern $\omega \in \{0,1\}^d$, we will say that the $i$th observation is an available case for $\omega$ when $\omega \preceq O_i$.

A \emph{data-dependent classifier} $\hat{C}$ is a measurable function from $\mathbb{R}^d \times (\mathbb{R}^d \times \{0,1\} \times \{0,1\}^d)^n$ to $\{0,1\}$, and we let $\mathcal{C}_n$ denote the set of all such data-dependent classifiers. Throughout this paper the latter $n$ arguments of $\hat{C}(x,D_n)$ will always be $D_n$ and thus will often be omitted; we will simply write $\hat{C}(x)$.   The performance of $\hat{C} \in \mathcal{C}_n$ will be measured by its \emph{test error}
\begin{equation*}
\label{eq:testerror}
L_P(\hat{C}) := \mathbb{P}_P\bigl\{\hat{C}(X) \neq Y | D_n\bigr\} = \int_{\mathbb{R}^d \times \{0,1\}} \mathbbm{1}_{\{ \hat{C}(x) \neq y\}}  \, dP(x,y),
\end{equation*}
which is minimised by the Bayes classifier $C^{\mathrm{Bayes}}(x) =\mathbbm{1}_{\{\eta(x) \geq 1/2\}}$. Therefore our main results will concern the (nonnegative) excess test error
\begin{equation}
\label{eq:excesserror}
   \mathcal{E}_P(\hat{C}) := L_P(\hat{C})- L_P(C^{\mathrm{Bayes}}) = \int_{\mathbb{R}^d} \mathbbm{1}_{\{ \hat{C}(x) \neq C^{\mathrm{Bayes}}(x)\}} |2 \eta(x) - 1| \, d\mu(x). 
\end{equation}

We now introduce our new framework for classification problems with missing data. First, in our main results, we will condition on the values of the observation indicators $O_1, \ldots, O_n$ in our training data. We will suppose these belong to a subset $\mathcal{O} \subseteq \{0,1\}^d$ and one consequence of this is that we only need to consider distributions $Q$ that assign positive mass to each element of $\mathcal{O}$. More formally, we employ the following general assumption on the missingness mechanism. 
\begin{definition}\label{def:missingness}
For $\mathcal{O} \subseteq \{0,1\}^d$, let $\mathcal{Q}_{\mathrm{Miss}}(\mathcal{O}) \subseteq \mathcal{Q}$ denote the class of distributions for which $\mathbb{P}_{Q}(O = o) > 0$ for all $o \in \mathcal{O}$, and 
\begin{equation}
\label{eq:missingness}
\mathbb{P}_Q(Y = 1 | X^{\omega} = x^{\omega}) = \mathbb{P}_Q(Y = 1 | X^{\omega} = x^{\omega}, O = o) 
\end{equation}
for all $x \in \mathrm{supp}(\mu)$, $\omega \in \{0,1\}^d$ and $o \in \mathcal{O}$ with $\omega \preceq o$.
\end{definition}

To understand the generality afforded by Definition~\ref{def:missingness}, first note that \eqref{eq:missingness} holds trivially if $O$ is independent of $(X,Y)$ (i.e.~the data is \emph{missing completely at random}) and $Q$ assigns positive probability to $O = o$, for all $o \in \mathcal{O}$, without any additional assumption on the distribution of $O$ nor the distribution $P$ of $(X,Y)$. However, the class $\mathcal{Q}_{\mathrm{Miss}}(\mathcal{O})$ also includes distributions where $O$ and $(X,Y)$ are dependent, and may even allow the data to be \emph{missing not at random}; see the examples below (the proofs of the claims are presented in Section~\ref{sec:intuitionfor1})
.  Proposition~\ref{prop:missingness}
in Section~\ref{sec:intuitionfor1} 
provides further understanding of Definition~\ref{def:missingness} by relating it to conditional independence.

\begin{example}[$(1-p)$-homogeneous missingness]
\label{ex:missingness1}
Fix $p \in (0,1)$.  Suppose $O \indep (X,Y)$ and the $\mathbb{P}(O = \omega) = p^{d_{\omega}}(1-p)^{d-d_{\omega}}$, for $\omega \in \{0,1\}^d$, in other words each component of $X$ is observed (independently) with probability $p$. Then $Q \in \mathcal{Q}_{\mathrm{Miss}}(\{0,1\}^d)$.
\end{example}

\begin{example}[MAR]
\label{ex:missingness2}
Suppose that $\mathrm{supp}(\mu) = [0,1]^4$, and that $O$ is conditionally independent of $Y$ given $X^{\omega}$ for all $\omega \in \{0,1\}^4$. For $x=(x_1,\ldots,x_4)^T\in[0,1]^4$, let $\IP\{O=(1,1,1,0)^T  \mid X=x\} = x_1$ and $\IP\{O=(1,0,0,1)^T \mid X=x\} = 1-x_1$.  Then $Q\in\mathcal{Q}_{\mathrm{Miss}}(\{(1,1,1,0)^T,(1,0,0,1)^T\})$. 
\end{example}

\begin{example}[MNAR]\label{ex:missingness3} Let $X \sim U([0,1]^2)$, and for $x = (x_1,x_2)^T \in [0,1]^2$, let $\eta(x) = 1/4 + x_1/2 + 1/4\cdot\cos(4\pi x_2)$. Suppose also that for $x = (x_1,x_2)^T \in [0,1]^2$ and $r \in \{0,1\}$, we have  
\[
\mathbb{P}_Q(O = o \mid X = x, Y = r) = \left\{ \begin{array}{c c}
1/4+1/4\cdot\mathbbm{1}_{\{x_2\leq1/2\}} & \text{ if } o = (1,1)^T \\ 
1/6+1/12\cdot\mathbbm{1}_{\{x_2>1/2\}} & \text{ otherwise.}\\
\end{array} \right.
\]
Then $Q \in \mathcal{Q}_{\mathrm{Miss}}(\{0,1\}^2)$. 
\end{example}

The aim of these examples is to provide straightforward understanding of how Definition~\ref{def:missingness} relates to some widely used assumptions in the missing data literature.  First, the homogeneous missingness probabilities in Example~\ref{ex:missingness1} provide a simple model amenable to theoretical understanding. Though this is at times too restrictive in practical applications, the simplicity allows us to quantify how much the amount of missingness affects the performance in different settings (see, e.g., \citet{loh2011high}  and \citet{cai2019high}). In Example~\ref{ex:missingness2}, we will never observe complete cases in our training data set, for instance this situation may arise if different data sources have been merged. As we will see later, this may cause serious problems when trying to apply some of the ad-hoc approaches mentioned in the introduction.  Moreover, the data in this example is \emph{Missing At Random} (MAR), in the sense used in, e.g., \citet[Assumption~2]{josse2019consistency}.    Example~\ref{ex:missingness3} covers the challenging setting where data is \emph{Missing Not At Random} (MNAR); the probability that the second component of the $2$-dimensional feature vector $(X_1,X_2)^T$ is missing depends on whether or not~$X_2 \leq 1/2$. 

Returning to more general settings, Lemma~\ref{lem:missingness}
in the supplementary material shows that~\eqref{eq:missingness} holds whenever the distribution of the observation indicator $O$ depends only on a subset of the features and the regression function depends only on another, disjoint subset of the features and these two subsets of features are independent.  More precisely, for $S_1, S_2 \subseteq [d]$ with $S_1\cap S_2 = \emptyset$, suppose the regression function $\eta(x)$ depends on $x = (x_1,\ldots,x_d) \in \mathbb{R}^d$ only via $x_{S_1} := \{x_j : j \in S_1\}$ and the conditional distribution of the observation indicator $O | \{X = x, Y = y\}$ depends only on $(x,y) \in \mathbb{R}^d \times \{0,1\}$ via $x_{S_2} := \{x_j : j \in S_2\}$. If $X_{S_1}$ and $X_{S_2}$ are independent then~\eqref{eq:missingness} holds.   

As outlined in the introduction, a key aspect of our framework is that any regression function $\eta :\mathbb{R}^d \rightarrow [0,1]$ may be decomposed into a sum of orthogonal functions $f_{\omega} : \mathbb{R}^d \rightarrow \mathbb{R}$ for $\omega \in \{0,1\}^d$. More precisely, these functions may be defined recursively for $x \in \mathbb{R}^d$ as follows: 
\begin{align}
    f_{\mathbf{0}_d}(x) & := \mathbb{E}_P\{\eta(X)\} -\frac{1}{2} \equiv \mathbb{P}_P(Y = 1)  -\frac{1}{2} \label{eq:def_f_zero} 
\\ f_{\omega}(x) & := \mathbb{E}_P\Bigl\{\eta(X)  - \frac{1}{2} - \sum_{\omega'\prec\omega}f_{\omega'}(X) \Bigm| X^{\omega} = x^{\omega}\Bigr\}  \text{ for } \omega \in \{0,1\}^d \text{ with } d_\omega \geq 1. \label{eq:def_f_omega}
\end{align} 
This is known as an \emph{anova decomposition} \citep{efron1981jackknife}, and indeed Proposition~\ref{prop:eta_decomposition} 
in Section~\ref{sec:propertiesofeta} 
confirms that 
\begin{equation}\label{eq:etadecomposition}
\eta(\cdot) = \frac{1}{2} + \sum_{\omega \in \{0,1\}^d} f_{\omega}(\cdot).
\end{equation}
We will seek to exploit this decomposition to estimate $\eta$, by estimating each of the functions $(f_\omega)_{\omega\in \{0,1\}^d}$ in turn. The main idea is that $f_{\omega}(x)$ only depends on $x$ via $x^{\omega}$, the entries of $x$ for which the corresponding entries of $\omega$ are $1$, and thus we may make use of all the available cases for $\omega$ when estimating $f_{\omega}$. Furthermore, if $f_\omega \equiv 0$ for some $\omega$, then we can hope that detecting this will be relatively easy. Indeed, our next definition specifies which of the $f_\omega$ functions may contribute to the regression decomposition and which are identically zero. 

Before stating Definition~\ref{def:effectivespace}, it is convenient to introduce some additional notation. For $Q\in\mathcal{Q}_{\mathrm{Miss}}(\mathcal{O})$ and $\omega \in \{0,1\}^d$, we write $\mu_{\omega}$ to denote the distribution of $X^{\omega}$ when $(X,Y,O) \sim Q$, i.e.~the distribution on $\mathbb{R}^{d}$ given by $\mu_{\omega}(A) = \mu\bigl(\{x \in \mathbb{R}^d : x^{\omega} \in A\} \bigr)$ for measurable $A \subseteq \mathbb{R}^d$.  Moreover, for $o \in \mathcal{O}$ and $\omega \preceq o$, we write $\mu_{\omega \mid o}$ to denote the distribution of $X^{\omega} \mid \{O = o\}$ when $(X,Y,O) \sim Q$, i.e.~the distribution on $\mathbb{R}^{d}$ given by $\mu_{\omega|o}(A) = \mathbb{P}_{Q}(X^\omega \in A \mid O = o)$, for measurable $A \subseteq \mathbb{R}^d$. 
For $Q \in \mathcal{Q}_{\mathrm{Miss}}(\mathcal{O})$ and $\omega \in \{0,1\}^d$, define  
\begin{equation}
\label{def:sigmaomega}
\sigma_{\omega}^2  := \min_{o \in \mathcal{O} : o \succeq \omega}  \mathbb{E}_{Q}\bigl\{f_\omega^2(X) \bigm| O= o \bigr\}.
\end{equation}
To make a precise statement of our main condition, which asks that $f_{\omega}$ is zero for certain values of $\omega$, we will also make use of  some basic notation concerning orderings of subsets of $\{0,1\}^d$.  Let $\mathcal{I}(\{0,1\}^d) := \{\Omega \subset \{0,1\}^d :\text{if } \omega,\omega' \in \Omega \text{ and } \omega \preceq \omega', \text{ then } \omega = \omega'\}$ be the set of \emph{antichains} in $\{0,1\}^d$, i.e.~the set of subsets $\Omega \subseteq \{0,1\}^d$ for which any two distinct elements of $\Omega$ are \emph{incomparable}.  Further, for $\Omega \in \mathcal{I}(\{0,1\}^d)$, let $L(\Omega) := \{\omega \in \{0,1\}^d : \text{ there exists } \omega' \in \Omega \text{ such that } \omega \prec \omega'\}$ denote the elements of $\{0,1\}^d$ that precede the elements of $\Omega$, and write $U(\Omega) := \{0,1\}^d \setminus \{\Omega \cup L(\Omega)\}$.

\begin{figure}[ht]
\centering
\includegraphics[width=0.25\linewidth]{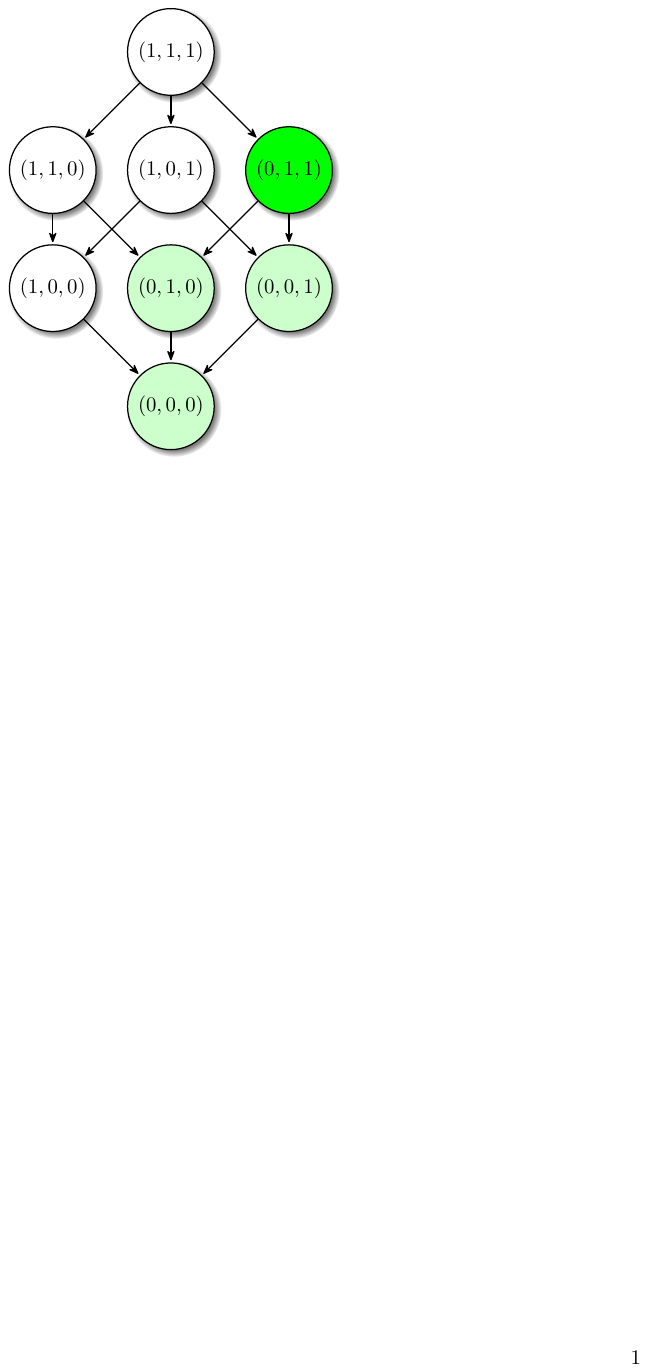}
\caption{A Hasse diagram representing the role of $\Omega_{\star}$ used in Definition~\ref{def:effectivespace}. In this example, $d = 3$ and $\Omega_{\star} = \{(0,1,1)^T\}$ (in dark green), we present $L(\Omega_{\star})=\{(0,1,0)^T,(0,0,1)^T, \mathbf{0}_d\}$ in light green and $U(\Omega_{\star})$ in white.}
\end{figure}

\begin{definition}
\label{def:effectivespace}
For $\Omega_{\star} \in \mathcal{I}(\{0,1\}^d)$,  $c_{\mathrm{E}}\in[0,1/4]$ and $\mathcal{O} \subseteq \{0,1\}^d$, let
\[
\mathcal{Q}_{\mathrm{E}}(\Omega_{\star},c_{\mathrm{E}},\mathcal{O}) := \Bigl\{Q \in \mathcal{Q}_{\mathrm{Miss}}(\mathcal{O}) : \sigma_\omega^2 \geq c_{\mathrm{E}}  \ \ \text{for all} \ \ \omega \in \Omega_{\star},\ f_{\omega} \equiv 0  \ \ \text{for all} \ \ \omega \in U(\Omega_{\star})\Bigr\}.
\]
\end{definition}
Definition~\ref{def:effectivespace} is very flexible. For example, if $\Omega_{\star} = \bigl\{\omega \in \{0,1\}^d : d_\omega = 1\bigr\}$, then it asks that $\eta$ is an \emph{additive} function of the components of $x$ (see, for example, \citet{hastie1986generalized}). As another example, suppose $S \subseteq [d]$ and $\Omega_{\star} = \{(\omega_1, \ldots, \omega_d)^T \in \{0,1\}^d : \omega_{j} = 0 \ \text{for}\  j \in S^c\}$, then $\eta$ only depends on the components of $x$ that are contained in $S$. Other, trivial, cases include $\Omega_{\star} = \{\mathbf{1}_d\}$, which (when $c_{\mathrm{E}} = 0$) corresponds to no assumption,  $\Omega_{\star} = \{\mathbf{0}_d\}$ asks that $\eta$ is constant,  and $\Omega_{\star} = \emptyset$ forces $\eta \equiv 1/2$.  Further examples are given in Section~\ref{sec:numericalresults}, based on specific instances of Examples~\ref{ex:missingness1} to~\ref{ex:missingness3}.  

The condition that $\sigma^2_\omega \geq c_{\mathrm{E}}$ for $\omega \in \Omega_{\star}$ is redundant if $c_{\mathrm{E}}$ is zero, and we may indeed take $c_{\mathrm{E}} = 0$ in  Theorem~\ref{thm:minmax_bounds} below.  The main role of the constant $c_{\mathrm{E}}$ is in Theorem~\ref{thm:nonadaptiveUBnew}, where it is assumed to be strictly positive.  The intuition is that this provides a natural measure of the signal strength, and that, with enough data, we are able to detect which patterns $\omega$ belong to $\Omega_{\star}$.  The minimum in the definition of $\sigma_{\omega}^2$ accounts for the fact that we will ultimately use the available cases in our training data to estimate the signal strength from each $\omega$, and that our results will be presented conditional on $O_1, \ldots, O_n$.

Our next condition concerns the regularity of the marginal feature distributions. For a probability distribution $\nu$ on $\mathbb{R}^d$ and $s \in [0,d]$ let $\rho_{\nu, s} : \mathbb{R}^d \rightarrow [0,1]$ denote the \emph{lower density} \citep{reeve2021adaptive} given by 
\begin{equation}\label{def:lower_density}
\rho_{\nu, s}(x) := \inf_{r \in (0,1)} \frac{\nu\bigl(B_{r}(x)\bigr)}{r^{s}}.
\end{equation}
The term `lower density' is used to highlight the relationship between $\rho_{\nu, d}$ and the density (if it exists) of $\nu$ with respect to Lebesgue measure on $\mathbb{R}^d$. Note, however, that our restriction on the marginal feature distributions in Definition~\ref{def:lowerdensity} below does not require that these marginal distributions have a density, nor that the lower density is uniformly bounded below by a positive constant.

\begin{definition}
\label{def:lowerdensity}
Fix $\boldsymbol{\gamma} = (\gamma_1, \ldots, \gamma_d)^T \in [0, \infty)^{d}$,  $C_{\mathrm{L}} \geq 1$ and $\mathcal{O} \subseteq \{0,1\}^d$. For $\omega = (\omega_1, \ldots, \omega_d)^T \in \{0,1\}^d \setminus \{\mathbf{0}_d\}$, let $\gamma_{\omega} := \min\{\gamma_j : \omega_j = 1\}$. Let $\mathcal{Q}_{\mathrm{L}}(\boldsymbol{\gamma}, C_{\mathrm{L}},\mathcal{O})$ denote the subclass of $\mathcal{Q}_{\mathrm{Miss}}(\mathcal{O})$ for which
\begin{equation}\label{eq:lowerdensity}
\mu_\omega\Bigl( \Bigl\{x \in \IR^d : \min_{o \in \mathcal{O} : \omega \preceq o} \rho_{\mu_{\omega \mid o}, d_{\omega}} (x) < \xi \Bigr\} \Bigr)\leq C_{\mathrm{L}}\cdot \xi^{\gamma_\omega}
\end{equation}
for all $\xi>0$ and all $\omega \in \{0,1\}^d \setminus \{\mathbf{0}_d\}$.
\end{definition}

To put Definition~\ref{def:lowerdensity} into context, first suppose that $O$ is independent of $(X,Y)$, so that $\mu_{\omega \mid o} = \mu_{\omega}$ for all $o \in \mathcal{O}$ and thus $\min_{o\in \mathcal{O} : \omega \preceq o} \rho_{\mu_{\omega \mid o}, d_{\omega}}(\cdot) =  \rho_{\mu_{\omega}, d_{\omega}}(\cdot)$. Then the parameters $\gamma_j$ for $j\in [d]$ control the tail properties of the univariate marginal distributions of $\mu$, where in particular a smaller value of $\gamma_j$ allows the tails to be heavier. Moreover, the definition of $\gamma_{\omega}$  ensures that $\gamma_{\omega} \geq \gamma_{\omega'}$ whenever $\omega \preceq \omega'$. Thus for $Q \in \mathcal{Q}_{\mathrm{L}}(\gamma,C_{\mathrm{L}}, \mathcal{O})$, the tails of the corresponding marginal distribution $\mu_{\omega'}$ are heavier compared with the $\mu_{\omega}$ one.  This is a natural restriction, concatenating a heavy tailed variable with a light tailed one shouldn't result in a light tailed joint distribution. Indeed, \citet[Proposition S8 and Lemma S9]{reeve2021adaptive} show that this property holds for product measures; in fact that paper provides many additional appealing aspects of lower densities. As an illustration, suppose that $(X,Y,O) \sim Q\in\mathcal{Q}_{\mathrm{Miss}}(\mathcal{O})$, where $X \sim U([0,1]) \otimes N(0,1)$, and $O$ is independent of $(X,Y)$. Then we may take any $\gamma_1 \in [0,\infty)$ and $\gamma_2\in[0,1]$ for the univariate marginals (see Section~\ref{sec:intuitionfor1}
). Then, letting $\boldsymbol{\gamma}=(\gamma_1,\gamma_2)$, \citet[Proposition S8]{reeve2021adaptive} confirms that $Q\in\mathcal{Q}_{\mathrm{L}}(\boldsymbol{\gamma}, C_{\mathrm{L}},\mathcal{O})$, for sufficiently large $C_{\mathrm{L}}$  (depending on $\boldsymbol{\gamma}$). 

In settings where the missingness $O$ depends on $(X, Y)$, the minimum over $o \in \mathcal{O}$ in~\eqref{eq:lowerdensity} provides control of the measure of the distributions of $X^{\omega} | \{O = o\}$ from which we observe training data. Note that the measure on the outside in~\eqref{eq:lowerdensity} is $\mu_{\omega}$, the marginal distribution of $X^{\omega}$ of our fully observed test point $X$.  Intuitively, \eqref{eq:lowerdensity} controls the relationship between the distributions of $X^{\omega}$ and $X^{\omega} | \{O=o\}$ for $o \in \mathcal{O}$.  The case $\omega=\mathbf{0}_d$ is excluded from Definition~\ref{def:lowerdensity}, since each one of the measures $\mu_{\mathbf{0}_{d}}$ and $\mu_{\mathbf{0}_{d}|o}$ for $o\in \{0,1\}^d$ are the point mass at zero, and thus \eqref{eq:lowerdensity} holds with any $\gamma_{\mathbf{0}_d} \geq 0$ and $C_{\mathrm{L}} = 1$. 

In Section~\ref{sec:intuitionfor1}
, we study specific instances of our examples above and in particular characterise which tail parameters we can choose in the class $\mathcal{Q}_{\mathrm{L}}(\boldsymbol{\gamma},C_{\mathrm{L}},\mathcal{O})$. In general, for compactly supported variables with lower densities bounded away from zero the tail parameter can be taken arbitrarily large, and for variables supported on all of $\mathbb{R}$ with infinitely many moments the tail parameter may be taken to be any value smaller than one. 

The remaining conditions on our class of distributions concern the smoothness of the functions $f_{\omega}$, which contribute to the decomposition in~\eqref{eq:etadecomposition}, and a standard margin assumption on the regression function $\eta$.  These two conditions only concern the marginal distribution of the pair $(X,Y)$, and are therefore presented as subsets of the class $\mathcal{P}$ of all distributions on $\mathbb{R}^d \times\{0,1\}$.  

\begin{definition}
\label{def:smoothness}
Fix $\boldsymbol{\beta} = (\beta_1, \ldots, \beta_d)^T \in (0,1]^{d}$ and $C_{\mathrm{S}} \geq 1$. For $\omega = (\omega_1, \ldots, \omega_d)^T \in \{0,1\}^d \setminus \{\mathbf{0}_d\}$, let $\beta_{\omega} := \min\{\beta_j : \omega_j = 1\}$.  Let $\mathcal{P}_{\mathrm{S}}(\boldsymbol{\beta}, C_{\mathrm{S}})$ denote the class of distributions for which 
\begin{align}\label{def:smoothness_eq}
|f_\omega(x_1)-f_\omega(x_2)|\leq C_{\mathrm{S}} \cdot \|x_1^{\omega}-x_2^{\omega}\|^{\beta_\omega}_2
\end{align}
for all $x_1, x_2 \in \mathrm{supp}(\mu)$ and all $\omega \in \{0,1\}^d\setminus \{\mathbf{0}_d\}$. 
\end{definition}

The form of $\beta_\omega$ in Definition~\ref{def:smoothness} imposes a mild natural ordering on the smoothness, namely that $f_{\omega'}$ is at least as smooth as $f_{\omega}$ when $\omega' \prec \omega$. This only rules out situations where higher-order interactions in the anova decomposition of $\eta$ in \eqref{eq:etadecomposition} are allowed to be smoother than the corresponding main univariate contributions.

In our main results below, we will consider distributions $Q$ in the intersection of $\mathcal{Q}_{\mathrm{E}}(\Omega_{\star}, c_{\mathrm{E}}, \mathcal{O})$ from Definition~\ref{def:effectivespace} and for which the corresponding $P$ belongs to $\mathcal{P}_{\mathrm{S}}(\boldsymbol{\beta}, C_{\mathrm{S}})$, thus~\eqref{def:smoothness_eq} is in fact only imposing an additional constraint for $\omega \in \Omega_{\star}\cup L(\Omega_\star)$.  A straightforward calculation shows that if $Q\in\mathcal{Q}_{\mathrm{E}}(\Omega_{\star},c_{\mathrm{E}},\mathcal{O})$ and $P\equiv P_Q \in \mathcal{P}_{\mathrm{S}}(\boldsymbol{\beta}, C_{\mathrm{S}})$, then the corresponding regression function satisfies
\[
|\eta(x_1)-\eta(x_2)|\leq 2^d \cdot C_{\mathrm{S}} \cdot \|x_1-x_2\|^{\min_{\omega\in\Omega_\star}\{\beta_\omega\}}_2
\]
for all $x_1, x_2 \in \mathrm{supp}(\mu)$.

Our final definition is the standard margin condition \citep[e.g.][]{polonik1995measuring,mammen1999smooth}. 
\begin{definition}
\label{def:margin}
Fix $\alpha\in[0,\infty)$ and $C_{\mathrm{M}} \geq 1$. Let $\mathcal P_\mathrm{M}(\alpha, C_\mathrm{M})$ denote the class of distributions for which
\begin{align*}
    \mu\Bigl(\bigl\{x \in \IR^d : \bigl|\eta(x)-1/2\bigr| < t \bigr\} \Bigr) \leq C_{\mathrm{M}} \cdot t^{\alpha}
\end{align*}
for all $t>0$.
\end{definition}
  
In Theorem~\ref{thm:minmax_bounds} below we provide our main minimax result, in which the class of distributions of interest is the intersection of the classes given in Definitions~\ref{def:missingness} to~\ref{def:margin}. More precisely, for $\Omega_{\star} \in \mathcal{I}(\{0,1\}^d)\setminus\{ \{\mathbf{0}_d\}, \emptyset \}$, $c_{\mathrm{E}} \in [0,1/4]$, $\boldsymbol{\gamma} \in [0, \infty)^d$, $C_{\mathrm{L}}>1$, $\boldsymbol{\beta} \in (0,1]^{d}$, $C_{\mathrm{S}} \geq 1$, $\alpha \in [0, \infty)$,  $C_{\mathrm{M}} \geq 1$ and $\mathcal{O} \subseteq \{0,1\}^d$, we write 
\begin{align}
\mathcal{Q}'_{\mathrm{Miss}} &\phantom{:}\equiv \mathcal{Q}'_{\mathrm{Miss}}( \Omega_{\star}, c_{\mathrm{E}}, \boldsymbol{\gamma}, C_{\mathrm{L}}, \boldsymbol{\beta}, C_{\mathrm{S}} , \alpha, C_{\mathrm{M}},\mathcal{O}) \label{def:Q_miss'}
\\ & := \mathcal{Q}_{\mathrm{E}}(\Omega_{\star},c_{\mathrm{E}},\mathcal{O}) \cap \mathcal{Q}_{\mathrm{L}}(\boldsymbol{\gamma},C_{\mathrm{L}}, \mathcal{O}) \cap \{Q \in \mathcal{Q} : P \equiv P_Q \in \mathcal P_{\mathrm{S}}(\boldsymbol{\beta},C_{\mathrm{S}})\cap\mathcal P_{\mathrm{M}}(\alpha,C_{\mathrm{M}})\}.\nonumber
\end{align}
As mentioned above, in the theorem we consider the excess risk conditional on the values of the observation indicators $O_1, \ldots, O_n$. To that end, for fixed $o_1, \ldots, o_n \in \mathcal{O}$, let $n_\omega  := \sum_{i=1}^n \mathbbm{1}_{\{\omega \preceq o_i\}}$ denote the number of available cases for the observation pattern $\omega \in \{0,1\}^d$, and define $\mathcal{N} := \{\omega \in \{0,1\}^d : n_{\omega} > 0\}$ to be the set of observation patterns for which we have training data. 

\begin{theorem}\label{thm:minmax_bounds}
Fix $d, n \in\mathbb{N}$, $\mathcal{O} \subseteq \{0,1\}^d$, $o_1, \ldots, o_n \in \mathcal{O}$, $\Omega_{\star} \in \mathcal{I}(\{0,1\}^d)\setminus\{ \{\mathbf{0}_d\} ,\emptyset\}$, $c_{\mathrm{E}} \in [0,1/4]$, $\boldsymbol{\gamma} \in [0, \infty)^d$, $C_{\mathrm{L}}>1$, $\boldsymbol{\beta}\in(0,1]^d$, $C_{\mathrm{S}} \geq 1$, $\alpha \in [0, \infty)$ and $C_{\mathrm{M}} \geq 1$.  Suppose that $\min_{\omega \in \Omega_{\star}} \gamma_{\omega} > \max_{\omega\in\Omega_{\star}} \beta_\omega/d_\omega$, $\max_{\omega\in\Omega_\star}\alpha\beta_\omega \leq \min_{\omega \in \Omega_{\star}} d_\omega$, $c_{\mathrm{E}}\leq (16|\Omega_\star|)^{-1}$, 
$C_{\mathrm{L}}\geq   (8d)^{(1+d) (1+\|\boldsymbol{\gamma}\|_{\infty})}$ and $C_{\mathrm{M}} \geq \max_{\omega\in\Omega_\star}\bigl(1 + 6\cdot 4^{d/\beta_\omega} \bigr)$.  
Let
\begin{equation}
\label{eq:minimaxrate}
R := \max_{\omega \in \Omega_{\star} \cap \mathcal{N}} \Bigl\{ n_\omega^{-\frac{\beta_\omega\gamma_\omega(1+\alpha)}{\gamma_\omega(2\beta_\omega+d_\omega)+\alpha\beta_\omega}} \Bigr
\}. 
\end{equation}
Then, there exist constants $0 < c < C$ (neither of which depend on $n$ nor $o_1, \ldots, o_n$) such that
\begin{align}
    c\cdot \Bigl(R
+ \mathbbm{1}_{\{\Omega_\star \cap \mathcal{N}^c \neq \emptyset\}} \Bigr) &\leq \inf_{\hat C\in\mathcal C_n} \sup_{Q \in \mathcal{Q}_{\mathrm{Miss}}'} \mathbb{E}_{Q}\Bigl\{\mathcal{E}_{P_Q}(\hat C) \Bigm| O_1 = o_1, \ldots, O_n = o_n \Bigr\}  \nonumber
\\ & \leq C \cdot \log_+^{\frac{1+\alpha}{2}}\Bigl( \min_{\omega \in \Omega_{\star}\cap \mathcal{N}} n_{\omega} \Bigr) \cdot R + \mathbbm{1}_{\{\Omega_\star \cap \mathcal{N}^c \neq \emptyset\}}. \label{eq:minmax_bound}
\end{align}
\end{theorem}

The proof of Theorem~\ref{thm:minmax_bounds} is given in Section~\ref{sec:proofs}. The upper bound in~\eqref{eq:minmax_bound} is attained by an \emph{oracle} version of our algorithm presented in Section~\ref{sec:algorithmUB}, which is allowed access to the set $\Omega_{\star}$. Note, however, that our HAM classifier in Algorithm~\ref{alg:nonadaptive}, which doesn't have access to $\Omega_{\star}$ also attains the same upper bound (up to poly-logarithmic factors), subject to a sample size condition; see Theorem~\ref{thm:nonadaptiveUBnew} for a precise statement. 

Theorem~\ref{thm:minmax_bounds} establishes the minimax rate of convergence (up to polylogarithmic factors) under our framework. In settings where $\Omega_{\star} \subseteq \mathcal{N}$, the rate is given by the quantity $R$ in~\eqref{eq:minimaxrate}, which elucidates the delicate interplay between the parameters that index our class $\mathcal{Q}_{\mathrm{Miss}}'$, as well as the number of available cases for each of the observation patterns in $\Omega_{\star}$. To provide some further intuition, the quantity $n_\omega^{-\frac{\beta_\omega\gamma_\omega(1+\alpha)}{\gamma_\omega(2\beta_\omega+d_\omega)+\alpha\beta_\omega}}$
reflects the contribution to the rate arising from the difficulty of estimating a non-zero $\omega$-way interaction term $f_{\omega}$ in the anova decomposition in \eqref{eq:ANOVAdecomp}.  Crucially, only the $n_{\omega}$ available cases in the training data set may be used to estimate this interaction term. Then the rate $R$ we obtain is given by the worst case over those $\omega$ for which the corresponding $f_{\omega}$ is non-zero as long as we have data available with which to estimate these, i.e. $\Omega_{\star} \subseteq \mathcal{N}$.  Given two observation patterns $\omega\succ\omega'\in\{0,1\}^d$, we have $n_\omega \leq n_{\omega'}$, $\beta_{\omega}\leq \beta_{\omega'}$, $\gamma_{\omega}\leq \gamma_{\omega'}$, and $d_{\omega} > d_{\omega'}$. Therefore we have
\[ 
n_\omega^{-\frac{\beta_\omega\gamma_\omega(1+\alpha)}{\gamma_\omega(2\beta_\omega+d_\omega)+\alpha\beta_\omega}}
>
n_{\omega'}^{-\frac{\beta_{\omega'}\gamma_{\omega'}(1+\alpha)}{\gamma_{\omega'}(2\beta_{\omega'}+d_{\omega'})+\alpha\beta_{\omega'}}}.
\]
The consequence of this is that if $f_{\omega}$ is non-zero, then estimating $f_{\omega'}$ will be easier than estimating $f_{\omega}$. This is why the maximum in $R$ is taken only over the observation patterns in $\Omega_\star$.  When $|\Omega_{\star}| \geq 2$, the maximum in $R$ may be attained by a pattern $\omega$ with large values of $\gamma_{\omega}$ and $\beta_{\omega}$, and low dimension $d_{\omega}$, but for which we only have few available cases. Alternatively, it may be that the rate is driven by a pattern $\omega$ for which a large quantity of data is available, but the corresponding $f_{\omega}$ is difficult to estimate (i.e.~with unfavourable $\gamma_{\omega}$, $\beta_{\omega}$, and $d_{\omega}$).  We provide explicit examples of the rate $R$ that we encounter in three different settings in Section~\ref{sec:numericalresults}.   

The indicator function on the left-hand side in~\eqref{eq:minmax_bound} shows that if there are no available cases for one of the patterns in $\Omega_{\star}$, then the worst case excess risk for any algorithm is bounded below by a constant $c >0$. The intuition is that, for the distributions in our class where the $\omega$-way interaction term $f_{\omega}$ is non-zero, this interaction may only be accurately estimated using observations that are available cases for $\omega$. The corresponding upper bound in this case follows from the fact that the excess risk is always bounded by $1$.    

Multiple strengths of our framework are highlighted by Theorem~\ref{thm:minmax_bounds}. We see that if $\mathbf{1}_d \notin \Omega_{\star}$, then consistent (and rate optimal) classification is still possible even when we have no complete cases in our training data set.  Moreover, when there is actually no missing data (so $n_{\omega} = n$ for every $\omega$), the rate under our framework is $n^{-\min_{\omega \in \Omega_{\star}}\frac{\beta_\omega\gamma_{\omega}(1+\alpha)}{\gamma_{\omega}(2\beta_\omega+d_{\omega})+\alpha\beta_\omega}}$. Comparing this rate to what one may expect when $\Omega_{\star} = \{\mathbf{1}_{d}\}$ and $c_{\mathrm{E}} = 0$ (thus placing no restriction in Definition~\ref{def:effectivespace}), we see that we obtain a faster rate of convergence than in the corresponding $d$-dimensional classification problem under just tail, smoothness and margin conditions, since $\min_{\omega \in \Omega_{\star}}\frac{\beta_\omega\gamma_{\omega}(1+\alpha)}{\gamma_{\omega}(2\beta_\omega+d_{\omega})+\alpha\beta_\omega} \geq \frac{\beta_{\min}\gamma_{\min}(1+\alpha)}{\gamma_{\min}(2\beta_{\min}+d)+\alpha\beta_{\min}}$, where $\beta_{\min} := \beta_{\mathbf{1}_d}$.  Here the inequality is strict if $\mathbf{1}_d \notin \Omega_{\star}$.  Finally, in situations where all of the variables are important but there are limited interactions between the covariates, for instance in an additive model with $\Omega_{\star} = \{\omega \in \{0,1\}^d : d_\omega = 1\}$, our rate $R$ does not depend on the ambient dimension $d$.  

In the theorem, asking that $\Omega_{\star}\not\in \{\{\mathbf{0}_d\},\emptyset\}$ is only ruling out trivial cases. When $\Omega_{\star} = \{\mathbf{0}_d\} $, the regression functions are constant and the classification problem reduces to a binomial testing problem, where based on just the labels $Y_1, \ldots, Y_n \sim \mathrm{Bernoulli}(1/2 + f_{\mathbf{0}_d})$ we would like to test whether $f_{\mathbf{0}_d}$ is greater or less than $0$.  If $\Omega_\star=\emptyset$, the regression function is the constant $1/2$, and thus we can obtain an excess risk of $0$ with a random guess (in fact, the excess risk is $0$ for any classifier). The restrictions on the relationships between the parameters are also mild. The condition that $\min_{\omega \in \Omega_{\star}} \gamma_{\omega} > \max_{\omega\in\Omega_{\star}} \beta_{\omega}/d_{\omega}$ (which is only needed for the upper bound) rules out distributions where the tails of the marginal distributions are so heavy that there is very little training data in regions where there is a reasonable chance that we may observe our test point.  The requirement that $\max_{\omega\in\Omega_{\star}}\alpha\beta_{\omega} \leq \min_{\omega \in \Omega_{\star}} d_{\omega}$  (which is only needed for the lower bound), is related to the condition that $\alpha \beta \leq d$ used in the standard $d$-dimensional nonparametric classification problem. This latter condition rules out \emph{super-fast rates} \citep{audibert2007fast}, and it holds under mild conditions (see \citet[Lemma~S15]{reeve2021adaptive}). The additional constraints on the constants $c_{\mathrm{E}}$, $C_{\mathrm{L}}$ and $C_{\mathrm{M}}$ are also only used in the proof of the lower bound; these restrictions may be relaxed slightly (indeed the bounds used in the proofs in Section~\ref{sec:lower_bound} are weaker), but for brevity we present the simpler bounds in the statement of the theorem. 

Finally, note that throughout this section, and in fact the whole paper, we suppose that our test point follows the marginal distribution $P \equiv P_Q$ of $(X,Y)$ when $(X,Y,O) \sim Q$. One may be interested instead in settings where the test point is arising from the conditional distribution $(X,Y) | \{O = \mathbf{1}_d\}$, say; in this case our theory may still be applied with minor changes, e.g., $\mu$ in \eqref{eq:excesserror} (and other places) should be replaced with $\mu_{\mathbf{1}_d | \mathbf{1}_d}$.

\section{The Hard-thresholding Anova Missing data classifier \label{sec:algorithmUB}}

In this section we introduce our new algorithm for classification with missing data.  As mentioned in the previous section, we do not need to know in advance which of the interaction terms in~\eqref{eq:etadecomposition} are nonzero. The full procedure, which we call the Hard-thresholding Anova Missing data (HAM) classifier, is given in Algorithm~\ref{alg:nonadaptive}.   The main result in this section (Theorem~\ref{thm:nonadaptiveUBnew}) establishes that the HAM classifier achieves the upper bound in Theorem~\ref{thm:minmax_bounds}, subject to a sample size condition, despite the fact that our algorithm does not have access to the set $\Omega_{\star}$ from Definition~\ref{def:effectivespace}.

Before presenting our theoretical results, we describe the key ideas underpinning our approach.  The first main step begins by estimating $\mathbb{P}(Y = 1)$ via the proportion of labels in the training data that belong to class $1$, which provides an estimate of $f_{\mathbf{0}_d}$.  The algorithm then proceeds to estimate each of the functions $f_\omega$ for $\omega \in \mathcal{N}$, in order of increasing dimension $d_{\omega}$.  More precisely, for each $\omega \in \mathcal{N}$, the estimator $\hat{f}_{\omega}$ of $f_{\omega}$ is based on a nearest neighbour type algorithm, which uses only the available cases for $\omega$, and where the distance metric is based on the covariates for which the corresponding entries of $\omega$ are equal to $1$. 

\begin{algorithm}[ht]
\caption{The Hard-thresholding Anova Missing data classifier $\hat{C}_{\mathrm{HAM}}$.}\label{alg:nonadaptive}
\begin{algorithmic}[1]
\Input:~Data $D_n= ((X_1^{o_1}, Y_1, o_1), \ldots, (X_n^{o_n}, Y_n,o_n)) \in (\mathbb{R}^d \times \{0,1\} \times \{0,1\}^d)^n$, $\boldsymbol{\gamma} \in [0,\infty)^d$, $\boldsymbol{\beta} \in (0,1]^d$, $\alpha \in [0, \infty)$, and a test point $x_0 \in \mathbb{R}^d$
\EndInput 
\State $\hat{f}_{\mathbf{0}}(\cdot) := \frac{1}{n}\sum_{i=1}^n Y_i -\frac{1}{2}$
\For{$d' = 1, \dots, d$} 
\For{$\omega \in \mathcal{N} $ such that $d_\omega=d'$}
\State $k_{\omega} := 1 + \lfloor n_{\omega}^{\frac{2\beta_{\omega}\gamma_{\omega}}{\gamma_{\omega}(2\beta_{\omega}+d_{\omega}) + \alpha \beta_{\omega}}}\rfloor$;   $\tau_\omega:= 2^{-4} \cdot n_{\omega}^{-\frac{\beta_{\omega}\gamma_{\omega}}{2[\gamma_{\omega}(2\beta_{\omega}+d_{\omega}) + \alpha \beta_{\omega}]}}$;
$N_{\omega} := \{i \in [n] : \omega \preceq o_i\}$
\For{$x \in \{X_i^{o_i} : i \in N_{\omega}\} \cup \{x_0\}$}
\State Let $(X_{(1)_{\omega}}^{\omega}(x),Y_{(1)_{\omega}}(x)),\ldots, (X_{(n_{\omega})}^{\omega}(x),Y_{(n_{\omega})_{\omega}}(x))$ be a reordering of the pairs $ \{(X_i^{o_i}, Y_i) : i \in N_{\omega}\}$ such that $\|X_{(1)_{\omega}}^{\omega}(x) - x^{\omega}\|\leq \cdots \leq \|X_{(n_{\omega})_{\omega}}^{\omega}(x) - x^{\omega}\|$
\State $\hat f_\omega(x): = \frac{1}{k_{\omega}}\sum_{j=1}^{k_{\omega}} Y_{(j)_\omega}(x) -\frac{1}{2} - \sum_{\omega' \prec \omega} \hat{f}_{\omega'}(x)$
\EndFor
\State $\hat\sigma^2_\omega := \frac{1}{n_\omega} \sum_{i \in N_{\omega}} \hat {f}^2_\omega(X_i^{\omega})$
\EndFor
\EndFor 
\State $\hat{\Omega} = \emptyset$
\For{$d' = d, d-1, \ldots,1$}
\For{$\omega \in \mathcal{N}$ with $d_\omega=d'$, $U(\{\omega\}) \cap \hat{\Omega} = \emptyset$}
\If{$\hat\sigma^2_\omega \geq \tau_{\omega}$} {$\hat{\Omega} = \hat{\Omega} \cup \{\omega\}$}
\EndIf
\EndFor
\EndFor
\State $\hat{\eta}(x_0):=\frac{1}{2} + \sum_{\omega\in\hat\Omega \cup L(\hat{\Omega})}\hat f_{\omega}(x_0)$
\Output:~$\hat{C}_{\mathrm{HAM}}(x_0) := \mathbbm{1}_{\{\hat{\eta}(x_0) \geq 1/2\}}$  
\EndOutput
\end{algorithmic}
\end{algorithm}

The second main step of the algorithm aims to determine which of the functions $f_{\omega}$ should be set to zero.  We first calculate a proxy for $\sigma^2_{\omega}$ given in~\eqref{def:sigmaomega} by taking the average after evaluating the estimator $\hat{f}_{\omega}^2$ at each of the available cases for $\omega$ in the training dataset.  We then seek to estimate $\Omega_{\star}$.  Starting with $\omega = \mathbf{1}_d$, and now working in order of decreasing dimension $d_{\omega}$, we do not include a pattern $\omega$ in our estimated set $\hat{\Omega}$ whenever $\hat{\sigma}^2_{\omega}$ is sufficiently small.  If, on the other hand, $\hat{\sigma}^2_{\omega}$ is large then $\omega$ is added to the set $\hat{\Omega}$ and we disregard all $\omega'$ for which $\omega' \preceq \omega$.  This approach ensures that $\hat{\Omega} \in \mathcal{I}(\{0,1\}^d)$.  Our final estimate of the regression function $\eta$ is then based on those $\hat{f}_{\omega}$ for which the pattern $\omega$ is contained in $\hat{\Omega} \cup L(\hat{\Omega})$, that is
\[
\hat{\eta}(\cdot):=\frac{1}{2} + \sum_{\omega\in\hat\Omega \cup L(\hat{\Omega})}\hat f_{\omega}(\cdot)
\] 
We classify the test point $x_0$ according to whether $\hat{\eta}(x_0) \geq 1/2$, or otherwise. 

An appealing feature of our algorithm is that we avoid the need for sample-splitting. However, since our proxy for $\sigma_{\omega}^2$ is calculated by evaluating $\hat{f}_{\omega}$ at the training data and we condition on the observation indicators in our results, we will require a slightly stronger assumption on the properties of the marginal feature distributions in our class, namely we ask for control of the $\mu_{\omega\mid o}$ tails. This amounts to a slightly extended version of Definition~\ref{def:lowerdensity}.

\begin{definitionp}{\ref*{def:lowerdensity}$^{\boldsymbol{+}}$}
\label{def:lowerdensity+}
Fix $\boldsymbol{\gamma} = (\gamma_1, \ldots, \gamma_d)^T \in [0, \infty)^{d}$, $C_{\mathrm{L}} \geq 1$ and $\mathcal{O} \subseteq \{0,1\}^d$. Then let $\mathcal{Q}^{+}_{\mathrm{L}}(\boldsymbol{\gamma}, C_{\mathrm{L}},\mathcal{O})$ denote the class of distributions for which \eqref{eq:lowerdensity} holds and
\begin{align*}
\max_{\tilde{o} \in \mathcal{O} : \omega \preceq \tilde{o}} \mu_{\omega|\tilde{o}}\Bigl( \Bigl\{x \in \IR^d : \min_{o \in \mathcal{O} : \omega \preceq o} \rho_{\mu_{\omega \mid o}, d_{\omega}} (x) < \xi \Bigr\} \Bigr)\leq C_{\mathrm{L}}\cdot \xi^{\gamma_\omega}
\end{align*}
for all $\xi>0$ and all $\omega \in \{0,1\}^d \setminus \{\mathbf{0}_d\}$.
\end{definitionp}

Now, to state our main theoretical result about Algorithm~\ref{alg:nonadaptive}, we will write $
\mathcal{Q}^+_{\mathrm{Miss}} \equiv \mathcal{Q}^+_{\mathrm{Miss}}( \Omega_{\star}, c_{\mathrm{E}}, \boldsymbol{\gamma}, C_{\mathrm{L}}, \boldsymbol{\beta}, C_{\mathrm{S}} , \alpha, C_{\mathrm{M}},\mathcal{O})$ for the class of distributions that arises when we replace 
$\mathcal{Q}_{\mathrm{L}}(\boldsymbol{\gamma},C_{\mathrm{L}}, \mathcal{O})$ in the class $\mathcal{Q}'_{\mathrm{Miss}}$ in~\eqref{def:Q_miss'} with $\mathcal{Q}^+_{\mathrm{L}}(\boldsymbol{\gamma},C_{\mathrm{L}}, \mathcal{O})$. 

\begin{theorem}\label{thm:nonadaptiveUBnew}
Fix $d,n \in\mathbb{N}$, $\mathcal{O} \subseteq \{0,1\}^d$, $o_1, \ldots, o_n \in \mathcal{O}$, $\Omega_{\star} \in \mathcal{I}(\{0,1\}^d)\setminus\{\{\mathbf{0}_d\},\emptyset\}$, $c_{\mathrm{E}} \in (0,1/4]$, $\boldsymbol{\gamma} \in [0, \infty)^d$, $C_{\mathrm{L}} \geq 1$,  $\boldsymbol{\beta}\in(0,1]^d$, $C_{\mathrm{S}} \geq 1$, $\alpha \in [0, \infty)$ and $C_{\mathrm{M}} \geq 1$.  Suppose that $\Omega_{\star} \subseteq \mathcal{N}$ and we have $\min_{\omega \in \Omega_\star} \gamma_{\omega} > \max_{\omega\in\Omega_\star} \beta_{\omega}/d_\omega$. 
Suppose that $n_\omega \geq \log^{\frac{4(\gamma_{\omega}(2\beta_{\omega} + d_{\omega}) + \alpha\beta_{\omega}) }{\beta_{\omega}\gamma_\omega}} \bigl(2|\mathcal{N}| \min_{\omega' \in \Omega_{\star}} n_{\omega'}\bigr)$,
for $\omega \in U(\Omega_{\star}) \cap \mathcal{N}$, then there exists a constant $C_{\mathrm{U}} \geq 1$ (not depending on $n$ nor $o_1, \ldots, o_n$) such that
\begin{align}\label{eq:nonadaptiveUBrate}
&  \sup_{Q \in \mathcal{Q}^+_{\mathrm{Miss}}}  \mathbb{E}_{Q}\Bigl\{\mathcal{E}_{P_Q}(\hat{C}_{\mathrm{HAM}}) \Bigm| O_1 = o_1,\ldots, O_n = o_n \Bigr\} \nonumber
 \\ & \hspace{130pt}  \leq  C_{\mathrm{U}} \cdot \log_+^{\frac{1+\alpha}{2}}\Bigl( \min_{\omega \in \Omega_{\star}} n_{\omega} \Bigr) \cdot \max_{\omega \in \Omega_{\star}} \Bigl\{ n_\omega^{-\frac{\beta_{\omega}\gamma_\omega(1+\alpha)}{\gamma_\omega(2\beta_{\omega}+d_\omega)+\alpha\beta_{\omega}}} \Bigr
\}.
\end{align}
\end{theorem}
The proof of Theorem~\ref{thm:nonadaptiveUBnew} is presented in Sections~\ref{subsec:proof2} and~\ref{subsec:proofnonadaptiveUBnew}
. The key to the proof is to establish in Proposition~\ref{lem:correct_thresholding} that $\hat{\Omega} = \Omega_{\star}$ with high probability, under a sample size condition. The result then follows by combining this with the proof of the upper bound in Theorem~\ref{thm:minmax_bounds}.

The main conclusion of Theorem~\ref{thm:nonadaptiveUBnew} is that our classifier $\hat{C}_{\mathrm{HAM}}$ attains the upper bound in Theorem~\ref{thm:minmax_bounds}, subject to a sample size condition.   Note that we do not impose any restrictions on the minimum number of available cases observed for $\omega \in \Omega_{\star} \cup L(\Omega_{\star})$ and that the condition on $n_\omega$ for $\omega \in U(\Omega_{\star}) \cap \mathcal{N}$ only rules out an extreme imbalance between the number of available cases for the different observation patterns. For instance, under our condition on the relationship between the parameters, if we have at least a polylogarithmic proportion of the data available for the patterns in $U(\Omega_{\star}) \cap \mathcal{N}$, e.g. $n_\omega \geq \log^{8 + 4(d+\alpha)/\beta_{\omega}} (2^{d+1} n)$, then the sample size condition holds. The intuition is that we only need a relatively small amount of data to detect that $f_\omega \equiv 0$ for $\omega \in U(\Omega_{\star})$.
 
A thorough investigation of the proof of Theorem~\ref{thm:nonadaptiveUBnew} in fact allows us to make nontrivial statements about the excess risk of $\hat{C}_{\mathrm{HAM}}$ even when our sample size condition is not satisfied.  In particular, we could elucidate the precise consequence of $\hat\Omega$ in Algorithm~\ref{alg:nonadaptive} not coinciding with $\Omega_{\star}$ in small sample settings.  First, whenever $\Omega_{\star} \subseteq \hat{\Omega} \cup L(\hat{\Omega})$, we may seek to control the excess risk in a similar way to \eqref{eq:nonadaptiveUBrate}, with the $\Omega_{\star}$ in the maximum replaced by $\hat{\Omega}$.  On the other hand, if $\Omega_{\star}  \nsubseteq \hat{\Omega} \cup L(\hat{\Omega})$ (or if $\Omega_{\star} \cap \mathcal{N}^c \neq \emptyset$), then $\hat{C}_{\mathrm{HAM}}$ will estimate a non-zero $f_{\omega}$ by $0$, and we will only have a trivial bound on the excess risk.

\section{Numerical results\label{sec:numericalresults}}
In this section we provide empirical evidence of the strength of the proposed methodology to complement the theoretical results from the previous section.  In particular, we will compare the performance of our $\hat{C}_{\mathrm{HAM}}$ classifier applied with canonical values of the class indices $\alpha$, $\boldsymbol{\beta}$, $\boldsymbol{\gamma}$, as well as a cross-validation version of our HAM algorithm (see Section~\ref{subsec:cvHAM}
) which only requires a set of possible values that $\alpha$, $\boldsymbol{\beta}$, $\boldsymbol{\gamma}$ may take. We also investigate the performance of an oracle version of the HAM classifier which has access to the unknown set $\Omega_{\star}$. The performance is compared with standard approaches used to deal with missing data in classification problems, namely complete case analysis and several different imputation methods.  We also include the comparison with an oracle imputation approach, which has access to the full dataset.

Our experiments\footnote{The code for the simulations is available at \href{https://github.com/TorbenSell/missing_data/}{\url{https://github.com/TorbenSell/missing\_data/}.}} will consider the following simulation settings: 

\medskip

\noindent \textbf{Setting 1}: We adopt the setting in Example~\ref{ex:missingness1}.  Fix $d=2$, $u=(2^{1/2},0)^T$ and $\Sigma = I$ the $2$-dimensional identity matrix.  Suppose that $\mathbb{P}(Y = 1) = 1/2$, and $X \mid \{Y = r\} \sim N_{d}((-1)^r u, \Sigma)$, for $r \in \{0,1\}$ and that $O \indep (X,Y)$. Let $\mathbb{P}(O = \omega) = p^{d_{\omega}}(1-p)^{d-d_{\omega}}$, for $\omega \in \{0,1\}^d$ and homogeneous observation probability $p=0.7$.  In this case, for $x = (x_1,x_2)^T \in \mathbb{R}^2$, the regression function is 
\[
\eta(x) := \frac{1}{2} + \frac{1}{2}\tanh(-2^{1/2} x_1).
\]
It follows that $f_{e_1}(x_1,x_2) = \frac{1}{2} \tanh(-2^{1/2}x_1)$ and that all other $f_{\omega}$ are $0$. The corresponding distribution $Q$ of $(X, Y, O)$ belongs to $\mathcal{Q}_{\mathrm{Miss}}^{+}$, with $\Omega_{\star} = \{(1,0)^T\}$, any $\boldsymbol{\gamma} = (\gamma_1,1)\in [0,1)\times \{1\}$, $\alpha = 1$, $\boldsymbol{\beta}=\mathbf{1}_{2}$, and $\mathcal{O} = \{0,1\}^2$,  see Section~\ref{sec:intuitionfor1} 
for full details.  Therefore, we have 
\[
R = n_{(1,0)}^{-\frac{2\gamma_1}{3\gamma_1 + 1}},
\] 
which can be arbitrarily close to $n_{(1,0)}^{-1/2}$ as $\gamma_1 \rightarrow 1$.

\noindent \textbf{Setting 2}: In Example~\ref{ex:missingness2}, take $\mu = U([0,1]^4)$ to be the uniform distribution on the $4$-dimensional unit cube. For $x=(x_1,x_2,x_3,x_4)^T\in[0,1]^4$, we let 
\[
\eta(x):= \frac{1}{2} + \frac{(x_2-1/2) \cdot x_3^2}{2} + \frac{x_4 - 1/2}{2}
\]
and suppose that $Y \indep O \mid X$.
Here the nonzero $f_{\omega}$ functions are 
\[
f_{e_2}(x) = \frac{(x_2-1/2)}{6}; \quad f_{e_2+e_3}(x)=\frac{(x_2-1/2)(x_3^2-1/3)}{2}; \quad f_{e_4}(x)=\frac{(x_4-1/2)}{2}.
\]
Here we have $Q \in \mathcal{Q}_{\mathrm{Miss}}^{+}$, with $\Omega_{\star} = \{(0,1,1,0)^T, (0,0,0,1)^T\}$, $\boldsymbol{\gamma} =(1,\gamma_2,\gamma_3, \gamma_4)^T$ for any $\gamma_2,\gamma_3,\gamma_4\in[0,\infty)$, $\alpha = 1$, $\boldsymbol{\beta}=\mathbf{1}_4$, and $\mathcal{O} = \{(1,1,1,0)^T, (1,0,0,1)^T\}$, see Section~\ref{sec:intuitionfor1}
. We thus have 
\[
R = n_{(0,1,1,0)}^{-\frac{2(\gamma_2 \wedge \gamma_3)}{4(\gamma_2 \wedge \gamma_3)+1}} \vee n_{(0,0,0,1)}^{-\frac{2\gamma_4}{3\gamma_4+1}},
\]
which, since we can take any arbitrarily large $\boldsymbol{\gamma}$, could be as fast as $R = n_{(0,1,1,0)}^{-1/2} \vee n_{(0,0,0,1)}^{-2/3}$.

\noindent \textbf{Setting 3}: We use Example~\ref{ex:missingness3}, where for $x = (x_1,x_2)^T \in [0,1]^2$ we have 
\[
\eta(x):= \frac{1}{4} + \frac{x_1}{4} + \frac{1}{4} \cos(4\pi x_2).
\]
Here the nonzero $f_{\omega}$ are
\[
f_{e_1}(x)= \frac{x_1-1/2}{4}; \quad f_{e_2}(x)=\frac{1}{4} \cos(4\pi x_2).
\]
We have $Q \in \mathcal{Q}_{\mathrm{Miss}}^{+}$, with $\Omega_{\star} = \{e_1,e_2\}$, any $\boldsymbol{\gamma}=(\gamma_1,\gamma_2) \in [0,\infty)^2$, $\alpha = 1$, $\boldsymbol{\beta}=\mathbf{1}_2$ and $\mathcal{O} = \{0,1\}^2$, see Section~\ref{sec:intuitionfor1}
. Here, we have \[
R = n_{(1,0)}^{-\frac{2\gamma_1}{3\gamma_1+1}} \vee n_{(0,1)}^{-\frac{2\gamma_2}{3\gamma_2+1}},
\] 
which approaches $(n_{(1,0)}\wedge n_{(0,1)})^{-2/3}$ as $\gamma_1,\gamma_2 \rightarrow \infty$. 

In each of these settings, we conduct an experiment by generating a training dataset of size $n$, for $n \in \{500,1000\}$.  In particular, let $((X_1,Y_1,O_1), \ldots, (X_n,Y_n,O_n)) \sim Q^n$, and a test dataset consisting of $m = 1000$ independent copies of $(X,Y)$, when $(X, Y, O) \sim Q$.  For each algorithm, which will only have access to $(X_1^{O_1},Y_1,O_1), \ldots, (X_n^{O_n},Y_n,O_n)$, we will present a boxplot of the empirical test error for 100 repeats of the experiment.  Here the empirical test error is the proportion of incorrect classifications made on the test dataset.  We will compare several different classification algorithms, the first three being our methods:

\smallskip

\noindent \textbf{HAM}: Our Algorithm~\ref{alg:nonadaptive} applied with $\boldsymbol{\beta} = \boldsymbol{\gamma} = \mathbf{1}_d$,  $\alpha  = 1$, and training data $(X_1^{O_1},Y_1,O_1)$, $\ldots, (X_n^{O_n},Y_n,O_n)$.

\noindent \textbf{cvHAM}: Our Algorithm~\ref{alg:nonadaptive} applied with $\boldsymbol{\beta} = \hat{\boldsymbol{\beta}} $, $\boldsymbol{\gamma} = \hat{\boldsymbol{\gamma}}$, and $\alpha = \hat{\alpha}$, and training data $(X_1^{O_1},Y_1,O_1), \ldots, (X_n^{O_n},Y_n,O_n)$. Here $\hat{\boldsymbol{\beta}}$, $\hat{\boldsymbol{\gamma}}$,  and $\hat{\alpha}$ are the outputs of our $5$-fold cross-validation procedure in Algorithm~\ref{alg:cv} 
applied with $A = \{0,1/2,1/d\}$, $B=\{1/4,2/4,3/4,1\}\cdot \mathbf{1}_d$ and $G=\{1/2,1,2,\infty\}\cdot \mathbf{1}_d$.

\noindent \textbf{Oracle HAM}: Our Algorithm~\ref{alg:nonadaptive} with $\boldsymbol{\beta} = \boldsymbol{\gamma} = \mathbf{1}_d$ and  $\alpha=1$, which has access to the true $\Omega_\star$ and thus can omit the estimation of it. 

The other methods apply a standard $k$-nearest neighbour ($k$nn) algorithm \citep{fix1952discriminatory,fix1989discriminatory} to different modifications of the training data, namely 
\smallskip

\noindent \textbf{Complete Case (CC)}: We use only the observations belonging to $N_{\mathbf{1}_d} = \{i\in [n] : O_i = \mathbf{1}_d\}$, and apply the $k$nn algorithm with $k = 1/3\cdot n_{\mathbf{1}_d}^{-1/(3+d)}$; in other words we apply the $k$nn algorithm directly to $\{(X_i,Y_i) : i \in N_{\mathbf{1}_d}\}$.
  
\noindent  \textbf{Zero Imputation (ZI)}: All missing values are imputed with the value $0$ and we then apply the $k$nn algorithm with $k = 1/3\cdot n^{-1/(3+d)}$; in other words the $k$nn algorithm is applied directly to $(X_1^{O_1},Y_1), \ldots, (X_n^{O_n},Y_n)$. 

\noindent  \textbf{Mean Imputation (MI)}: For each $j \in [d]$, let $\hat{\nu}_j := \frac{1}{n_{e_j}} \sum_{i \in N_{e_j}} X_i^{e_j}$ be the mean of the observations in the training data for which the $j$th variable is not missing and define $\hat{\nu} := \sum_{j\in [d]} \hat{\nu}_j$.  Then for each observation $i \in [n]$, let $\tilde{X}_i = X_{i}^{o_i} + (\mathbf{1}_d - o_i) \odot \hat{\nu}$. We then apply the $k$nn algorithm with $k = 1/3\cdot n^{-1/(3+d)}$ to $(\tilde{X}_1,Y_1), \ldots, (\tilde{X}_n,Y_n)$. 

\noindent  \textbf{Multiple Imputation by Chained Equations (MICE)}: All missing values in the training set are imputed using the `miceforest' python package to give $5$ different datasets. We then apply the $k$nn algorithm with $k = 1/3\cdot n^{-1/(3+d)}$ to each of these datasets and classify a test point based on a majority vote. 

\noindent  \textbf{MissForest (MF)}: All missing values are imputed using the missforest algorithm \citep{stekhoven2012missforest} which is implemented in the python package `missingpy'. A $k$nn algorithm with $k = 1/3\cdot n^{-1/(3+d)}$ is then employed using the imputed dataset as training data.

\noindent \textbf{No missing data (no MD)}: We apply a $k$nn algorithm with $k = 1/3\cdot n^{-1/(3+d)}$ to the full training dataset. This can be viewed as an oracle imputation method, which can impute the missing values with the unobserved truth. 

\begin{figure}[ht]
\centering
\includegraphics[width=0.98\linewidth]{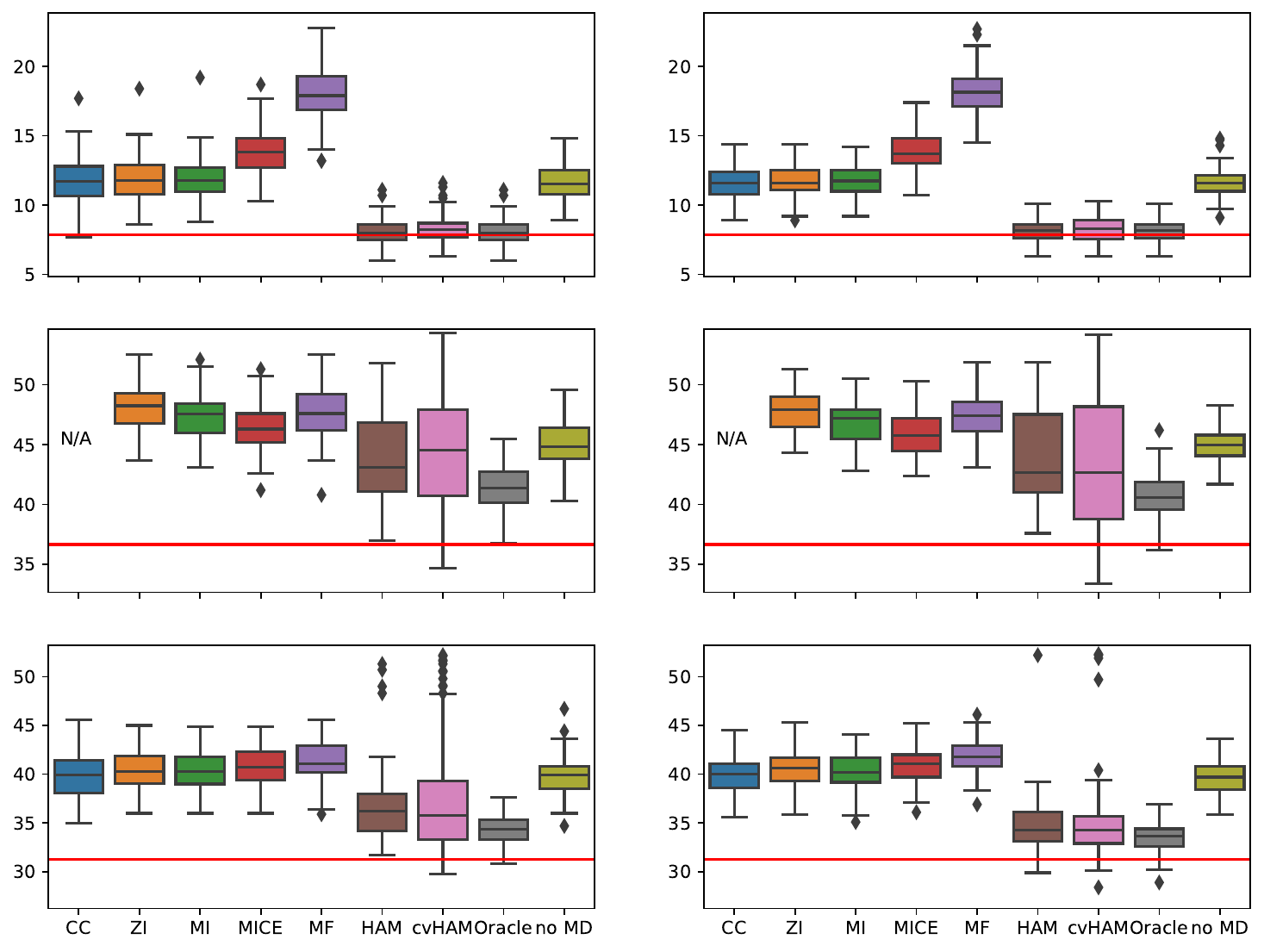}
\caption{Box plots of the empirical percentage test errors for the Complete Case (CC), Zero Imputation (ZI), Mean Imputation (MI), MICE, and MissForest (MF) approaches and our classifier with known parameters (HAM) and our classifier with parameters estimated by a cross-validation procedure (cvHAM).  For reference, we also include our Oracle HAM (Oracle) classifier, and a classifier that has access to all data, with no missing values (no MD).  The Bayes risk is shown as the red horizontal line. We present the results for Setting~1 (top row), Setting~2 (middle row) and Setting~3 (bottom row), with $n=500$ (left column) and $n=1000$ (right column). \label{fig:results}}
\end{figure}

The results of our experiments are presented in Figure~\ref{fig:results}.  We see that across our three settings, the HAM classifier performs well, and typically outperforms the widely-used competing methods.  In Setting~2, there are no fully observed data points available, so a complete case analysis is not possible. Further we see, especially in Setting~1, that the HAM classifier performs almost as well as its oracle version.  In Settings~2 and~3, the slight difference between the performance of HAM and its oracle is due in part to the fact that we have relatively small number of available cases for some observation patterns. For instance, in Setting~3 we need to observe both variables to detect that $f_{\mathbf{1}_2} = 0$ in this example, but on average we only have $188$ (when $n=500$) and $375$ (when $n=1000$) complete cases. Indeed, by inspecting the results, we see that we are slightly conservative in our thresholding step for this setting, nevertheless HAM still performs well.  Some further practical improvement may be achievable by tuning the thresholding cutoff depending on the data (i.e.~via cross validation). 

The results of an additional numerical study with no missing data are presented in Section~\ref{sec:simulationsappendix} 
in the supplementary material, we see that our HAM classifier remains effective in this case.

\subsection{Real data example\label{sec:real_data}}
In this section, we illustrate the performance of our HAM algorithm in a real data setting using the `Mammographic Mass' data set \citep{misc_mammographic_mass_161} and consider the task of classifying breast cancer tumours as either benign (class $0$) or malignant (class $1$) based on the shape, margin, and density of the tumour mass, as well as the patient's age.  There are $961$ observations in the dataset, of which $5$ have missing entries for age, $31$ have no entry for mass shape, $48$ are missing entries for mass margin, and $76$ have no entry for mass density.  There are 516 benign tumours (53.7\%) and 445 malignant tumours (46.3\%).

We compare our HAM and cvHAM algorithms to the complete case classifier (CC), as well as Zero Imputation (ZI), Mean Imputation (MI), MICE, and MissForest (MF), all implemented as described in the previous section.  We randomly split the data into a training set and test set by sampling $161$ of the complete cases to form our test set, leaving $n = 800$ observations to form the training set which contains missing values.  We record the test error of each algorithm over $100$ different train-test splits.  The results of this experiment are shown in Figure~\ref{fig:real_data}, where we see that our HAM and cvHAM approaches are outperforming the existing methods. 

\begin{figure}[ht]
\centering
\includegraphics[width=0.65\linewidth]{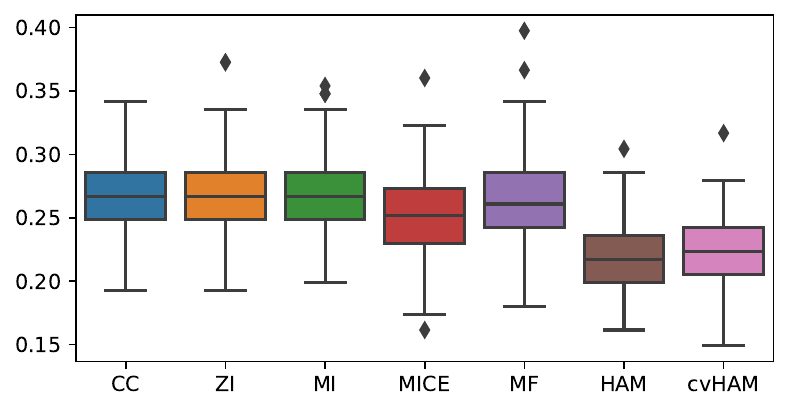}
\caption{Box plots of the empirical percentage test errors for the Complete Case (CC), Zero Imputation (ZI), Mean Imputation (MI), MICE, and MissForest (MF) approaches and our classifier with known parameters (HAM) and our classifier with parameters estimated by a cross-validation procedure (cvHAM). \label{fig:real_data}}
\end{figure}

\section{Proofs\label{sec:proofs}}

\subsection{Constructions used in the lower bound proof of Theorem~\ref{thm:minmax_bounds}\label{sec:lower_bound}}
Our proof of the lower bound in Theorem \ref{thm:minmax_bounds} makes use of a new version of Assouad's Lemma, modified to our classification with missing data setting (see Lemma~\ref{lem:assouad} 
in Section~\ref{sec:appendixLBproofs}
). To illustrate the key ideas of our proof, we initially consider the case that $\Omega_\star = \{\omega\}$ contains only one element. We construct a finite set of distributions $Q$ of $(X, Y, O)$ belonging to $\mathcal{Q}'_{\mathrm{Miss}}$.  The full constructions for our proofs are somewhat complicated and many of the details, as well as the corresponding theory and ultimately the full proof of the lower bound in Theorem~\ref{thm:minmax_bounds} for a general $\Omega_{\star}$, are presented in Section~\ref{sec:appendixLBproofs}
.  

Here we present some of the key aspects of our proof.  First, for each distribution $Q$ that we construct, we take $(X,Y)$ and $O$ to be independent and take the distribution of $O$ to be uniform on $\mathcal{O}$, so $Q \in \mathcal{Q}_{\mathrm{Miss}}(\mathcal{O})$.  In order to ensure that our regression functions decompose as  in~\eqref{eq:etadecomposition}, we require a carefully designed marginal feature distribution, which in particular satisfies many symmetries about the origin.  The full definition of these, which differ slightly in the \emph{light tailed} case (where $\gamma_{\omega} \geq 1$) and the \emph{heavy tailed} case (where $\gamma_{\omega} < 1$), are given in Sections~\ref{subsubsec:LBlight} and~\ref{subsubsec:LBheavy}.  In both cases, we have a discrete lattice of support points on which the regression function construction will change between distributions, and a set where the regression functions does not change between distributions in our construction. Sections~\ref{sec:LB_small_r}
and~\ref{sec:LB_big_r} 
in the supplementary material show that the distributions are constructed in such as way that they belong to the class~$\mathcal{Q}_\mathrm{L}^{+}(\boldsymbol{\gamma}, C_{\mathrm{L}}, \mathcal{O})$. 

We then, in Sections~\ref{subsec:LBlighteta} 
and~\ref{subsec:LBheavyeta}
, define the regression function $\eta$ for our respective light- and heavy-tailed initial constructions when $\Omega_\star = \{\omega\}$.  The symmetry in marginal distributions allows us to define $\eta$ via a single nonzero function $f_{\omega}$ in the anova decomposition in~\eqref{eq:etadecomposition}.  This function $f_\omega$ is defined first on the region $(0, \infty)^d$ and then mirrored or flipped on the appropriate axes to define the function on the whole of $\mathbb{R}^d$, ensuring that $\eta=1/2+f_\omega$, i.e.~\eqref{eq:etadecomposition} holds with all other $f_{\omega'} = 0$ (see Lemmas~\ref{lem:E_and_S_def_satisfied} 
and~\ref{lem:E_and_S_def_satisfied_heavy_tails})
.  On the lattice, $f_{\omega}$ is taken to be either $\epsilon$  or $-\epsilon$ for some small $\epsilon>0$. The different possible combinations of $\pm \epsilon$ on these points lead to the different distributions used in our application of Assouad's Lemma. On the remainder of the marginal support, $f_\omega$ is taken to be a particular $\beta$-smooth function.  This choice allows us to extend the definition of $f_{\omega}$ to a $\beta$-smooth function on the whole of $\mathbb{R}^d$ (via McShane's extension Theorem -- see Corollaries~\ref{cor:f_beta_holder_on_IRplus}
and~\ref{cor:f_beta_holder_on_IRplus_heavy_tails}
).  Lemmas~\ref{lem:E_and_S_def_satisfied} 
and~\ref{lem:E_and_S_def_satisfied_heavy_tails} 
also show that the corresponding distributions belong to $\mathcal{P}_\mathrm{S}(\boldsymbol{\beta}, 1)$.  Concluding Sections~\ref{subsec:LBlighteta} 
and~\ref{subsec:LBheavyeta}
, Lemmas~\ref{lem:M_def_satisfied} 
and~\ref{lem:M_def_satisfied_heavy_tails} 
show that the distributions belong to $\mathcal{P}_\mathrm{M}(\alpha, C_{\mathrm{M}})$. 

Sections~\ref{subsec:LBproof_lighttails} 
and~\ref{subsec:LBproof_heavytails} 
complete the proof of the lower bound when $\Omega_\star = \{\omega\}$ for the light- and heavy-tailed cases, respectively.  More precisely, we apply our version of Assouad's Lemma to carefully selected versions of the constructions described above.  The full proof of the lower bound in Theorem~\ref{thm:minmax_bounds} is presented in Section~S2.8.  The proof extends the constructions above to the general $\Omega_{\star}$ case, using a mixture of distributions, which establishes the lower bound over $\mathcal{Q}'_{\mathrm{Miss}}( \Omega_{\star}, c_{\mathrm{E}}, \boldsymbol{\gamma}, C_{\mathrm{L}}, \boldsymbol{\beta}, C_{\mathrm{S}} , \alpha, C_{\mathrm{M}}, \mathcal{O})$; in fact we will see that our constructions belong to the smaller class  $\mathcal{Q}^+_{\mathrm{Miss}}$, and thus that our lower bound holds even over this more restrictive class.

\subsubsection{The marginal feature distribution in the light tailed case\label{subsubsec:LBlight}}

 It is convenient to introduce some additional notation at this point.  For a set $A^+\subseteq [0,\infty)^d$ and $s = (s_1,\ldots, s_d)^T \in \{-1,1\}^d$, we will write $s \odot A^+:=\{x = (x_1,\ldots,x_d)^T \in \mathbb{R}^d : x = s \odot x^+, x^+ \in A^+ \} \subseteq \mathbb{R}^d$.  We now define the marginal distribution of the feature vectors, which will in fact be the same for all $Q$ in our class.  First, for $q \in \mathbb{N}$ and $r>0$, let 
\[
\mathcal{T}_{q,r} = \biggl\{\Bigl(1+\frac{rv_1}{q},\ldots, 1+\frac{rv_d}{q}\Bigr)^T : v_j \in [q] \text{ for } j\in [d]\biggr\}.
\]
Then, fixing $\omega \in \{0,1\}^d \setminus \{\mathbf{0}_d\}$, let 
$\mathcal{T}_{q,r}^{\omega} = \{ z^\omega \in\mathbb{R}^d : z \in \mathcal{T}_{q, r}\}$,  $T := \bigl|\mathcal{T}_{q,r}^{\omega}\bigr| = q^d$ and enumerate this set as $\{z^\omega_{1}, \ldots, z^{\omega}_{T}\} \subseteq [0,1 + r]^{d}$.  Further let $\Pi_{\omega} : \mathbb{R}^d \rightarrow \mathbb{R}^{d_{\omega}}$ be the projection of $\mathbb{R}^d$ onto the $d_{\omega}$ non-zero coordinates of $\omega$, i.e. for $x := (x_1, \ldots, x_d)^T \in \mathbb{R}^d$ and $\omega := (\omega_1, \ldots, \omega_d)^T \in \{0,1\}^d$, we have $\Pi_{\omega}(x) := (x_j)_{j \in [d]: \omega_j = 1} \in \mathbb{R}^{d_{\omega}}$. For $A \subseteq \mathbb{R}^d$, let $\Pi_{\omega}(A) := \{\Pi_\omega(x) : x\in A\}$.  Further, for $\kappa > 0$, let $\mathcal{R} :=  [1+2r,1+2r+\kappa]^{d}$ and $\mathcal{R}^\omega := \{x^\omega\in\mathbb{R}^d: x \in\mathcal R\}$.  Let $\mathcal{R}_0 := [1,1+r]^d$, $\mathcal{R}_0^\omega := \{x^\omega\in\mathbb{R}^d: x \in\mathcal R_0\}$ and, for $j \in \{\tilde{j} \in [d] : \omega_{\tilde{j}} =1\}$, let $\mathcal{R}_j^\omega = \mathcal{R}_0^\omega + (1 + 2r + \kappa) \cdot e_j$.

We define the marginal $X$ distribution as follows: for $a,b \in [0,1/2]$ and $A \subseteq \mathbb{R}^d$ (measurable), let
\begin{align}
\begin{split}\label{lb:marginal_construction}
    \mu(A) \equiv \mu^{(\omega)}_{\kappa,r,q,a,b}(A) & := 
    \frac{a}{2^{d}T} \sum_{s\in\{-1,1\}^d}  \sum_{t=1}^{T}\mathbbm{1}_{\{s^\omega\odot z_t^\omega\in A\}}\\
    & \hspace{30pt} +\frac{b}{2^{d}d_\omega r^{d_\omega}} \sum_{j \in [d] : \omega_j = 1} \sum_{s\in\{-1,1\}^d} \mathcal{L}_{d_\omega}\Bigl(\Pi_\omega\bigl(A\cap(s\odot \mathcal{R}_j^{\omega})\bigr)\Bigr)\\
    & \hspace{60pt} + \frac{1-a-b}{2^{d}\kappa^{d_{\omega}}} \sum_{s\in\{-1,1\}^d} \mathcal{L}_{d_\omega}\Bigl(\Pi_\omega\bigl(A\cap(s\odot \mathcal{R}^{\omega})\bigr)\Bigr).
\end{split}
\end{align}

\begin{figure}[ht]
\centering
\includegraphics[width=0.75\linewidth]{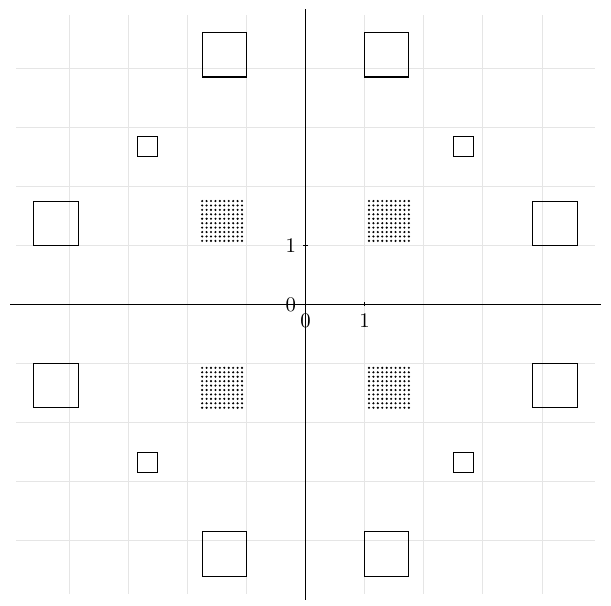}
\caption{\label{fig:LB_light_tails} The support of the light tailed marginal for an observation pattern $\omega$ with $d_\omega=2$ and $\gamma_\omega\geq1$. The lattice in the top right quadrant is $\mathcal{T}^\omega_{q,r}$, the boxes are $\mathcal{R}^\omega$, $\mathcal{R}^\omega_1$ and $\mathcal{R}^\omega_2$. The boxes carry enough mass to ensure that projecting onto one variable results in light tails. Specifically here we plot the support of $\mu^{(\mathbf{1}_2)}_{\kappa,r,q,a,b}$ with $\kappa = 0.35$, $r = 0.75$ and $q = 10$.}
\end{figure}

A particular instance of the support of this construction is given in Figure~\ref{fig:LB_light_tails}. The regression functions in this case are given in Section~\ref{subsec:LBlighteta}
. The main idea is that we will take $\eta = 1/2 + f_{\omega}$ to be slightly above or below $1/2$  on each of the grid points in $\mathcal{T}_{q,r}^{\omega}$, and bounded away from $1/2$ on the remainder of the support.  This is initially defined only on the ``top right" quadrant of the support and then reflected asymmetrically to ensure the $f_{\omega}$ satisfies the orthogonality constraint in the anova decomposition -- see~\eqref{def:f_omega}
in Section~\ref{subsec:LBlighteta}
. 

Note that the construction of the marginal distribution used in this subsection shares some aspects with the construction used by~\citet{reeve2021adaptive} in their transfer learning problem.  In particular, the regions $\mathcal{T}^{\omega}_{q,r}$ and $\mathcal{R}^\omega$ are used similarly to how they were used in~\citet{reeve2021adaptive}, while here we introduce the sets $\mathcal{R}^\omega_j$ for $j \in [d]$ and carefully \emph{reflect} the upper-right quadrant in the coordinate axes to give our final $X$ marginal construction.  These latter aspects are crucial for constructing distributions $Q$ satisfying Definitions~1 to~5 in our missing data problem.  

\subsubsection{The marginal feature distribution in the heavy tailed case\label{subsubsec:LBheavy}}

The marginal construction for the heavy tailed case is somewhat simpler, in particular the marginal distribution of $X$ can be taken to be a product measure.  Let $\nu_0$ denote a point mass at zero, i.e. $\nu_0(A) = \mathbbm{1}_{\{0\in A\}}$, for $A \subseteq \mathbb{R}$.   For $\tilde{q} \in \mathbb{N}$, let $\nu_1 \equiv \nu_{1,\tilde{q}}$ denote the distribution on $\mathbb{R}$ given by 
\[
\nu_{1,\tilde{q}}(A) := \frac{1}{2 \tilde{q}} \sum_{j=1}^{\tilde{q}} \mathbbm{1}_{\{ 1+j/\tilde{q} \in A\}} + \mathbbm{1}_{\{-1-j/\tilde{q} \in A\}},
\]
for $A \subseteq \mathbb{R}$. Further, for $a \in [0,1], r > 1$ and $q \in \mathbb{N}$, let $\nu_2 \equiv \nu_{2,q,r,a}$ denote the distribution on $\mathbb{R}$ given by 
\begin{align*}
    \nu_{2,q,r,a}(A)   &:=  \frac{a}{2 q} \sum_{j=1}^{q} \Bigl(\mathbbm{1}_{\{1+ \frac{rj}{q} \in A\}} + \mathbbm{1}_{\{-1-\frac{rj}{\tilde{q}} \in A\}}\Bigr) \\
    &\hspace{60pt}+ \frac{1-a}{2} \mathcal{L}_1\Bigl( A \cap \bigl\{(-r-2,-r-1) \cup (r+1,r+2)\bigr\}\Bigr).
\end{align*}
Now fix $\omega \in \{0,1\}^d\setminus \{\mathbf{0}_d\}$ and $\tilde{j} \in \{j \in [d] : \omega_j =1\}$. Let 
\begin{equation}
\label{eq:muheavy}
\mu \equiv \mu^{(\omega)}_{q,\tilde{q},r,a,\tilde{j}} = \bigotimes_{j=1}^{\tilde{j}-1} \nu_{\omega_j} \otimes \nu_{2,q,r,a} \otimes \bigotimes_{j = \tilde{j}+1}^{d} \nu_{\omega_j}.
\end{equation}

\begin{figure}[ht]
\centering
\includegraphics[width=0.9\linewidth]{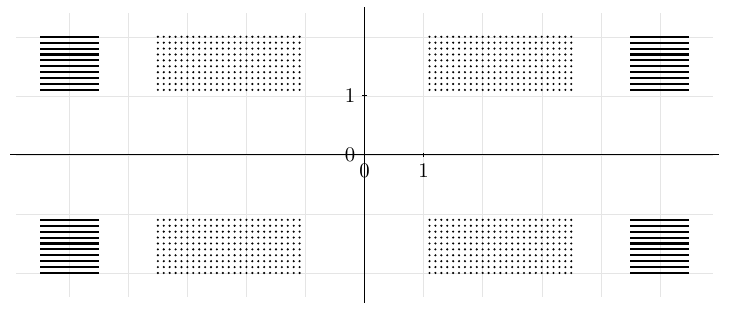}
\caption{The support of the heavy tailed marginal for an observation pattern $\omega$ with $d_\omega=2$ and $\gamma_\omega<1$. Specifically here we plot the support of $\mu^{(\mathbf{1}_2)}_{q,\tilde{q},r,a,1} = \nu_{2,q,r,a} \otimes \nu_{1,\tilde{q}}$, with $q = 25$, $\tilde{q} = 10$, $r = 2.5$. The lines on the sides ensure that projecting onto the $x_2$ axis results in light tails.  \label{fig:LB_heavy_tails} }
\end{figure}

A particular instance of the support of this construction is given in Figure~\ref{fig:LB_heavy_tails}. The corresponding regression function in this case is conceptually similar to that in the light tailed case above and is therefore deferred to Section~\ref{subsec:LBheavyeta}
.

\subsection{Outline of the proof of the upper bound in Theorem~\ref{thm:minmax_bounds}\label{subsec:proofUB}}

In this section, we provide the proof of the upper bound in Theorem~\ref{thm:minmax_bounds}. Our strategy for the proof is to introduce two high-probability events $E_{1}^{\delta}(x)$ and $E_{2}^\delta(x)$. The first provides control of the nearest neighbour distances to a point $x$ in the support of $\mu$ for each observation pattern $\omega \in \mathcal{N}$. Our second high-probability event asks for control of the empirical average of the labels of the nearest neighbours in each observation pattern.  Then, working on these events, we provide in Lemma~\ref{lem:f_is_not_zero} control of the error $|\hat{f}_\omega(x)-f_\omega(x)|$ for each $\omega \in \mathcal{N}$.  This establishes the rate at which we are able to estimate the functions $f_{\omega}$ for $\omega \in \mathcal{N}$. 

We next show that in regions where $\eta$ is sufficiently far from a half, a version of our classifier agrees with the Bayes classifier with high probability.  The main result in this subsection is Proposition~\ref{prop:rate_if_Omega_correct}, which provides a tail bound on the excess risk of this version of our algorithm.   We then prove the upper bound in Theorem~\ref{thm:minmax_bounds} at the end of this subsection, by integrating the tail bound for an oracle version of the algorithm, which has access to the true $\Omega_{\star}$.  The proofs of the results in subsection, as well as some additional lemmas are given in Section~\ref{sec:appendixUBproofs} 
in the supplement. 

Before presenting the results, it is convenient to fix some notation that we will use throughout this section.  First, recall from the statement of Theorem~\ref{thm:minmax_bounds} that we are treating the missingness indicators $o_1, \ldots, o_n \in \mathcal{O}\subseteq\{0,1\}^d$ as fixed and all probability statements in this section should be interpreted as being conditional on $O_1 = o_1, \ldots, O_n = o_n$.  For $\omega \in \{0,1\}^d$, we will write $\rho_{\omega}(\cdot)$ in place of $\min_{o\in\mathcal{O}:\omega\preceq o}\{\rho_{\mu_{\omega|o}, d_{\omega}}(\cdot)\}$. Let $k_{\mathbf{0}_d}=n$ and recall from Algorithm~\ref{alg:nonadaptive} that $k_{\omega} = 1 + \lfloor n_{\omega}^{\frac{2\beta_{\omega}\gamma_{\omega}}{\gamma_{\omega}(2\beta_{\omega}+d_{\omega}) + \alpha \beta_{\omega}}}\rfloor$, for $\omega\in \mathcal{N} \setminus \{\mathbf{0}\}$.  Let $\mathcal{X} := \{x \in \mathbb{R}^d : \rho_{\omega}(x^{\omega}) > 0, \  \text{for all} \ \omega \in \mathcal{N}\}$. For $\delta \in (0,1]$, $\omega \in \mathcal{N}$, $k_\omega \in [n_{\omega}]$, and $x \in \mathcal{X}$, let 
$\tilde{k}_{\omega} \equiv \tilde{k}_{\omega} (x) := \lceil 4\log_{+}(|\mathcal{N}|/\delta)\rceil \mathbbm{1}_{\{k_{\omega} <\lceil4\log_{+}(|\mathcal{N}|/\delta)\rceil
\}}  + k_{\omega} \mathbbm{1}_{\{\lceil4\log_{+}(|\mathcal{N}|/\delta)\rceil \leq k_{\omega} < n_{\omega}\rho_{\omega}(x^{\omega})/2 \}} +  (\lceil n_{\omega}\rho_{\omega}(x^{\omega})/2 \rceil - 1 )\mathbbm{1}_{\{ 
n_{\omega}\rho_{\omega}(x^{\omega})/2 \leq k_{\omega}\}}$
 if $\lceil4\log_{+}(|\mathcal{N}|/\delta)\rceil<n_{\omega}\rho_{\omega}(x^{\omega})/2$, and let $\tilde{k}_{\omega} \equiv \tilde{k}_{\omega}(x) := 0$ if $\lceil4\log_{+}(|\mathcal{N}|/\delta)\rceil\geq n_{\omega}\rho_{\omega}(x^{\omega})/2$. 
 
We now formally introduce our two high-probability events. First, for $x \in \mathcal{X}$ and $\omega \in \mathcal{N}$, write $X_{(0)_\omega}^{\omega}(x^\omega) = x^{\omega}$ and define the event 
\begin{align*}
    A_{\omega}^\delta(x)&:=  \biggl\{\lVert X_{(\tilde{k}_\omega)_\omega}^{\omega}(x^\omega)-x^\omega\rVert\leq\Bigl(\frac{2\tilde k_\omega }{n_\omega\rho_\omega(x^\omega)}\Bigr)^{1/d_\omega} \biggr\}.
\end{align*}
Let $E_{1}^\delta(x):=\bigcap_{\omega \in \mathcal{N}} A_{\omega}^\delta(x)$.  Lemma~\ref{lem:E_1} 
in Section~\ref{sec:appendixUBproofs} 
bounds the probability of the event $E_{1}^\delta(x)$. 

We now define the high-probability event on which the empirical average of labels of the nearest neighbours is close to the corresponding average of the regression function evaluated at the feature vectors.  For $\omega \in \{0,1\}^d$, write $\eta_\omega : \mathbb{R}^d \rightarrow [0,1]$ for the regression function of $Y$ given $X^{\omega}$ given by $\eta_\omega(x) := \mathbb{P}_Q(Y =1 \mid X^{\omega} = x^{\omega})$.  For $\delta \in (0,1)$, $\omega \in \mathcal{N}$ and $x \in \mathbb{R}^d$ let 
\begin{align*}
    B^\delta_\omega(x) := \biggl\{\Bigl|\frac{1}{k_\omega}\sum_{i=1}^{k_\omega} \bigl\{Y_{(i)_\omega}(x) - \eta_\omega(X_{(i)_\omega}(x))\bigr\} \Bigr|\leq\sqrt{\frac{\log_+(2 |\mathcal{N}|/\delta)}{2k_\omega}}\biggr\},
\end{align*}
and let $E_{2}^\delta(x) := \bigcap_{\omega \in \mathcal{N}} B^\delta_\omega(x)$. Lemma~\ref{lem:E_2_prob} 
in Section~\ref{sec:appendixUBproofs} 
bounds the probability of this event.

Our first result here controls the bias and variance of $\hat{f}_{\omega}$ for $\omega \in \mathcal{N}$ on the event $E_1^{\delta}(x) \cap E_2^{\delta}(x)$.  To this end, we will make use of a universal bound on $\|f_\omega\|_\infty$, given by $C_{\mathrm{B}}$ in Proposition~\ref{prop:f_bounded}
.  For $\omega \in \mathcal{N} \setminus \{\mathbf{0}\}$ and $x \in \mathcal{X}$, let
\begin{align}
\label{eq:biasvariance}
    R_{\omega,1}(x) =  2C_{\mathrm{B}}C_{\mathrm{S}} \cdot \max_{\omega' \preceq \omega} \Bigl(\frac{2\tilde k_{\omega'}}{n_{\omega'}\rho_{\omega'}(x^{\omega'})}\Bigr)^{\beta_{\omega'}/d_{\omega'}}; \quad R_{\omega,2} := \sqrt{\frac{\log_+(2 |\mathcal{N}|/\delta)}{2k_\omega}}, 
\end{align}
and let $R_{\mathbf{0}_d,1}(x) = 0$ and $R_{\mathbf{0}_d,2} = \sqrt{\frac{\log_+(2 |\mathcal{N}|/\delta)}{2n}}$.

\begin{lemma}\label{lem:f_is_not_zero}
Fix $\omega\in \mathcal{N}$ and $x\in \mathcal X$, 
then on the event $E_1^{\delta}(x) \cap E_{2}^{\delta}(x)$ we have that 
\begin{align}\label{ineq:f_bound}
    \bigl|\hat{f}_{\omega}(x)-f_{\omega}(x)\bigr| \leq \kappa_{\omega}\cdot \bigl(2^{d_\omega} \cdot R_{\omega,1}(x) +  R_{\omega,2}\bigr),
\end{align}
for some $1 \leq \kappa_{\omega} \leq 2(d_\omega+1)^{d_\omega}$, which is given explicitly in the proof. 
\end{lemma}

It is now convenient to define a version of our HAM classifier that depends on a deterministic set $\Omega \subseteq \mathcal{N}$.  Indeed, let $\hat{\eta}_{\Omega}(x_0) := 1/2+ \sum_{\omega\in \Omega\cup L(\Omega)} \hat{f}_{\omega}(x_0)$ and let $\hat{C}_{\Omega}(x_0) := \mathbbm{1}_{\{\hat{\eta}_{\Omega}(x_0) \geq 1/2\}}$. 
Our next two results concern the properties of $\hat{C}_{\Omega}(x_0)$, and we begin by establishing the rate at which $\hat{\eta}_{\Omega}(x)$ approximates $\eta(x)$ on the event $E_1^{\delta}(x) \cap E_2^{\delta}(x)$.  For $x\in \mathcal{X}$ let $R^\delta_{\mathbf{0}_d}(x):= \sqrt{\frac{\log_+(2 |\mathcal{N}|/\delta)}{2n}}$ and for $\omega \in \mathcal{N}\setminus\{\mathbf{0}_d\}$ and $x\in \mathcal{X}$ define 
\[
R^{\delta}_{\omega}(x) := \biggl(\frac{n_{\omega}^{-\frac{\alpha\beta_{\omega}^2/d_{\omega}}{\gamma_{\omega}(2\beta_{\omega}+d_{\omega}) + \alpha \beta_{\omega}}}}{\rho_{\omega}^{\beta_{\omega}/d_{\omega}}(x^{\omega}) }  + \log^{1/2}_+(2|\mathcal{N}|/\delta) \biggr) n_{\omega}^{-\frac{\beta_{\omega}\gamma_{\omega}}{\gamma_{\omega}(2\beta_{\omega}+d_{\omega}) + \alpha \beta_{\omega}}}. 
\]
This represents the rate at which $\hat{f}_{\omega}(x)$ estimates $f_{\omega}(x)$ on our high-probability event.  Further, for $\Omega \in \{0,1\}^d$ satisfying $\Omega_{\star} \subseteq \Omega \cup L(\Omega)$ and $x \in \mathcal{X}$, let  
\[ 
R^{\delta}_{\Omega}(x) := \max_{\omega \in \Omega} R^{\delta}_{\omega}(x). 
\]
Lemma~\ref{lem:hatC=BayesC} shows that $\hat{C}_{\Omega}$ agrees with $C^{\mathrm{Bayes}}$ on the event $E_1^\delta(x) \cap E_2^\delta(x)$, whenever $\eta$ is not too close to $1/2$.

\begin{lemma}\label{lem:hatC=BayesC}
Fix $d,n\in\IN$, $\mathcal{O} \subseteq \{0,1\}^d$, $o_1, \ldots, o_n \in \mathcal{O}$, $\Omega_{\star}, \Omega \in \mathcal{I}(\{0,1\}^d)\setminus\{\{\mathbf{0}_d\},\emptyset\}$, $ c_{\mathrm{E}} \in (0,1)$, $\boldsymbol{\gamma} \in [0,\infty)^d$, $C_{\mathrm{L}}\geq 1$, $\boldsymbol{\beta} \in (0,1]^{d}$, $C_{\mathrm{S}} \geq 1$, $\alpha \in [0,\infty)$, $C_{\mathrm{M}}\geq1$ and $\delta \in (0,1)$. Suppose that $\Omega_{\star} \subseteq \Omega\cup L(\Omega) \subseteq \mathcal{N}$. Fix $Q\in\mathcal{Q}'_{\mathrm{Miss}}$. Let 
\[
C := \max_{\omega \in \Omega}\Bigl\{2 \kappa_{\omega} |\Omega\cup L(\Omega)|\bigl(2^{3 + d_{\omega}}C_{\mathrm{B}}
C_{\mathrm{S}}+\sqrt{1/2}\bigr)\Bigr\}.
\]
If $x\in\mathcal{X}$ is such that
\begin{align*}
    |\eta(x)-1/2 |\geq C \cdot R_{\Omega}^{\delta}(x),
\end{align*}
then, on the event $E_1^\delta(x) \cap E_2^\delta(x)$, we have $\hat{ C}_{\Omega}(x)=C^{\mathrm{Bayes}}(x)$.
\end{lemma}

We are now in a position to state Proposition~\ref{prop:rate_if_Omega_correct}.

\begin{proposition}\label{prop:rate_if_Omega_correct}
Fix $d,n\in\IN$, $\mathcal{O}\subseteq\{0,1\}^d$, $o_1, \ldots, o_n \in \mathcal{O}$, $\Omega_{\star},\Omega \in \mathcal{I}(\{0,1\}^d)\setminus\{\{\mathbf{0}_d\},\emptyset\}$, $ c_{\mathrm{E}} \in (0,1)$, $\boldsymbol{\gamma} \in [0,\infty)^d$, $C_{\mathrm{L}}\geq 1$, $\boldsymbol{\beta} \in (0,1]^d$, $C_{\mathrm{S}} \geq 1$, $\alpha \in [0,\infty)$, $C_{\mathrm{M}} \geq 1$ and $\delta \in (0,1)$. Let $\delta' := \frac{\delta}{2} \max_{\omega \in \Omega}\{1/n_{\omega}^{1+\alpha}\}$.  Suppose that $\Omega_{\star} \subseteq \Omega \cup L(\Omega) \subseteq \mathcal{N}$ and $\min_{\omega \in \Omega} \gamma_{\omega} > \max_{\omega\in\Omega} \beta_{\omega}/d_{\omega}$. 
There exists a constant $K_{\Omega} \geq 2$ such that, for any $Q\in\mathcal{Q}'_{\mathrm{Miss}}$, we have
\begin{align}\label{eq:tailbound}
    \mathbb{P}_{Q}\biggl\{\mathcal{E}_P(\hat C_{\Omega}) > K_{\Omega} \cdot \log_+^{\frac{1+\alpha}{2}}(2|\mathcal N|/\delta') \cdot \max_{\omega \in \Omega} n_\omega^{-\frac{\beta_{\omega}\gamma_\omega(1+\alpha)}{\gamma_\omega(2\beta_{\omega}+d_\omega)+\alpha\beta_{\omega}}} \biggm| O_1=o_1, \ldots, O_n=o_n\biggr\}\!\leq\!\delta.
\end{align}
\end{proposition}
The idea behind the proof of this result is to integrate the bound on the regression function in Lemma~\ref{lem:hatC=BayesC} over $\mathcal{X}$. This involves ruling out an exceptional set on which the event $E_{1}^{\delta'}(x) \cap E_{2}^{\delta'}(x)$ fails to hold; see Lemma~\ref{lem:bound_prob_mu_big} 
in Section~\ref{sec:appendixUBproofs} 
for full details.  The proof of the upper bound in Theorem~\ref{thm:minmax_bounds} now follows by bounding the expectation of the excess risk using~\eqref{eq:tailbound}, and replacing $\Omega$ with $\Omega_{\star}$. The details are also deferred to Section~\ref{sec:appendixUBproofs}
.

\subsection{Outline of the proof of Theorem~\ref{thm:nonadaptiveUBnew}\label{subsec:proof2}}
In this subsection we provide an outline for the proof of Theorem~\ref{thm:nonadaptiveUBnew}. The full details of the proof are somewhat lengthy and are therefore given in Section~\ref{subsec:proofnonadaptiveUBnew}
in the supplementary material.  The key idea is to show that our estimator $\hat{\Omega}$ used in the HAM classifier in Algorithm~\ref{alg:nonadaptive} correctly identifies $\Omega_\star$ with high probability; see Proposition~\ref{lem:correct_thresholding} below. The proof of Theorem~\ref{thm:nonadaptiveUBnew} then follows by carefully combining this proposition with the upper bound in Theorem~\ref{thm:minmax_bounds}.

Before stating the proposition, we will use the constant $C_{\mathrm{B}}$ (which provides an upper bound on $\|f_{\omega}\|_\infty$ and only depends on $d$) from Proposition~\ref{prop:f_bounded}
, and define $\phi_{\omega} := \gamma_{\omega}(2\beta_{\omega} + d_{\omega}) + \alpha\beta_{\omega}$, $\kappa_{\omega,\circ} := 2\phi_{\omega} /(\beta_{\omega}\gamma_\omega)$ and $C_{\rho} := (32C^2_{\mathrm{B}}C_{\mathrm{L}})^{-\frac{1}{\min_{\omega\in \mathcal{N}}\gamma_\omega}}$.  Further, we will make use of a constant $\kappa_{\omega,\mathrm{T}}$, which only depends on $d_\omega$; see the proof of Lemma~\ref{lem:f_bound_at_training_points} 
in Section~\ref{subsec:proofnonadaptiveUBnew} 
for a precise definition. 

\begin{proposition}
\label{lem:correct_thresholding} 
Fix $d,\!n \in\mathbb{N}$, $\mathcal{O}\subseteq\{0,\!1\}^d$, $o_1, \ldots, o_n \in \mathcal{O}$, $\Omega_{\star} \in \mathcal{I}(\{0,1\}^d)\setminus\{\{\mathbf{0}_d\}, \emptyset\}$, $c_{\mathrm{E}} \in (0,1/4]$, $\boldsymbol{\gamma} \in [0, \infty)^d$, $C_{\mathrm{L}} \geq 1$,  $\boldsymbol{\beta} \in (0,1]^{d}$, $C_{\mathrm{S}} \geq 1$, $\alpha \in [0, \infty)$,  $C_{\mathrm{M}} \geq 1$ and $\delta \in (0,1)$. Suppose that $\Omega_{\star} \subseteq \mathcal{N}$ and $\min_{\omega \in \Omega_{\star}} \gamma_\omega > \max_{\omega \in \Omega_{\star}} \beta_{\omega}/d_{\omega}$. Suppose further that $(n_\omega)_{\omega \in \mathcal{N}}$ satisfies 
\[
n_{\omega} \geq 
\Bigl(\frac{2^7 C_{B}^2 C_{\mathrm{L}}}{ c_{\mathrm{E}}} \Bigr)^{\kappa_{\omega,\circ} (1\vee \frac{1}{2\alpha})} \Bigl(\frac{20 C_{\mathrm{S}} \kappa_{\omega,\mathrm{T}} \kappa_{\omega,\circ} \cdot \log_{+}(2|\mathcal{N}|n_{\omega}/\delta)}{C_{\rho}^{1 \vee \gamma_{\omega}}} \Bigr)^{2\kappa_{\omega,\circ}} 
\]
for all $\omega \in \Omega_{\star}$, and 
\[
n_{\omega} \geq \Bigl(16\kappa_{\omega,\circ} \kappa_{\omega,T}^2  \cdot \log(2|\mathcal{N}|n_{\omega}/\delta) \Bigr)^{\kappa_{\omega,\circ}},
\]
for all $\omega \in \mathcal{N} \cap U(\Omega_{\star})$. 
Then we have 
\[
\sup_{Q \in \mathcal{Q}_{\mathrm{Miss}}^+} \mathbb{P}_Q\Bigl(\hat{\Omega} \neq \Omega_{\star} \Bigm| O_1 = o_1, \ldots, O_n = o_n\Bigr) \leq 4\delta .
\]
\end{proposition}
Proposition~\ref{lem:correct_thresholding} shows that, under a sample size condition, we have $\hat{\Omega} = \Omega_{\star}$ with high probability, uniformly over the class $\mathcal{Q}_{\mathrm{Miss}}^{+}$. For the proof of this result, given in Section~\ref{subsec:proofnonadaptiveUBnew}
, we introduce several additional high-probability events, which control how well we can estimate $\sigma_{\omega}^2$ as well as the error of $\hat{f}$ at each of the training data points simultaneously.  On the intersection of these events, we are then able to show that $\hat{\sigma}^2_\omega$ is on the correct side of the threshold $\tau_{\omega}$ for each relevant $\omega \in \mathcal{N}$. The sample size condition ensures that we have enough signal in our training data to detect that $f_\omega$ is nonzero.  Note, however, that in Theorem~\ref{thm:nonadaptiveUBnew} we are in fact able to remove the condition on $n_{\omega}$ for $\omega \in \Omega_{\star}$ and simplify the condition for $\omega \in  \mathcal{N} \cap U(\Omega_{\star})$.

\section*{Funding} 
The research of TS and TIC is supported by an Engineering and Physical Sciences Research Council (EPSRC) New Investigator Award (EP/V002694/1), and the research of TBB is supported by an EPSRC New Investigator Award (EP/W016117/1).

\setcounter{section}{0}
\setcounter{equation}{0}
\setcounter{algorithm}{0}
\setcounter{figure}{0}
\setcounter{theorem}{0}
\setcounter{definition}{0}
\setcounter{example}{0}
\def\theequation{S\arabic{equation}}
\def\thesection{S\arabic{section}}
\def\thetheorem{S\arabic{theorem}}
\def\thefigure{S\arabic{figure}}
\def\thedefinition{S\arabic{definition}}
\def\theexample{S\arabic{example}}
\def\theremark{S\arabic{remark}}
\def\thealgorithm{S\arabic{algorithm}}

\newpage

\section{Additional results and proofs of the claims in Section~\ref{sec:statistical_setting}
\label{sec:proofsfor2}}

\subsection{Intuition behind Definition~\ref{def:missingness} 
and proofs of the claim in the examples\label{sec:intuitionfor1}}

Our first result below provides additional understanding of the class $\mathcal{Q}_{\mathrm{Miss}}(\mathcal{O})$ in Definition~\ref{def:missingness}
, based on conditional independence properties. 

\begin{proposition}\label{prop:missingness}
Fix $Q \in \mathcal{Q}$ and suppose that $(X,Y,O) \sim Q$. Let $\mathcal{O}_{Q} := \{o \in \{0,1\}^d : \mathbb{P}_Q(O = o) > 0\}$. If, for all $\omega \in \{0,1\}^d$, we have that $Y$ and $O$ are conditionally independent given $X^\omega$, then $Q \in \mathcal{Q}_{\mathrm{Miss}}(\mathcal{O}_Q)$.
\end{proposition}
\begin{proof}
Fix $o \in \mathcal{O}_Q$ and $\omega\preceq o$, then by Bayes' Theorem and the assumed conditional independence we have
\begin{align*}
    \mathbb{P}_Q(Y = 1 | X^{\omega} = x^{\omega}, O = o) &= \frac{\mathbb{P}_Q( O = o| Y = 1, X^{\omega} = x^{\omega}) \mathbb{P}_Q(Y = 1 | X^{\omega} = x^{\omega})}{\mathbb{P}_Q( O = o| X^{\omega} = x^{\omega}) }\\
    &= 
\mathbb{P}_Q(Y = 1 | X^{\omega} = x^{\omega}) 
\end{align*}
whenever $\mathbb{P}_Q( O = o| X^{\omega} = x^{\omega})>0$. On the other hand, if $\mathbb{P}_Q( O = o| X^{\omega} = x^{\omega})=0$, then the right hand side in~\eqref{eq:missingness} 
is conditioning on a null set under the distribution $Q$. Thus we can define a version of regression function satisfying 
\[
\mathbb{P}_Q(Y = 1 | X^{\omega} = x^{\omega}, O = o) = \mathbb{P}_Q(Y=1 | X^\omega = x^\omega),
\]
for all $o \in \mathcal{O}_Q$ and $\omega\preceq o$ and $x\in\mathrm{supp}(\mu)$, without changing the distribution $Q$. This completes the proof.
\end{proof}

\begin{lemma}\label{lem:missingness}
Fix $S_1, S_2 \subseteq [d]$ with $S_1\cap S_2 = \emptyset$ and $Q\in\mathcal{Q}$.  Suppose that $(X,Y,O)\sim Q$, and let $\mathcal{O}_{Q} := \{o \in \{0,1\}^d : \mathbb{P}_Q(O = o) > 0\}$.  Suppose further that the regression function $\eta(x)$ depends on $x = (x_1,\ldots,x_d) \in \mathbb{R}^d$ only via $x_{S_1} := \{x_j : j \in S_1\}$, and that $O \mid \{X = x, Y = y\}$ depends on $(x,y) \in \mathbb{R}^d \times \{0,1\}$ only via $x_{S_2} := \{x_j : j \in S_2\}$. If $X_{S_1}$ and $X_{S_2}$ are independent, then $Q \in \mathcal{Q}_{\mathrm{Miss}}(\mathcal{O}_Q)$.
\end{lemma}
\begin{proof}
First write $\omega' = (\omega'_1, \ldots, \omega'_d)$ with $\omega'_j = 1$ if $j\in S_1$ and $\omega'_j =0$, otherwise. Similarly write $\omega'' = (\omega''_1, \ldots, \omega''_d)$ with $\omega''_j = 1$ if $j\in S_2$ and $\omega''_j =0$ otherwise.   

Fix $\omega = (\omega_1,\ldots, \omega_d) \in \{0,1\}^d$.  We then have that for $x \in \mathrm{supp}(\mu)$
\begin{align*}
 \mathbb{P}(Y=1 \mid X^\omega = x^\omega ) &= \int_{\mathbb{R}^d} \mathbb{P}(Y=1 \mid X = x) \, d\mu_{\mathbf{1}_d-\omega}(x) 
 \\ & =  \int_{\mathbb{R}^d} \mathbb{P}(Y=1 \mid X^{\omega'} = x^{\omega'}) \, d\mu_{\mathbf{1}_d-\omega}(x) 
 \\ & = \int_{\mathbb{R}^d} \mathbb{P}(Y=1 \mid X_{S_1} = x_{S_1}) \, d\mu_{(\mathbf{1}_d-\omega) \wedge \omega'}(x)
\end{align*}
Similarly, for $o \in \mathcal{O}_Q$ with $o\preceq\omega$ and $x \in \mathrm{supp}(\mu)$, we have that 
\[
     \mathbb{P}(O=o \mid X^\omega = x^\omega ) =  \int_{\mathbb{R}^d} \mathbb{P}(O=o \mid X_{S_2} = x_{S_2}) \, d\mu_{(\mathbf{1}_d-\omega) \wedge \omega''}(x).
\]

Now, since $X_{S_1}$ and $X_{S_2}$ are independent, we have that $X^{\omega'}$ and $X^{\omega''}$ are independent and hence $X^{(\mathbf{1}_d-\omega) \wedge \omega'}$ and $X^{(\mathbf{1}_d-\omega) \wedge \omega''}$ are independent. It follows that, for $o \in \mathcal{O}_Q$ with $o\preceq\omega$ and $x \in \mathrm{supp}(\mu)$, we have
\begin{align*}
    &\mathbb{P}(Y=1,O=o \mid X^\omega = x^\omega ) \\
    &= \int_{\mathbb{R}^d} \mathbb{P}(Y=1,O=o \mid X = x )  \, d\mu_{\mathbf{1}_d-\omega}(x)  \\
    &= \int_{\mathbb{R}^d} \mathbb{P}(Y=1 \mid X = x )
    \mathbb{P}(O=o \mid Y=1, X = x) \, d\mu_{\mathbf{1}_d-\omega}(x)  \\
    & = \int_{\mathbb{R}^d} \mathbb{P}(Y=1 \mid  X_{S_1}=x_{S_1}) \mathbb{P}(O=o \mid  X_{S_2}=x_{S_2}) \, d\mu_{\mathbf{1}_d-\omega}(x) \\
   & = \int_{\mathbb{R}^d} \mathbb{P}(Y=1 \mid  X_{S_1}=x_{S_1}) \mathbb{P}(O=o \mid  X_{S_2}=x_{S_2}) \, d\mu_{\mathbf{1}_d-\omega}(x) \\
    & = \int_{\mathbb{R}^d} \int_{\mathbb{R}^d} \mathbb{P}(Y=1 \mid  X_{S_1}\!=\!x_{S_1}) \mathbb{P}(O=o \mid  X_{S_2}\!=\!z_{S_2}) \, d\mu_{(\mathbf{1}_d-\omega) \wedge \omega'}(x^{\omega'}) \, d\mu_{(\mathbf{1}_d-\omega) \wedge \omega''}(z^{\omega''}) \\
    & = \int_{\mathbb{R}^d} \int_{\mathbb{R}^d} \mathbb{P}(Y=1 \mid  X_{S_1}\!=\!x_{S_1}) \mathbb{P}(O=o \mid  X_{S_2}\!=\!z_{S_2}) \, d\mu_{(\mathbf{1}_d-\omega) \wedge \omega'}(x) \, d\mu_{(\mathbf{1}_d-\omega) \wedge \omega''}(z) \\
    & = \Bigl\{ \int_{\mathbb{R}^d} \mathbb{P}(Y\!=\!1 \mid  X_{S_1}\!=\!x_{S_1}) d\mu_{(\mathbf{1}_d-\omega) \wedge \omega'}(x) \Bigr\} \Bigl\{ \int_{\mathbb{R}^d} \mathbb{P}(O\!=\!o \mid   X_{S_2}\!=\!z_{S_2}) d\mu_{(\mathbf{1}_d-\omega) \wedge \omega''}(z) \Bigr\}.
\end{align*}
It follows that
\[
    \mathbb{P}(Y=1 \mid O=o, X^\omega = x^\omega ) = \frac{\mathbb{P}(Y=1,O=o \mid X^\omega = x^\omega )}{\mathbb{P}(O=o \mid X^\omega = x^\omega )} = \mathbb{P}(Y=1 \mid X^\omega = x^\omega )
\]
as required.
\end{proof}

We now prove the claims made in the examples in Section~\ref{sec:statistical_setting}
, as well as provide additional detail on the particular classes the corresponding distributions belong to. 

\begin{proof}[Proof of the claim in the Example~\ref{ex:missingness1}
] 
We have $\IP_Q(O=o)>0$ for any $o\in\{0,1\}^d$, and thus $Q \in \mathcal{Q}_{\mathrm{Miss}}(\{0,1\}^d)$ by Proposition~\ref{prop:missingness}, since $(X,Y)$ and $O$ are independent.  In the special case with $u = (a,0)^T$ for some $a > 0$, and $\Sigma = I$ (in Setting~1 in Section~\ref{sec:numericalresults} 
we have $a=2^{1/2})$), we will in fact show that 
\[
Q \in \mathcal{Q}_{\mathrm{Miss}}^+ \Bigl( \bigl\{(1,0)^T\bigr\}, 0.074, (\gamma_1,1)^T, C_{\mathrm{L}}, \mathbf{1}_2, 1\vee(a/2), 1, 4/a , \{0,1\}^2 \Bigr),
\]
for any $\gamma_1 \in [0,1)$ and sufficiently large $C_{\mathrm{L}}$. 

Indeed, for $x = (x_1,x_2)^T\in \mathbb{R}^2$, we have
\[
\eta(x) = \frac{e^{-ax_1}} {e^{-ax_1 } + e^{ax_1}  } = \frac{1}{2} + \frac{1}{2} \tanh(-ax_1).
\]
Therefore, in the decomposition in~\eqref{eq:etadecomposition} 
the only nonzero term is $f_{(1,0)^T}(x)=\tanh(-ax_1)/2$, and all other $f_\omega$s are zero. Further, 
\begin{align*}
\sigma_{(1,0)^T}^2 &= \mathbb{E}_Q\{\tanh^2(aX_1)/4\} = \int_{-\infty}^{\infty} \frac{1}{4\sqrt{2\pi}} e^{-x_1^2/2 - a^2/2} \cosh(ax_1) \tanh^2(ax_1) \, dx_1
\\ & \geq \int_{-\infty}^{\infty} \frac{1}{4\sqrt{2\pi}} e^{-x_1^2/2 - a^2/2-a^2x_1^2/2} \sinh^2(ax_1) \, dx_1 = \frac{e^{-a^2/2}\{e^{2a^2/(1+a^2)} - 1\}}{8(1+a^2)^{1/2}},
\end{align*} 
where we used the inequality $\cosh(ax)\leq e^{a^2x^2/2}$. Thus, if for example $a=2^{1/2}$, then $Q \in \mathcal{Q}_{\mathrm{E}}(\{(1,0)^T\}, 0.074, \{0,1\}^2)$. 

Turning now to the parameters in Definition~\ref{def:lowerdensity}
, first consider $\omega = (0,1)^T$, in this case the marginal distribution of $X_2$ is a standard Gaussian random variable.  Hence, for $x = (x_1,x_2)^T \in \{0\} \times [2,\infty)$, we have by convexity
\begin{align*}
    \rho_{\mu_\omega,1}(x) &= \inf_{r \in (0,1)} \frac{\Phi(x_2+r) - \Phi(x_2-r)}{r} = 2\phi(x_2)
\end{align*}
where $\Phi$ and $\phi$ denote the cumulative distribution function and density, respectively, of a standard Gaussian random variable.   It follows that, for $\xi \leq \xi_0 := 2\phi(2)$, we have 
\[
\mu_{\omega} \bigl( \{x \in \mathbb{R}^2 : \rho_{\mu_\omega,1}(x) < \xi\} \bigr) \leq \mu_{\omega} \bigl( \{x \in \mathbb{R}^2 : 2\phi(x_2) < \xi\} \bigr) = 2\Phi(-b) \leq e^{-b^2/2} = \frac{\sqrt{\pi}\xi}{\sqrt{2}} ,
\]
where $b := \Bigl(2\log(\frac{\sqrt{2}}{\sqrt{\pi}\xi}) \Bigr)^{1/2}$.  On the other hand, for $\xi > \xi_0$, we have $\mu_{\omega} \bigl( \{x \in \mathbb{R}^2 : \rho_{\mu_\omega,1}(x) < \xi\} \bigr) \leq 1 \leq (\xi/\xi_0)$.  We conclude that we may take any $\gamma_2 \in[0,1]$.  Now, for $\omega=(1,0)^T$, the marginal distribution of $X_1$ is a mixture of two Gaussian random variables with means $-a$ and $a$, respectively.  In this case, for $x = (x_1,x_2)^T \in \mathbb{R} \times \{0\}$ with $|x_1| > a+2$ we have 
\[
\rho_{\mu_\omega,1}(x) = \phi(x_1-a) + \phi(x_1+a) \geq  2\phi(|x_1|+a).
\]
Then similarly to above, we have  for $\xi \leq \xi_1 := 2\phi(2a+2)$, we have 
\begin{align*} 
\mu_{\omega} \bigl( \{x \in \mathbb{R}^2 : \rho_{\mu_\omega,1}(x) < \xi\} \bigr) &\leq \mu_{\omega} \bigl( \{x \in \mathbb{R}^2 : 2\phi(|x_1| + a) < \xi\} \bigr) \\ & = \Phi(-(c+a)) + \Phi(-(c-a)) \leq 2\Phi(-(c-a)) 
\\ & \leq e^{-(c-a)^2/2} = e^{-(c+a)^2/2 + 2ac} 
\\ & \leq \frac{\sqrt{\pi}\xi}{\sqrt{2}} \exp\Bigl(2a \sqrt{2\log(\sqrt{2}/(\sqrt{\pi}\xi))} \Bigr),
\end{align*}
where $c := \Bigl(2\log(\frac{\sqrt{2}}{\sqrt{\pi}\xi}) \Bigr)^{1/2} - a$.  On the other hand, for $\xi > \xi_1$, we have $\mu_{\omega} \bigl( \{x \in \mathbb{R}^2 : \rho_{\mu_\omega,1}(x) < \xi\} \bigr) \leq 1 \leq (\xi/\xi_1)^{\gamma_1}$, for any $\gamma_1 < 1$.  Finally, by \citet[Lemma~S8]{reeve2021adaptive}, we also deduce that
\[
\mu\bigl( \{x \in \mathbb{R}^2 : \rho_{\mu,2}(x) < \xi\} \bigr) \leq C_{\mathrm{L}}\xi^{\gamma_1}.
\]
for any $\gamma_1 < 1$ and some appropriate choice of $C_{\mathrm{L}}$.  We therefore deduce that $Q \in \mathcal{Q}_{\mathrm{L}}(\boldsymbol{\gamma}, C_{\mathrm{L}}, \{0,1\}^d)$, for any $\boldsymbol{\gamma} \in [0,1) \times [0,1]$ and the corresponding choice of $C_{\mathrm{L}}$.  

Regarding smoothness, for $x=(x_1,x_2)^T,x'=(x'_1,x'_2)^T\in\mathrm{supp}(\mu)$, we have 
\[
|f_{(1,0)^T}(x)-f_{(1,0)^T}(x')| = \frac{1}{2}|\tanh(-ax_1)-\tanh(-ax'_1)| \leq \frac{a}{2} |x_1-x'_1|,
\]
and thus $P_Q \in \mathcal{P}_{\mathrm{S}}(\mathbf{1}_2, 1 \vee (a/2))$ since all other $f_\omega$ are $0$.

Finally, for $t>0$, we have 
\begin{align*}
\mathbb{P}_P(|\eta(X) - 1/2| < t) &= 2\mathbb{P}_P\Bigl\{0 \leq \tanh(aX_1) < 2t\Bigr\} = 2\mathbb{P}_P\Bigl\{0 \leq X_1 < \frac{\tanh^{-1}(2t)}{a}\Bigr\} 
\\ & = \Phi\Bigl(2^{1/2} + \frac{\tanh^{-1}(2t)}{a}\Bigr) -  \Phi\Bigl(a - \frac{\tanh^{-1}(2t)}{a}\Bigr) 
\\ & \leq \frac{2\tanh^{-1}(2t)}{a} \leq \frac{4}{a}\cdot t.
\end{align*}
It follows that $P_Q \in \mathcal{P}_{\mathrm{M}}(1, 1\vee 4/a)$. 

\end{proof}

\begin{proof}[Proof of the claim in the Example~\ref{ex:missingness2}
] 
First, for $o\in\{(1,1,1,0)^T,(1,0,0,1)^T\}$, we have that $\IP_Q(O=o)>0$ and no other observation patterns can be observed in the training data. Since $Y$ and $O$ are conditionally independent given $X^{\omega}$ for all $\omega \in \{0,1\}^4$, we have $Q \in \mathcal{Q}_{\mathrm{Miss}}(\{(1,1,1,0)^T,(1,0,0,1)^T\})$ by Proposition~\ref{prop:missingness}.

In the special case used in Setting~2 in Section~\ref{sec:numericalresults}
, where for $x = (x_1,x_2,x_3,x_4)^T \in [0,1]^4$ we have
\[
\eta(x):= \frac{1}{2} +  \frac{(x_2-1/2)\cdot x_3^2}{2} + \frac{x_4-1/2}{2},
\]
the nonzero $f_{\omega}$ functions are 
\[
f_{e_2}(x) = \frac{(x_2-1/2)}{6}; \quad f_{e_2+e_3}(x)=\frac{(x_2-1/2)(x_3^2-1/3)}{2}; \quad f_{e_4}(x)=\frac{(x_4-1/2)}{2}.
\]

To see that $Y$ and $O$ are conditionally independent given $X^{\omega}$ in this case, observe that for any $\omega \in \{0,1\}^d$ and $o=(1,1,1,0)^T$ and writing $X = (X_1,X_2,X_3,X_4)^T$ we have
\begin{align*}
    \mathbb{P}\bigl(Y=1 , O=o \bigm| X^\omega = x^\omega\bigr) &=\mathbb{E}\Bigl\{  \mathbb{P}(Y=1, O=o \mid X) \Bigm| X^{\omega} = x^{\omega} \Bigr\} \\
    &=\mathbb{E}\Bigl\{ \mathbb{P}(Y=1\mid X) \cdot \mathbb{P}(O=o \mid X) \Bigm| X^{\omega} = x^{\omega} \Bigr\} \\
    &= \mathbb{E}\bigl\{\eta(X) \cdot X_1 \bigm| X^{\omega} = x^{\omega} \bigr\}  
    \\ & = \mathbb{E}\bigl\{\eta(X) \bigm| X^{\omega} = x^{\omega} \bigr\} \cdot \mathbb{E}\bigl(X_1 \bigm| X^{\omega} = x^{\omega} \bigr),
\end{align*}
since $\eta(x)$ does not depend on $x_1$.  If $o=(1,0,0,1)^T$, then similarly for any $\omega \in \{0,1\}^d$ we have that $\mathbb{P}(Y=1 , O=o \mid X^\omega = x^\omega) =\mathbb{E}\{\eta(X) \mid X^{\omega} = x^{\omega} \} \cdot \mathbb{E}(1-X_1 \mid X^{\omega} = x^{\omega})$.

We now show that $Q \in \mathcal{Q}_{\mathrm{E}}$ for appropriate choices of the parameters.  For $\omega = (\omega_1,\omega_2, \omega_3,\omega_4)^T \in \{0,1\}^4$, we have $\mu_{\omega} = \otimes_{j=1}^4 \bigl\{U([0,1]) \mathbbm{1}_{\{\omega_j = 1\}} + \delta_0 \mathbbm{1}_{\{\omega_j = 0\}}\bigr\}$, where $\delta_0$ denotes a Dirac mass at $0$. Further, for $\omega \preceq (1,1,1,0)^T$, we have 
\begin{align*} 
\mu_{\omega | (1,1,1,0)^T} & = \bigl\{\mathrm{Beta}(2,1) \mathbbm{1}_{\{\omega_1 = 1\}} + \delta_0 \mathbbm{1}_{\{\omega_1 = 0\}}\bigr\} \otimes \bigl\{U([0,1]) \mathbbm{1}_{\{\omega_2 = 1\}} + \delta_0 \mathbbm{1}_{\{\omega_2 = 0\}}\bigr\} 
\\ & \hspace{30pt} \otimes \bigl\{U([0,1]) \mathbbm{1}_{\{\omega_3 = 1\}} + \delta_0 \mathbbm{1}_{\{\omega_3 = 0\}}\bigr\} \otimes \delta_{0}
\end{align*}
and for $\omega \preceq (1,0,0,1)^T$, we have \[
\mu_{\omega | (1,0,0,1)^T} =\bigl\{\mathrm{Beta}(1,2) \mathbbm{1}_{\{\omega_1 = 1\}} + \delta_0 \mathbbm{1}_{\{\omega_1 = 0\}}\bigr\} \otimes \delta_{0}^{\otimes 2} \otimes \bigl\{U([0,1]) \mathbbm{1}_{\{\omega_4 = 1\}} + \delta_0 \mathbbm{1}_{\{\omega_4 = 0\}}\bigr\} .
\]
To see this, note that for a measurable set $A \subseteq [0,1]^3 $ and $\omega = o = (1,1,1,0)^T$, we let $\tilde{A}= A \times \{0\}$ and have
\begin{align*}
    \mu_{\omega | o}(\tilde{A}) & = \IP_Q(X^\omega\in \tilde{A}\mid O=o) 
    = 2\cdot \IP_Q(O=o \mid X^\omega\in \tilde{A}) \cdot \IP_Q(X^\omega\in \tilde{A}) \\
    &= 2\cdot \int_{A} x_1  dx_1 \, dx_2 \, dx_3,
\end{align*}
which proves the claim for $\mu_{(1,1,1,0)^T|(1,1,1,0)^T}$, the other cases follows by similar calculations.

Now focusing on the nonzero $\sigma_{\omega}^2$ for $\omega \in \Omega_{\star}$, we have 
\[
\sigma_{(0,1,1,0)^T}^2 =  \mathbb{E}\Bigl\{\frac{(X_2 - 1/2)^2 (X_3^2 - 1/3)^2}{4} \Bigm| O = (1,1,1,0)^T\Bigr\} = \frac{1}{540} \approx 0.002,
\]
\[
\sigma_{(0,1,0,0)^T}^2 =  \mathbb{E}\Bigl\{\frac{(X_2 - 1/2)^2}{36} \Bigm| O = (1,1,1,0)^T\Bigr\} = \frac{1}{432}  \approx 0.002
\]
and
\[
\sigma_{(0,0,0,1)^T}^2 =  \mathbb{E}\Bigl\{\frac{(X_4 - 1/2)^2}{4} \Bigm| O = (1,0,0,1)^T\Bigr\} = \frac{1}{48}  \approx 0.02
\]
therefore we have $Q\in \mathcal{Q}_{\mathrm{E}}(\{(0,1,1,0)^T,(0,0,0,1)^T\},1/540,\{(1,1,1,0)^T, (1,0,0,1)^T\})$.

Turning to Definition~\ref{def:lowerdensity}
, consider first the lower density of $\mu_{(1,0,0,0)^T|(1,1,1,0)^T}$. For $x_1 \in [0,1]$ and $r \in (0,x_1\wedge (1-x_1))$, we have 
\[
2\int_{x_1-r}^{x_1+r} x \, dx =  \bigl\{(x_1+r)^{2}-(x_1-r)^{2}\bigr\}= 4x_1r,
\]
and for $0 \leq x_1 < r$ we have 
\[
2\int_{0}^{(x_1+r)\wedge 1} x \, dx  = (x_1+r)^2  \wedge 1.
\]
For $0<(1-x_1)<r$ we have that
\[
2\int_{(x_1-r)\vee 0}^{1} x \, dx  =  1 - \{(x_1-r) \vee 0\}^2.
\]
It follows that, for $x_1 \in [0,1]$, we have 
\[
\rho_{\mu_{(1,0,0,0)^T|(1,1,1,0)^T},1}\bigl((x_1,0,0,0)^T\bigr) = \inf_{r\in(0,1)} \frac{2\int_{(x_1-r)\vee 0}^{(x_1+r)\wedge1} x\, dx}{r} = 4x_1 \wedge 1
\]
A similar calculation shows that
\[
\rho_{\mu_{(1,0,0,0)^T|(1,0,0,1)^T},1}\bigl((x_1,0,0,0)^T\bigr) = \inf_{r\in(0,1)} \frac{2\int_{(x_1-r)\vee 0}^{(x_1+r)\wedge1} x\, dx}{r} = \{4(1- x_1)\} \wedge 1,
\]
for $x_1 \in [0,1]$.  Therefore, for $\xi < 1$, we have 
\[
\mu_{(1,0,0,0)^T} \Bigl( \Bigl\{x \in [0,1]^4 : \min_{o \in \{(1,1,1,0)^T,(1,0,0,1)^T \}  } \rho_{\mu_{(1,0,0,0)^T|o},1}(x) < \xi \Bigr\} \Bigr) = \xi/2 \leq \xi,
\]
whereas, for $\xi \geq 1$, we have 
\[
\mu_{(1,0,0,0)^T} \Bigl( \Bigl\{x \in [0,1]^4 : \min_{o \in \{(1,1,1,0)^T,(1,0,0,1)^T \}  } \rho_{\mu_{(1,0,0,0)^T|o},1}(x) < \xi \Bigr\} \Bigr) = 1 \leq \xi. 
\]
For $\omega \in \{e_2, e_3, e_4\}$, and $o \in \{(1,1,1,0)^T, (1,0,0,1)^T\}$ with $\omega \preceq o$, we have $\mu_{\omega}=\mu_{\omega | o}$ and thus $\rho_{\mu_\omega,1}(\cdot) = 1$ on the support of $X^{\omega}$, and
\[
\mu_{\omega} \Bigl( \Bigl\{x \in [0,1]^4 : \min_{o \in \{(1,1,1,0)^T,(1,0,0,1)^T \}  } \rho_{\mu_{\omega|o},1}(x) < \xi \Bigr\} \Bigr) = \mathbbm{1}_{\{\xi > 1\}} \leq \xi^{\gamma}
\]
for any choice of $\gamma\geq0$. For $\omega=o=(1,1,1,0)^T$, we have that $\mu_{(1,1,1,0)^T | (1,1,1,0)^T} = \mathrm{Beta}(2,1)  \otimes U([0,1])^{\otimes 2} \otimes \delta_{0}$. Thus, for $x=(x_1,x_2,x_3,0)^T \in [0,1]^3 \times \{0\}$, since $\{z \in \mathbb{R}^3 : \|z - (x_1,x_2,x_3)^T\|_{\infty} < r/\sqrt{3}\} \subseteq \{z \in \mathbb{R}^2 : \|z - (x_1,x_2,x_3)^T\| < r\}$, we have 
\[
\rho_{\mu_{(1,1,1,0)^T | (1,1,1,0)^T},3}(x) \geq \inf_{r\in (0,1)} \frac{\nu((x_1-r/\sqrt{3},x_1+r/\sqrt{3}))  \cdot r^2/3}{ r^3} \geq \frac{(4x_1\wedge 1)}{3\sqrt{3}},
\]
where $\nu$ denotes the measure of a $\mathrm{Beta}(2,1)$ distribution. Thus, for $\xi < 1/(3\sqrt{3})$, we have 
\[
\mu_{(1,1,1,0)^T} \Bigl( \Bigl\{x \in [0,1]^4 : \rho_{\mu_{(1,1,1,0)^T|(1,1,1,0)^T},3}(x) < \xi \Bigr\} \Bigr) \leq \frac{2\xi}{3^{3/2}}.  
\]
For $\xi\geq 1/(3\sqrt{3})$ we have
\[
\mu_{(1,1,1,0)^T} \Bigl( \Bigl\{x \in [0,1]^4 : \rho_{\mu_{(1,1,1,0)^T|(1,1,1,0)^T},2}(x) < \xi \Bigr\} \Bigr) \leq 1 \leq 3^{3/2} \xi.  
\]
For $\omega  \prec o=(1,1,1,0)^T$, and $\omega  \preceq o=(1,0,0,1)^T$ similar calculations show that~\eqref{eq:lowerdensity} 
holds when $\boldsymbol{\gamma} = (1,\gamma_2,\gamma_3,\gamma_4) \in \{1\} \times [0,\infty)^3$.  Therefore 
\[
Q\in\mathcal{Q}_{\mathrm{L}}\bigl((1,\gamma_2,\gamma_3,\gamma_4)^T,3^{3/2},\{(1,1,1,0)^T, (1,0,0,1)^T\}\bigr)
\]
for any $\gamma_2,\gamma_3, \gamma_4 \in[0,\infty)$. Further, similar (and slightly simpler) calculations confirm that in fact $Q$ belongs to the $\mathcal{Q}_{\mathrm{L}}^+$ class with the same parameters.

In terms of smoothness and the margin parameter, we have $P_Q \in \mathcal{P}_{\mathrm{S}}(\mathbf{1}_4, 1)$ and 
\begin{align*}
\mathbb{P}_P&\Bigl(|\eta(X) - 1/2| < t\Bigr) \\
&= \int_{0}^{1} \int_{0}^{1} \int_{0}^{1} \mathbb{P}_P\Bigl(|\eta(X) - 1/2| < t \Bigm|  X_1 = x_1,X_2=x_2,X_3=x_3\Bigr) \, dx_1\, dx_2\, dx_3 \\
&\leq  \int_{0}^{1} \int_{0}^{1} \int_{0}^{1} \int_{-2t-1/2-(x_2-1/2)x_3^2}^{2t-1/2-(x_2-1/2)x_3^2}\, dx_4 \, dx_1\, dx_2\, dx_3  = 4t,
\end{align*}
thus $P_Q \in \mathcal{P}_{\mathrm{M}}(1, 4)$. 
\end{proof}

\begin{proof}[Proof of the claim in the Example~\ref{ex:missingness3}
] 
We have $\IP_Q(O=o)>0$ for any $o\in\{0,1\}^2$, and further $\IP_Q(O=o|X^\omega=x^\omega)>0$ for any $o,\omega\in\{0,1\}^2$ and $x\in[0,1]^2$. 

First, for $x = (x_1,x_2)^T \in [0,1]^2$ with $x_2 \leq 1/2$, $\omega\in\{0,1\}^d$, we have that $\IP_Q( O = \mathbf{1}_{2} \mid   X^{\omega} = x^{\omega}, Y = 1) = 1/2 =  \IP_Q( O = \mathbf{1}_{2} \mid   X^{\omega} = x^{\omega}, Y = 0)$, and thus $\IP_Q( O = \mathbf{1}_{2} \mid   X^{\omega} = x^{\omega}, Y = 1) = 1/2$. It follows that, for $x = (x_1,x_2)^T \in [0,1]^2$ with $x_2 \leq 1/2$ and $\omega \in \{0,1\}^2$,
\begin{align*}
\IP_Q(Y = 1 \mid X^{\omega} = x^{\omega},  O = \mathbf{1}_{2}) &= \frac{\IP_Q( O = \mathbf{1}_{2} \mid Y = 1, X^{\omega} = x^{\omega}) \cdot \IP_Q(Y = 1 \mid X^{\omega} = x^{\omega})}{\IP_Q( O = \mathbf{1}_{2} \mid X^{\omega} = x^{\omega}) }  \\
&= \frac{1/2 \cdot \IP_Q(Y = 1 \mid X^{\omega} = x^{\omega})}{1/2} = \IP_Q(Y = 1 \mid X^{\omega} = x^{\omega}) .
\end{align*}
Similarly,  for $x = (x_1,x_2)^T \in [0,1]^2$ with $x_2 \leq 1/2$, $r \in \{0,1\}$, and $o,\omega\in\{0,1\}^2\setminus\{\mathbf{1}_{2}\}$ with $\omega\preceq o$, we have that $\IP_Q( O = o \mid Y = r,  X^{\omega} = x^{\omega}) = 1/6$.  Therefore~\eqref{eq:missingness} 
holds for $x = (x_1,x_2)^T \in [0,1]^2$ with $x_2 \leq 1/2$. Similar calculations show that~\eqref{eq:missingness} 
also holds $x = (x_1,x_2)^T \in [0,1]^2$ with $x_2> 1/2$, and we deduce that $Q \in \mathcal{Q}_{\mathrm{Miss}}(\{0,1\}^2)$. 

Regarding the class $\mathcal{Q}_{\mathrm{E}}$, we recall that $\eta(x) = 1/4 + x_1/2 + 1/4\cdot\cos(4\pi x_2)$ and thus the nonzero $f_\omega$ functions are
\[
f_{e_1}=\frac{x_1-1/2}{2};\qquad f_{e_2}=\frac{\cos(4\pi x_2)}{4}.
\]
We therefore have 
\begin{align*}
\sigma_{e_1}^2 = \min\Bigl\{ \mathbb{E}_{Q}\bigl\{f_{e_1}^2(X) \bigm| O= e_1 \bigr\} , \mathbb{E}_{Q}\bigl\{f_{e_1}^2(X) \bigm| O= \mathbf{1}_2 \bigr\} \Bigr\} = 1/48 \approx 0.02
\end{align*}
and
\begin{align*}
\sigma_{e_2}^2 = \min\Bigl\{ \mathbb{E}_{Q}\bigl\{f_{e_2}^2(X) \bigm| O= e_2 \bigr\} , \mathbb{E}_{Q}\bigl\{f_{e_2}^2(X) \bigm| O= \mathbf{1}_2 \bigr\} \Bigr\} = 1/32 \approx 0.03,
\end{align*}
and thus $Q\in\mathcal{Q}_{\mathrm{E}}(\{e_1,e_2\},1/48,\{0,1\}^2)$.

Further, in this example, all relevant lower densities are bounded below by a universal constant, and therefore $Q \in \mathcal{Q}_{\mathrm{L}}^+(\boldsymbol{\gamma}, C_{\mathrm{L}}, \{0,1\}^2)$, 
for any $\boldsymbol{\gamma}\in[0,\infty)^2$ and sufficiently large $C_{\mathrm{L}}\geq 1$ (depending on $\boldsymbol{\gamma}$ and the lower bound on the lower densities).  Finally, we have $P_Q \in \mathcal{P}_{\mathrm{S}}(\mathbf{1}_2, \pi)$ and $P_Q \in \mathcal{P}_{\mathrm{M}}(1, 8)$, since
\[
\mathbb{P}_{P}(|\eta(X) -1/2| \leq t ) \leq \int_{0}^{1} \int_{1/2- 4t - \cos(4\pi x_2)}^{1/2+4t - \cos(4\pi x_2)} \, dx_1 \, dx_2 \leq 8t. 
\]
\end{proof}

\subsection{Properties of the regression function and its anova decomposition\label{sec:propertiesofeta}}

\begin{proposition}\label{prop:eta_decomposition}Fix $Q\in\mathcal{Q}$. For any $\omega \in \{0,1\}^d$, we have for the conditional regression functions $\eta_\omega(x):=\mathbb{P}_P(Y = 1 \mid X^\omega = x^\omega)$ that
\begin{align*}
\eta_{\omega}(\cdot) = \frac{1}{2} + \sum_{\omega' \preceq \omega} f_{\omega'}(\cdot).
\end{align*}
In particular, $\eta(\cdot) = \frac{1}{2} + \sum_{\omega\in\{0,1\}^d} f_{\omega}(\cdot)$.
\end{proposition}
\begin{proof}
Fix an observation pattern $\omega \in \{0,1\}^d$ and $x\in\IR^d$, and recall that 
\[
f_{\omega}(x) = \mathbb{E}_P\Bigl\{\eta(X)  - \frac{1}{2} - \sum_{\omega'\prec\omega}f_{\omega'}(X) \Bigm| X^{\omega} = x^{\omega}\Bigr\} .
\]
It follows that
\begin{align*}
    \eta_\omega(x)&=\IE_P(Y \mid X^\omega=x^\omega)
    =\IE_Q\bigl\{\eta(X) \mid X^\omega=x^\omega\bigr\} \\
    &= \IE_P\Bigl\{\frac{1}{2} + \sum_{\omega'\prec \omega} f_{\omega'}(X) \Bigm| X^\omega=x^\omega\Bigr\} + f_{\omega}(x)\\
    &= \frac{1}{2} + \sum_{\omega'\preceq\omega}f_{\omega'}(x^\omega)= \frac{1}{2} + \sum_{\omega'\preceq\omega}f_{\omega'}(x).
\end{align*}
This completes the proof. 
\end{proof}

\begin{proposition}\label{prop:f_bounded}
Fix $d\in\IN$. There exists a constant $C_\mathrm{B} \geq 1$ only depending on $d$ such that for any $x\in\IR^d$, any $P \in \mathcal{P}$, and the estimators $\hat{f}_{\omega}$ for $\omega \in \{0,1\}^d$ defined in Section~\ref{sec:algorithmUB} 
we have $|f_\omega(x)| \leq C_\mathrm{B}$, $|\hat f_\omega(x)|\leq C_\mathrm{B}$, and $|\hat f_\omega(x)-f_\omega(x)|\leq 2C_\mathrm{B}$, for all $\omega \in \{0,1\}^d$. In particular, we may take $C_\mathrm{B}$ to be the $d$th ordered Bell number. 
\end{proposition}
\begin{proof}
We will in fact show that $\|f_{\omega}\|_{\infty}$  and $\|\hat{f}_{\omega}\|_{\infty}$ are bounded by the $d_{\omega}$th ordered Bell number $B_{d_\omega} := \frac{1}{2} \sum_{j=0}^{\infty} \frac{j^{d_\omega}}{2^j}$, which are increasing and satisfy the recurrence relation that $B_{0} = B_{1} = 1$ and $B_{k} = 1 + \sum_{j=1}^{k-1} \binom{k}{j} B_{j}$, for $k \in \mathbb{N}$. 

Now, since $\eta(x) \in [0,1]$ for all $x$, we have $f_{\mathbf{0}_d} = \mathbb{E}_P\{\eta(X)\} -1/2 \in [-1/2,1/2]$ and $|f_{\mathbf{0}_d}| \leq 1/2 < B_0$. By induction over $d_\omega$, suppose that $\|f_{\omega'}\|_{\infty} \leq B_{d_{\omega'}}$ for $\omega' \preceq \omega$, then we have that
\begin{align*}
|f_{\omega}(x)| &= \Bigl|\mathbb{E}_P\Bigl\{\eta(X) -\frac{1}{2} - \sum_{\omega'\prec\omega}f_{\omega'}(X) \Bigm| X^{\omega} = x^{\omega}\Bigr\} \Bigr| \\
&\leq \mathbb{E}_P\Bigl\{|\eta(X) - 1/2 - f_{\mathbf{0}_d}| + \sum_{j=1}^{d_\omega-1} \sum_{\omega'\prec \omega : d_{\omega'} = j} |f_{\omega'}(X)| \Bigm| X^{\omega} = x^{\omega}\Bigr\} 
\\ & \leq 1 + \sum_{j=1}^{d_\omega-1} \binom{d_{\omega}}{j} B_{j} =: B_{d_\omega}.
\end{align*} 
This proves the result for $f_\omega(x)$. Similar steps using \eqref{def:f_tilde} give a bound for the estimator $\hat f_\omega$. Finally, $|\hat f_\omega(x)-f_\omega(x)|\leq|\hat f_\omega(x)|+|f_\omega(x)|\leq 2C_{\mathrm{B}}$.
\end{proof}

\section{Technical results and additional lemmas for the proof of the lower bound in Theorem~\ref{thm:minmax_bounds} 
\label{sec:appendixLBproofs}} 

The proof of the lower bound in Theorem~\ref{thm:minmax_bounds} 
makes use of several intermediate results.  The first in Section~\ref{subsec:Assouad} is a version of Assouad's Lemma (Lemma~\ref{lem:assouad}), modified to our missing data setting.  The remainder of the proof involves constructing a set of distributions belonging to $\mathcal{Q}_{\mathrm{Miss}}'$ for which we can apply Assouad's Lemma and obtain the desired lower bound. The proof is separated into the light and heavy tailed settings, and begins with the case that $\Omega_{\star} = \{\omega\}$. For the light tailed case,  the marginal constructions are given in Section~\ref{sec:lower_bound} 
in the main text. Section~\ref{sec:LB_small_r} establishes that the corresponding distributions belong to the class $\mathcal{Q}_{\mathrm{L}}(\boldsymbol{\gamma},C_{\mathrm{L}},\mathcal{O})$, Section~\ref{subsec:LBlighteta} defines the associated regression functions, and establishes that the corresponding distributions belong to the classes $\mathcal{Q}_{\mathrm{E}}(\{\omega\}, c_{\mathrm{E}},\mathcal{O})$, $\mathcal{P}_{\mathrm{S}}(\boldsymbol{\beta}, C_{\mathrm{S}})$ and $\mathcal{P}_{\mathrm{M}}(\alpha,C_{\mathrm{M}})$. The proofs for the special case when $\Omega_{\star}$ is a singleton in the light tailed case is completed in Section~\ref{subsec:LBproof_lighttails}. The corresponding arguments for the heavy tailed case are presented in Sections 5.1.2, \ref{sec:LB_big_r}, \ref{subsec:LBheavyeta} and~\ref{subsec:LBproof_heavytails}.  Finally, the proof for  general $\Omega_{\star}$, which builds on the earlier arguments, is presented in Section~\ref{subsec:LBproof_general}. 

\subsection{Assouad's Lemma for classification with missing data}
\label{subsec:Assouad}
\begin{lemma}[Assouad]\label{lem:assouad}
Fix $d,n, m\in\mathbb{N}$ and $\omega, o_1, \ldots, o_n \in\{0,1\}^d$ and $r > 0$. Let $\mathcal P$ be a set of distributions on $\mathbb{R}^d\times\{0,1\}$, let  $\Sigma=\{-1,1\}^{m}$, $u\in[0,1/(m\cdot2^{d_\omega})]$, and $\epsilon\in[0,1/4]$. Fix $z_1, \ldots, z_m \in (0,r)^d$ satisfying $z_i^\omega\neq z_j^\omega$ for $i\neq j$. Assume that the class of distributions $\{Q^\sigma :\sigma\in\Sigma\} \subseteq \mathcal{Q}$, which have associated marginal distributions $P^{\sigma}$ of $(X,Y)$, regression functions $\eta^\sigma:\IR^d\rightarrow[0,1]$ and each with $X$-marginal $\mu$, satisfy
\begin{enumerate}
    \item[(i)] $ 2^{d_{\omega}+5}n_\omega \epsilon^2u \leq1$, where $n_\omega=\sum_{i=1}^n\mathbbm{1}_{\{\omega\preceq o_i\}}$;
    \item[(ii)] we have $\mu(\{s^{\omega} \odot z_t^{\omega}\})=u$ for all $t\in[m]$ and $s\in\{-1,1\}^d$;
    \item[(iii)] for $\sigma=(\sigma_t)_{ t\in [m]}\in\Sigma$, we have $\eta^\sigma(s^{\omega} \odot z_t^{\omega}) = 1/2 + \bigl(\prod_{j \in [d] :(s^{\omega})_{j} =-1} (s^{\omega})_{j} \bigr)\cdot\sigma_t\epsilon$, for all $t\in[m]$ and $s\in\{-1,1\}^d$, in particular we have $\eta^\sigma(z^{\omega}_t) = 1/2 + \sigma_t \epsilon$, for  $t\in[m]$; 
    \item[(iv)] for $\sigma,\sigma'\in\Sigma$, we have $\eta^\sigma(x)=\eta^{\sigma'}(x)$ for all $x\in \mathrm{supp}(\mu) \setminus \bigcup_{t\in [m], s\in \{-1,1\}^d} \{s^{\omega}\odot z_t^{\omega}\}$;
    \item[(v)] we have $\Bigl(\mathrm{supp}(\mu) \setminus \bigcup_{t\in [m], s\in \{-1,1\}^d} \{s^{\omega}\odot z_t^{\omega}\} \Bigr) \bigcap (-r-1,r+1)^d = \emptyset$;
    \item[(vi)] for $\sigma \in \Sigma$, if $(X,Y,O) \sim Q^{\sigma}$, then $O \indep (X,Y)$.
\end{enumerate}
Then\begin{align}\label{eq:lb_assouad}
    \inf_{\hat C\in\mathcal C}\sup_{Q\in\mathcal{Q}}\mathbb{E}_{Q}\bigl\{\mathcal{E}_{P_Q}(\hat C) \bigm| O_1 = o_1, \ldots, O_n = o_n \bigr\}\geq 2^{d_\omega-1} m \epsilon u.
\end{align}
\end{lemma}

\begin{proof}
For $\sigma = (\sigma_t)_{t\in [m]}\in\Sigma$ and $t' \in [m]$, define $\sigma' \equiv \sigma(t') := (\sigma_t(t'))_{t\in [m]}\in\Sigma$, with $\sigma_t( t'):= -\sigma_t$ if $t' = t$, and $\sigma_t(t'):= \sigma_t$ otherwise. Then, for $s \in \{-1,1\}^d$, we have $\eta^{\sigma'}(s^{\omega} \odot z_{t'}^{\omega}) = 1 - \eta^\sigma(s^{\omega}\odot z_{t'}^{\omega})$ and $\eta^{\sigma'}(x) = \eta^\sigma(x)$ for $x \in \mathrm{supp}(\mu) \setminus \bigcup_{s\in \{-1,1\}^d}\{ s^{\omega}\odot z_{t'}^{\omega}\} $. Now, for $\omega_0 \in \{0,1\}^d$, let $P^{\sigma}_{\omega_0}$ denote the marginal distribution of $(X^{\omega_0}, Y)$ when $(X, Y) \sim P^{\sigma}$.  By \textit{(vi)}, we have that $P^{\sigma}_{\omega_0} = P^{\sigma}_{\omega_0 | o}$, for $o \in \{0,1\}^d$, where $P^{\sigma}_{\omega_0 | o}$ denotes the conditional distribution of  $(X^{\omega_0}, Y) | \{O = o
\}$ when $(X, Y, O) \sim Q^{\sigma}$.     For $\omega \preceq \omega_0$ and $x \in (-r-1,r+1)^d \cap \mathrm{supp}(\mu)$, we have that 
\[
\eta^{\sigma}_{\omega_0}(x) :=   \mathbb{E}_{P_{\omega_0}^{\sigma}}(Y | X^{\omega_0} = x^{\omega_0}) = \mathbb{E}_{P^{\sigma}}(Y | X^{\omega} = x^{\omega}) = \eta^{\sigma}(x) .
\]
Thus, for $\omega \preceq \omega_0$, we have
\begin{align*}
\mathrm{KL}(&P_{\omega_0}^\sigma,P_{\omega_0}^{\sigma'})\\
&:= \! \! \int_{\mathbb{R}^{d_{\omega_0}}} \Bigl[
\eta_{\omega_0}^\sigma(x^{\omega_0})\log\Bigl(\frac{\eta_{\omega_0}^\sigma(x)}{\eta_{\omega_0}^{\sigma'}(x)}\Bigr)
    +\{1-\eta_{\omega_0}^\sigma(x^{\omega_0})\}\log\Bigl(\frac{1-\eta_{\omega_0}^\sigma(x^{\omega_0})}{1-\eta_{\omega_0}^{\sigma'}(x^{\omega_0})}\Bigr) \Bigr] \, d\mu_{\omega_0}(x^{\omega_0})\\
    & \phantom{:}= \frac{1}{2^{d-d_{\omega}}}\sum_{s\in\{-1,1\}^d} \bigl\{2\eta^\sigma(s^\omega\odot z_{t'}^{\omega})-1\bigr\}\log\Bigl(\frac{\eta^\sigma(s^\omega\odot z_{t'}^{\omega})}{1-\eta^\sigma(s^\omega\odot z_{t'}^{\omega})}\Big)\cdot\mu( \{s^\omega\odot z_{t'}^{\omega}\})\\
    & \phantom{:}= 2^{d_\omega+1}\epsilon \sigma_{t'} \cdot \log\Bigl(\frac{1+2\epsilon\sigma_{t'}}{1-2\epsilon\sigma_{t'}}\Bigr)\cdot u  \leq \frac{2^{d_\omega+3} \epsilon^2 u }{1-2\epsilon} \leq 2^{d_\omega+4}\epsilon^2 u,
\end{align*}
where the penultimate inequality uses the inequality $\log a\leq a-1$ for $a\geq1$.  On the other hand, if $\omega \npreceq \omega_0$, then
\[
\eta^{\sigma}_{\omega_0}(x) =   \mathbb{E}_{P_{\omega_0}^{\sigma}}(Y | X^{\omega \wedge \omega_0} = x^{\omega \wedge \omega_0} ) = 1/2,
\]
and therefore $P_{\omega_0}^\sigma  = P_{\omega_0}^{\sigma'}$, and in particular $ \mathrm{KL}(P_{\omega_0}^\sigma,P_{\omega_0}^{\sigma'})=0$. Then, writing $\Pi_n^\sigma := \bigotimes_{i \in [n]} P^\sigma_{o_i}$, by Pinsker's inequality and (i), we have that
\begin{align*}
\mathrm{TV}\bigl(\Pi_n^\sigma, \Pi_n^{\sigma'}\bigr) \leq \sqrt{\frac{\sum_{i=1}^n  \mathrm{KL}(P_{o_i}^\sigma,P_{o_i}^{\sigma'})}{2}}\leq \sqrt{n_{\omega} 2^{d_\omega+3}\epsilon^2u}\leq 1/2.
\end{align*}

Now, for $\hat{C} \in \mathcal{C}_n$ and $\sigma \in \Sigma$, let $C^{\sigma}(x) := \mathbbm{1}_{\{\eta^{\sigma}(x) \geq 1/2\}}$ be the Bayes classifier under $P^{\sigma}$, then we have 
\begin{align*} 
\mathcal E_{P^{\sigma}}(\hat C) &= \int_{\mathbb{R}^d} \mathbbm{1}_{\{\hat C(x)\neq C^{\sigma}(x)\}}|2\eta^{\sigma}(x)-1| \, d\mu(x) \\
&\geq  \frac{ 2\epsilon  u}{2^{d-d_{\omega}}} \sum_{t \in [m]}\sum_{s\in\{-1,1\}^d} \mathbbm{1}_{\{\hat{C}(s^{\omega} \odot z^{\omega}_{t})\neq C^{\sigma}(s^\omega\odot z^{\omega}_{t})\}}.
\end{align*}
Further, for $\hat{C} \in \mathcal{C}_n$, we write $\tilde{C} : \mathbb{R}^d \times (\mathbb{R}^d \times \{0,1\})^n \rightarrow \{0,1\}$ when $\hat{C}$ is applied with the observation indicators being equal to $o_1, \ldots, o_n$, i.e. 
\[
\tilde{C}(x, (x_1^{o_1},y_1), \ldots, (x_n^{o_n},y_n)) = \hat{C}(x, (x_1^{o_1},y_1,o_1), \ldots, (x_n^{o_n},y_n,o_n))
\]
Then, using \textit{(vi)},
\[
\mathbb{E}_{Q^\sigma} \bigl\{ \mathcal E_{P^\sigma}(\hat C) \bigm| O_1 = o_1, \ldots, O_n = o_n \bigr\} = \mathbb{E}_{\Pi_{n}^\sigma} \bigl\{ \mathcal E_{P^\sigma}(\tilde{C})\bigr\}.
\]
Thus
\[
\mathbb{E}_{\Pi_{n}^\sigma} \bigl\{ \mathcal E_{P^\sigma}(\tilde C)\bigr\} \geq \frac{ 2\epsilon  u}{2^{d-d_{\omega}}} \sum_{t \in [m]} \sum_{s\in\{-1,1\}^d} \mathbb{P}_{\Pi_{n}^\sigma}\Bigl\{\tilde{C}(s^\omega\odot z^{\omega}_{t})\neq C^{\sigma}(s^\omega\odot z^{\omega}_{t})\Bigr\}.
\]
It follows that, for $t' \in [m]$, we have
\begin{align*}
\mathbb{E}_{\Pi_{n}^\sigma}\bigl\{ \mathcal E_{P^\sigma}(\tilde C)\bigr\} &+ \mathbb{E}_{\Pi_{n}^{\sigma(t')}} \bigl\{ \mathcal E_{P^{\sigma(t')}}(\tilde C)\bigr\}  \\& \geq \frac{2^{d_\omega+1}\epsilon  u}{2^d} \sum_{s \in \{-1,1\}^d} \Bigl[\mathbb{P}_{\Pi_{n}^\sigma}\Bigl\{\tilde{C}(s^\omega\odot z^{\omega}_{t'})\neq C^{\sigma}(s^\omega\odot z^{\omega}_{t'})\Bigr\} 
\\ & \hspace{40pt} + \mathbb{P}_{\Pi_{n}^{\sigma(t')}}\Bigl\{\tilde{C}(s^\omega\odot z^{\omega}_{t'})\neq C^{\sigma(t')}(s^\omega\odot z^{\omega}_{t'})\Bigr\}\Bigr] \\ 
& \geq 2^{d_\omega+1}\epsilon  u\Bigl\{1 - \mathrm{TV}\bigl(\Pi^\sigma_n, \Pi^{\sigma'}_n\bigr)\Bigr\} \geq 2^{d_\omega}\epsilon u,
\end{align*}
since $C^\sigma(s^{\omega}\odot z_{t'}^{\omega}) = 1 - C^{\sigma(t')}(s^{\omega} \odot z_{t'}^{\omega})$. We conclude that, for $\hat{C} \in \mathcal{C}_n$,  we have 
\begin{align*}
\sup_{Q\in\mathcal{Q}} \mathbb{E}_{Q} \bigl\{ & \mathcal{E}_{P_Q}(\hat C) \bigm| O_1 = o_1, \ldots, O_n = o_n \bigr\} 
\\ & \geq \max_{\sigma \in \Sigma} \IE_{\Pi_n^{\sigma}}\bigl\{\mathcal E_{P^\sigma}(\tilde C) \bigr\} \geq \frac{1}{2^{m}} \sum_{\sigma\in \Sigma} \IE_{\Pi_n^{\sigma}}\bigl\{\mathcal E_{P^{\sigma}}(\tilde C)  \bigr\} \\
& = \frac{1}{2^{m}} \sum_{t \in [m]} \sum_{\sigma \in \Sigma : \sigma_t = 1} \Bigl[ \IE_{\Pi_n^{\sigma}}\bigl\{\mathcal E_{P^{\sigma}}(\tilde C)\bigr\} + \IE_{\Pi_n^{\sigma(t)}}\bigl\{\mathcal E_{P^{\sigma(t)}}(\tilde C)\bigr\} \Bigr] \\
& \geq  \frac{2^{d_\omega}}{2^{m}} \sum_{t \in [m]} \sum_{\sigma \in \Sigma : \sigma_t = 1} \epsilon u = 2^{d_\omega-1}m\epsilon u,
\end{align*} 
as required. 
\end{proof}

\subsection{Properties of the marginal construction in the light tailed case\label{sec:LB_small_r}}
Recall the construction of $\mu^{(\omega)}_{\kappa,r,q,a,b}$ given in Section~\ref{subsubsec:LBlight}
. We show in this section that this marginal feature distribution indeed falls into our proposed class of distributions $\mathcal{Q}_{\mathrm{L}}(\boldsymbol{\gamma},C_{\mathrm{L}},\mathcal{O})$.  We start with three technical lemmas. 

\begin{lemma}\label{lem:LB_lebesgue_bound}
Fix $d\in\IN$, $\omega \in \{0,1\}^d\setminus \{\mathbf{0}_d\}$, $\kappa\in(0,1/\sqrt{d_{\omega}})$, $r>0$, and $\omega'\in\{0,1\}^d \setminus \{\mathbf{0}_d\}$. For any $x\in\mathcal{R}^{\omega}$ we have that
\[
\mathcal{L}_{d_{\omega}} \Bigl(\Pi_{\omega}\bigl(\bigl\{ z\in\mathbb{R}^d: z^{\omega'}\in B_\zeta(x^{\omega'}) \bigr\}\cap\mathcal R^\omega\bigr)\Bigr) \geq \kappa^{d_{\omega}} \cdot \zeta^{d_{\omega\wedge\omega'}} 
\]
for all $\zeta\in[0,1]$. For any $x\in\mathcal{R}_j^{\omega}$ we have that
\[
\mathcal{L}_{d_{\omega}} \Bigl(\Pi_{\omega}\bigl(\bigl\{ z\in\mathbb{R}^d: z^{\omega'}\in B_\zeta(x^{\omega'}) \bigr\}\cap\mathcal{R}_j^\omega\bigr)\Bigr) \geq r^{d_\omega} \cdot \min\{1, (r\sqrt{d_{\omega\wedge \omega'}})^{-d_{\omega\wedge\omega'}}\} \cdot \zeta^{d_{\omega\wedge\omega'}} 
\]
for all $\zeta\in[0,1]$.  
\end{lemma}
\begin{proof} 
First, fix $x\in\mathcal{R}^\omega$. For $\zeta\in[0,1]$, consider the map on $\IR^{d_{\omega\wedge\omega'}}$ given by
\[
\psi_{x,\zeta}: \tilde{x} \mapsto \Pi_{\omega\wedge\omega'}(x) + \zeta \cdot\Bigl(\tilde{x}-\Pi_{\omega\wedge\omega'}(x)\Bigr).
\] 
Observe that $\psi_{x,\zeta}(\Pi_{\omega\wedge\omega'}(\mathcal{R}^{\omega}))\subset\Pi_{\omega\wedge\omega'}(\mathcal{R}^{\omega})$. Furthermore, since 
\[
\Pi_{\omega\wedge\omega'}(x^{\omega'}) = \Pi_{\omega\wedge\omega'}(x^{\omega\wedge\omega'}) \in \psi_{x,\zeta}(\Pi_{\omega\wedge\omega'}(\mathcal{R}^\omega))
\]
and $\mathrm{diam}(\psi_{x,\zeta}(\Pi_{\omega\wedge\omega'}(\mathcal{R}^\omega)) \leq \zeta \kappa d^{1/2}_{\omega\wedge\omega'} < \zeta$, 
we also have $\psi_{x,\zeta}(\Pi_{\omega\wedge\omega'}(\mathcal{R}^\omega))\subseteq\Pi_{\omega\wedge\omega'}(\{z \in \mathbb{R}^d: z^{\omega'} \in B_\zeta(x^{\omega'})\})$. We conclude that
\begin{align*} 
\mathcal{L}_{d_\omega} \Bigl( \Pi_\omega\bigl( &\bigl\{ z\in\mathbb{R}^d: z^{\omega'}\in B_\zeta(x^{\omega'}) \bigr\} \cap\mathcal R^\omega\bigr) \Bigr) 
\\ & = \kappa^{d_{\omega}-d_{\omega\wedge\omega'}} \cdot \mathcal{L}_{d_{\omega\wedge\omega'}} \Bigl( \Pi_{\omega\wedge\omega'} \Bigl(\bigl\{ z\in\mathbb{R}^d: z^{\omega'}\in B_\zeta(x^{\omega'}) \bigr\}\cap\mathcal R^\omega\Bigr) \Bigr) 
\\ &\geq \kappa^{d_\omega-d_{\omega\wedge\omega'}} \cdot \mathcal{L}_{d_{\omega\wedge\omega'}} \bigl(\psi_{x,\zeta}(\Pi_{\omega\wedge\omega'}(\mathcal{R}^\omega))\bigr) \geq \kappa^{d_{\omega}} \cdot \zeta^{d_{\omega\wedge\omega'}}.
\end{align*}
This concludes the proof of the first part of the lemma. 

For the second part, fix $x\in\mathcal{R}_j^\omega$. If $\zeta \in [0, (r \sqrt{d_{\omega \wedge \omega'}}) \wedge 1]$, then there exists a $d_{\omega\wedge \omega'}$-dimensional hypercube $A$ with side length $\zeta/\sqrt{d_{\omega \wedge \omega'}}$ such that  $A \subseteq \Pi_{\omega' \wedge \omega'}\Bigl(\Pi_{\omega}\bigl(\bigl\{ z\in\mathbb{R}^d: z^{\omega'}\in B_\zeta(x^{\omega'}) \bigr\}\cap\mathcal R^{\omega}_j\bigr)\Bigr)$, thus for $\zeta \in [0, (r \sqrt{d_{\omega \wedge \omega'}})\wedge1]$,
\[
\mathcal{L}_{d_{\omega}} \Bigl(\Pi_{\omega}\bigl(\bigl\{ z\in\mathbb{R}^d: z^{\omega'}\in B_\zeta(x^{\omega'}) \bigr\}\cap\mathcal{R}_j^\omega\bigr)\Bigr) \geq d_{\omega\wedge\omega'}^{-d_{\omega\wedge\omega'}/2} \cdot r^{d_\omega-d_{\omega\wedge\omega'}} \zeta^{d_{\omega\wedge\omega'}} 
\]

If $\zeta \in [(r \sqrt{d_{\omega \wedge \omega'}}) \wedge 1, 1]$, then $\mathcal{R}_j^\omega \subseteq \Pi_{\omega}\bigl(\bigl\{ z\in\mathbb{R}^d: z^{\omega'}\in B_\zeta(x^{\omega'}) \bigr\}\cap\mathcal{R}_j^\omega\bigr)$, and 
\[
\mathcal{L}_{d_{\omega}} \Bigl(\Pi_{\omega}\bigl(\bigl\{ z\in\mathbb{R}^d: z^{\omega'}\in B_\zeta(x^{\omega'}) \bigr\}\cap\mathcal{R}_j^\omega\bigr)\Bigr) = r^{d_\omega} \geq r^{d_\omega} \zeta^{d_{\omega \wedge \omega'}}.
\]
The result follows. 
\end{proof}

\begin{lemma}\label{lem:LB_numberofpointsinball}
Fix $d\in\IN$, $\omega,\omega'\in\{0,1\}^d \setminus \{\mathbf{0}_d\}$ , $\kappa \in (0,1/\sqrt{d_{\omega}})$, $r>0$, $q \in \mathbb{N}$, $x \in \mathcal T_{q,r}^{\omega}$, and $\zeta>0$. 
If $\zeta \leq 4r\sqrt{d_{\omega\wedge\omega'}}$, then we have that
\[
|\mathcal {T}_{q,r}^{\omega\wedge\omega'} \cap B_\zeta(x^{\omega'})| \geq \Bigl(\frac{q\zeta}{8r\sqrt{d_{\omega\wedge\omega'}}} \Bigr)^{d_{\omega\wedge\omega'}}.
\]
\end{lemma}

\begin{proof}
First observe that since $x \in \mathcal T_{q,r}^{\omega}$, we have $x^{\omega'}\in\mathcal {T}_{q,r}^{\omega\wedge\omega'}$. If $q = 1$, then $|\mathcal{ T}_{q,r}^{\omega\wedge\omega'} \cap B_\zeta(x^{\omega'})| = 1 \geq (\zeta/(8r\sqrt{d_{\omega\wedge\omega'}}))^{d_{\omega\wedge\omega'}}$.  Now if $q \geq 2$, since $\frac{\zeta}{8\sqrt{d_{\omega\wedge\omega'}}} \leq \frac{r}{2}$ we can find a $d_{\omega \wedge \omega'}$-dimensional, axis-aligned hypercube $A$ with vertex $x^{\omega'}$ and side length $\zeta/(8\sqrt{d_{\omega\wedge\omega'}})$ containing at least $\lceil q\zeta/(8r \sqrt{d_{\omega\wedge\omega'}})\rceil^{d_{\omega\wedge\omega'}}$ elements of $\mathcal{T}_{q,r}^{\omega\wedge\omega'}$. We conclude that $|\mathcal {T}_{q,r}^{\omega\wedge\omega'}\cap B_\zeta(x^{\omega'})| \geq |\mathcal{T}_{q,r,\tilde{r}}^{\omega\wedge\omega'}\cap A| \geq \bigl(\frac{q\zeta}{8r \sqrt{d_{\omega\wedge\omega'}}} \bigr)^{d_{\omega\wedge\omega'}}$.
\end{proof}

Now recall the definition of the lower density $\rho_{\nu, s}$ from~\eqref{def:lower_density}
. For $\mu \equiv \mu^{(\omega)}_{\kappa,r,q,a,b}$ and $\omega' \in \{0,1\}^d$, we'll write $\rho_{\omega'}$ in place of $\rho_{(\mu^{(\omega)}_{\kappa,r,q,a,b})_{\omega'},d_{\omega'}}$. Our next result derives lower bounds on $\rho_{\omega'}$ in different regions of the feature space.

\begin{lemma}\label{lem:LB_boundonmu}
Fix $d\in\IN$, $\omega \in \{0,1\}^d \setminus \{\mathbf{0}_d\}$, $\kappa \in (0,1/\sqrt{d_{\omega}})$, $r>0$, $q \in \mathbb{N}$, and $a,b \in [0,1/4]$. Then, for $\omega' \in \{0,1\}^d$, such that $\omega \wedge \omega' \neq \mathbf{0}_d$, we have
\begin{enumerate}
    \item[(i)]  $\rho_{\omega'}(x^{\omega'}) \geq  \frac{1-a-b}{2^{d_{\omega\wedge\omega'}}}$ for all $x \in \mathcal{R}^\omega$,
    \item[(ii)] $\rho_{\omega'}(x^{\omega'}) \geq  \frac{b}{d_{\omega} \cdot 2^{d_{\omega\wedge\omega'}}} \cdot \min\Bigl\{1, (r \sqrt{d_{\omega \wedge \omega'}})^{-d_{\omega\wedge\omega'}}\Bigr\}$ for all $x \in \mathcal{R}_j^\omega$ and $j\in[d]$ with $\omega_j = 1$,
    \item[(iii)] if $\omega \npreceq \omega'$, then $\rho_{\omega'}(x^{\omega'}) \geq  \frac{b}{d_{\omega} \cdot 2^{d_{\omega\wedge\omega'}}} \cdot \min\Bigl\{1, (r \sqrt{d_{\omega \wedge \omega'}})^{-d_{\omega\wedge\omega'}}\Bigr\}$ for all $x\in\mathcal T^\omega_{q,r}$,
      \item[(iv)] if $\omega \preceq \omega'$, then $\rho_{\omega'}(x^{\omega'}) \geq 2^{-3d_{\omega}} \cdot \min\Bigl\{1, a(4r \sqrt{d_{\omega}})^{-d_{\omega}} \Bigr\}$ for all $x\in\mathcal T^\omega_{q,r}$.
\end{enumerate}
\end{lemma}
\begin{proof}
To prove (i), fix $x\in\mathcal{R}^\omega$, $\zeta\in(0,1)$ and  $\omega'\in\{0,1\}^d \setminus \{\mathbf{0}_d\}$. Then, by Lemma~\ref{lem:LB_lebesgue_bound}, we have
\begin{align}
    \nonumber \mu_{\omega'}\bigl(B_\zeta(x^{\omega'})\bigr)
    &= \mu\Bigl(\bigl\{ z\in\mathbb{R}^d: z^{\omega'}\in B_\zeta(x^{\omega'}) \bigr\}\Bigr)
    \\ \nonumber &\geq \frac{(1-a-b)2^{d-d_{\omega}}}{2^d\kappa^{d_\omega}} \mathcal{L}_{d_\omega} \Bigl( \Pi_\omega\bigl( \bigl\{ z\in\mathbb{R}^d: z^{\omega'}\in B_\zeta(x^{\omega'}) \bigr\}\cap\mathcal R^\omega \bigr)\Bigr)
    \\\label{eq:casei} 
    & \geq \frac{1-a-b}{2^{d_{\omega\wedge\omega'}}}\zeta^{d_{\omega\wedge\omega'}}.
\end{align}
Hence 
\[
\rho_{\omega'}(x^{\omega'}) = \inf_{\zeta\in (0,1)} \frac{\mu_{\omega'}\bigl(B_\zeta(x^{\omega'})\bigr)} {\zeta^{d_{\omega'}}} \geq  \frac{1-a-b}{2^{d_{\omega\wedge\omega'}}} \cdot \inf_{\zeta\in (0,1)} \frac{1}{\zeta^{d_{\omega'}-d_{\omega\wedge\omega'}}} = \frac{1-a-b}{2^{d_{\omega\wedge\omega'}}}. 
\]

To prove (ii), fix $x\in\mathcal{R}_j^\omega$, $\zeta\in(0,1)$ and  $\omega'\in\{0,1\}^d \setminus \{\mathbf{0}_d\}$. Then, by Lemma~\ref{lem:LB_lebesgue_bound}, we have
\begin{align}
    \nonumber \mu_{\omega'}\bigl(B_\zeta(x^{\omega'})\bigr)
    &= \mu\Bigl(\bigl\{ z\in\mathbb{R}^d: z^{\omega'}\in B_\zeta(x^{\omega'}) \bigr\}\Bigr)
    \\ \nonumber &\geq \frac{b2^{d-d_{\omega\wedge\omega'}}}{2^{d}d_\omega r^{d_\omega}} \mathcal{L}_{d_\omega} \Bigl( \Pi_\omega\bigl( \bigl\{ z\in\mathbb{R}^d: z^{\omega'}\in B_\zeta(x^{\omega'}) \bigr\}\cap\mathcal R_j^\omega \bigr)\Bigr)
    \\ \nonumber &\geq \frac{b}{2^{d_{\omega\wedge\omega'}}d_\omega r^{d_\omega}} \cdot r^{d_\omega} \cdot \min\{1, (r\sqrt{d_{\omega\wedge \omega'}})^{-d_{\omega\wedge\omega'}}\} \cdot \zeta^{d_{\omega\wedge\omega'}} 
    \\\nonumber &\geq \frac{b}{2^{d_{\omega\wedge\omega'}}d_\omega } \cdot \min\{1, (r\sqrt{d_{\omega\wedge \omega'}})^{-d_{\omega\wedge\omega'}}\} \cdot \zeta^{d_{\omega'}}.
\end{align}
The result of part (ii) follows via the same argument used in part (i). 

To proof (iii), fix $x \in\mathcal{T}^\omega_{q,r}$ and $\omega'\in\{0,1\}^d\setminus \{\mathbf{0}_d\}$, with $\omega \npreceq \omega'$. Note that for any $x\in\mathcal T^\omega_{q,r}$ there exists a $j\in[d]$ such that $\omega_j=1$ and $\tilde{x}\in\mathcal{R}^\omega_j$ such that $\tilde{x}^{\omega'}=x^{\omega'}$, and we have by Lemma~\ref{lem:LB_lebesgue_bound}
\begin{align*}
\mu_{\omega'}\bigl(B_\zeta(x^{\omega'})\bigr)
    &= \mu\Bigl(\bigl\{ z\in\mathbb{R}^d: z^{\omega'}\in B_\zeta(x^{\omega'}) \bigr\}\Bigr) 
    = \mu\Bigl(\bigl\{ z\in\mathbb{R}^d: z^{\omega'}\in B_\zeta(\tilde{x}^{\omega'}) \bigr\}\Bigr)  
    \\ &\geq \frac{b2^{d-d_{\omega\wedge\omega'}}}{2^{d}d_\omega r^{d_\omega}} \mathcal{L}_{d_\omega} \Bigl( \Pi_\omega\bigl( \bigl\{ z\in\mathbb{R}^d: z^{\omega'}\in B_\zeta(\tilde{x}^{\omega'}) \bigr\}\cap\mathcal R_j^\omega \bigr)\Bigr)
    \\ &\geq \frac{b}{2^{d_{\omega\wedge\omega'}}d_\omega r^{d_\omega}} \cdot r^{d_\omega} \cdot \min\{1, (r\sqrt{d_{\omega\wedge \omega'}})^{-d_{\omega\wedge\omega'}}\} \cdot \zeta^{d_{\omega\wedge\omega'}} 
    \\ &\geq \frac{b}{2^{d_{\omega\wedge\omega'}}d_\omega } \cdot \min\{1, (r\sqrt{d_{\omega\wedge \omega'}})^{-d_{\omega\wedge\omega'}}\} \cdot \zeta^{d_{\omega'}}.
\end{align*}

To prove (iv), fix  $x \in\mathcal{T}^\omega_{q,r}$ and $\omega'\in\{0,1\}^d\setminus \{\mathbf{0}_d\}$, with $\omega \preceq \omega'$. If $\zeta \in(4r\sqrt{d_{\omega}},1]$, then letting $z:=\bigl(1+2r,\ldots,1+2r)$, we have $z^{\omega} \in \mathcal{R}^{\omega}$ and $\|x^{\omega} - z^{\omega}\| \leq 2r\sqrt{d_{\omega}} < \zeta/2$. Hence, by \eqref{eq:casei}, we have 
\[
\mu_{\omega'}\bigl(B_\zeta(x^{\omega'})\bigr)  \geq \mu_{\omega'}\bigl(B_{\zeta/2}(z^{\omega'})\bigr) \geq \frac{1-a-b}{2^{d_{\omega\wedge\omega'}}} (\zeta/2)^{d_{\omega}} \geq 2^{-3d_{\omega}} \cdot \zeta^{d_{\omega}}.
\]
On the other hand, if $\zeta\in(0,(4r\sqrt{d_{\omega}}) \wedge 1]$ 
we have by Lemma~\ref{lem:LB_numberofpointsinball},
\begin{align*}
\mu_{\omega'}\bigl(B_\zeta(x^{\omega'})\bigr)
    &= \mu\Bigl(\bigl\{ z\in\mathbb{R}^d: z^{\omega'}\in B_\zeta(x^{\omega'}) \bigr\}\Bigr) 
    \geq \frac{a}{(2q)^{d_{\omega}}}  \cdot\bigl|\mathcal{T}_{q,r}^{\omega}\cap  B_\zeta(x^{\omega'}) \bigr|  \\
    &\geq \frac{a}{(2q)^{d_{\omega}}} \Bigl(\frac{q\zeta}{8r\sqrt{d_{\omega}}} \Bigr)^{d_{\omega}}  \geq  \frac{a}{2^{3d_{\omega}}} \Bigl(\frac{1}{4r\sqrt{d_{\omega}}}\Bigr)^{d_{\omega}}\zeta^{d_{\omega'}}.
\end{align*}

\end{proof}

The next corollary exploits the previous lemmas and gives conditions under which a marginal distribution of the type~\eqref{lb:marginal_construction} 
falls into the class $\mathcal{Q}_{\mathrm{L}}(\boldsymbol{\gamma}, C_{\mathrm{L}},\mathcal{O})$.

\begin{corollary}\label{cor:L_def_satisfied}
Fix $d\in\IN$, $\mathcal{O} \subseteq \{0,1\}^d$, $\boldsymbol{\gamma} = (\gamma_1,\dots,\gamma_d)^T \in [0, \infty)^d$, $r> 0$, $q \in \mathbb{N}$, $\omega\in\{0,1\}^d\setminus\{\mathbf{0}_d\}$,  $\kappa \in (0,1/\sqrt{d_{\omega}})$ and $a,b \in [0,1/4]$. Let 
$\gamma_{\omega} := \min\{\gamma_j : \omega_j = 1\}$ and let $\gamma_{\max} := \max\{\gamma_j : j\in[d]\}$. Let $a_0 := 2^{-3d_{\omega}\gamma_{\omega}}$ and $b_0 = 2^{-1}(d_\omega 2^{d_{\omega}})^{-\gamma_{\max}}$. Fix $C_{\mathrm{L}}>4^{\gamma_{\max}(d_\omega+1)}$. Let $Q \equiv Q_{\omega,\kappa,r,q,a,b}$ denote a distribution on $\mathbb{R}^d \times \{0,1\}\times \{0,1\}^d$ with corresponding $X$-marginal distribution $\mu=\mu^{(\omega)}_{\kappa,r,q,a,b}$, that also satisfies $\min_{o \in \mathcal{O}} \mathbb{P}_Q(O = o) >0$ and for which  $O \indep (X,Y)$, when $(X, Y, O) \sim Q$. If $a \leq a_0 \wedge b$ and 
\begin{align}
&\label{eq:a_bound} a^{1-\gamma_{\omega}} (4r\sqrt{d_\omega})^{d_{\omega}\gamma_{\omega}}\leq a_0,\\
&\label{eq:b_bound1} b^{1-\gamma_{\omega}} \leq C_{\mathrm{L}} \cdot b_0 \cdot \Bigl( \min\Bigl\{1, (r\sqrt{d_{\omega}})^{-d_{\omega}}\Bigr\} \Bigr)^{\gamma_{\omega}},\\
&\label{eq:b_bound2} b^{1-\gamma_{\max}} \leq C_{\mathrm{L}} \cdot b_0 \cdot \Bigl( \min\Bigl\{1, (r\sqrt{d_{\omega\wedge\omega'}})^{-d_{\omega\wedge\omega'}}\Bigr\} \Bigr)^{\gamma_{\max}},
\end{align}
for all $\omega'$ with $\omega \wedge \omega' \neq \mathbf{0}_d$, then $Q \in \mathcal{Q}_{\mathrm{L}}(\boldsymbol{\gamma},C_{\mathrm{L}},\mathcal{O})$.
\end{corollary}
\begin{proof} 
First, since $O$ is independent of $(X,Y)$, we have $\mu_{\omega' \mid o} = \mu_{\omega'}$ and thus, for any $\omega',o\in\{0,1\}^d$ and $x\in\IR^d$, $\rho_{\mu_{\omega' \mid o}, d_{\omega'}}(x) = \rho_{\mu_{\omega'}, d_{\omega'}}(x)$. Therefore it suffices to show that~\eqref{eq:lowerdensity} 
holds with $\rho_{\mu_{\omega'}, d_{\omega'}}(x)$ instead of $\min_{o \in \mathcal{O}: \omega' \preceq o} \rho_{\mu_{\omega' \mid o}, d_{\omega'}}(x)$.

Now fix $\omega'\in\{0,1\}^d$. First, if $\omega\wedge\omega'=\mathbf{0}$, then~\eqref{eq:lowerdensity} 
holds with $C_{L} = 1$ and any $\gamma_{\omega'}  > 0$, since $\mu_{\mathbf{0}}$ is a point mass at $\mathbf{0} \in \mathbb{R}^d$.  If $\omega \preceq \omega'$ and
\[
\xi \in \Biggl(0, \Bigl (2^{-3d_{\omega}}\cdot \min\Bigl\{1,a(4r\sqrt{d_{\omega}})^{-d_{\omega}}\Bigr\} \Bigr) \wedge \Bigl( \frac{b}{d_{\omega} \cdot 2^{d_{\omega}}} \cdot \min\Bigl\{1, (r \sqrt{d_{\omega}})^{-d_{\omega}}\Bigr\} \Bigr)\Biggr], 
\] 
then by the symmetry of $\mu$ and Lemma \ref{lem:LB_boundonmu}, for any $x\in \mathcal{T}^\omega_{q,r}\cup \mathcal R^\omega\cup\bigcup_{j\in[d_\omega]}\mathcal{R}_j^\omega$ and $s\in\{-1,1\}^d$,  we have $\rho_{\omega'}(s \odot x^{\omega'})>\xi$.  It follows that 
\begin{align*}
    \mu_{\omega'}\bigl(\{x\in\IR^d:\rho_{\omega'}(x)<\xi\}\bigr) = 0 \leq C_{\mathrm{L}}\cdot\xi^{\gamma_{\omega'}}.
\end{align*}

Next, if $2^{-3d_{\omega}}\cdot \min\Bigl\{1,a(4r\sqrt{d_{\omega}})^{-d_{\omega}}\Bigr\} <  \frac{b}{d_{\omega} \cdot 2^{d_{\omega}}} \cdot \min\Bigl\{1, (r \sqrt{d_{\omega}})^{-d_{\omega}} \Bigr\}$ and 
\[
\xi \in \Biggl(2^{-3d_{\omega}}\cdot \min\Bigl\{1,a(4r\sqrt{d_{\omega}})^{-d_{\omega}}\Bigr\} , \frac{b}{d_{\omega} \cdot 2^{d_{\omega}}} \cdot \min\Bigl\{1, (r \sqrt{d_{\omega}})^{-d_{\omega}}\Bigr\} \Biggr],
\]
then, by \eqref{eq:a_bound}, we have
\begin{align*}
 \mu_{\omega'}\bigl(\{x\in\IR^d:\rho_{\omega'}(x)<\xi\}\bigr)   &  \leq a \leq 2^{-3d_{\omega}\gamma_{\omega}} \cdot \min\Bigl\{1,a^{\gamma_{\omega}}(4r\sqrt{d_{\omega}})^{-\gamma_{\omega}d_{\omega}}\Bigr\}\\
& \leq C_{\mathrm{L}}\cdot\xi^{\gamma_{\omega}}
 \leq C_{\mathrm{L}}\cdot\xi^{\gamma_{\omega'}}.
\end{align*}
If $\xi \in \Biggl( \frac{b}{d_{\omega} \cdot 2^{d_{\omega}}} \cdot \min\Bigl\{1, (r \sqrt{d_{\omega}})^{-d_{\omega}}\Bigr\} , \frac{1-a-b}{2^{d_{\omega}}} \Biggr]$, then we have by \eqref{eq:b_bound1} and since $a\leq b$, that
\begin{align*}
 \mu_{\omega'}\bigl(\{x\in\IR^d:\rho_{\omega'}(x)<\xi\}\bigr)   &  \leq a+b \leq 2b \leq 2 C_{\mathrm{L}} \cdot b_0 \cdot  \Bigl(b \cdot \min\Bigl\{1, (r \sqrt{d_{\omega}})^{-d_{\omega}}\Bigr\} \Bigr)^{\gamma_{\omega}} \\ 
  &\leq C_{\mathrm{L}}\cdot \Bigl( \frac{b}{d_{\omega} \cdot 2^{d_{\omega}}} \cdot \min\Bigl\{1, (r \sqrt{d_{\omega}})^{-d_{\omega}}\Bigr\} \Bigr)^{\gamma_{\omega}}\\
 &\leq C_{\mathrm{L}}\cdot\xi^{\gamma_{\omega}}
 \leq C_{\mathrm{L}}\cdot\xi^{\gamma_{\omega'}}.
\end{align*}
If $\xi > \frac{1-a-b}{2^{d_{\omega}}}$, then we have
\[
\mu_{\omega'}\bigl(\{x\in\IR^d:\rho_{\omega'}(x)<\xi\}\bigr) \leq 1 = 1^{\gamma_{\omega'}} \leq 2^{\gamma_{\omega'}(d_{\omega}+1)}\cdot\xi^{\gamma_{\omega'}} \leq C_{\mathrm{L}}\cdot\xi^{\gamma_{\omega'}}
\]
since $1-a-b>1/2$.

Now consider $\omega \npreceq \omega'$ satisfying $\omega\wedge\omega'\neq\mathbf{0}_d$. First for
\[
\xi \in \Biggl(0, \frac{b}{d_{\omega} \cdot 2^{d_{\omega\wedge\omega'}}} \cdot \min\Bigl\{1, (r \sqrt{d_{\omega \wedge \omega'}})^{-d_{\omega\wedge\omega'}}\Bigr\} \Biggr]
\] 
we have, similarly to above, that
\begin{align*}
    \mu_{\omega'}\bigl(\{x\in\IR^d:\rho_{\omega'}(x)<\xi\}\bigr) = 0 \leq C_{\mathrm{L}}\cdot\xi^{\gamma_{\omega'}}.
\end{align*}
If 
\[
\xi \in \Biggl( \frac{b}{d_{\omega} \cdot 2^{d_{\omega\wedge\omega'}}} \cdot \min\Bigl\{1, (r \sqrt{d_{\omega \wedge \omega'}})^{-d_{\omega\wedge\omega'}}\Bigr\} , \frac{1-a-b}{2^{d_{\omega\wedge\omega'}}} \Biggr],
\] 
then, by \eqref{eq:b_bound2}, we obtain
\begin{align*}
\mu_{\omega'}\bigl(\{x\in\IR^d:\rho_{\omega'}(x)<\xi\}\bigr)   &  \leq a+b  \leq 2b \\
&\leq C_{\mathrm{L}}\cdot \Bigl( \frac{b}{d_{\omega} 2^{d_{\omega\wedge\omega'}}} \cdot \min\Bigl\{1, (r \sqrt{d_{\omega \wedge \omega'}})^{-d_{\omega\wedge\omega'}}\Bigr\} \Bigr)^{\gamma_{\max}}
\\ & \leq C_{\mathrm{L}}\cdot\xi^{\gamma_{\max}} \leq C_{\mathrm{L}}\cdot\xi^{\gamma_{\omega'}}.
\end{align*}
Finally, if $\xi > \frac{1-a-b}{2^{d_{\omega\wedge\omega'}}}$, then we have
\[
\mu_{\omega'}\bigl(\{x\in\IR^d:\rho_{\omega'}(x)<\xi\}\bigr) \leq 1 = 1^{\gamma_{\omega'}} \leq 2^{\gamma_{\omega'}(d_{\omega\wedge\omega'}+1)}\cdot\xi^{\gamma_{\omega'}} \leq C_{\mathrm{L}}\cdot\xi^{\gamma_{\omega'}}
\]
since $1-a-b>1/2$. This completes the proof. 
\end{proof}

\subsection{Regression function construction in the light tailed case\label{subsec:LBlighteta}}
We now construct different regression functions $\eta_\omega^\sigma$, for $\sigma \in \{-1,1\}^T$, which  will satisfy the assumptions of our version of Assouad's Lemma. To this end, we will use the marginal distribution $\mu^{(\omega)}_{\kappa,r,q,a,b}$ constructed in Section~\ref{subsubsec:LBlight}
, along with the additional quantities $\boldsymbol{\beta}=(\beta_1,\ldots,\beta_d)\in(0,1]^{d}$, $c_{\mathrm{E}} \in [0, 1/4]$,  $\epsilon \in (0,1/4)$ and $\sigma\in\{-1,1\}^{T}$. Let $\mathcal{S}^{\omega} = \{x^\omega \in \mathbb{R}^d : (\Pi_\omega(x^\omega))_j=0$ for some $j\in[d_\omega]\}$ be the coordinate axes lying in the $d_\omega$-dimensional submanifold $(\mathbb{R}^d)^\omega=\{x\in\mathbb{R}^d:x_j=0$ for all $j\in[d]$ with $\omega_j=0\}$. Recalling also that $z:=(1+2r,\ldots,1+2r)^{T} \in \mathbb{R}^d$, we define the function $f^{\circ,+}_{\epsilon,q,r,\sigma}: \mathcal{R}^\omega \cup \mathcal{T}^{\omega}_{q,r} \cup  \mathcal{S}^{\omega}\cup \bigcup_{j\in[d]:\omega_j=1}\mathcal{R}_j^\omega \rightarrow\IR$ by 
\begin{align}\label{def:f_circ_plus}
    f^{\circ,+}_{\epsilon,q,r,\sigma}(x):=
    \begin{cases}\epsilon + \frac14\lVert x^\omega-z^{\omega}\rVert^{\beta_{\omega}}  &\mbox{if } x^\omega\in\mathcal{R}^\omega\\ 
    \sigma_t\cdot\epsilon &\mbox{if } x^{\omega}=x^{\omega}_t \in\mathcal{T}^\omega_{q,r}\\ 
    0 &\mbox{if } x^\omega \in \mathcal{S}^{\omega}\\ 
    \frac{1}{2} &\mbox{if } x^\omega \in \bigcup_{j\in[d]:\omega_j=1}\mathcal{R}_j^\omega.
    \end{cases}
\end{align}

Further, we claim there exists a $\beta_{\omega}$-H\"older continuous extension $f_{\epsilon,q,r,\sigma}$ of $f^{\circ,+}_{\epsilon,q,r,\sigma}$ onto $\mathbb{R}^d$. Let $Q$ be the distribution of $(X,Y,O)$ with $O\indep(X,Y)$, $X$-marginal $\mu=\mu^{(\omega)}_{\kappa,r,q,a,b}$ and regression function $\eta^{\sigma}(x) = 1/2 + f_{\epsilon,q,r,\sigma}(x)$, then $Q$ belongs to $\mathcal{Q}_{\mathrm{E}}(\{\omega\}, c_{\mathrm{E}},\mathcal{O})$, and $P\equiv P_Q$ falls into the classes $\mathcal{P}_{\mathrm{S}}(\boldsymbol{\beta}, C_{\mathrm{S}})$ and $\mathcal{P}_{\mathrm{M}}(\alpha, C_{\mathrm{M}})$ for appropriate choices of $ C_{\mathrm{S}}$, $\alpha$, $C_{\mathrm{M}}$ and $\mathcal{O}$. This is made precise in the next three lemmas.

\begin{lemma}\label{lem:f_beta_holder_on_Rplus_Tplus_S}
Fix $d\in\IN$, $\omega \in \{0,1\}^d \setminus \{\mathbf{0}_d\}$, $\kappa \in (0,1/\sqrt{d_{\omega}})$, $q \in \mathbb{N}$, $r>0$, $\beta_{\omega} \in (0,1]$, $\epsilon \in (0,(1/4)\wedge(1/8\cdot(r/q)^{\beta_{\omega}})]$. Let further $\sigma=(\sigma_t)_{t\in[T]}\in\{-1,1\}^T$. Then
\begin{align*}
    |f^{\circ,+}_{\epsilon,q,r,\sigma}(x_1)-f^{\circ,+}_{\epsilon,q,r,\sigma}(x_2)|\leq \frac{1}{2} \| x_1-x_2\|^{\beta_{\omega}},
\end{align*}
for all $x_1,x_2 \in \mathcal{R}^\omega \cup \mathcal{T}^{\omega}_{q,r} \cup  \mathcal{S}^{\omega} \cup \bigcup_{j\in[d]:\omega_j=1}\mathcal{R}_j^\omega$.
\end{lemma}
\begin{proof}
If $x_1=x_2$, $x_1, x_2 \in \mathcal{S}^{\omega}$, or $x_1,x_2\in\bigcup_{j\in[d]:\omega_j=1}\mathcal{R}_j^\omega$, then $|f^{\circ,+}_{\epsilon,q,r,\sigma}(x_1)-f^{\circ,+}_{\epsilon,q,r,\sigma}(x_2)|=0\leq1/2\|x_1-x_2\|^{\beta_{\omega}}$. If $x_1 \in \mathcal{T}^{\omega}_{q,r} \cup \mathcal{S}^{\omega}$ and $x_2 \in\mathcal{T}^{\omega}_{q,r}$, with $x_1 \neq x_2$, then 
\[
|f^{\circ,+}_{\epsilon,q,r,\sigma}(x_1)-f^{\circ,+}_{\epsilon,q,r,\sigma}(x_2)| \leq 2\epsilon \leq \frac{1}{4} (r/q)^{\beta_{\omega}} \leq \frac{1}{4}\|x_1-x_2\|^{\beta_{\omega}}.
\]
If $x_1,x_2 \in \mathcal{R}^\omega$, then, by Minkowski's inequality,
\[
|f^{\circ,+}_{\epsilon,q,r,\sigma}(x_1)-f^{\circ,+}_{\epsilon,q,r,\sigma}(x_2)| = \frac{1}{4}  \bigl| \|x_1^\omega-z^\omega\|^{\beta_{\omega}}-\|x_2^\omega-z^\omega\|^{\beta_{\omega}} \bigr| \leq \frac14 \|x_1-x_2\|^{\beta_{\omega}}.
\]
If $x_1 \in \mathcal{R}^\omega$ and $x_2 \in \mathcal{T}^{\omega}_{q,r} \cup \mathcal{S}^{\omega}$, then
\begin{align*}
|f^{\circ,+}_{\epsilon,q,r,\sigma}(x_1)-f^{\circ,+}_{\epsilon,q,r,\sigma}(x_2)| &\leq  2\epsilon+\frac14\|x_1^\omega-z^\omega\|^{\beta_{\omega}} \\
&\leq \frac{r^{\beta_{\omega}}}{4}  \frac{\|x_1-x_2\|^{\beta_{\omega}}}{r^{\beta_{\omega}}} + \frac14\|x_1^\omega-x_2^\omega\|^{\beta_{\omega}} \leq \frac{1}{2} \|x_1- x_2\|^{\beta_{\omega}}.
\end{align*}
If $x_1\in\bigcup_{j\in[d]:\omega_j=1}\mathcal{R}_j^\omega$ and $x_2\in \mathcal{R}^\omega \cup \mathcal{T}^{\omega}_{q,r} \cup  \mathcal{S}^{\omega}$, then $|f^{\circ,+}_{\epsilon,q,r,\sigma}(x_1)-f^{\circ,+}_{\epsilon,q,r,\sigma}(x_2)|\leq1/2\leq1/2\|x_1- x_2\|^{\beta_{\omega}}$, and the result follows.
\end{proof}

\begin{corollary}\label{cor:f_beta_holder_on_IRplus}
Fix $d\in\IN$, $\omega \in \{0,1\}^d \setminus \{\mathbf{0}_d\}$, $\kappa \in (0,1/\sqrt{d_{\omega}})$, $q \in \mathbb{N}$, $r>0$, $\beta_{\omega} \in (0,1]$, $\epsilon \in (0,(1/4)\wedge(1/8\cdot(r/q)^{\beta_{\omega}})]$, $\sigma=(\sigma_t)_{t\in[T]}\in\{-1,1\}^T$. There exists a function $f^{+}_{\epsilon,q,r,\sigma}: [0,\infty)^d \rightarrow\IR$ such that for all $x\in\mathcal{R}^\omega \cup \mathcal{T}^{\omega}_{q,r} \cup \mathcal{S}^{\omega}\bigcup_{j\in[d]:\omega_j=1}\mathcal{R}_j^\omega$ we have $f^{+}_{\epsilon,q,r,\sigma}(x)=f^{\circ,+}_{\epsilon,q,r,\sigma}(x)$; and further we have $f^{+}_{\epsilon,q,r,\sigma}(x_1)=f^{+}_{\epsilon,q,r,\sigma}(x_1^\omega)$, as well as $|f^{+}_{\epsilon,q,r,\sigma}(x_1)-f^{+}_{\epsilon,q,r,\sigma}(x_2)|\leq (1/2)\cdot\lVert x_1-x_2\rVert^{\beta_{\omega}}$ for all $x_1,x_2\in [0,\infty)^d$.

\end{corollary}
\begin{proof}
Note that \eqref{def:f_circ_plus} depends only on $x^\omega$, such that by McShane's extension theorem \citep{mcshane1934extension} on $\IR^{d_\omega}$ there exists a $\beta_{\omega}$-H\"older continuous extension of $f^{\circ,+}_{\epsilon,q,r,\sigma}(x^{\omega})$ onto $([0,\infty)^d)^\omega$. The result follows by setting $f^{+}_{\epsilon,q,r,\sigma}(x)=f^{+}_{\epsilon,q,r,\sigma}(x^\omega)$ for any $x\in[0,\infty)^d$.
\end{proof}

We now define the function $f_{\epsilon,q,r,\sigma}^{\omega}:\IR^d\rightarrow\IR$. Given $x = (x_1, \ldots, x_d)^T \in \mathbb{R}^d$, let $s = (s_1,\ldots, s_d)^T \in \{-1,1\}^d$ be $s_j := \mathrm{sign}(x_j)$ and define   
\begin{align}\label{def:f_omega}
    f^{(\omega)}_{\epsilon,q,r,\sigma}(x):= \Bigl(\prod_{j \in [d_{\omega}]} \{\Pi_{\omega}(s)\}_j \Bigr) \cdot f^+_{\epsilon,q,r,\sigma}(s\odot x).
\end{align}
Finally, define $\eta_{\epsilon,q,r,\sigma}^{(\omega)}(x):=1/2 + f_{\epsilon,q,r,\sigma}^{(\omega)}(x)$, for $x \in \mathbb{R}^d$.

\begin{lemma}\label{lem:E_and_S_def_satisfied}
Fix $d\in\IN$, $\mathcal{O}\subseteq\{0,1\}^d$, $\omega \in \{0,1\}^d \setminus \{\mathbf{0}_d\}$, $\kappa \in (0,1/\sqrt{d_{\omega}})$, $q \in \mathbb{N}$, $r>0$, $a,b\in(0,1/4]$, and $\boldsymbol{\beta} \in (0,1]^{d}$, and let $\beta_\omega:=\min\{\beta_j:\omega_j=1\}$. Fix further $\epsilon \in (0,(1/4)\wedge(1/8\cdot(r/q)^{\beta_{\omega}})]$, $\sigma=(\sigma_t)_{t\in[T]}\in\{-1,1\}^T$ and $c_{\mathrm{E}} \leq b/4$. Let $P = P^{(\omega)}_{\epsilon,q,r,\sigma,a,b}$ denote the distribution on $\mathbb{R}^d\times\{0,1\}$ with regression function $\eta=\eta^{(\omega)}_{\epsilon,q,r,\sigma}$ and marginal feature distribution $\mu=\mu^{(\omega)}_{\kappa,r,q,a,b}$.  Further, let $Q = P \otimes U(\mathcal{O})$ be the distribution of $(X,Y,O)$ with $(X,Y)$-marginal $P$, with $O$ distributed uniformly on $\mathcal{O}$ and $(X,Y) \indep O$. Then $Q\in\mathcal{Q}_{\mathrm{E}}(\{\omega\}, c_{\mathrm{E}},\mathcal{O})$ and $P\equiv P_Q\in \mathcal{P}_{\mathrm{S}}(\boldsymbol{\beta}, 1)$.
\end{lemma}
\begin{proof}
First, since $O\indep(X,Y)$, we have $\mu_{\omega'} \equiv \mu_{\omega'|o}$ for any $o\in\mathcal{O}$ and $\omega'\preceq o$, thus $\sigma_{\omega'}^2=\IE_Q(f_{\omega'}^2(X))$, for $\omega' \in \{0,1\}^d$. Next, we show that the functions given by Equations~\eqref{eq:def_f_zero} 
and~\eqref{eq:def_f_omega} 
satisfy $f_{\omega} = f^{(\omega)}_{\epsilon,q,r,\sigma}$ and $f_{\omega'} \equiv 0$, for all $\omega' \in\{0,1\}^d \setminus \{\omega\}$. 
By the symmetry of $ f^{(\omega)}_{\epsilon,q,r,\sigma}$, and letting $j^\star=\min\{j\in[d]:~\omega_j=1\}$ we have 
\begin{align*}
f_{\mathbf{0}_d} &= \mathbb{E}_P\bigl\{\eta(X)\bigr\} -\frac{1}{2} = \mathbb{E}_P\{f^{(\omega)}_{\epsilon,q,r,\sigma}(X)\} = \int_{\mathbb{R}^d} f^{(\omega)}_{\epsilon,q,r,\sigma}(x) \, d\mu(x) \\
&= \int_{\mathbb{R}^d} f^{(\omega)}_{\epsilon,q,r,\sigma}(x) \mathbbm{1}_{\{x_{j^\star}>0\}} \, d\mu(x) + \int_{\mathbb{R}^d} f^{(\omega)}_{\epsilon,q,r,\sigma}(x) \mathbbm{1}_{\{x_{j^\star}<0\}} \, d\mu(x)  \\
&= \frac{1}{2^{d-d_\omega}} \sum_{s\in \{-1,1\}^d:s_{j^\star}=1} \Bigl(\prod_{j = 2 }^{d_{\omega}} \{\Pi_{\omega}(s)\}_j \Bigr) \cdot \int_{\mathcal{T}^{\omega}_{q,r} \cup \mathcal{R}^{\omega}}   f^+_{\epsilon,q,r,\sigma}(x) \, d\mu(x) \\
&\hspace{8mm} - \frac{1}{2^{d-d_\omega}} \sum_{s\in \{-1,1\}^d:s_{j^\star}=-1} \Bigl(\prod_{j = 2}^{d_{\omega}} \{\Pi_{\omega}(s)\}_j \Bigr) \cdot \int_{\mathcal{T}^{\omega}_{q,r} \cup \mathcal{R}^{\omega}}   f^+_{\epsilon,q,r,\sigma}(x) \, d\mu(x)=0,
\end{align*}
where the fourth equality holds since $f^{(\omega)}_{\epsilon,q,r,\sigma}(x)=0$ wherever $x_{j^\star}=0$.  Further, for $\mathbf{0} \prec \omega' \prec \omega$, now letting $j^\star=\min\{j\in[d]:~\omega_j=1-\omega'_j=1\}$ and by induction we see that 
\begin{align*}
    f_{\omega'}(x) &= \mathbb{E}_P\Bigl\{\eta(X) - \frac{1}{2} - \sum_{\omega''\prec\omega'}f_{\omega''}(X) \Big| X^{\omega'} = x^{\omega'}\Bigr\} = \mathbb{E}_P\Bigl\{f^{(\omega)}_{\epsilon,q,r,\sigma}(X) \Big| X^{\omega'} = x^{\omega'}\Bigr\} \\
    &= \mathbb{E}_P\Bigl\{f^{(\omega)}_{\epsilon,q,r,\sigma}(X) \bigl( \mathbbm{1}\{x_{j^\star}>0\} + \mathbbm{1}\{x_{j^\star}<0\} \bigr) \Big| X^{\omega'} = x^{\omega'}\Bigr\} = 0
\end{align*}
using the same argument as for the $\omega'=\mathbf{0}_d$ case.  We further have
\begin{align*}
f_{\omega}(x) &= \mathbb{E}_P\Bigl\{\eta(X) - \frac{1}{2} - \sum_{\omega'\prec\omega}f_{\omega'}(X) \Big| X^{\omega} = x^{\omega}\Bigr\} \\
&= \mathbb{E}_P\Bigl\{f^{(\omega)}_{\epsilon,q,r,\sigma}(X) \Big| X^{\omega} = x^{\omega}\Bigr\} = f^{(\omega)}_{\epsilon,q,r,\sigma}(x^\omega),
\end{align*}
i.e. $f_{\omega} = f^{(\omega)}_{\epsilon,q,r,\sigma}$. For $\omega' \npreceq \omega$, there exists either a $j^\star=\min\{j\in[d]:~\omega_j=1-\omega'_j=1\}$, in which case we proceed as above, or there exists no such $j^\star$, in which case $\omega\prec\omega'$. In this case we have
\begin{align*}
    f_{\omega'}(x) &= \mathbb{E}_P\Bigl\{\eta(X) - \sum_{\omega''\prec\omega'}f_{\omega''}(X) \Bigm| X^{\omega'} = x^{\omega'}\Bigr\} = \mathbb{E}_P\Bigl\{0\Bigm| X^{\omega'} = x^{\omega'}\Bigr\} = 0.
\end{align*}
We deduce that $\sigma_{\omega'}^2=0$ for all $\omega'\neq \omega$.  Finally, we have 
\[
\sigma_\omega^2 = \int_{\mathbb{R}^d} f_\omega^2(x) \,d\mu(x)  =  a \epsilon^2  + \frac{b}{4} + 2^{d_\omega} \int_{\mathcal{R}^\omega}  \Bigl(\epsilon + \frac{1}{4}\|x^\omega-z^\omega\|^{\beta_\omega}\Bigr)^2 \, d\mu(x) \geq b/4.
\]
Thus $Q\in\mathcal{Q}_{\mathrm{E}}(\{\omega\},c_{\mathrm{E}},\mathcal{O})$, for $c_{\mathrm{E}} \leq b/4$. 

For the final part of the result, if $x_1,x_2\in\mathbb{R}^d$ are such that $\mathrm{sign}(x_1^\omega)=\mathrm{sign}(x_2^\omega)$, then by Corollary~\ref{cor:f_beta_holder_on_IRplus},
\[
|f_\omega(x_1) - f_\omega(x_2)| = |f^+_{\epsilon,q,r,\sigma}(s\odot x_1) - f^+_{\epsilon,q,r,\sigma}(s\odot x_2)| \leq \frac{1}{2} \|s\odot x_1-s\odot x_2\|^{\beta_{\omega}} \leq \|x_1-x_2\|^{\beta_{\omega}}. 
\]
On the other hand, if $s_1:=\mathrm{sign}(x_1^\omega) \neq \mathrm{sign}(x_2^\omega)=:s_2$, then there exists $z \in \mathcal{S}^\omega$ with $z=x_1+\zeta(x_2-x_1)$ for some $\zeta\in[0,1]$. Then, again by Corollary~\ref{cor:f_beta_holder_on_IRplus},
\begin{align*}
|f_\omega(x_1)-&f_\omega(x_2)| \leq |f_\omega(x_1)-f_\omega(z)| + |f_\omega(z)-f_\omega(x_2)| \\
&\leq |f^+_{\epsilon,q,r,\sigma}(s_1\odot x_1)-f^+_{\epsilon,q,r,\sigma}(s_1\odot z)| + |f^+_{\epsilon,q,r,\sigma}(s_2\odot z) - f^+_{\epsilon,q,r,\sigma}(s_2\odot x_2)| \\
&\leq \frac{1}{2} \|s_1\odot x_1-s_1\odot z\|^{\beta_{\omega}} + \frac{1}{2} \|s_2\odot z-s_2\odot x_2\|^{\beta_{\omega}} = \frac{1}{2} \|x_1-z\|^{\beta_{\omega}} + \frac{1}{2} \|z-x_2\|^{\beta_{\omega}} \\
&= \frac{1}{2} (\zeta^{\beta_{\omega}}+(1-\zeta)^{\beta_{\omega}}) \|x_1-x_2\|^{\beta_{\omega}} \leq \frac{1}{2} 2^{1-\beta_{\omega}} \|x_1-x_2\|^{\beta_{\omega}} \leq \|x_1-x_2\|^{\beta_{\omega}}.
\end{align*}
Finally, for $\omega'\neq\omega$ we have $|f_{\omega'}(x_1)-f_{\omega'}(x_2)|=0\leq\|x_1-x_2\|^{\beta_{\omega'}}$; In other words, we have $P\in\mathcal{P}_{\mathrm{S}}(\boldsymbol{\beta},1)$. 
\end{proof}

\begin{lemma}\label{lem:M_def_satisfied}
Fix $d\in\IN$, $\omega \in \{0,1\}^d \setminus \{\mathbf{0}_d\}$, $\kappa \in (0,1/\sqrt{d_{\omega}})$, $q \in \mathbb{N}$, $r>0$, $\beta_{\omega} \in (0,1]$, $\epsilon \in (0,(1/4)\wedge(1/8\cdot(r/q)^{\beta_{\omega}})]$, $\sigma=(\sigma_t)_{t\in[T]}\in\{-1,1\}^T$.
Fix $C_{\mathrm{M}} \geq \max\{1+4^{d_\omega/\beta_{\omega}}(2\kappa)^{-d_\omega}V_{d_\omega}, 2^{\alpha}\}$, $\alpha\in[0,d_{\omega}/\beta_{\omega}]$, and $a\in[0,\epsilon^\alpha]$, $b\in(0,1/4]$. Let $P = P^{(\omega)}_{\epsilon,q,r,\sigma,a,b}$ denote the distribution on $\mathbb{R}^d\times\{0,1\}$ with regression function $\eta=\eta^{(\omega)}_{\epsilon,q,r,\sigma}$ and marginal feature distribution $\mu=\mu^{(\omega)}_{\kappa,r,q,a,b}$. Then $P\in\mathcal{P}_{\mathrm{M}}(\alpha,C_{\mathrm{M}})$.
\end{lemma}
\begin{proof}
First, if $t\in (0,\epsilon)$, then $\mu\bigl(\bigl\{x \in \IR^d : |\eta(x)- 1/2| < t \bigr\} \bigr) = 0 \leq C_{\mathrm{M}} \cdot t^{\alpha}$.  For $t\in [\epsilon,1/2)$, by \eqref{def:f_circ_plus} and the definition of $\eta$, if $x\in\mathrm{supp}(\mu)\setminus \bigcup_{s\in \{-1,1\}^d}(s\odot \mathcal{T}^{\omega}_{q,r})$ satisfies $|\eta(x)-1/2|<t$, then $\|x-z^\omega\|\leq(4t)^{1/\beta_{\omega}}$. Thus
\begin{align*}
\mu\Bigl(\Bigl\{x \in \mathbb{R}^d : \bigl|\eta(x)-1/2\bigr| < t \Bigr\} \Bigr) &=  2^{d_\omega} \mu\bigl(\mathcal{T}^\omega_{q,r}\bigl) + 2^{d_{\omega}} \mu\bigl( B_{(4t)^{1/\beta_{\omega}}}(z^\omega) \cap \mathcal{R}^{\omega} \bigr) \\
&=  a + \frac{(1-a-b)}{\kappa^{d_\omega}}\mathcal{L}_{d_\omega}\Bigl(\Pi_\omega\bigl(B_{(4t)^{1/\beta_{\omega}}}(z^\omega) \cap \mathcal{R}^\omega\bigr)\Bigr)\\
& \leq   \epsilon^\alpha + \frac{ (1-a-b)(4t)^{d_{\omega}/\beta_{\omega}}V_{d_{\omega}}}{(2\kappa)^{d_\omega}} \leq C_{\mathrm{M}} \cdot t^{\alpha}.
\end{align*}
Finally, if $t\geq1/2$, then
\begin{align*}
\mu\Bigl(\Bigl\{x \in \mathbb{R}^d : \bigl|\eta(x)-1/2\bigr| < t \Bigr\} \Bigr) \leq 1 \leq (2t)^{\alpha}  \leq C_{\mathrm{M}} \cdot t^{\alpha},
\end{align*}
as required. 
\end{proof}

\subsection{Proof of the lower bound in the light tailed case\label{subsec:LBproof_lighttails}}

We can now complete the proof in the case that $\Omega_{\star} = \{\omega\}$ is a singleton for which $\gamma_\omega = \min\{\gamma_j : \omega_j = 1\} \geq 1$.  Recall that $n_{\omega} = \sum_{i=1}^n \mathbbm{1}_{\{\omega \preceq o_i\}}$.  

\begin{lemma}\label{lem:LB_light_tails}
Fix $d, n \in \mathbb{N}$, $\mathcal{O}\subseteq\{0,1\}^d$, $o_1, \ldots, o_n \in \mathcal{O}$,  $\omega \in \{0,1\}^d \setminus \{\mathbf{0}_d\}$, $c_{\mathrm{E}} \in [0,1/4]$, $\boldsymbol{\gamma} \in [0,\infty)^d$, $C_{\mathrm{L}}\geq 1$, $\boldsymbol{\beta}\in (0,1]^{d}$, $C_{\mathrm{S}} \geq 1$, $\alpha \in [0, \infty)$ and $C_{\mathrm{M}} \geq 1$. Let $\beta_{\omega} := \min\{\beta_j : \omega_j = 1\}$. Suppose that $\gamma_{\omega} = \min\{\gamma_j : \omega_j = 1\} \geq 1$, $\alpha\beta_{\omega} \leq d_\omega$, $C_{\mathrm{M}}\geq\max\{1+4^{d_\omega/\beta_{\omega}}(d_\omega)^{-d_\omega/2}V_{d_\omega}, 2^{\alpha}\}$, $c_{\mathrm{E}} \leq 1/16$ and $C_{\mathrm{L}}\geq 4^{\gamma_{\max}(d_{\omega}+1)}$, where $\gamma_{\max} = \max_{j\in[d]}\{\gamma_j\}$.
Then there exists a constant $c$, depending only on $\gamma_\omega$, $d_{\omega}$, $\beta_{\omega}$ and $\alpha$, such that 
\begin{align*}
    \inf_{\hat C\in\mathcal C_n} \sup_{Q} \mathbb E_{Q}\Bigl\{\mathcal E_{P_Q}(\hat C) \Bigm| O_1 = o_1, \ldots, O_n = o_n \Bigr\}\geq c \cdot \max\Bigl\{ n_\omega^{-\frac{\beta_{\omega}\gamma_\omega(1+\alpha)}{\gamma_\omega(2\beta_{\omega}+d_\omega)+\alpha\beta_{\omega}}}, 1 \Bigr\},
\end{align*}
where the supremum is taken over $Q\in\mathcal{Q}_{\mathrm{Miss}}'(\{\omega\},c_{\mathrm{E}}, \boldsymbol{\gamma}, C_{\mathrm{L}}, \boldsymbol{\beta}, C_{\mathrm{S}} , \alpha, C_{\mathrm{M}},\mathcal{O})$.
\end{lemma}
\begin{proof}
First, let
$a_0 := 2^{-3d_{\omega}\gamma_{\omega}}$, $a_1 := 2^{3\gamma_{\omega} d_{\omega}/\beta_{\omega}+2\gamma_{\omega}d_{\omega}}
 \cdot d_{\omega}^{\gamma_{\omega}d_{\omega}/2}$, $b_0 = 2^{-1}(d_\omega 2^{d_{\omega}})^{-\gamma_{\max}}$,
and
\[
q_0:=\min\biggl\{ a_0^{(2+\alpha)/(\alpha \wedge 1)}2^{5+d_\omega} , \Bigr(\frac{a_0}{a_1} \cdot 2^{(d_{\omega} + 5)\frac{\alpha\beta_{\omega}(1-\gamma_{\omega})+\gamma_{\omega}d_{\omega}}{\beta_{\omega}(2+\alpha)} } \Bigl)^{\frac{(2+\alpha)\beta_{\omega}}{\alpha\beta_{\omega}+\gamma_{\omega}d_{\omega}+2\gamma_{\omega}\beta_{\omega}}}, 2^{\frac{d_\omega(d_\omega-3\alpha-1)}{d_\omega+(2+\alpha)\beta_{\omega}}}\biggr\}.
\]
Further let
\begin{align*}
\rho &:= \frac{\gamma_{\omega}(d_\omega-\alpha\beta_{\omega})+\alpha\beta_{\omega}}{\gamma_{\omega}(2\beta_{\omega}+d_\omega)+\alpha\beta_{\omega}}\in[0,1); \quad q :=\lfloor (q_{0} n_\omega^{\rho})^{1/d_\omega}\rfloor, 
\\ m &:= q^{d_{\omega}}, \quad 
\epsilon:= \min\Bigl\{\Bigl(\frac{m}{2^{d_\omega+5} n_\omega}\Bigr)^{\frac{1}{2+\alpha}},1/4\Bigr\} , \quad
u := \frac{\epsilon^\alpha}{m},
\\  \kappa &:=\frac{1}{2\sqrt{d_\omega}}, \quad
a:=\epsilon^\alpha, \quad 
b:=\frac{1}{4}, \quad
r:= (8\epsilon)^{1/\beta_{\omega}}\cdot q.
\end{align*}
Suppose initially that $n_{\omega}  >  q_{0}^{-1/\rho}$, so that $q \geq 1$. For $\sigma \in\{-1,1\}^m$, let $Q^\sigma$ denote the distribution with marginal $\mu=\mu^{(\omega)}_{\kappa,r,q,a,b}$, regression function $\eta^\sigma=\eta^{(\omega)}_{\epsilon,q,r,\sigma}$ and for which $O$ is independent of $(X, Y)$ and uniformly distributed on $\mathcal{O}$. 

We now show that $Q^\sigma \in \mathcal{Q}_{\mathrm{Miss}}'(\{\omega\},c_{\mathrm{E}}, \boldsymbol{\gamma}, C_{\mathrm{L}}, \boldsymbol{\beta}, C_{\mathrm{S}} , \alpha, C_{\mathrm{M}},\mathcal{O})$, by applying in turn the results from the previous subsections.  First, $Q^\sigma\in\mathcal{Q}_{\mathrm{Miss}}(\mathcal{O})$ since the missingness indicator $O$ is independent of $(X,Y)$ and $\IP(O=o)>0$ for all $o\in\mathcal{O}$. Next, to apply Corollary \ref{cor:L_def_satisfied}, using that $\rho-1 = \frac{-\gamma_{\omega} \beta_{\omega} (2+\alpha)}{\gamma_{\omega}(2\beta_{\omega} + d_{\omega}) + \alpha\beta_{\omega}}$, we have  
\begin{align*} 
a & = \epsilon^{\alpha} \leq \Bigl(\frac{m}{2^{d_{\omega} + 5} n_{\omega}} \Bigr)^{\frac{\alpha}{2+\alpha}} \leq \Bigl(\frac{q_0 n_{\omega}^{\rho}} {2^{d_{\omega} + 5} n_{\omega}} \Bigr)^{\frac{\alpha}{2+\alpha}} \\
&= \Bigl(\frac{q_0}{2^{5 + d_{\omega}}}\Bigr)^{\alpha/(2+\alpha)}  n_{\omega}^{-\frac{\gamma_{\omega}\alpha\beta_{\omega}}{\gamma_{\omega}(2\beta+d_{\omega}) + \alpha \beta_{\omega}}} \leq  a_0^{\alpha \wedge 1} \leq a_0,
\end{align*}
by the first term in the minimum in the definition of $q_0$. Further, this also implies that $\bigl(\frac{m}{2^{d_{\omega} + 5} n_{\omega}} \bigr)^{\frac{1}{2+\alpha}} < \frac{1}{4}$, so $\epsilon < 1/4$. Moreover, since $\alpha\beta_{\omega} \leq d_{\omega}$, 
\begin{align*}
a^{1-\gamma_{\omega}} (4r\sqrt{d_{\omega}})^{\gamma_{\omega}d_{\omega}}
&= \epsilon^{\alpha(1-\gamma_{\omega})+\gamma_{\omega}d_{\omega}/\beta_{\omega}} \cdot q^{\gamma_{\omega}d_{\omega}} \cdot 2^{3d_{\omega}\gamma_{\omega}/\beta_{\omega}+2\gamma_{\omega}d_{\omega}} \cdot d_{\omega}^{\gamma_{\omega}d_{\omega}/2}\\
&\leq \Bigl(\frac{q_0 n_{\omega}^{\rho-1}} {2^{d_{\omega} + 5}} \Bigr)^{\frac{\alpha(1-\gamma_{\omega})+\gamma_{\omega}d_{\omega}/\beta_{\omega}}{2+\alpha}} \cdot \bigl(q_0 n_{\omega}^{\rho}\bigr)^{\gamma_{\omega}} \cdot a_1\\
&= \Bigl(\frac{q_0} {2^{d_{\omega} + 5}} \Bigr)^{\frac{\alpha(1-\gamma_{\omega})+\gamma_{\omega}d_{\omega}/\beta_{\omega}}{2+\alpha}} \cdot q_0^{\gamma_{\omega}} \cdot a_1 \leq a_0,
\end{align*} 
where we have used the second term in the minimum in the definition of $q_0$, thus \eqref{eq:a_bound} holds.  To show that~\eqref{eq:b_bound1} and~\eqref{eq:b_bound2} hold, first note that
\begin{align*}
    r&=(8\epsilon)^{1/\beta_{\omega}}\cdot q 
    \leq 8^{1/\beta_{\omega}} \Bigl(\frac{m}{2^{d_{\omega} + 5} n_{\omega}} \Bigr)^{\frac{1}{(2+\alpha)\beta_{\omega}}} \cdot  (q_{0} n_\omega^{\rho})^{1/d_\omega} 
    \\ & = 8^{1/\beta_{\omega}} \cdot \Bigl(\frac{q_0}{2^{5 + d_{\omega}}}\Bigr)^{\frac{1}{(2+\alpha)\beta_{\omega}}} \cdot q_0^{1/d_\omega} \cdot n_{\omega}^{\frac{(1-\gamma_\omega)\alpha\beta_{\omega}}{(\gamma_{\omega}(2\beta_{\omega}+d_{\omega}) + \alpha \beta_{\omega})d_\omega}}  \leq 1,
\end{align*}
where we have used the fact that $\gamma_{\omega} \geq 1$ and the third term in the definition of $q_0$. Thus, the minima in \eqref{eq:b_bound1} and \eqref{eq:b_bound2} are both attained by $1$, and we have
\begin{align*}
b^{1-\gamma_{\omega}} \leq b^{1-\gamma_{\max}} = 4^{\gamma_{\max}-1}\leq C_{\mathrm{L}} \cdot b_0,
\end{align*}
since $C_{\mathrm{L}} \geq 4^{\gamma_{\max}(d_\omega+1)} \geq 4^{\gamma_{\max}-1}b_0^{-1}$. We deduce that \eqref{eq:b_bound1} and \eqref{eq:b_bound2} hold, and therefore that $Q^\sigma \in \mathcal{Q}_{\mathrm{L}}(\boldsymbol{\gamma}, C_{\mathrm{L}},\mathcal{O})$ by Corollary~\ref{cor:L_def_satisfied}. Further, since $b = 1/4$, we have $c_{\mathrm{E}} \leq 1/16 = b/4$, thus by Lemma~\ref{lem:E_and_S_def_satisfied}, the upper bound on $C_{\mathrm{M}}$ and Lemma~\ref{lem:M_def_satisfied}, we deduce that $P^{\sigma} \equiv P_{Q^{\sigma}} \in \mathcal{P}_{\mathrm{S}}(\boldsymbol{\beta}, 1) \cap \mathcal{P}_{\mathrm{M}}(\alpha, C_{\mathrm{M}})$ and $Q^\sigma\in\mathcal{Q}_{\mathrm{E}}(\{\omega\}, c_{\mathrm{E}},\mathcal{O}) $.  

To complete the proof we apply Lemma~\ref{lem:assouad}. We first verify that Assumptions~(i)-(vi) in that lemma hold: 
\begin{enumerate}
\item[(i)] By the definition of $\epsilon$ and $u$, we have $2^{d_\omega+5}n_\omega\epsilon^2u=1$, 
\item[(ii)] Letting $z_t:=z_t^\omega$ for $t\in[m]$, we have that $\mu(\{s^{\omega} \odot z_t^{\omega}\})=u$ for all $t\in[m]$ and $s\in\{-1,1\}^d$ by the construction of the marginal measure in~\eqref{lb:marginal_construction}
,
\item[(iii)]  
By the construction of the regression function \eqref{def:f_circ_plus}, we have $\eta^\sigma(z^{\omega}_t) = 1/2 + \sigma_t \epsilon$ for $t\in[m]$, and $\eta^\sigma(s^{\omega} \odot z_t^{\omega}) = 1/2 + \bigl(\prod_{j \in [d] :(s^{\omega})_{j} =-1} (s^{\omega})_{j} \bigr)\cdot\sigma_t\epsilon$, for all $t\in[m]$ and $s\in\{-1,1\}^d$.
\item[(iv)] Since, for $\sigma,\sigma'\in\Sigma$, the support is given by $\mathrm{supp}(\mu)=\bigcup_{s\in\{-1,1\}^d}s\odot(\mathcal{T}_{q,r}^\omega \cup \mathcal{R}^\omega \cup \bigcup_{j\in[d]:\omega_j=1}\mathcal{R}_j^\omega)$, and since we have $\eta^\sigma=\eta^{\sigma'}$ on $\mathcal{R}^\omega \cup \bigcup_{j\in[d]:\omega_j=1}\mathcal{R}_j^\omega$, we indeed have $\eta^\sigma(x)=\eta^{\sigma'}(x)$ for all $x\in \mathrm{supp}(\mu) \setminus \bigcup_{t\in [m], s\in \{-1,1\}^d} \{s^{\omega}\odot z_t^{\omega}\}$,
\item[(v)] 
Again by construction, we have $\Bigl(\mathrm{supp}(\mu) \setminus \bigcup_{t\in [m], s\in \{-1,1\}^d} \{s^{\omega}\odot z_t^{\omega}\} \Bigr) \bigcap (-r-1,r+1)^d = \Bigl(\bigcup_{s\in\{-1,1\}^d}s \odot \bigl(\mathcal{R}^\omega\cup\bigcup_{j\in[d]:\omega_j=1}\mathcal{R}_j^\omega\bigr) \Bigr) \bigcap (-r-1,r+1)^d  = \emptyset$,
\item[(vi)]  By construction of the distribution of $O$, we have $O\indep(X,Y)$.
\end{enumerate} 
Then, for $n_{\omega} > (q_0)^{-1/\rho}$, since $\epsilon = \bigl(\frac{m}{2^{d_{\omega}+5}n_\omega}\bigr)^{\frac{1}{2+\alpha}}$ the lower bound in~\eqref{eq:lb_assouad} in Lemma~\ref{lem:assouad} gives
\begin{align*}
     \inf_{\hat C\in\mathcal C_n} \sup_{Q}\ &\mathbb E_{Q}\{\mathcal E_{P}(\hat C) | O_1 = o_1, \ldots, O_n = o_n\} \geq  \frac{m u \epsilon}{2} = \frac{\epsilon^{1+\alpha}}{2} = \frac{1}{2} \left(\frac{m}{2^{d_{\omega}+5}n_\omega}\right)^{\frac{1+\alpha}{2+\alpha}} \\
    &\geq \frac{1}{2} \Bigl( \frac{q_0 n_\omega^\rho} {2^{2d_{\omega}+5} n_\omega} \Bigr)^{\frac{1+\alpha}{2+\alpha}} = 2^{-1-\frac{(2d_{\omega}+5)(1+\alpha)}{2+\alpha}} q_0^{\frac{1+\alpha}{2+\alpha}} n_\omega^{\frac{(1+\alpha)(\rho-1)}{(2+\alpha)}}
    =: \tilde{c} \cdot  n_\omega^{-\frac{\beta_{\omega}\gamma_\omega(1+\alpha)}{\gamma_\omega(2\beta_{\omega}+d_\omega)+\alpha\beta_{\omega}}}.
\end{align*}
Finally, if $n_{\omega} \leq q_0^{-1/\rho}$ (including possibly $n_{\omega} = 0$), then using the fact that the excess risk is decreasing in $n_\omega$, we conclude that the result follows with $c := \tilde{c} \cdot q_0^{\frac{\beta_{\omega}\gamma_\omega(1+\alpha)}{\rho\{\gamma_\omega(2\beta_{\omega}+d_\omega)+\alpha\beta_{\omega}\}}}$. 
\end{proof}

\subsection{Properties of the marginal construction in the heavy tailed case\label{sec:LB_big_r}}
We now turn to the proof of the lower bound when $\gamma_{\omega} < 1$. Recall the definitions of $\nu_0$, $\nu_{1,\tilde{q}}$,  $\nu_{2,q,r,a}$ and $\mu^{(\omega)}_{q,\tilde{q},r,a,\tilde{j}}$ from Section~\ref{subsubsec:LBheavy} 
(and see Figure~\ref{fig:LB_heavy_tails}
). 

As mentioned above, the ideas behind the proof in this case are conceptually similar to (and in fact slightly simpler than) the results in the light tailed case. The overall outline of the remainder of the proof in this case follows the steps in Sections~\ref{sec:LB_small_r}, \ref{subsec:LBlighteta} and \ref{subsec:LBproof_lighttails}.

\begin{lemma}
\label{lem:tailparametersformarginals}
Fix $d\in\IN$, $\tilde{q} \in \mathbb{N}$, $d_0, d_1 \in \mathbb{N}$, with $d_0+d_1\geq 1$, and $\gamma_1 > 0$. Let $\nu := \nu_0^{\otimes d_0} \otimes \nu_1^{\otimes d_1}$. Then 
\[
\nu\Bigl(\{x\in \mathbb{R}^{d_0+d_1} : \rho_{\nu,d_1}(x) < \xi \}\Bigr) \leq \bigl\{(2\sqrt{d_1})^{d_1} \xi\bigr\}^{\gamma_1}. 
\] 
Further, fixing $C_0 > 1$, $q \in \mathbb{N}$, $r>1$ and $\gamma_2 
\in (0,1)$. If $a < \min\{1/2, 1-1/C_0^{1/{\gamma_2}}\}$ and satisfies $a^{1-\gamma_2} r^{\gamma_2} \leq 1$, then 
\[
\nu_{2}\Bigl(\{x\in \mathbb{R} : \rho_{\nu_2,1}(x) \leq \xi \}\Bigr) \leq C_0\cdot(2 \xi)^{\gamma_2}.
\]
As a consequence, for $\gamma_1 > 2$, we have
\[
(\nu \otimes \nu_{2})\Bigl(\{x\in \mathbb{R}^{d_0+d_1+1} : \rho_{\nu \otimes \nu_2,d_1+1}(x) \leq \xi \}\Bigr)\leq C_0 \cdot 2^{2 + 3\gamma_{2} + 3(d_1 + 1)\gamma_{1}/2}  \cdot d_1^{d_1\gamma_{1}/2} \xi^{\gamma_2}.
\]
\end{lemma} 
\begin{proof} 
First note that $\rho_{\nu_1,1}(x) = 0$, for all $x \in \mathbb{R} \setminus \bigcup_{j \in [\tilde{q}]}\{1+j/\tilde{q}\} \cup \{-1-j/\tilde{q}\}$.  On the other hand, for $x \in  \bigcup_{j \in [\tilde{q}]}\{1+j/\tilde{q}\} \cup \{-1-j/\tilde{q}\}$ and $\zeta \in (0,1)$, we have $\nu_1\bigl([x-\zeta, x+\zeta]\bigr) \geq \zeta/2$. Thus $\rho_{\nu_1,1}(x) \geq 1/2$, for all $x \in \bigcup_{j \in [\tilde{q}]}\{1+j/\tilde{q}\} \cup \{-1-j/\tilde{q}\}$. Hence, if $\xi \in (0,1/2)$, then we have 
\[
\nu_{1}\Bigl(\{x\in \mathbb{R} : \rho_{\nu_1,1}(x) \leq \xi \}\Bigr) = 0 \leq (2 \xi)^{\gamma_1}. 
\]
If $\xi > 1/2$, then we have 
\[
\nu_{1}\Bigl(\{x\in \mathbb{R} : \rho_{\nu_1,1}(x) \leq \xi \}\Bigr) \leq 1 \leq (2 \xi)^{\gamma_1}. 
\]
Now turning to the measure $\nu := \nu_0^{\otimes d_0} \otimes \nu_1^{\otimes d_1}$, for $d_0, d_1 \in \mathbb{N}$. By similar arguments, for $\zeta \in (0,1)$ and $x = (0,\ldots,0,x_{d_0+1},\ldots, x_{d_0+d_1})^T \in \mathbb{R}^{d_0+d_1}$, with $x_{d_0+k} \in \bigcup_{j \in [\tilde{q}]}\{1+j/\tilde{q}\} \cup \{-1-j/\tilde{q}\}$, for $k \in [d_1]$, we have
\[
\frac{\nu(B_{\zeta}(x))}{\zeta^{d_1}} \geq d_1^{-d_1/2} \prod_{j \in [d_1]} \frac{\nu_{1}\Bigl((x_{d_0+j} - \zeta/\sqrt{d_1}, x_{d_0+j} + \zeta/\sqrt{d_1})\Bigr) }{\zeta/\sqrt{d_1}} \geq  \frac{1}{(2\sqrt{d_1})^{d_1}}.
\]
Therefore, if $\xi \in (0, (2\sqrt{d_1})^{-d_1})$, then we have 
\[
\nu\Bigl(\{x\in \mathbb{R} : \rho_{\nu,d_1}(x) < \xi \}\Bigr) = 0 \leq \{(2\sqrt{d_1})^{d_1} \xi\bigr\}^{\gamma_1}.
\]
Further, if $\xi \geq (2\sqrt{d_1})^{-d_1}$, then we have 
\[
\nu\Bigl(\{x\in \mathbb{R} : \rho_{\nu,d_1}(x) < \xi \}\Bigr) \leq 1 \leq \bigl\{(2\sqrt{d_1})^{d_1} \xi\bigr\}^{\gamma_1}. 
\]
This completes the proof of the first part of the lemma. 

For the second part, first consider the $\nu_2$ measure. We have $\rho_{\nu_2,1}(x) = 0$, for all $x \in \mathbb{R} \setminus \Bigl\{\bigcup_{j \in [q]}\{1+rj/q\} \cup \{-1-rj/q\}  \bigcup (1+r, 2+r) \cup (-2-r,-1-r)\Bigr\}$.  For $x \in  \bigcup_{j \in [q]}\{1+rj/q\} \cup \{-1-rj/q\}$ and $\zeta \in (0,1)$, since $r>1$, we have $\nu_2\bigl([x-\zeta, x+\zeta]\bigr) \geq \frac{a\zeta}{2r}$. Thus $\rho_{\nu_2,1}(x) \geq a/(2r)$, for all $x \in \bigcup_{j \in [q]}\{1+rj/q\} \cup \{-1-rj/q\}$. Moreover, if $x \in (1+r, 2+r) \cup (-2-r,-1-r)$, then 
$\nu_2\bigl([x-\zeta, x+\zeta]\bigr) \geq \frac{(1-a)\zeta}{2}$. It follows that, if $\xi \in (0,a/(2r))$, then we have 
\[
\nu_{2}\Bigl(\{x\in \mathbb{R} : \rho_{\nu_0,1}(x) \leq \xi \}\Bigr) = 0 \leq (2 \xi)^{\gamma_2}. 
\]
If $\xi \in [a/(2r), (1-a)/2)$, then we have 
\[
\nu_{2}\Bigl(\{x\in \mathbb{R} : \rho_{\nu_2,1}(x) \leq \xi \}\Bigr) \leq a \leq 2^{\gamma_2} \bigl(\frac{a}{2r}\bigr)^{\gamma_2} \leq (2\xi)^{\gamma_2}. 
\]
Finally, if $\xi \geq (1-a)/2$, then we have 
\[
\nu_{2}\Bigl(\{x\in \mathbb{R} : \rho_{\nu_2,1}(x) \leq \xi \}\Bigr) \leq 1 \leq \Bigl(\frac{2 \xi}{1-a}\Bigr)^{\gamma_2} \leq C_0(2 \xi)^{\gamma_2} . 
\]

For the final conclusion, by \citet[Proposition~S8]{reeve2021adaptive}, we have 
\[
(\nu \otimes \nu_2)\Bigl(\{x\in \mathbb{R}^{d_0 + d_1 + 1} : \rho_{\nu \otimes \nu_2, d_1+1}(x) < \xi \}\Bigr) \leq C_0 \cdot 2^{2 + 3\gamma_{2} + 3(d_1 + 1)\gamma_{1}/2}  \cdot d_1^{d_1\gamma_{1}/2} \cdot \xi^{\gamma_{2}}.
\]
\end{proof}

\begin{corollary}\label{cor:L_def_satisfied_heavy_tails}
Fix $d\in\IN$, $\mathcal{O}\subseteq\{0,1\}^d$, $\boldsymbol{\gamma}=(\gamma_1,\ldots,\gamma_d)\in[0,\infty)^d$, $r>0$, $\tilde{q},q\in\mathbb{N}$, $\omega\in\{0,1\}^d\setminus\{\mathbf{0}_d\}$, $C_0>1$, $\tilde{j} \in \{j\in [d] : \omega_{j} = 1\}$ and $a\in[0,1)$. Let $\gamma_\omega:=\min\{\gamma_j:\omega_j=1\}$ and let $\gamma_{\max}:=2\vee\max\{\gamma_j:j\in[d]\}$. Let $Q\equiv Q_{q,\tilde{q},r,a,\tilde{j}}$ be the distribution of $(X,Y,O)$ on $\mathbb{R}^d\times\{0,1\}\times\{0,1\}^d$ with $X$-marginal distribution $\mu = \mu^{(\omega)}_{q,\tilde{q},r,a,\tilde{j}}$, with $O\indep(X,Y)$ and with $\IP(O=o)>0$ for all $o\in\mathcal{O}$. If $\gamma_{\tilde{j}} = \gamma_{\omega} < 1$,  $a < \min\{1/2, 1-1/C_0^{1/{\gamma_\omega}}\}$ and $a^{1-\gamma_{\omega}}r^{\gamma_{\omega}}\leq 1$,
then $Q\in\mathcal{Q}_{\mathrm{L}}(\boldsymbol{\gamma},C_{\mathrm{L}},\mathcal{O})$ for any $C_{\mathrm{L}}\geq C_0 \cdot 2^{2 + 3\gamma_{\omega} + 3d_\omega \gamma_{\max}/2}  \cdot d_{\omega}^{d_{\omega}\gamma_{\max}/2}$.
\end{corollary}
\begin{proof}
For $\omega' \in \{0,1\}^d\setminus \{\mathrm{0}_{d}\}$, we have
\begin{align*}
\mu_{\omega'} = \left\{
\begin{array}{c c} 
\bigotimes_{j=1}^{\tilde{j}-1} \nu_{\omega_j\wedge \omega'_j} \otimes \nu_{2} \otimes \bigotimes_{j = \tilde{j}+1}^{d} \nu_{\omega_j\wedge \omega'_{j}} & \text{if } \omega'_{\tilde{j}}= 1 \\
\bigotimes_{j=1}^{d} \nu_{\omega_j \wedge \omega'_j} & \text{otherwise.}
\end{array} \right.
\end{align*}
If $\omega'_{\tilde{j}} = 0$, then by permuting the order of the coordinates if necessary and the first part of Lemma~\ref{lem:tailparametersformarginals} with $\gamma_1 = \gamma_{\omega'}$, $d_0 = d - d_{\omega \wedge \omega'}$, $d_1 = d_{\omega \wedge \omega'} \leq \max\{d_{\omega'}, d_{\omega}-1\}$, we have
\[
\mu_{\omega'}\bigl(\{x \in \mathbb{R}^d : \rho_{\omega'}(x) < \xi \}\bigr) \leq C_{\mathrm{L}} \cdot \xi^{\gamma_{\omega'}},
\]
for all $\xi >0$, since $C_{\mathrm{L}} \geq (2\sqrt{d_{\omega}})^{d_{\omega}\gamma_{\max} }$. Further, if $\omega' = e_{\tilde{j}}$, by the second part of Lemma~\ref{lem:tailparametersformarginals} with $\gamma_2 = \gamma_{\omega}$ we have  
\[
\mu_{\omega'}\bigl(\{x \in \mathbb{R}^d : \rho_{\omega'}(x) < \xi \}\bigr)  \leq  
C_{\mathrm{L}} \cdot \xi^{\gamma_{\omega}},
\]
for all $\xi >0$, since $C_{\mathrm{L}} \geq C_0 2^{\gamma_{\omega}}$.
Finally, by the third part of Lemma~\ref{lem:tailparametersformarginals} with $\gamma_1 = \gamma_{\max} \geq 2$, $\gamma_2 = \gamma_\omega$, $d_0 = d - d_{\omega \wedge \omega'}$, $d_1 = d_{\omega \wedge \omega'} - 1$, we have 
\[
\mu_{\omega'}\bigl(\{x \in \mathbb{R}^d : \rho_{\omega'}(x) < \xi \}\bigr)  \leq  
C_{\mathrm{L}} \cdot \xi^{\gamma_{\omega}},
\]
since $C_{\mathrm{L}} \geq C_0 \cdot 2^{2 + 3\gamma_{\omega} + 3d_\omega \gamma_{\max}/2}  \cdot d_{\omega}^{d_{\omega}\gamma_{\max}/2}.$ Since $O\indep(X,Y)$ and thus $\mu_{\omega'}=\mu_{\omega'|o}$ for any $o\in\mathcal{O}$, this completes the proof. 
\end{proof}

\subsection{Regression function construction in the heavy tailed case\label{subsec:LBheavyeta}}
We now construct the regression functions $\eta_\omega^\sigma$, for $\sigma \in \{-1,1\}^T$, for the heavy tailed case ($\gamma_{\omega} <1$), using the marginal distribution constructed in Section~\ref{subsec:LBheavyeta}, along with the additional quantities $\boldsymbol{\beta}=(\beta_1,\ldots,\beta_d) \in (0,1]^{d}$, $c_{\mathrm{E}} \in [0, 1/4)]$,  $\epsilon \in (0,1/4)$ and $\sigma\in\{-1,1\}^{T}$. Let $T := \tilde{q}^{d_\omega-1} \cdot q$ and let $z_1^\omega, \ldots, z_T^\omega$ be an enumeration of the set 
\[
\mathcal{T}^{\omega} := \Bigl\{ \omega + \Bigl(\Bigl(\frac{v_1}{\tilde{q}}, \ldots,\frac{v_{\tilde{j}-1}}{\tilde{q}},  \frac{rv_{\tilde{j}}}{q},   \frac{v_{\tilde{j}+1}}{\tilde{q}}, \ldots, \frac{v_{d}}{\tilde{q}}\Bigr)^T\Bigr)^\omega : v_j \in [\tilde{q}] \text{ for } j \in [d]\setminus \{\tilde{j}\}, v_{\tilde{j}} \in [q] \Bigl\}.
\]
Let 
\[
\mathcal{R}_{\tilde{j}}^{\omega}\! :=\! \Bigl\{\omega\! +\! \Bigl(\!\Bigl(\frac{v_1}{\tilde{q}}, \ldots,\frac{v_{\tilde{j}-1}}{\tilde{q}},  r + u_{\tilde{j}}, \frac{v_{\tilde{j}+1}}{\tilde{q}},\ldots, \frac{v_{d}}{\tilde{q}}\Bigr)^T\Bigr)^\omega\! : v_j \in [\tilde{q}] \text{ for } j\! \in\! [d]\setminus \{\tilde{j}\}, u_{\tilde{j}}\! \in\! (0,1) \Bigl\}.
\]
Recall that $\mathcal{S}^{\omega} = \{x^\omega \in \mathbb{R}^d : (\Pi_\omega(x^\omega))_j=0$ for some $j\in[d_\omega]\}$ is the coordinate axes lying in the $d_\omega$-dimensional submanifold $(\mathbb{R}^d)^\omega=\{x\in\mathbb{R}^d:x_j=0$ for all $j\in[d]$ with $\omega_j=0\}$. We define the function $f^{\circ,+}_{\epsilon,\tilde{q},q,r,\sigma}: \mathcal{R}_{\tilde{j}}^\omega \cup \mathcal{T}^{\omega} \cup  \mathcal{S}^{\omega} \rightarrow\IR$ by
\begin{align}\label{def:f_circ_plus_heavy_tails}
    f^{\circ,+}_{\epsilon,\tilde{q},q,r,\sigma}(x):=
    \begin{cases} 
    \sigma_t\cdot\epsilon &\mbox{if } x^{\omega} = z^{\omega}_t \\ 
    0 &\mbox{if } x^\omega \in \mathcal{S}^{\omega}\\ 
    \frac{1}{2} &\mbox{if } x^{\omega} \in \mathcal{R}^{\omega}_{\tilde{j}} .
    \end{cases}
\end{align}

Recalling that $\beta_{\omega} := \min\{\beta_j : \omega_j = 1\}$, we will now show that there exists a $\beta_{\omega}$-H\"older continuous extension $f_{\epsilon,q,r,\sigma}$ of $f^{\circ,+}_{\epsilon,\tilde{q},q,r,\sigma}$ onto $\mathbb{R}^d$. Then, letting $\eta^{\sigma}(x) = 1/2 + f_{\epsilon,\tilde{q},q,r,\sigma}(x)$, the corresponding distribution $Q$ of $(X,Y,O)$ belongs to $\mathcal{Q}_{\mathrm{E}}(\{\omega\}, c_{\mathrm{E}},\mathcal{O})$, and the distribution $P=P_Q$ of $(X,Y)$ belongs to $\mathcal{P}_{\mathrm{S}}(\boldsymbol{\beta}, C_{\mathrm{S}})$ and $\mathcal{P}_{\mathrm{M}}(\alpha, C_{\mathrm{M}})$ for appropriate choices of the parameters, see the subsequent lemmas.

\begin{lemma}\label{lem:f_beta_holder_on_Rplus_Tplus_S_heavy_tails}
Fix $d\in\IN$, $\omega \in \{0,1\}^d \setminus \{\mathbf{0}_d\}$, $\tilde{q},q \in \mathbb{N}$, $r>1$, $\boldsymbol{\beta}\in(0,1]^d$ and let $\beta_{\omega} = \min\{\beta_j : \omega_j = 1\}$. Fix further $\tilde{j} \in \{j\in [d] : \omega_{j} = 1\}$ and $\epsilon \in (0,(1/4\cdot(r/q)^{\beta_{\omega}}) \wedge (1/4\cdot(1/\tilde{q})^{\beta_{\omega}})]$. Let further $\sigma=(\sigma_t)_{t\in[T]}\in\{-1,1\}^T$. Then
\begin{align*}
    |f^{\circ,+}_{\epsilon,\tilde{q},q,r,\sigma}(x_1)-f^{\circ,+}_{\epsilon,\tilde{q},q,r,\sigma}(x_2)|\leq \frac{1}{2} \| x_1-x_2\|^{\beta_{\omega}},
\end{align*}
for all $x_1,x_2 \in \mathcal{T}^{\omega} \cup  \mathcal{S}^{\omega} \cup \mathcal{R}_{\tilde{j}}^\omega$.
\end{lemma}
\begin{proof}
If $x_1=x_2$, $x_1, x_2 \in \mathcal{S}^{\omega}$, or $x_1,x_2\in\mathcal{R}_{\tilde{j}}^\omega$, then $|f^{\circ,+}_{\epsilon,\tilde{q},q,r,\sigma}(x_1)-f^{\circ,+}_{\epsilon,\tilde{q},q,r,\sigma}(x_2)|=0\leq1/2\|x_1-x_2\|^{\beta_{\omega}}$. If $x_1 \in \mathcal{T}^{\omega} \cup \mathcal{S}^{\omega}$ and $x_2 \in\mathcal{T}^{\omega}$, with $x_1 \neq x_2$, then 
\[
|f^{\circ,+}_{\epsilon,\tilde{q},q,r,\sigma}(x_1)-f^{\circ,+}_{\epsilon,\tilde{q},q,r,\sigma}(x_2)| \leq 2\epsilon \leq \frac{1}{2} \Bigl( \bigl(\frac{r}{q}\bigr)^{\beta_{\omega}} \wedge \bigl(\frac{1}{\tilde{q}}\bigr)^{\beta_{\omega}} \Bigr) \leq \frac{1}{2}\|x_1-x_2\|^{\beta_{\omega}}.
\]
If $x_1\in\mathcal{R}_{\tilde{j}}^\omega$ and $x_2\in \mathcal{T}^{\omega} \cup  \mathcal{S}^{\omega}$, then $|f^{\circ,+}_{\epsilon,\tilde{q},q,r,\sigma}(x_1)-f^{\circ,+}_{\epsilon,\tilde{q},q,r,\sigma}(x_2)| \leq 1/2 \leq 1/2 \cdot \|x_1- x_2\|^{\beta_{\omega}}$, and the result follows.
\end{proof}

\begin{corollary}\label{cor:f_beta_holder_on_IRplus_heavy_tails}
Fix $d\in\IN$, $\omega \in \{0,1\}^d \setminus \{\mathbf{0}_d\}$, $\tilde{q},q \in \mathbb{N}$, $r>1$, $\beta_{\omega} \in (0,1]$, $\tilde{j} \in \{j\in [d] : \omega_{j} = 1\}$, $\epsilon \in (0,(1/4\cdot(r/q)^{\beta_{\omega}}) \wedge (1/4\cdot(1/\tilde{q})^{\beta_{\omega}})]$ and $\sigma=(\sigma_t)_{t\in[T]}\in\{-1,1\}^T$. There exists a function $f^{+}_{\epsilon,\tilde{q},q,r,\sigma}: [0,\infty)^d \rightarrow\IR$ such that for all $x\in \mathcal{T}^{\omega} \cup \mathcal{S}^{\omega} \cup \mathcal{R}_{\tilde{j}}^\omega$ we have $f^{+}_{\epsilon,\tilde{q},q,r,\sigma}(x)=f^{\circ,+}_{\epsilon,\tilde{q},q,r,\sigma}(x)$; and further we have $f^{+}_{\epsilon,\tilde{q},q,r,\sigma}(x_1)=f^{+}_{\epsilon,\tilde{q},q,r,\sigma}(x_1^\omega)$, as well as $|f^{+}_{\epsilon,\tilde{q},q,r,\sigma}(x_1)-f^{+}_{\epsilon,\tilde{q},q,r,\sigma}(x_2)|\leq (1/2)\cdot\lVert x_1-x_2\rVert^{\beta_{\omega}}$ for all $x_1,x_2\in [0,\infty)^d$.

\end{corollary}
\begin{proof}
Note that \eqref{def:f_circ_plus_heavy_tails} depends only on $x^\omega$, such that by McShane's extension theorem \citep{mcshane1934extension} on $\IR^{d_\omega}$ there exists a $\beta_{\omega}$-H\"older continuous extension of $f^{\circ,+}_{\epsilon,\tilde{q},q,r,\sigma}(x^{\omega})$ onto $([0,\infty)^d)^\omega$. The result follows by setting $f^{+}_{\epsilon,\tilde{q},q,r,\sigma}(x)=f^{+}_{\epsilon,\tilde{q},q,r,\sigma}(x^\omega)$ for any $x\in[0,\infty)^d$.
\end{proof}

We now define the function $f_{\epsilon,\tilde{q},q,r,\sigma}^{(\omega)}:\IR^d\rightarrow\IR$. Given $x = (x_1, \ldots, x_d)^T \in \mathbb{R}^d$, let $s = (s_1,\ldots, s_d)^T \in \{-1,1\}^d$ be $s_j := \mathrm{sign}(x_j)$ and define   
\begin{align}\label{def:f_omega_heavytails}
    f^{(\omega)}_{\epsilon,\tilde{q},q,r,\sigma}(x):= \Bigl(\prod_{j \in [d_{\omega}]} \{\Pi_{\omega}(s)\}_j \Bigr) \cdot f^+_{\epsilon,\tilde{q},q,r,\sigma}(s\odot x).
\end{align}

Finally, define $\eta_{\epsilon,\tilde{q},q,r,\sigma}^{(\omega)}(x):=1/2 + f_{\epsilon,\tilde{q},q,r,\sigma}^{(\omega)}(x)$, for $x \in \mathbb{R}^d$.

\begin{lemma}\label{lem:E_and_S_def_satisfied_heavy_tails}
Fix $d\in\IN$, $\mathcal{O}\subseteq\{0,1\}^d$, $\omega \in \{0,1\}^d \setminus \{\mathbf{0}_d\}$, $\tilde{q},q \in \mathbb{N}$, $r>1$, $a\in(0,1/2]$, $\boldsymbol{\beta}\in(0,1]^d$ and let $\beta_{\omega} = \min\{\beta_j : \omega_j = 1\}$. Fix further $\tilde{j} \in \{j\in [d] : \omega_{j} = 1\}$, $\epsilon \in (0,(1/4\cdot(r/q)^{\beta_{\omega}}) \wedge (1/4\cdot(1/\tilde{q})^{\beta_{\omega}})]$, $\sigma=(\sigma_t)_{t\in[T]}\in\{-1,1\}^T$, $c_{\mathrm{E}} \leq 1/8$. Let $Q\equiv Q_{q,\tilde{q},r,a,\tilde{j}}$ be the distribution of $(X,Y,O)$ on $\mathbb{R}^d\times\{0,1\}\times\{0,1\}^d$ with $X$-marginal distribution $\mu = \mu^{(\omega)}_{q,\tilde{q},r,a,\tilde{j}}$, with regression function $\eta=\eta^{(\omega)}_{\epsilon,\tilde{q},q,r,\sigma}$, with $O\indep(X,Y)$ and with $\IP(O=o)>0$ for all $o\in\mathcal{O}$. Then $Q$ belongs to $\mathcal{Q}_{\mathrm{E}}(\{\omega\}, c_{\mathrm{E}},\mathcal{O})$ and $P_Q$, the distribution of $(X,Y)$ belongs to $\mathcal{P}_{\mathrm{S}}(\boldsymbol{\beta}, 1)$.
\end{lemma}
\begin{proof}
We first show that the functions given by Equations~\eqref{eq:def_f_zero} 
and~\eqref{eq:def_f_omega} 
satisfy $f_{\omega} = f^{(\omega)}_{\epsilon,\tilde{q},q,r,\sigma}$ and $f_{\omega'} \equiv 0$, for all $\omega' \in\{0,1\}^d \setminus \{\omega\}$. 
First, by the symmetry of $ f^{(\omega)}_{\epsilon,\tilde{q},q,r,\sigma}$, and letting $j^\star=\min\{j\in[d]:~\omega_j=1\}$ we have 
\begin{align*}
f_{\mathbf{0}} &= \mathbb{E}_P\bigl\{\eta(X)\bigr\} - \frac{1}{2}  =  \mathbb{E}_P\{f^{(\omega)}_{\epsilon,\tilde{q},q,r,\sigma}(X)\} =  \int_{\mathbb{R}^d} f^{(\omega)}_{\epsilon,\tilde{q},q,r,\sigma}(x) \, d\mu(x) \\
&= \int_{\mathbb{R}^d} f^{(\omega)}_{\epsilon,\tilde{q},q,r,\sigma}(x) \mathbbm{1}_{\{x_{j^\star}>0\}} \, d\mu(x) + \int_{\mathbb{R}^d} f^{(\omega)}_{\epsilon,\tilde{q},q,r,\sigma}(x) \mathbbm{1}_{\{x_{j^\star}<0\}} \, d\mu(x)  \\
&= \frac{1}{2^{d-d_\omega}} \sum_{s\in \{-1,1\}^d:s_{j^\star}=1} \Bigl(\prod_{j = 2 }^{d_{\omega}} \{\Pi_{\omega}(s)\}_j \Bigr) \cdot \int_{\mathcal{T}^{\omega} \cup \mathcal{R}_{\tilde{j}}^{\omega}}   f^+_{\epsilon,\tilde{q},q,r,\sigma}(x) \, d\mu(x) \\
&\hspace{8mm} - \frac{1}{2^{d-d_\omega}} \sum_{s\in \{-1,1\}^d:s_{j^\star}=-1} \Bigl(\prod_{j = 2}^{d_{\omega}} \{\Pi_{\omega}(s)\}_j \Bigr) \cdot \int_{\mathcal{T}^{\omega} \cup \mathcal{R}_{\tilde{j}}^{\omega}}   f^+_{\epsilon,\tilde{q},q,r,\sigma}(x) \, d\mu(x)=0.
\end{align*}
where the fourth equality holds since $f^{(\omega)}_{\epsilon,\tilde{q},q,r,\sigma}(x)=0$ wherever $x_{j^\star}=0$.  Further, for $\mathbf{0} \prec \omega' \prec \omega$, now letting $j^\star=\min\{j\in[d]:~\omega_j=1-\omega'_j=1\}$ and by induction we get 
\begin{align*}
    f_{\omega'}(x) &= \mathbb{E}_P\Bigl\{\eta(X) - \frac{1}{2} - \sum_{\omega''\prec\omega'}f_{\omega''}(X) \Big| X^{\omega'} = x^{\omega'}\Bigr\} = \mathbb{E}_P\Bigl\{f^{(\omega)}_{\epsilon,\tilde{q},q,r,\sigma}(X) \Big| X^{\omega'} = x^{\omega'}\Bigr\} \\
    &= \mathbb{E}_P\Bigl\{f^{(\omega)}_{\epsilon,\tilde{q},q,r,\sigma}(X) \bigl( \mathbbm{1}\{x_{j^\star}>0\} + \mathbbm{1}\{x_{j^\star}<0\} \bigr) \Big| X^{\omega'} = x^{\omega'}\Bigr\} = 0
\end{align*}
using the same argument as for the $\omega'=\mathbf{0}_d$ case.  We further get
\begin{align*}
f_{\omega}(x) & = \mathbb{E}_P\Bigl\{\eta(X) - \frac{1}{2} - \sum_{\omega'\prec\omega}f_{\omega'}(X) \Bigm| X^{\omega} = x^{\omega}\Bigr\} 
\\ & = \mathbb{E}_P\Bigl\{f^{(\omega)}_{\epsilon,\tilde{q},q,r,\sigma}(X) \Bigm| X^{\omega} = x^{\omega}\Bigr\} = f^{(\omega)}_{\epsilon,\tilde{q},q,r,\sigma}(x^\omega),
\end{align*} 
i.e. $f_{\omega} = f^{(\omega)}_{\epsilon,\tilde{q},q,r,\sigma}$. For $\omega' \npreceq \omega$, there exists either a $j^\star=\min\{j\in[d]:~\omega_j=1-\omega'_j=1\}$, in which case we proceed as above, or there exists no such $j^\star$, in which case $\omega\prec\omega'$. In this case we have
\begin{align*}
    f_{\omega'}(x) &= \mathbb{E}_P\Bigl\{\eta(X) - \sum_{\omega''\prec\omega'}f_{\omega''}(X) \Big| X^{\omega'} = x^{\omega'}\Bigr\} = \mathbb{E}_P\Bigl\{0\Big| X^{\omega'} = x^{\omega'}\Bigr\} = 0.
\end{align*}
We deduce that $\sigma_{\omega'}^2=0$ for all $\omega'\neq \omega$.  Finally, we have
\begin{align*}
\sigma_\omega^2 &= \int_{\mathbb{R}^d} f_\omega^2(x) \,d\mu(x)  =  a \epsilon^2 + 2^{d_\omega} \int_{\mathcal{R}_{\tilde{j}}^\omega}  1/4 \, d\mu(x) 
=  a \epsilon^2 + \frac{1-a}{4} \geq 1/8,
\end{align*}
thus $Q\in\mathcal{Q}_{\mathrm{E}}(\{\omega\},c_{\mathrm{E}},\mathcal{O})$, for $c_{\mathrm{E}} \leq 1/8$. 

For the final part of the result, if $x_1,x_2\in\mathbb{R}^d$ are such that $\mathrm{sign}(x_1^\omega)=\mathrm{sign}(x_2^\omega)$, then by Corollary~\ref{cor:f_beta_holder_on_IRplus_heavy_tails},
\begin{align*}
|f_\omega(x_1) - f_\omega(x_2)| &= |f^+_{\epsilon,\tilde{q},q,r,\sigma}(s\odot x_1) - f^+_{\epsilon,\tilde{q},q,r,\sigma}(s\odot x_2)| \\
&\leq \frac{1}{2} \|s\odot x_1-s\odot x_2\|^{\beta_{\omega}} \leq \|x_1-x_2\|^{\beta_{\omega}}.  
\end{align*}
On the other hand, if $s_1:=\mathrm{sign}(x_1^\omega) \neq \mathrm{sign}(x_2^\omega)=:s_2$, then there exists $z \in \mathcal{S}^\omega$ with $z=x_1+\zeta(x_2-x_1)$ for some $\zeta\in[0,1]$. Then, again by Corollary~\ref{cor:f_beta_holder_on_IRplus_heavy_tails},
\begin{align*}
|f_\omega(x_1) &-f_\omega(x_2)| \leq |f_\omega(x_1)-f_\omega(z)| + |f_\omega(z)-f_\omega(x_2)| \\
&\leq |f^+_{\epsilon,\tilde{q},q,r,\sigma}(s_1\odot x_1)-f^+_{\epsilon,\tilde{q},q,r,\sigma}(s_1\odot z)| + |f^+_{\epsilon,\tilde{q},q,r,\sigma}(s_2\odot z) - f^+_{\epsilon,\tilde{q},q,r,\sigma}(s_2\odot x_2)| \\
&\leq \frac{1}{2} \|s_1\odot x_1-s_1\odot z\|^{\beta_{\omega}} + \frac{1}{2} \|s_2\odot z-s_2\odot x_2\|^{\beta_{\omega}} = \frac{1}{2} \|x_1-z\|^{\beta_{\omega}} + \frac{1}{2} \|z-x_2\|^{\beta_{\omega}} \\
&= \frac{1}{2} (\zeta^{\beta_{\omega}}+(1-\zeta)^{\beta_{\omega}}) \|x_1-x_2\|^{\beta_{\omega}} \leq \frac{1}{2} 2^{1-\beta_{\omega}} \|x_1-x_2\|^{\beta_{\omega}} \leq \|x_1-x_2\|^{\beta_{\omega}}.
\end{align*}
Finally, for $\omega'\neq\omega$ we have $|f_{\omega'}(x_1)-f_{\omega'}(x_2)|=0\leq\|x_1-x_2\|^{\beta_{\omega'}}$; In other words, we have $P\in\mathcal{P}_{\mathrm{S}}(\boldsymbol{\beta},1)$. 
\end{proof}

\begin{lemma}\label{lem:M_def_satisfied_heavy_tails}
Fix $d\in\IN$, $\omega \in \{0,1\}^d \setminus \{\mathbf{0}_d\}$, $\tilde{q},q \in \mathbb{N}$, $r>1$, $\beta_{\omega} \in (0,1]$, $\tilde{j} \in \{j\in [d] : \omega_{j} = 1\}$, $\epsilon \in (0,(1/4\cdot(r/q)^{\beta_{\omega}}) \wedge (1/4\cdot(1/\tilde{q})^{\beta_{\omega}})]$, $\sigma=(\sigma_t)_{t\in[T]}\in\{-1,1\}^T$.
Fix $C_{\mathrm{M}} \geq 2^\alpha$, $\alpha\in[0,d_{\omega}/\beta_{\omega}]$, and $a\in[0,\epsilon^\alpha]$. Let $P = P^{(\omega)}_{\epsilon,\tilde{q},q,r,\sigma,a}$ denote the distribution on $\mathbb{R}^d\times\{0,1\}$ with regression function $\eta=\eta^{(\omega)}_{\epsilon,\tilde{q},q,r,\sigma}$ and marginal feature distribution $\mu=\mu^{(\omega)}_{\tilde{q},q,r,a}$. Then $P\in\mathcal{P}_{\mathrm{M}}(\alpha,C_{\mathrm{M}})$.
\end{lemma}
\begin{proof}
First, if $t\in (0,\epsilon)$, then $\mu\bigl(\bigl\{x \in \IR^d : |\eta(x)- 1/2| < t \bigr\} \bigr) = 0 \leq C_{\mathrm{M}} \cdot t^{\alpha}$.  For $t\in [\epsilon,1/2)$, we have by the definition of $\eta$ that 
\begin{align*}
\mu\Bigl(\Bigl\{x \in \mathbb{R}^d : \bigl|\eta(x)-1/2\bigr| < t \Bigr\} \Bigr) &=  2^{d_\omega} \mu\bigl(\mathcal{T}^\omega\bigl) 
=  a  \leq   \epsilon^\alpha \leq C_{\mathrm{M}} \cdot t^{\alpha}.
\end{align*}
Finally, if $t\geq1/2$, then
\begin{align*}
\mu\Bigl(\Bigl\{x \in \mathbb{R}^d : \bigl|\eta(x)-1/2\bigr| < t \Bigr\} \Bigr) \leq 1 \leq (2t)^{\alpha}  \leq C_{\mathrm{M}} \cdot t^{\alpha},
\end{align*}
as required. 
\end{proof}

\subsection{Proof of the lower bound in the heavy tailed case\label{subsec:LBproof_heavytails}}

\begin{lemma}\label{lem:LB_heavy_tails}
Fix $d, n \in \mathbb{N}$, $\mathcal{O} \subseteq \{0,1\}^d$, $o_1, \ldots, o_n \in \mathcal{O}$, $\omega \in \{0,1\}^d \setminus \{\mathbf{0}_d\}$, $c_{\mathrm{E}} \in [0,1/4]$, $\boldsymbol{\gamma} \in [0,\infty)^d$, $C_{\mathrm{L}}\geq 1$, $\boldsymbol{\beta}\in(0,1]^d$ and let $\beta_{\omega} = \min\{\beta_j : \omega_j = 1\}$. Fix further $C_{\mathrm{S}} \geq 1$, $\alpha \in [0, \infty)$ and $C_{\mathrm{M}} \geq 1$. Suppose that $\gamma_{\omega} = \min\{\gamma_j : \omega_j = 1\} < 1$, $\alpha\beta_{\omega} \leq d_\omega$, $C_{\mathrm{M}}\geq 2^\alpha$, $c_{\mathrm{E}} \leq 1/8$ and $C_{\mathrm{L}}\geq C_0 \cdot 2^{3 + 3\gamma_{\omega} + 3d_\omega \gamma_{\max}/2}  \cdot d_{\omega}^{d_{\omega}\gamma_{\max}/2}$.
Then there exists a constant $c$, depending only on $\gamma_\omega$, $C_{\mathrm{L}}$, $d_{\omega}$, $\beta_{\omega}$, $\alpha$, such that 
\begin{align*}
    \inf_{\hat C\in\mathcal C_n} \sup_{Q}\mathbb{E}_{Q}\bigl\{\mathcal{E}_{P}(\hat C) \bigm| O_1 = o_1, \ldots, O_n = o_n\bigr\}\geq c \cdot \min\Bigl\{ n_\omega^{-\frac{\beta_{\omega}\gamma_\omega(1+\alpha)}{\gamma_\omega(2\beta_{\omega}+d_\omega)+\alpha\beta_{\omega}}}, 1 \Bigr\},
\end{align*}
where the supremum is taken over $Q\in\mathcal{Q}_{\mathrm{Miss}}'(\{\omega\},c_{\mathrm{E}}, \boldsymbol{\gamma}, C_{\mathrm{L}}, \boldsymbol{\beta}, C_{\mathrm{S}} , \alpha, C_{\mathrm{M}},\mathcal{O})$.
\end{lemma}
\begin{proof}
First, let $a_1 := 2^{3\gamma_{\omega}/\beta_{\omega}}$ and
\[
q_0:=\min\biggl\{ 2^{5+d_\omega} , \Bigr(2^{(d_{\omega} + 5)\frac{\alpha(1-\gamma_{\omega})+\gamma_{\omega}/\beta_{\omega}}{2+\alpha} } \cdot a_1^{-1} \Bigl)^{\frac{2+\alpha}{\alpha(1-\gamma_{\omega})+\gamma_{\omega}/\beta_{\omega}+\gamma_{\omega}(2+\alpha)}}, 2^{(2+d_\omega)(2+\alpha)}\biggr\}.
\]
Further let
\begin{align*}
\rho &:= \frac{\gamma_{\omega}(d_\omega-\alpha\beta_{\omega})+\alpha\beta_{\omega}}{\gamma_{\omega}(2\beta_{\omega}+d_\omega)+\alpha\beta_{\omega}}\in[0,1); \quad \rho_1 := \frac{\gamma_{\omega}(1-\alpha\beta_{\omega})+\alpha\beta_{\omega}}{\gamma_{\omega}(2\beta_{\omega}+d_\omega)+\alpha\beta_{\omega}} \in [0, \rho]; 
\\ q &:= \lfloor q_0 n_{\omega}^{\frac{\gamma_{\omega} + \alpha \beta_{\omega} (1-\gamma_{\omega})}{\gamma_{\omega} (2\beta_{\omega} +d_{\omega}) + \alpha\beta_{\omega}}} \rfloor, \quad \tilde{q} := \lfloor n_{\omega}^{\frac{\gamma_{\omega}}{\gamma_{\omega} (2\beta_{\omega} +d_{\omega}) + \alpha\beta_{\omega}}} \rfloor;  \quad m := q\cdot \tilde{q}^{d_{\omega}-1}
\\ \epsilon &:= \min\Bigl\{\Bigl(\frac{m}{2^{d_\omega+5} n_\omega}\Bigr)^{\frac{1}{2+\alpha}},1/4\Bigr\}, \quad r := (8\epsilon)^{1/\beta_{\omega}} \cdot q, \quad u := \frac{\epsilon^\alpha}{m},\quad
a:=\epsilon^\alpha.
\end{align*}
Fix $\tilde{j} \in \argmin_{j \in \{j' \in [d]:\omega_{j'}=1\}} \gamma_j$, so that $\gamma_{\tilde{j}} = \gamma_{\omega}$. 

Suppose initially that $n_{\omega}  >  q_{0}^{-1/\rho_1}$, so that $q \geq 1$. For $\sigma \in\{-1,1\}^m$, let $Q^\sigma\equiv Q^\sigma_{q,\tilde{q},r,a,\tilde{j}}$ be the distribution of $(X,Y,O)$ on $\mathbb{R}^d\times\{0,1\}\times\{0,1\}^d$ with $X$-marginal distribution $\mu = \mu^{(\omega)}_{q,\tilde{q},r,a,\tilde{j}}$, with regression function $\eta=\eta^{(\omega)}_{\epsilon,\tilde{q},q,r,\sigma}$, with $O\indep(X,Y)$ and with $\IP(O=o)=1/|\mathcal{O}|$ for all $o\in\mathcal{O}$. We now show that $Q^\sigma\in\mathcal{Q}_{\mathrm{Miss}}'(\{\omega\},c_{\mathrm{E}}, \boldsymbol{\gamma}, C_{\mathrm{L}}, \boldsymbol{\beta}, C_{\mathrm{S}} , \alpha, C_{\mathrm{M}},\mathcal{O})$, by applying in turn the results from the previous subsections.  First, to apply Corollary \ref{cor:L_def_satisfied}, using that $\rho-1 = \frac{-\gamma_{\omega} \beta_{\omega} (2+\alpha)}{\gamma_{\omega}(2\beta_{\omega} + d_{\omega}) + \alpha\beta_{\omega}}$ and that 
\[
m\leq q_0 n_{\omega}^{\frac{\gamma_{\omega} + \alpha \beta_{\omega} (1-\gamma_{\omega})}{\gamma_{\omega} (2\beta_{\omega} +d_{\omega}) + \alpha\beta_{\omega}}} \cdot n_{\omega}^{\frac{(d_\omega-1)\gamma_{\omega}}{\gamma_{\omega} (2\beta_{\omega} +d_{\omega}) + \alpha\beta_{\omega}}} = q_0 n_\omega^{\rho-1},
\]
it follows that 
\begin{align*} 
a & = \epsilon^{\alpha} \leq \Bigl(\frac{q_0 n_{\omega}^{\rho}} {2^{d_{\omega} + 5} n_{\omega}} \Bigr)^{\frac{\alpha}{2+\alpha}} = \Bigl(\frac{q_0}{2^{5 + d_{\omega}}}\Bigr)^{\alpha/(2+\alpha)} \cdot n_{\omega}^{-\frac{\gamma_{\omega}\alpha\beta_{\omega}}{\gamma_{\omega}(2\beta_{\omega}+d_{\omega}) + \alpha \beta_{\omega}}} \leq  1,
\end{align*}
by the first term in the minimum in the definition of $q_0$. Further, this also implies that $\bigl(\frac{m}{2^{d_{\omega} + 5} n_{\omega}} \bigr)^{\frac{1}{2+\alpha}}  < \frac{1}{4\tilde{q}}$, so $\epsilon < 1/(4\tilde{q}) \leq 1/4$. Moreover, if 
\[
n_{\omega} > N_0 := \max\Bigl\{q_0^{-1/\rho_1},  \Bigl(\frac{8}{2^{\beta_{\omega}}} \Bigl(\frac{1}{2^{2d_{\omega}+5} }\Bigr)^{\frac{1}{2+\alpha}} q_0^{\frac{1 + 2\beta_{\omega} + \alpha\beta_{\omega}}{2+\alpha}}\Bigr)^{-\frac{\gamma_\omega(2\beta_{\omega}+d_{\omega}) + \alpha \beta_{\omega}}{\alpha\beta_{\omega}^2(1-\gamma_{\omega})}}\Bigr\},
\] 
then, using that $\gamma_{\omega} < 1$, we have
\[
r^{\beta_{\omega}} = 8 \epsilon q^{\beta_{\omega}} \geq 4\Bigl(\frac{q_0 n_{\omega}^{\rho}}{2^{2d_{\omega}+5} n_{\omega}}\Bigr)^{\frac{1}{2+\alpha}} q_0^{\beta_{\omega}}  n_{\omega}^{\beta_{\omega}\rho_1} = \frac{8}{2^{\beta_{\omega}}} \Bigl(\frac{1}{2^{2d_{\omega}+5} }\Bigr)^{\frac{1}{2+\alpha}} q_0^{\frac{1 + 2\beta_{\omega} + \alpha\beta_{\omega}}{2+\alpha}} \cdot n_{\omega}^{\frac{\alpha\beta_{\omega}^2(1-\gamma_{\omega})}{\gamma_\omega(2\beta_{\omega}+d_{\omega}) + \alpha \beta_{\omega}}} > 1,
\]
so $r >1$. Suppose now that $n_{\omega} > N_0$. Then, since $\alpha\beta_{\omega} \leq d_{\omega}$,
\begin{align*}
a^{1-\gamma_{\omega}} r^{\gamma_{\omega}} & = 2^{3\gamma_{\omega}/\beta_{\omega}} \epsilon^{\alpha(1-\gamma_{\omega}) + \gamma_{\omega}/\beta_{\omega}} \cdot q^{\gamma_{\omega}} 
\leq 2^{3\gamma_{\omega}/\beta_{\omega}}  \Bigl(\frac{q_0 n_{\omega}^{\rho-1}} {2^{d_{\omega} + 5}} \Bigr)^{\frac{\alpha(1-\gamma_{\omega})+\gamma_{\omega}/\beta_{\omega}}{2+\alpha}} \cdot q_0^{\gamma_{\omega}} n_{\omega}^{\frac{\gamma_{\omega}\{\gamma_\omega + \alpha \beta_{\omega}(1-\gamma_{\omega})\}}{\gamma_{\omega}(2\beta_{\omega} + d_{\omega}) + \alpha\beta_{\omega}} } \\
&= a_1 \cdot \Bigl(\frac{q_0} {2^{d_{\omega} + 5}} \Bigr)^{\frac{\alpha(1-\gamma_{\omega})+\gamma_{\omega}/\beta_{\omega}}{2+\alpha}} \cdot q_0^{\gamma_{\omega}} \leq 1,
\end{align*} 
where we have used the second term in the minimum in the definition of $q_0$. Thus by Corollary~\ref{cor:L_def_satisfied_heavy_tails} we have $Q^\sigma \in \mathcal{Q}_{\mathrm{L}}(\boldsymbol{\gamma}, C_{\mathrm{L}},\mathcal{O})$. Further we have
\begin{align*} 
8\epsilon\cdot \tilde{q}^{\beta_{\omega}} \leq 8 \Bigl(\frac{q_0 n_{\omega}^{\rho}} {2^{d_{\omega} + 5} n_{\omega}} \Bigr)^{\frac{1}{2+\alpha}} \tilde{q}^{\beta_{\omega}} = 8 \Bigl(\frac{q_0}{2^{5 + d_{\omega}}}\Bigr)^{\frac{1}{2+\alpha}} 
\leq 1
\end{align*}
by the third term in the definition of $q_0$.
Finally, by the upper bound on $c_{\mathrm{E}}$, Lemma~\ref{lem:E_and_S_def_satisfied_heavy_tails}, the upper bound on $C_{\mathrm{M}}$ and Lemma~\ref{lem:M_def_satisfied_heavy_tails}, we deduce that $Q^{\sigma} \in \mathcal{Q}_{\mathrm{E}}(\{\omega\}, c_{\mathrm{E}},\mathcal{O})$, as well as $P^\sigma\in \mathcal{P}_{\mathrm{S}}(\boldsymbol{\beta}, 1) \cap \mathcal{P}_{\mathrm{M}}(\alpha, C_{\mathrm{M}})$ for the distribution $P^\sigma$ of $(X,Y)$.  

To complete the proof we apply Lemma~\ref{lem:assouad}. We first verify that Assumptions~(i)-(vi) in that lemma hold: 
\begin{enumerate}
\item[(i)] By the definition of $\epsilon$ and $u$, we have $2^{d_\omega+5}n_\omega\epsilon^2u=1$, 
\item[(ii)] Letting $z_t:=z_t^\omega$ for $t\in[m]$, we have that $\mu(\{s^{\omega} \odot z_t^{\omega}\})=u$ for all $t\in[m]$ and $s\in\{-1,1\}^d$ by the construction of the marginal measure in~\eqref{eq:muheavy}
,
\item[(iii)]  
By the construction of the regression function \eqref{def:f_circ_plus}, we have $\eta^\sigma(z^{\omega}_t) = 1/2 + \sigma_t \epsilon$ for $t\in[m]$, and $\eta^\sigma(s^{\omega} \odot z_t^{\omega}) = 1/2 + \bigl(\prod_{j \in [d] :(s^{\omega})_{j} =-1} (s^{\omega})_{j} \bigr)\cdot\sigma_t\epsilon$, for all $t\in[m]$ and $s\in\{-1,1\}^d$,
\item[(iv)] Since, for $\sigma,\sigma'\in\Sigma$, the support is given by $\mathrm{supp}(\mu)=\bigcup_{s\in\{-1,1\}^d}s\odot(\mathcal{T}^\omega \cup \mathcal{R}_{\tilde{j}}^\omega)$, and since we have $\eta^\sigma=\eta^{\sigma'}$ on $\mathcal{R}_{\tilde{j}}^\omega$, we indeed have $\eta^\sigma(x)=\eta^{\sigma'}(x)$ for all $x\in \mathrm{supp}(\mu) \setminus \bigcup_{t\in [m], s\in \{-1,1\}^d} \{s^{\omega}\odot z_t^{\omega}\}$,
\item[(v)] 
Again by construction, we have $\Bigl(\mathrm{supp}(\mu) \setminus \bigcup_{t\in [m], s\in \{-1,1\}^d} \{s^{\omega}\odot z_t^{\omega}\} \Bigr) \bigcap (-r-1,r+1)^d = \emptyset$,
\item[(vi)] By construction of the distribution of $O$, we have $O\indep(X,Y)$.
\end{enumerate} 
Then, for $n_{\omega} > N_0$, since $\epsilon = \bigl(\frac{m}{2^{d_{\omega}+5}n_\omega}\bigr)^{\frac{1}{2+\alpha}}$ the lower bound in~\eqref{eq:lb_assouad} in Lemma~\ref{lem:assouad} gives
\begin{align*}
     \inf_{\hat C\in\mathcal C_n} \sup_{Q}\ &\mathbb E_{Q}\{\mathcal E_{P}(\hat C) \mid O_1=o_1, \ldots, O_n=o_n\} \geq  \frac{m u \epsilon}{2} = \frac{\epsilon^{1+\alpha}}{2} 
    = \frac{1}{2} \left(\frac{m}{2^{d_{\omega}+5}n_\omega}\right)^{\frac{1+\alpha}{2+\alpha}} \\
    &\geq \frac{1}{2} \Bigl( \frac{q_0 n_\omega^\rho} {2^{d_{\omega}+6} n_\omega} \Bigr)^{\frac{1+\alpha}{2+\alpha}} = 2^{-1-\frac{(d_{\omega}+6)(1+\alpha)}{2+\alpha}} q_0^{\frac{1+\alpha}{2+\alpha}} n_\omega^{\frac{(1+\alpha)(\rho-1)}{(2+\alpha)}}
    =: \tilde{c} \cdot  n_\omega^{-\frac{\beta_{\omega}\gamma_\omega(1+\alpha)}{\gamma_\omega(2\beta_{\omega}+d_\omega)+\alpha\beta_{\omega}}}.
\end{align*}
Finally, if $n_{\omega} \leq N_0$ (including possibly $n_{\omega} = 0$), then using the fact that the excess risk is decreasing in $n_\omega$, we conclude that the result follows with $c := \tilde{c} \cdot N_0^{\frac{\beta_{\omega}\gamma_\omega(1+\alpha)}{\gamma_\omega(2\beta_{\omega}+d_\omega)+\alpha\beta_{\omega}}}$. 
\end{proof}

\subsection{Proof of the lower bound in Theorem~\ref{thm:minmax_bounds}
\label{subsec:LBproof_general}}
In this section, we extend the ideas from this section to the case that $|\Omega_{\star}| \geq 2.$  If $\Omega_\star = \{\omega^\star\}$ for one $\omega^\star \in\{0,1\}^d\setminus\{\mathbf{0}_d\}$, then to prove the lower bound in Theorem~\ref{thm:minmax_bounds}
, we can directly apply Lemma~\ref{lem:LB_light_tails} if $\gamma_{\omega^\star}\geq 1$, or Lemma~\ref{lem:LB_heavy_tails} if $\gamma_{\omega^\star}< 1$. 
To deal with the case that $|\Omega_\star|\geq 2$ we construct a class of mixture distributions. One component in the mixture follows the same construction as in the preceding sections, whereas the other components consist of point masses, which ensure that the mixture falls into the class $\mathcal{Q}_{\mathrm{E}}(\Omega_{\star}, c_{\mathrm{E}}, \mathcal{O})$.  To this end, for $|\Omega_\star|\geq2$, fix 
\[
\omega^\star \in \argmax_{\omega\in\Omega_\star\cap\mathcal{N}} \bigl\{ n_\omega^{-\frac{\beta_{\omega}\gamma_\omega(1+\alpha)}{\gamma_\omega(2\beta_{\omega}+d_\omega)+\alpha\beta_{\omega}}} \bigr\}.
\]
If $\gamma_{\omega^\star}\geq 1$, we will use the construction from Section~\ref{sec:LB_small_r} and Lemma~\ref{lem:LB_light_tails}. On the other hand, if $\gamma_{\omega^\star}< 1$, we will use the construction from Section~\ref{sec:LB_big_r} and Lemma~\ref{lem:LB_heavy_tails}. To be more precise, let $r\geq1$, $\kappa>0$, $\tilde{q},q\in\mathbb{N}$, $a,b\in[0,1/2]$, $\tilde{j}\in[d]$, and define
\begin{align}\label{def:mu_mixture_LB}\begin{split}
\mu^{(\Omega_{\star})}(A) \equiv \mu^{(\Omega_{\star})}_{\kappa,r,\tilde{q},q,a,b,\tilde{j}}(A) &:= \frac{1}{|\Omega_\star|} \cdot \bigl\{ \mu^{(\omega^\star)}_{\kappa,r,q,a,b}(A)\mathbbm{1}_{\{\gamma_{\omega^\star}\geq 1\}} + \mu^{(\omega^\star)}_{\tilde{q},q,r,a,\tilde{j}}(A)\mathbbm{1}_{\{\gamma_{\omega^\star}<1\}} \bigr\} \\
&\hspace{18mm}+ \frac{1}{2^d|\Omega_\star|} \cdot \sum_{\omega\in\Omega_\star\setminus\{\omega^\star\}} \sum_{s\in\{-1,1\}^d}  \mathbbm{1}_{\{(1+r) \cdot (s\odot \omega)\in A\}}.
\end{split}
\end{align}
Furthermore, let 
\[
\mathcal{R}^\pm := \Bigl\{ x \in\mathbb{R}^d:~\exists s\in\{-1,1\}^d, \omega\in\Omega_\star\setminus\{\omega^\star\} \text{ such that }x=(1+r) \cdot (s\odot\omega) \Bigr\}
\]
and define, for $\epsilon\in(0,1/4)$ and $\sigma\in\{-1,1\}^T$ the regression function as
\begin{align}\label{def:eta_mixture_LB}
\eta^{(\Omega_{\star})}_{\epsilon,\tilde{q},q,r,\sigma}(x) := 
\begin{cases}
\eta^{(\omega^\star)}_{\epsilon,q,r,\sigma}(A)\mathbbm{1}_{\{\gamma_{\omega^\star}\geq 1\}} + \eta^{(\omega^\star)}_{\epsilon,\tilde{q},q,r,\sigma}\mathbbm{1}_{\{\gamma_{\omega^\star}<1\}} 
&\mbox{if } x=x^{\omega^\star}\\ 
1/2+1/2\cdot\Bigl(\prod_{j \in [d_{\omega}]} \{\Pi_{\omega}(s)\}_j \Bigr)  &\mbox{if } x=(1+r) \cdot (s\odot\omega) \in\mathcal{R}^\pm
\end{cases}.
\end{align}

Let $Q\equiv Q^{(\Omega_{\star})}_{\kappa,r,\tilde{q},q,a,b,\tilde{j},\sigma}$ denote the distribution on $\mathbb{R}^d\times\{0,1\}\times\{0,1\}^d$ with regression function $\eta^{(\Omega_{\star})}_{\epsilon,\tilde{q},q,r,\sigma}$,  marginal feature distribution $\mu^{(\Omega_{\star})}_{\kappa,r,\tilde{q},q,a,b,\tilde{j}}$, and for which $O \indep (X,Y)$ with $O \sim U(\mathcal{O})$ when $(X,Y,O) \sim Q$. We will now show that $Q \in \mathcal{Q}_{\mathrm{E}}(\Omega_\star,c_{\mathrm{E}},\mathcal{O})\cap\mathcal{Q}_{\mathrm{L}}(\boldsymbol{\gamma},C_{\mathrm{L}},\mathcal{O})$ and that $P\equiv P_Q\in\mathcal P_{\mathrm{S}}(\boldsymbol{\beta},C_{\mathrm{S}})\cap\mathcal P_{\mathrm{M}}(\alpha,C_{\mathrm{M}})$.

\begin{lemma}\label{lem:generalLB_L_def_satisfied}
Fix $d,n\in\IN$, $\mathcal{O}\subseteq\{0,1\}^d$, $\boldsymbol{\gamma} = (\gamma_1,\dots,\gamma_d) \in [0, \infty)^d$, $C_{\mathrm{L}}>1$, $\boldsymbol{\beta} = (\beta_1,\dots,\beta_d)\in(0,1]^{d}$, $\alpha\in[0,\infty)$, $\Omega_\star\in\mathcal I(\{0,1\}^d)\setminus\{\{\mathbf{0}_d\},\emptyset\}$ with $|\Omega_\star|\geq 2$.  Fix further $o_1,\ldots,o_n\in\mathcal{O}$ and $\omega^\star \in \argmax_{\omega\in\Omega_\star\cap\mathcal{N}} \bigl\{ n_\omega^{-\frac{\beta_{\omega}\gamma_\omega(1+\alpha)}{\gamma_\omega(2\beta_{\omega}+d_\omega)+\alpha\beta_{\omega}}} \bigr\}$, as well as $r> 0$, $\tilde{q},q \in \mathbb{N}$, $\kappa \in (0,1/\sqrt{d_{\omega^\star}})$, $C_0>1$, $\tilde{j} \in \{j\in [d] : \omega^\star_{j} = 1\}$, and $a,b \in [0,1/4]$. Let $\gamma_{\omega^\star} := \min\{\gamma_j : \omega^\star_j = 1\}$ and $\gamma_{\max}:=2\cdot\mathbbm{1}_{\{\gamma_{\omega^\star}<1\}}\vee\max\{\gamma_j:j\in[d]\}$, $a_0 := 2^{-3d_{\omega^\star}\gamma_{\omega^\star}}$ and $b_0 = 2^{-1}(d_{\omega^\star} 2^{d_{\omega^\star}})^{-\gamma_{\max}}$. Consider the set of assumptions defined by 
\begin{enumerate}
    \item[A1.] 
    \begin{enumerate}
        \item[(i)] $\gamma_{\omega^\star}\leq 1$
        \item[(ii)] $a \leq a_0 \wedge b$
        \item[(iii)] $a^{1-\gamma_{\omega^\star}} (4r\sqrt{d_{\omega^\star}})^{d_{\omega^\star}\gamma_{\omega^\star}}\leq a_0$
        \item[(iv)] $b^{1-\gamma_{\omega^\star}} \leq C_{\mathrm{L}} \cdot b_0 \cdot \Bigl( \min\Bigl\{1, (r\sqrt{d_{\omega^\star}})^{-d_{\omega^\star}}\Bigr\} \Bigr)^{\gamma_{\omega^\star}}$
        \item[(v)] For all $\omega'$ with $\omega^\star \wedge \omega' \neq \mathbf{0}_d$: $b^{1-\gamma_{\max}} \leq C_{\mathrm{L}} \cdot b_0 \cdot \Bigl( \min\Bigl\{1, (r\sqrt{d_{\omega^\star\wedge\omega'}})^{-d_{\omega^\star\wedge\omega'}}\Bigr\} \Bigr)^{\gamma_{\max}}$
        \item[(vi)] $C_{\mathrm{L}}>4^{\gamma_{\max}(d_{\omega^\star}+1)}$
    \end{enumerate}
    
    \item[A2.]
    \begin{enumerate}
        \item[(i)] $\gamma_{\omega^\star}> 1$
        \item[(ii)] $a < \min\{1/2, 1-1/C_0^{1/{\gamma_{\omega^\star}}}\}$
        \item[(iii)] $a^{1-\gamma_{\omega^\star}}r^{\gamma_{\omega^\star}}\leq 1$
        \item[(iv)] $C_{\mathrm{L}}\geq C_0 \cdot 2^{2 + 3\gamma_{\omega^\star} + 3d_{\omega^\star} \gamma_{\max}/2}  \cdot d_{\omega^\star}^{d_{\omega^\star}\gamma_{\max}/2}$
    \end{enumerate}
\end{enumerate}
and assume that either A1 or A2 holds, and additionally that $C_{\mathrm{L}}\geq (2^{\max_{\omega\in\Omega_\star}\{d_\omega\}}|\Omega_\star|)^{\gamma_{\max}}$. Let $Q\equiv Q^{(\Omega_{\star})}_{\kappa,r,\tilde{q},q,a,b,\tilde{j}}$ denote a distribution on $\mathbb{R}^d\times\{0,1\}\times\{0,1\}^d$ of $(X,Y,O)$ with corresponding $X$-marginal feature distribution $\mu\equiv \mu^{(\Omega_{\star})}_{\kappa,r,\tilde{q},q,a,b,\tilde{j}}$, for which $\IP(O=o)>0$ for all $o\in\mathcal{O}$ and for which $O\indep(X,Y)$. Then $Q \in \mathcal Q_{\mathrm{L}}(\boldsymbol{\gamma},C_{\mathrm{L}},\mathcal{O})$. 
\end{lemma}
\begin{proof}
Fix any $\omega'\in\{0,1\}^d$. For any $x\in\mathcal{R}^{\pm}$ and any $\zeta\in(0,1)$, we have 
\begin{align*}
    \mu_{\omega'}\bigl( B_\zeta(x^{\omega'}) \bigr) &= \mu\bigl( \{ z\in\mathbb{R}^d:z^{\omega'}\in B_\zeta(x^{\omega'}) \} \bigr)\\
    &\geq \frac{1}{2^d|\Omega_\star|}\cdot \sum_{\omega\in\Omega_\star\setminus\{\omega^\star\}}\sum_{s\in\{-1,1\}^d}\mathbbm{1}_{\{(1+r) \cdot (s\odot\omega) \in \{ z\in\mathbb{R}^d:z^{\omega'}\in B_\zeta(x^{\omega'}) \} \}} \\
    &\geq \frac{1}{2^{|x|}|\Omega_\star|}
    \geq \frac{1}{2^{\max_{\omega\in\Omega_\star}\{d_\omega\}}|\Omega_\star|},
\end{align*}
where we used that $|x|$ are the number of observed variables for $x\in\mathcal{R}^{\pm}$. Thus $\rho_{\omega'}(x)\geq (2^{\max_{\omega\in\Omega_\star}\{d_\omega\}}|\Omega_\star|)^{-1}$ for any $x\in\mathcal{R}^{\pm}$. Now, by Corollary~\ref{cor:L_def_satisfied} or Corollary~\ref{cor:L_def_satisfied_heavy_tails}, we note that if $\xi\in (0 ,(2^{\max_{\omega\in\Omega_\star}\{d_\omega\}}|\Omega_\star|)^{-1} ]$, then
\begin{align*}
    \mu_{\omega'}\bigl( \{x\in\mathbb{R}^d:~\rho_{\omega'}(x)<\xi\} \bigr) &\leq \frac{1}{|\Omega_\star|} \cdot \bigl\{ \mu^{(\omega^\star)}_{\kappa,r,q,a,b}(A)\mathbbm{1}_{\{\gamma_{\omega^\star}\geq 1\}} + \mu^{(\omega^\star)}_{\tilde{q},q,r,a,\tilde{j}}(A)\mathbbm{1}_{\{\gamma_{\omega^\star}<1\}} \bigr\} \\&\leq \frac{C_{\mathrm{L}}}{|\Omega_\star|} \cdot \xi^{\gamma_{\omega'}} \leq C_{\mathrm{L}} \cdot \xi^{\gamma_{\omega'}}.
\end{align*}

If $\xi > (2^{\max_{\omega\in\Omega_\star}\{d_\omega\}}|\Omega_\star|)^{-1}$, then
\begin{align*}
    \mu_{\omega'}\bigl( \{x\in\mathbb{R}^d:~\rho_{\omega'}(x)<\xi\} \bigr) &\leq 1 \leq C_{\mathrm{L}} \cdot \xi^{\gamma_{\omega'}},
\end{align*}
since $C_{\mathrm{L}}\geq (2^{\max_{\omega\in\Omega_\star}\{d_\omega\}}|\Omega_\star|)^{\gamma_{\max}}$. The result follow since $\mu_{\omega'}=\mu_{\omega'|o}$ by the independence of $O$ and $(X,Y)$.
\end{proof}

\begin{lemma}\label{lem:generalLB_E_and_S_def_satisfied}
Fix $d,n\in\IN$, $\mathcal{O}\subseteq\{0,1\}^d$, $\boldsymbol{\gamma} = (\gamma_1,\dots,\gamma_d) \in [0, \infty)^d$, $\boldsymbol{\beta}=(\beta_1,\ldots,\beta_d) \in (0,1]^{d}$, $\alpha\in[0,\infty)$, and $\Omega_\star\in\mathcal I(\{0,1\}^d)\setminus\{\{\mathbf{0}_d\},\emptyset\}$ with $|\Omega_\star|\geq 2$.  Fix further $o_1,\ldots,o_n\in\mathcal{O}$ and   $\omega^\star \in \argmax_{\omega\in\Omega_\star\cap\mathcal{N}} \bigl\{ n_\omega^{-\frac{\beta_{\omega}\gamma_\omega(1+\alpha)}{\gamma_\omega(2\beta_{\omega}+d_\omega)+\alpha\beta_{\omega}}} \bigr\}$.  Let $\gamma_{\omega^\star} := \min\{\gamma_j : \omega^\star_j = 1\}$.  Fix further $r>0$, $\tilde{q},q \in \mathbb{N}$, $\kappa \in (0,1/\sqrt{d_{\omega^\star}})$, $\tilde{j} \in \{j\in [d] : \omega_{j} = 1\}$, $a,b\in(0,1/4]$, $\epsilon \in (0,1]$, $\sigma=(\sigma_t)_{t\in[T]}\in\{-1,1\}^T$,  $c_{\mathrm{E}} \leq 1$. 
Consider the set of assumptions defined by 
\begin{enumerate}
    \item[B1.] 
    \begin{enumerate}
        \item[(i)] $\gamma_{\omega^\star}\leq 1$
        \item[(ii)] $\epsilon \leq (1/4)\wedge(1/8\cdot(r/q)^{\beta_{\omega^\star}})$
        \item[(iii)] $c_{\mathrm{E}} \leq b/(4|\Omega_\star|)$
    \end{enumerate}
    
    \item[B2.]
    \begin{enumerate}
        \item[(i)] $\gamma_{\omega^\star}> 1$
        \item[(ii)] $\epsilon \leq (1/4\cdot(r/q)^{\beta_{\omega^\star}}) \wedge (1/4\cdot(1/\tilde{q})^{\beta_{\omega^\star}})$
        \item[(iii)] $c_{\mathrm{E}} \leq 1/(8|\Omega_\star|)$
        \item[(iv)] $b\geq 1/2$
    \end{enumerate}
\end{enumerate}
Let $Q\equiv Q^{(\Omega_{\star})}_{\kappa,r,\tilde{q},q,a,b,\tilde{j},\sigma}$ be a distribution of $(X,Y,O)$ on $\mathbb{R}^d\times\{0,1\}\times\{0,1\}^d$ with $X$-marginal distribution $\mu^{(\Omega_{\star})}_{\kappa,r,\tilde{q},q,a,b,\tilde{j}}$, with regression function $\eta=\eta^{(\Omega_{\star})}_{\epsilon,\tilde{q},q,r,\sigma}$, with $O\indep(X,Y)$ and with $\IP(O=o)>0$ for all $o\in\mathcal{O}$. Suppose that either B1 or B2 holds, then $Q$ belongs to $\mathcal{Q}_{\mathrm{E}}(\Omega_\star, c_{\mathrm{E}},\mathcal{O})$ and $P_Q$, the distribution of $(X,Y)$, belongs to $\mathcal{P}_{\mathrm{S}}(\boldsymbol{\beta}, 1)$. 
\end{lemma}
\begin{proof}
For $\omega^\star$, we set $f^{(\omega^\star)}(x):=f^{(\omega^\star)}_{\epsilon,q,r,\sigma}$ from Definition~\ref{def:f_omega} (if $\gamma_{\omega^\star}\geq1$) or $f^{(\omega^\star)}(x):=f^{(\omega^\star)}_{\epsilon,\tilde{q},q,r,\sigma}$ from Definition~\ref{def:f_omega_heavytails} (otherwise).
For $\omega\in\Omega_\star\setminus\{\omega^\star\}$, recall that $S^\omega:=\{x^\omega\in\mathbb{R}^d:(\Pi_\omega(x^\omega))_j=0\text{ for some }j\in[d_\omega]\}$. Now define functions $f^{\omega,\circ,+}$ on the set $\{(1+r)\cdot\omega\}\cup S^\omega$ as 
\begin{align*}
    f^{\omega,\circ,+}(x)=\begin{cases}
    1/2  &\mbox{if } x=(1+r) \cdot \omega \\ 
    0  &\mbox{if } x\in S^\omega
    \end{cases}.
\end{align*}
We note that, for any $x_1,x_2\in \{(1+r) \cdot\omega\}\cup S^\omega$,  $f^{\omega,\circ,+}$ satisfies
\begin{align}\label{eq:holder_smooth_f_omega}
    |f^{\omega,\circ,+}(x_1)-f^{\omega,\circ,+}(x_2)| \leq \frac{1}{2} \|x_1-x_2\|^{\beta_{\omega}}.
\end{align}
Thus, by McShane's extension theorem \citep{mcshane1934extension} on $\IR^{d_\omega}$, there exists a $\beta_{\omega}$-H\"older continuous extension of $f^{\omega,+}$ onto $([0,\infty)^d)^\omega$. We now set $f^{\omega,+}(x)=f^{\omega,+}(x^\omega)$ for any $x\in[0,\infty)^d$.
Recall that for $x=(x_1,\ldots,x_d)^T\in\mathbb{R}^d$, we let $s=(s_1,\ldots,s_d)^T\in\{-1,1\}^d$ be $s_j:=\mathrm{sign}(x_j)$. We now define for all $\omega\in\Omega_\star\setminus\{\omega^\star\}$ the functions $f^{(\omega)}$ as
\begin{align}\label{def:f_omega_general}
    f^{(\omega)}(x):=\Bigl( \prod_{j\in[d_\omega]} \{\Pi_\omega(s)\}_j \Bigr) \cdot f^{\omega,+}(s\odot x).
\end{align}

With the choice of $f^{(\omega)}$ as in~\eqref{def:f_omega_general}, we can write 
\[
\eta_\sigma(x) = \frac{1}{2} + f^{(\omega^\star)}_\sigma(x) + \sum_{\omega\in\Omega_\star\setminus\{\omega^\star\}} f^{(\omega)}(x).
\]
Throughout the remainder of the proof, we will drop the dependence of $f^{(\omega^\star)}_\sigma$ on $\sigma$, and simply write $f^{(\omega)}$ instead of $f^{(\omega^\star)}_\sigma$, whenever $\omega=\omega^\star$, such that we have $\eta_\sigma(x) = \frac{1}{2} + \sum_{\omega\in\Omega_\star} f^{(\omega)}(x)$. We now show by induction over $d'=|\omega'|=0,\ldots,d$ that the functions given by Equations~\eqref{eq:def_f_omega} 
and \eqref{def:f_omega_general} satisfy $f_\omega=f^{(\omega)}$ for all $\omega\in\Omega_\star$
and that $f_{\omega} \equiv 0$, for all $\omega \in\{0,1\}^d \setminus \Omega_\star$. 

Firstly, for $d'=|\omega'|=0$ we have $\omega'=\mathbf{0}_d\notin\Omega_\star$. There exists, for any $\omega\in\Omega_\star$
, a $j^{\star,\omega}=\min\{j\in[d]:~\omega_j=1\}$, such that
\begin{align}\begin{split}
    \mathbb{E}_P\{f^{(\omega)}&(X)\} = \int_{\mathbb{R}^d} f^{(\omega)}(x) \mathbbm{1}_{\{x_{j^{\star,\omega}}>0\}} \, d\mu(x) + \int_{\mathbb{R}^d} f^{(\omega)}(x) \mathbbm{1}_{\{x_{j^{\star,\omega}}<0\}} \, d\mu(x)  \\
&=  \frac{1}{2^{d-d_\omega}} \sum_{s\in \{-1,1\}^d:s_{j^{\star,\omega}}=1} \Bigl(\prod_{j = 2}^{d_{\omega}} \{\Pi_{\omega}(s)\}_j \Bigr) \cdot \int_{\mathcal{R}^{\pm}\cap[0,\infty)^d}   f^{\omega,+}(x) \, d\mu(x) \\
&\hspace{8mm} - \frac{1}{2^{d-d_\omega}} \sum_{s\in \{-1,1\}^d:s_{j^{\star,\omega}}=-1} \Bigl(\prod_{j = 2}^{d_{\omega}} \{\Pi_{\omega}(s)\}_j \Bigr) \cdot \int_{\mathcal{R}^{\pm}\cap[0,\infty)^d}   f^{\omega,+}(x) \, d\mu(x)=0.\label{eq:LB_f_omega_is_zero}
\end{split}
\end{align}
We conclude that
\begin{align*}
f_{\mathbf{0}} &= \mathbb{E}_P\bigl\{\eta(X)\bigr\} -\frac{1}{2} = \mathbb{E}_P \Bigl\{  \sum_{\omega\in\Omega_\star} f^{(\omega)}(X) \Bigr\} = 0,
\end{align*}
such that the induction base case holds.

For the induction step, assume that for all $\omega$ with $|\omega|<d'$, we have $f_\omega=f^{(\omega)}$ if $\omega\in\Omega_\star$
and that $f_{\omega} \equiv 0$, if $\omega \in\{0,1\}^d \setminus \Omega_\star$. For any $\omega'$ with $|\omega'|=d'$, we consider two cases.

Case 1: $\omega'\notin \Omega_\star$.
In this case, for all $\omega\in\Omega_\star$ with $\omega\nprec\omega'$, one can take $j^{\star,\omega}=\min\{j\in[d]:~\omega_j=1-\omega'_j=1\}$, and deduce similarly to \eqref{eq:LB_f_omega_is_zero} that 
\[
\mathbb{E}_P\Bigl\{f^{(\omega)}(X) \Big| X^{\omega'} = x^{\omega'}\Bigr\} = \mathbb{E}_P\Bigl\{f^{(\omega)}(X) \bigl( \mathbbm{1}\{x_{j^{\star,\omega}}>0\} + \mathbbm{1}\{x_{j^{\star,\omega}}<0\} \bigr) \Big| X^{\omega'} = x^{\omega'}\Bigr\} = 0.
\]
Using the induction hypothesis, we conclude that 
\begin{align*}
    f_{\omega'}(x)
    &= \mathbb{E}_P\Bigl\{ \eta(X) - \frac{1}{2} - \sum_{\omega''\prec\omega'}f_{\omega''}(X) \Big| X^{\omega'} = x^{\omega'} \Bigr\}  \\
    &= \mathbb{E}_P\Bigl\{ \sum_{\omega\in\Omega_\star} f^{(\omega)}(X) - \sum_{\omega''\prec\omega'}f_{\omega''}(X) \Big| X^{\omega'} = x^{\omega'} \Bigr\} \\
    &= \mathbb{E}_P\Bigl\{ \sum_{\omega\in\Omega_\star:~\omega\nprec\omega'} f^{(\omega)}(X) \Big| X^{\omega'} = x^{\omega'} \Bigr\} = 0.
\end{align*}

Case 2: $\omega'\in\Omega_\star\setminus\{\omega^\star\}$. This case is similar to the first one, with the difference in the conclusion:
\begin{align*}
    f_{\omega'}(x)
    &= \mathbb{E}_P\Bigl\{ \eta(X) - \frac{1}{2} - \sum_{\omega''\prec\omega'}f_{\omega''}(X) \Big| X^{\omega'} = x^{\omega'} \Bigr\}  \\
    &= \mathbb{E}_P\Bigl\{ \sum_{\omega\in\Omega_\star} f^{(\omega)}(X) - \sum_{\omega''\prec\omega'}f_{\omega''}(X) \Big| X^{\omega'} = x^{\omega'} \Bigr\} \\
    &= \mathbb{E}_P\Bigl\{ f^{(\omega')}(X) + \sum_{\omega\in\Omega_\star:~\omega\nprec\omega'} f^{(\omega)}(X) \Big| X^{\omega'} = x^{\omega'} \Bigr\} = f^{(\omega')}(x),
\end{align*}
proving the induction step. Furthermore, for $\omega\in\Omega_\star\setminus\{\omega^\star\}$ we have
\begin{align*}
\sigma_\omega^2 &= \int_{\mathbb{R}^d} f_\omega^2(x) \,d\mu(x)  =  \frac{1}{4|\Omega_\star|}\geq c_{\mathrm{E}},
\end{align*}
and for $\omega^\star$, by Lemma~\ref{lem:E_and_S_def_satisfied} or Lemma~\ref{lem:E_and_S_def_satisfied_heavy_tails}, 
\begin{align*}
\sigma_{\omega^\star}^2 &= \int_{\mathbb{R}^d} (f_\sigma^{({\omega^\star})}(x))^2 \,d\mu(x)  \geq  \frac{\min\{b/4,1/8\}}{|\Omega_\star|} \geq c_{\mathrm{E}},
\end{align*}
thus $Q\in\mathcal{Q}_{\mathrm{E}}(\Omega_\star,c_{\mathrm{E}},\mathcal{O})$. 

Now, for any $\omega\in\Omega_\star\setminus\{\omega^\star\}$, if $x_1,x_2\in\mathbb{R}^d$ are such that $\mathrm{sign}(x_1^\omega)=\mathrm{sign}(x_2^\omega)$, then by Equation~\ref{eq:holder_smooth_f_omega},
\[
|f^{(\omega)}(x_1) - f^{(\omega)}(x_2)| = |f^{(\omega)}(s\odot x_1) - f^{(\omega)}(s\odot x_2)| \leq \frac{1}{2} \|s\odot x_1-s\odot x_2\|^{\beta_{\omega}} \leq \|x_1-x_2\|^{\beta_{\omega}}. 
\]
On the other hand, if $s_1:=\mathrm{sign}(x_1^\omega) \neq \mathrm{sign}(x_2^\omega)=:s_2$, then there exists $z \in \mathcal{S}^\omega$ with $z=x_1+\zeta(x_2-x_1)$ for some $\zeta\in[0,1]$. Then, again by Equation~\ref{eq:holder_smooth_f_omega},
\begin{align*}
|f^{(\omega)}(x_1) &-f^{(\omega)}(x_2)| \leq |f^{(\omega)}(x_1)-f^{(\omega)}(z)| + |f^{(\omega)}(z)-f^{(\omega)}(x_2)| \\
&\leq |f^{(\omega)}(s_1\odot x_1)-f^{(\omega)}(s_1\odot z)| + |f^{(\omega)}(s_2\odot z) - f^{(\omega)}(s_2\odot x_2)| \\
&\leq \frac{1}{2} \|s_1\odot x_1-s_1\odot z\|^{\beta_{\omega}} + \frac{1}{2} \|s_2\odot z-s_2\odot x_2\|^{\beta_{\omega}} = \frac{1}{2} \|x_1-z\|^{\beta_{\omega}} + \frac{1}{2} \|z-x_2\|^{\beta_{\omega}} \\
&= \frac{1}{2} (\zeta^{\beta_{\omega}}+(1-\zeta)^{\beta_{\omega}}) \|x_1-x_2\|^{\beta_{\omega}} \leq \frac{1}{2} 2^{1-\beta_{\omega}} \|x_1-x_2\|^{\beta_{\omega}} \leq \|x_1-x_2\|^{\beta_{\omega}}.
\end{align*}
Using additionally that $f_\omega \equiv 0$ for $\omega\notin\Omega_\star$, as well as either Lemma~\ref{lem:E_and_S_def_satisfied} or \ref{lem:E_and_S_def_satisfied_heavy_tails}, we conclude that $P\in\mathcal{P}_{\mathrm{S}}(\boldsymbol{\beta},1)$.
\end{proof}

\begin{lemma}\label{lem:generalLB_M_def_satisfied}
Fix $d,n\in\mathbb{N}$, $\mathcal{O}\subseteq\{0,1\}^d$, $\boldsymbol{\gamma} = (\gamma_1,\dots,\gamma_d) \in [0, \infty)^d$, $\boldsymbol{\beta} = (\beta_1,\ldots,\beta_d)\in (0,1]^{d}$, $\alpha\in[0,\infty)$, $C_{\mathrm{M}}\geq1$, and $\Omega_\star\in\mathcal I(\{0,1\}^d)\setminus\{\{\mathbf{0}_d\},\emptyset\}$ with $|\Omega_\star|\geq 2$.  Fix further $o_1,\ldots,o_n\in\mathcal{O}$, and fix an $\omega^\star \in \argmax_{\omega\in\Omega_\star\cap\mathcal{N}} \bigl\{ n_\omega^{-\frac{\beta_{\omega}\gamma_{\omega}(1+\alpha)}{\gamma_{\omega}(2\beta_{\omega}+d_{\omega})+\alpha\beta_{\omega}}} \bigr\}$.  Let $\gamma_{\omega^\star} := \min\{\gamma_j : \omega^\star_j = 1\}$.  Fix further $r>0$, $\tilde{q},q \in \mathbb{N}$, $\kappa \in (0,1/\sqrt{d_{\omega^\star}})$, $\tilde{j} \in \{j\in [d] : \omega_{j} = 1\}$, $a,b\in(0,1/4]$, $\epsilon \in (0,1]$, $\sigma=(\sigma_t)_{t\in[T]}\in\{-1,1\}^T$,  $c_{\mathrm{E}} \leq 1$. 
Consider the set of assumptions defined by 
\begin{enumerate}
    \item[C1.] 
    \begin{enumerate}
        \item[(i)] $\gamma_{\omega^\star}\leq 1$
        \item[(ii)] $\epsilon \leq (1/4)\wedge(1/8\cdot(r/q)^{\beta_{\omega}})$
        \item[(iii)] $C_{\mathrm{M}} \geq \max\{1+4^{d_{\omega^\star}/\beta_{\omega}}(2\kappa)^{-d_{\omega^\star}}V_{d_{\omega^\star}}, 2^{\alpha}\}$
    \end{enumerate}
    
    \item[C2.]
    \begin{enumerate}
        \item[(i)] $\gamma_{\omega^\star}> 1$
        \item[(ii)] $\epsilon \leq (1/4\cdot(r/q)^{\beta_{\omega}}) \wedge (1/4\cdot(1/\tilde{q})^{\beta_{\omega}})$ 
        \item[(iii)] $C_{\mathrm{M}} \geq 2^\alpha$
    \end{enumerate}
\end{enumerate}
Let $P\equiv P^{(\Omega_\star)}_{\kappa,r,\tilde{q},q,a,b,\tilde{j},\epsilon,\sigma}$ denote the distribution on $\mathbb{R}^d\times\{0,1\}$ with regression function $\eta^{(\Omega_\star)}_{\epsilon,\tilde{q},q,r,\sigma}$ and marginal feature distribution $\mu^{(\Omega_\star)}_{\kappa,r,\tilde{q},q,a,b,\tilde{j}}$. Suppose that either C1 or C2 holds, that $\alpha\in[0,d_{\omega^\star}/\beta_{\omega^\star}]$ and that $a\in[0,\epsilon^\alpha]$, then $P \in \mathcal P_{\mathrm{M}}(\alpha,C_{\mathrm{M}})$. 
\end{lemma}
\begin{proof}
First, if $t\in (0,\epsilon)$, then $\mu\bigl(\bigl\{x \in \IR^d : |\eta(x)- 1/2| < t \bigr\} \bigr) = 0 \leq C_{\mathrm{M}} \cdot t^{\alpha}$.  For $t\in [\epsilon,1/2)$, by Lemma~\ref{lem:M_def_satisfied} or Lemma~\ref{lem:M_def_satisfied_heavy_tails}, we have
\begin{align*}
\mu\Bigl(\Bigl\{x \in \mathbb{R}^d : \bigl|\eta(x)-1/2\bigr| < t \Bigr\} \Bigr) &=  \frac{1}{|\Omega_\star|} \mu^{(\omega^\star)}\Bigl(\Bigl\{x \in \mathbb{R}^d : \bigl|\eta(x)-1/2\bigr| < t \Bigr\} \Bigr) \\
&\leq \frac{C_{\mathrm{M}}}{|\Omega_\star|}  \cdot t^{\alpha} \leq C_{\mathrm{M}} \cdot t^{\alpha}.
\end{align*}
Finally, if $t\geq1/2$, then
\begin{align*}
\mu\Bigl(\Bigl\{x \in \mathbb{R}^d : \bigl|\eta(x)-1/2\bigr| < t \Bigr\} \Bigr) \leq 1 \leq (2t)^{\alpha}  \leq C_{\mathrm{M}} \cdot t^{\alpha},
\end{align*}
as required. 
\end{proof}

\begin{proof}[Proof of the lower bound in Theorem~\ref{thm:minmax_bounds}
]
Firstly, if $\Omega_\star = \{\omega^\star\}$ for one $\omega^\star \in\{0,1\}^d\setminus\{\mathbf{0}_d\}$, we can directly apply Lemma~\ref{lem:LB_light_tails} if $\gamma_{\omega^\star}\geq 1$, or Lemma~\ref{lem:LB_heavy_tails} if $\gamma_{\omega^\star}< 1$. Otherwise $|\Omega_\star|\geq2$, and we will prove the lower bound using the constructions in~\eqref{def:mu_mixture_LB} and~\eqref{def:eta_mixture_LB}.

To this end, let \[\omega^\star=\argmax_{\omega\in\Omega_\star\cap\mathcal{N}} \bigl\{ n_\omega^{-\frac{\beta_{\omega}\gamma_\omega(1+\alpha)}{\gamma_\omega(2\beta_{\omega}+d_\omega)+\alpha\beta_{\omega}}} \bigr\},\]
and we consider the two different cases $\gamma_{\omega^\star}\geq 1$ and $\gamma_{\omega^\star}< 1$.

\textbf{Case 1:} $\gamma_{\omega^\star}\geq 1$.
Let
$a_0 := 2^{-3d_{\omega^\star}\gamma_{\omega^\star}}$, $a_1 := 2^{3\gamma_{\omega^\star}d_{\omega^\star}/\beta_{\omega^\star}+2\gamma_{\omega^\star}d_{\omega^\star}}
 \cdot d_{\omega^\star}^{\gamma_{\omega^\star}d_{\omega^\star}/2}$, $b_0 = 2^{-1}(d_{\omega^\star} 2^{d_{\omega^\star}})^{-\gamma_{\max}}$,
and
\[
q_0:=\min\biggl\{ a_0^{\frac{2+\alpha}{\alpha \wedge 1}}2^{5+d_{\omega^\star}} , \Bigr(\frac{a_0}{a_1} \cdot 2^{(d_{\omega^\star} + 5)\frac{\alpha\beta_{\omega^\star}(1-\gamma_{\omega^\star})+\gamma_{\omega^\star}d_{\omega^\star}}{\beta_{\omega^\star}(2+\alpha)} } \Bigl)^{\frac{(2+\alpha)\beta_{\omega^\star}}{\alpha\beta_{\omega^\star}+\gamma_{\omega^\star}d_{\omega^\star}+2\gamma_{\omega^\star}\beta_{\omega^\star}}}, 2^{\frac{d_{\omega^\star}(d_{\omega^\star}-3\alpha-1)}{d_{\omega^\star}+(2+\alpha)\beta_{\omega^\star}}}\biggr\}.
\]
Further let
\begin{align*}
\rho &:= \frac{\gamma_{\omega^\star}(d_{\omega^\star}-\alpha\beta_{\omega^\star})+\alpha\beta_{\omega^\star}}{\gamma_{\omega^\star}(2\beta_{\omega^\star}+d_{\omega^\star})+\alpha\beta_{\omega^\star}}\in[0,1); \quad q :=\lfloor (q_{0} n_{\omega^\star}^{\rho})^{1/d_{\omega^\star}}\rfloor, 
\\ m &:= q^{d_{\omega^\star}}, \quad 
\epsilon:= \min\Bigl\{\Bigl(\frac{m}{2^{d_{\omega^\star}+5} n_{\omega^\star}}\Bigr)^{\frac{1}{2+\alpha}},1/4\Bigr\} , \quad
u := \frac{\epsilon^\alpha}{m},
\\  \kappa &:=\frac{1}{2\sqrt{d_{\omega^\star}}}, \quad
a:=\epsilon^\alpha, \quad 
b:=\frac{1}{4}, \quad
r:= (8\epsilon)^{1/\beta_{\omega^\star}}\cdot q,
\\ \tilde{q}&:=\tilde{j}:=1.
\end{align*}
Suppose initially that $n_{\omega^\star} >  q_{0}^{-1/\rho}$, so that $q \geq 1$. For $\sigma \in\{-1,1\}^m$, let $Q^\sigma\equiv Q^{(\Omega_{\star})}_{\kappa,r,\tilde{q},q,a,b,\tilde{j},\sigma}$ be the distribution of $(X,Y,O)$ on $\mathbb{R}^d\times\{0,1\}\times\{0,1\}^d$ with $X$-marginal distribution $\mu = \mu^{(\Omega_\star)}_{\kappa,r,,\tilde{q},q,a,b,\tilde{j}}$, with regression function $\eta=\eta^{(\Omega_\star)}_{\epsilon,\tilde{q},q,r,\sigma}$, with $O\indep(X,Y)$ and with $O\sim U(\mathcal{O})$. We now show that $Q^\sigma\in\mathcal{Q}'_{\mathrm{Miss}}$.

First, we verify that the set of Assumptions A1 in Lemma~\ref{lem:generalLB_L_def_satisfied} is satisfied. Using that $\rho-1 = \frac{-\gamma_{\omega^\star} \beta_{\omega^\star} (2+\alpha)}{\gamma_{\omega^\star}(2\beta_{\omega^\star} + d_{\omega^\star}) + \alpha\beta_{\omega^\star}}$, we have  
\begin{align*} 
a & = \epsilon^{\alpha} \leq \Bigl(\frac{m}{2^{d_{\omega^\star} + 5} n_{\omega^\star}} \Bigr)^{\frac{\alpha}{2+\alpha}} \leq \Bigl(\frac{q_0 n_{\omega^\star}^{\rho}} {2^{d_{\omega^\star} + 5} n_{\omega^\star}} \Bigr)^{\frac{\alpha}{2+\alpha}} \\
&= \Bigl(\frac{q_0}{2^{5 + d_{\omega^\star}}}\Bigr)^{\alpha/(2+\alpha)} \cdot n_{\omega^\star}^{-\frac{\gamma_{\omega^\star}\alpha\beta_{\omega^\star}}{\gamma_{\omega^\star}(2\beta_{\omega^\star}+d_{\omega^\star}) + \alpha \beta_{\omega^\star}}} \leq  a_0^{\alpha \wedge 1} \leq a_0,
\end{align*}
by the first term in the minimum in the definition of $q_0$. This further implies that $\bigl(\frac{m}{2^{d_{\omega^\star} + 5} n_{\omega^\star}} \bigr)^{\frac{1}{2+\alpha}}  < \frac{1}{4}$, so $\epsilon < 1/4$. Moreover, since $\alpha\beta_{\omega^\star} \leq d_{\omega^\star}$, 
\begin{align*}
a^{1-\gamma_{\omega^\star}} (4r\sqrt{d_{\omega^\star}})^{\gamma_{\omega^\star}d_{\omega^\star}}
&= \epsilon^{\alpha(1-\gamma_{\omega^\star})+\gamma_{\omega^\star}d_{\omega^\star}/\beta_{\omega^\star}} \cdot q^{\gamma_{\omega^\star}d_{\omega^\star}} \cdot 2^{3d_{\omega^\star}\gamma_{\omega^\star}/\beta_{\omega^\star}+2\gamma_{\omega^\star}d_{\omega^\star}} \cdot d_{\omega^\star}^{\gamma_{\omega^\star}d_{\omega^\star}/2}\\
&\leq \Bigl(\frac{q_0 n_{\omega^\star}^{\rho-1}} {2^{d_{\omega^\star} + 5}} \Bigr)^{\frac{\alpha(1-\gamma_{\omega^\star})+\gamma_{\omega^\star}d_{\omega^\star}/\beta_{\omega^\star}}{2+\alpha}} \cdot \bigl(q_0 n_{\omega^\star}^{\rho}\bigr)^{\gamma_{\omega^\star}} \cdot a_1\\
&= \Bigl(\frac{q_0} {2^{d_{\omega^\star} + 5}} \Bigr)^{\frac{\alpha(1-\gamma_{\omega^\star})+\gamma_{\omega^\star}d_{\omega^\star}/\beta_{\omega^\star}}{2+\alpha}} \cdot q_0^{\gamma_{\omega^\star}} \cdot a_1 \leq a_0,
\end{align*} 
where we have used the second term in the minimum in the definition of $q_0$, thus A1 (iii) holds. To show that A1 (iv) and (v) hold, first note that
\begin{align*}
    r &\leq 8^{1/\beta_{\omega^\star}} \Bigl(\frac{m}{2^{d_{\omega^\star} + 5} n_{\omega^\star}} \Bigr)^{\frac{1}{(2+\alpha)\beta_{\omega^\star}}} \cdot  (q_{0} n_{\omega^\star}^{\rho})^{1/d_{\omega^\star}} 
    \\ & = 8^{1/\beta_{\omega^\star}} \cdot \Bigl(\frac{q_0}{2^{5 + d_{\omega^\star}}}\Bigr)^{\frac{1}{(2+\alpha)\beta_{\omega^\star}}} \cdot q_0^{1/d_{\omega^\star}} \cdot n_{\omega^\star}^{\frac{(1-\gamma_{\omega^\star})\alpha\beta_{\omega^\star}}{(\gamma_{\omega^\star}(2\beta_{\omega^\star}+d_{\omega^\star}) + \alpha \beta_{\omega^\star})d_{\omega^\star}}}  \leq 1,
\end{align*}
where we have used the fact that $\gamma_{\omega^\star} \geq 1$ and the third term in the definition of $q_0$. Thus, the minima in A1 (iv) and (v) are both attained by $1$, and we have
\begin{align*}
b^{1-\gamma_{\omega^\star}} \leq b^{1-\gamma_{\max}} = 4^{\gamma_{\max}-1}\leq C_{\mathrm{L}} \cdot b_0,
\end{align*}
since $C_{\mathrm{L}} \geq 4^{\gamma_{\max}(d_{\omega^\star}+1)} \geq 4^{\gamma_{\max}-1}b_0^{-1}$. We deduce that A1 (iv) and (v) hold, and therefore that $Q^\sigma \in \mathcal{Q}_{\mathrm{L}}(\boldsymbol{\gamma}, C_{\mathrm{L}},\mathcal{O})$ by Corollary~\ref{lem:generalLB_L_def_satisfied}. Further, since $b = 1/4$, we have $c_{\mathrm{E}} \leq 1/16 = b/4$, thus by Lemma~\ref{lem:generalLB_E_and_S_def_satisfied}, the upper bound on $C_{\mathrm{M}}$ and Lemma~\ref{lem:generalLB_M_def_satisfied}, we deduce that $Q^{\sigma} \in \mathcal{Q}_{\mathrm{E}}(\Omega_\star, c_{\mathrm{E}},\mathcal{O})$ and $P_{Q^\sigma}\in \mathcal{P}_{\mathrm{S}}(\boldsymbol{\beta}, 1) \cap \mathcal{P}_{\mathrm{M}}(\alpha, C_{\mathrm{M}})$ for the distribution $P_{Q^\sigma}$ of $(X,Y)$, since the set of Assumptions B1 in Lemma~\ref{lem:generalLB_E_and_S_def_satisfied} is satisfied as well as the Assumptions C1 in Lemma~\ref{lem:generalLB_M_def_satisfied}.  

\textbf{Case 2:} $\gamma_{\omega^\star}< 1$.
First, let $a_1 := 2^{3\gamma_{\omega^\star}/\beta_{\omega^\star}}$ and
\[
q_0:=\min\biggl\{ 
2^{5+d_{\omega^\star}} , \Bigr(2^{(d_{\omega^\star} + 5)\frac{\alpha(1-\gamma_{\omega^\star})+\gamma_{\omega^\star}/\beta_{\omega^\star}}{2+\alpha} } \cdot a_1^{-1} \Bigl)^{\frac{2+\alpha}{\alpha(1-\gamma_{\omega^\star})+\gamma_{\omega^\star}/\beta_{\omega^\star}+\gamma_{\omega^\star}(2+\alpha)}}, 2^{(2+d_{\omega^\star})(2+\alpha)}\biggr\}.
\]
Further let
\begin{align*}
\rho &:= \frac{\gamma_{\omega^\star}(d_{\omega^\star}-\alpha\beta_{\omega^\star})+\alpha\beta_{\omega^\star}}{\gamma_{\omega^\star}(2\beta_{\omega^\star}+d_{\omega^\star})+\alpha\beta_{\omega^\star}}\in[0,1); \quad \rho_1 := \frac{\gamma_{\omega^\star}(1-\alpha\beta_{\omega^\star})+\alpha\beta_{\omega^\star}}{\gamma_{\omega^\star}(2\beta_{\omega^\star}+d_{\omega^\star})+\alpha\beta_{\omega^\star}} \in [0, \rho]; 
\\ q &:= \lfloor q_0 n_{\omega^\star}^{\frac{\gamma_{\omega^\star} + \alpha \beta_{\omega^\star} (1-\gamma_{\omega^\star})}{\gamma_{\omega^\star} (2\beta_{\omega^\star} +d_{\omega^\star}) + \alpha\beta_{\omega^\star}}} \rfloor, \quad \tilde{q} := \lfloor n_{\omega^\star}^{\frac{\gamma_{\omega^\star}}{\gamma_{\omega^\star} (2\beta_{\omega^\star} +d_{\omega^\star}) + \alpha\beta_{\omega^\star}}} \rfloor;  \quad m := q\cdot \tilde{q}^{d_{\omega^\star}-1}
\\ \epsilon &:= \min\Bigl\{\Bigl(\frac{m}{2^{d_{\omega^\star}+5} n_{\omega^\star}}\Bigr)^{\frac{1}{2+\alpha}},1/4\Bigr\}, \quad r := (8\epsilon)^{1/\beta_{\omega^\star}} \cdot q, \quad u := \frac{\epsilon^\alpha}{m},\quad
a:=\epsilon^\alpha\\
\kappa&=1/(2\sqrt{d}),\quad 
b:=1/2.
\end{align*}
Fix $\tilde{j} \in \argmin_{j \in \{j' \in [d]:\omega^\star_{j'}=1\}} \gamma_j$, so that $\gamma_{\tilde{j}} = \gamma_{\omega^\star}$. 

Suppose initially that $n_{\omega^\star}  >  q_{0}^{-1/\rho_1}$, so that $q \geq 1$. Let $Q^\sigma\equiv Q^{(\Omega_{\star})}_{\kappa,r,\tilde{q},q,a,b,\tilde{j},\sigma}$ be the distribution of $(X,Y,O)$ on $\mathbb{R}^d\times\{0,1\}\times\{0,1\}^d$ with $X$-marginal distribution $\mu = \mu^{(\Omega_\star)}_{\kappa,r,\tilde{q},q,a,b,\tilde{j}}$, with regression function $\eta=\eta^{(\Omega_\star)}_{\epsilon,\tilde{q},q,r,\sigma}$, with $O\indep(X,Y)$ and with $O\sim U(\mathcal{O})$, for $\sigma \in\{-1,1\}^m$. We now show that $Q^\sigma \in \mathcal{Q}_{\mathrm{Miss}}'$, by applying in turn the results from the previous subsections. 
First, we verify that the set of Assumptions A2 in Lemma~\ref{lem:generalLB_L_def_satisfied} is satisfied. Using that $\rho-1 = \frac{-\gamma_{\omega^\star} \beta_{\omega^\star} (2+\alpha)}{\gamma_{\omega^\star}(2\beta_{\omega^\star} + d_{\omega^\star}) + \alpha\beta_{\omega^\star}}$ and that 
\[
m\leq q_0 n_{\omega^\star}^{\frac{\gamma_{\omega^\star} + \alpha \beta_{\omega^\star} (1-\gamma_{\omega^\star})}{\gamma_{\omega^\star} (2\beta_{\omega^\star} +d_{\omega^\star}) + \alpha\beta_{\omega^\star}}} \cdot n_{\omega^\star}^{\frac{(d_{\omega^\star}-1)\gamma_{\omega^\star}}{\gamma_{\omega^\star} (2\beta_{\omega^\star} +d_{\omega^\star}) + \alpha\beta_{\omega^\star}}} = q_0 n_{\omega^\star}^{\rho-1},
\]
it follows that 
\begin{align*} 
a & = \epsilon^{\alpha} \leq \Bigl(\frac{q_0 n_{\omega^\star}^{\rho}} {2^{d_{\omega^\star} + 5} n_{\omega^\star}} \Bigr)^{\frac{\alpha}{2+\alpha}} = \Bigl(\frac{q_0}{2^{5 + d_{\omega^\star}}}\Bigr)^{\alpha/(2+\alpha)} \cdot n_{\omega^\star}^{-\frac{\gamma_{\omega^\star}\alpha\beta_{\omega^\star}}{\gamma_{\omega^\star}(2\beta_{\omega^\star}+d_{\omega^\star}) + \alpha \beta_{\omega^\star}}} \leq  1,
\end{align*}
by the first term in the minimum in the definition of $q_0$. This further implies that $\bigl(\frac{m}{2^{d_{\omega^\star} + 5} n_{\omega^\star}} \bigr)^{\frac{1}{2+\alpha}}  < \frac{1}{4\tilde{q}}$, so $\epsilon < 1/(4\tilde{q}) \leq 1/4$. Moreover, if 
\[
n_{\omega^\star} > N_0 := \max\Bigl\{q_0^{-1/\rho_1},  \Bigl(\frac{8}{2^{\beta_{\omega^\star}}} \Bigl(\frac{1}{2^{2d_{\omega^\star}+5} }\Bigr)^{\frac{1}{2+\alpha}} q_0^{\frac{1 + 2\beta_{\omega^\star} + \alpha\beta_{\omega^\star}}{2+\alpha}}\Bigr)^{-\frac{\gamma_{\omega^\star}(2\beta_{\omega^\star}+d_{\omega^\star}) + \alpha \beta_{\omega^\star}}{\alpha\beta_{\omega^\star}^2(1-\gamma_{\omega^\star})}}\Bigr\},
\] 
then, using that $\gamma_{\omega^\star} < 1$, we have
\begin{align*}
r^{\beta_{\omega^\star}} &= 8 \epsilon q^{\beta_{\omega^\star}} \geq 4\Bigl(\frac{q_0 n_{\omega^\star}^{\rho}}{2^{2d_{\omega^\star}+5} n_{\omega^\star}}\Bigr)^{\frac{1}{2+\alpha}} q_0^{\beta_{\omega^\star}}  n_{\omega^\star}^{\beta_{\omega^\star}\rho_1} \\
&= \frac{8}{2^{\beta_{\omega^\star}}} \Bigl(\frac{1}{2^{2d_{\omega^\star}+5} }\Bigr)^{\frac{1}{2+\alpha}} q_0^{\frac{1 + 2\beta_{\omega^\star} + \alpha\beta_{\omega^\star}}{2+\alpha}} \cdot n_{\omega^\star}^{\frac{\alpha\beta_{\omega^\star}^2(1-\gamma_{\omega^\star})}{\gamma_{\omega^\star}(2\beta_{\omega^\star}+d_{\omega^\star}) + \alpha \beta_{\omega^\star}}} > 1,
\end{align*}
so $r >1$. Suppose now that $n_{\omega^\star} > N_0$. Then,
since $\alpha\beta_{\omega^\star} \leq d_{\omega^\star}$,
\begin{align*}
a^{1-\gamma_{\omega^\star}} r^{\gamma_{\omega^\star}} & = 2^{3\gamma_{\omega^\star}/\beta_{\omega^\star}} \epsilon^{\alpha(1-\gamma_{\omega^\star}) + \gamma_{\omega^\star}/\beta_{\omega^\star}} \cdot q^{\gamma_{\omega^\star}} \\
&\leq 2^{3\gamma_{\omega^\star}/\beta_{\omega^\star}} \cdot \Bigl(\frac{q_0 n_{\omega^\star}^{\rho-1}} {2^{d_{\omega^\star} + 5}} \Bigr)^{\frac{\alpha(1-\gamma_{\omega^\star})+\gamma_{\omega^\star}/\beta_{\omega^\star}}{2+\alpha}} \cdot q_0^{\gamma_{\omega^\star}} n_{\omega^\star}^{\frac{\gamma_{\omega^\star}\{\gamma_{\omega^\star} + \alpha \beta_{\omega^\star}(1-\gamma_{\omega^\star})\}}{\gamma_{\omega^\star}(2\beta_{\omega^\star} + d_{\omega^\star}) + \alpha\beta_{\omega^\star}} } \\
&= a_1 \cdot \Bigl(\frac{q_0} {2^{d_{\omega^\star} + 5}} \Bigr)^{\frac{\alpha(1-\gamma_{\omega^\star})+\gamma_{\omega^\star}/\beta_{\omega^\star}}{2+\alpha}} \cdot q_0^{\gamma_{\omega^\star}} \leq 1,
\end{align*} 
where we have used the second term in the minimum in the definition of $q_0$. Thus the Assumptions A2 in Lemma~\ref{lem:generalLB_L_def_satisfied} are satisfied such that we have $Q^\sigma \in \mathcal{Q}_{\mathrm{L}}(\boldsymbol{\gamma}, C_{\mathrm{L}},\mathcal{O})$. Further we have
\begin{align*} 
8\epsilon\cdot \tilde{q}^{\beta_{\omega^\star}} \leq 8 \Bigl(\frac{q_0 n_{\omega^\star}^{\rho}} {2^{d_{\omega^\star} + 5} n_{\omega^\star}} \Bigr)^{\frac{1}{2+\alpha}} \tilde{q}^{\beta_{\omega^\star}} = 8 \Bigl(\frac{q_0}{2^{5 + d_{\omega^\star}}}\Bigr)^{\frac{1}{2+\alpha}} 
\leq 1
\end{align*}
by the third term in the definition of $q_0$.
Finally, by the upper bound on $c_{\mathrm{E}}$, Lemma~\ref{lem:generalLB_E_and_S_def_satisfied}, the upper bound on $C_{\mathrm{M}}$ and Lemma~\ref{lem:generalLB_M_def_satisfied}, we deduce that $Q^{\sigma} \in \mathcal{Q}_{\mathrm{E}}(\Omega_\star, c_{\mathrm{E}},\mathcal{O})$ and that the $(X,Y)$ distribution $P_{Q^\sigma}$ falls into the class $\mathcal{P}_{\mathrm{S}}(\boldsymbol{\beta}, 1) \cap \mathcal{P}_{\mathrm{M}}(\alpha, C_{\mathrm{M}})$, since the set of Assumptions B2 in Lemma~\ref{lem:generalLB_E_and_S_def_satisfied} is satisfied as well as the Assumptions C2 in Lemma~\ref{lem:generalLB_M_def_satisfied}.

The penultimate step of the proof is to verify that the Assumptions~(i)-(vi) of Lemma~\ref{lem:assouad} are satisfied, where we set $\omega\equiv\omega^\star$: 
\begin{enumerate}
\item[(i)] We have, for both $\gamma_{\omega^\star}\geq1$ and $\gamma_{\omega^\star}<1$ that $2^{d_{\omega^\star}+5}n_{\omega^\star}\epsilon^2u=1$ by the definition of $\epsilon$ and $u$.
\item[(ii)] Letting $z_t:=z_t^{\omega^\star}$ for $t\in[m]$, we have that $\mu(\{s^{\omega^\star} \odot z_t^{\omega^\star}\})=u$ for all $t\in[m]$ and $s\in\{-1,1\}^d$ by the construction of the marginal measure $\mu = \mu^{(\Omega^\star)}_{\kappa,r,\tilde{q},q,a,b,\tilde{j}}$,
\item[(iii)]  
By the construction of the regression function $\eta^\sigma = \eta^{(\Omega^\star)}_{\epsilon,\tilde{q},q,r,\sigma}$, we have $\eta^\sigma(z^{\omega^\star}_t) = 1/2 + \sigma_t \epsilon$ for $t\in[m]$, and $\eta^\sigma(s^{\omega^\star} \odot z_t^{\omega^\star}) = 1/2 + \bigl(\prod_{j \in [d] :(s^{\omega^\star})_{j} =-1} (s^{\omega^\star})_{j} \bigr)\cdot\sigma_t\epsilon$, for all $t\in[m]$ and $s\in\{-1,1\}^d$.
\item[(iv)] For $\sigma,\sigma'\in\Sigma$, the support of $\mu^{(\Omega^\star)}$ is given by $\mathrm{supp}(\mu^{(\Omega^\star)})=\bigcup_{s\in\{-1,1\}^d}s\odot(\mathcal{T}^\omega \cup \bigcup_{j\in[d]:\omega^\star_j=1} \mathcal{R}^\omega_j \cup\mathcal{R}^{\pm}$ if $\gamma_{\omega^\star}\geq1$, and by $\mathrm{supp}(\mu^{(\Omega^\star)})=\bigcup_{s\in\{-1,1\}^d}s\odot(\mathcal{T}^\omega \cup \mathcal{R}_{\tilde{j}}^\omega) \cup\mathcal{R}^{\pm}$ if $\gamma_{\omega^\star}<1$. We have $\eta^\sigma=\eta^{\sigma'}$ on $\mathcal{R}^\omega\cup \bigcup_{j\in[d]:\omega^\star_j=1} \mathcal{R}^\omega_j \cup\mathcal{R}^{\pm}$ (if $\gamma_{\omega^\star}\geq1$), or on $\mathcal{R}_{\tilde{j}}^\omega\cup \mathcal{R}^{\pm}$ (if $\gamma_{\omega^\star}<1$); either way we indeed have $\eta^\sigma(x)=\eta^{\sigma'}(x)$ for all $x\in \mathrm{supp}(\mu) \setminus \bigcup_{t\in [m], s\in \{-1,1\}^d} \{s^{\omega^\star}\odot z_t^{\omega^\star}\}$,
\item[(v)] 
Again by construction, we have $\Bigl(\mathrm{supp}(\mu) \setminus \bigcup_{t\in [m], s\in \{-1,1\}^d} \{s^{\omega^\star}\odot z_t^{\omega^\star}\} \Bigr) \bigcap (-r-1,r+1)^d = \emptyset$,
\item[(vi)] Once more by construction, we have $O\indep (X,Y)$.
\end{enumerate} 
Then, for $n_{\omega^\star} > q_0$ if $\gamma_{\omega^\star}\geq1$ ($n_{\omega^\star} > N_0$ if $\gamma_{\omega^\star}<1$), since $\epsilon = \bigl(\frac{m}{2^{d_{\omega^\star}+5}n_{\omega^\star}}\bigr)^{\frac{1}{2+\alpha}}$ the lower bound in~\eqref{eq:lb_assouad} in Lemma~\ref{lem:assouad} gives
\begin{align*}
     \inf_{\hat C\in\mathcal C_n} & \sup_{Q}\ \mathbb{E}_{Q}\{\mathcal {E}_{P}(\hat C) | O_1 = o_1, \ldots, O_n = o_n\} \geq  \frac{m u \epsilon}{2} = \frac{\epsilon^{1+\alpha}}{2} 
    = \frac{1}{2} \left(\frac{m}{2^{d_{\omega^\star}+5}n_{\omega^\star}}\right)^{\frac{1+\alpha}{2+\alpha}} \\
    &\geq \frac{1}{2} \Bigl( \frac{q_0 n_{\omega^\star}^\rho} {2^{d_{\omega^\star}+6} n_{\omega^\star}} \Bigr)^{\frac{1+\alpha}{2+\alpha}} = 2^{-1-\frac{(d_{\omega^\star}+6)(1+\alpha)}{2+\alpha}} q_0^{\frac{1+\alpha}{2+\alpha}} n_{\omega^\star}^{\frac{(1+\alpha)(\rho-1)}{(2+\alpha)}}
    =: \tilde{c} \cdot  n_{\omega^\star}^{-\frac{\beta_{\omega^\star}\gamma_{\omega^\star}(1+\alpha)}{\gamma_{\omega^\star}(2\beta_{\omega^\star}+d_{\omega^\star})+\alpha\beta_{\omega^\star}}}.
\end{align*}
Finally, if $n_{\omega^\star} \leq q_0$ ($n_{\omega^\star} \leq N_0$), including possibly $n_{\omega^\star} = 0$, then using the fact that the excess risk is decreasing in $n_{\omega^\star}$, we conclude that the result follows with $c := \tilde{c} \cdot q_0^{\frac{\beta_{\omega^\star}\gamma_{\omega^\star}(1+\alpha)}{\gamma_{\omega^\star}(2\beta_{\omega^\star}+d_{\omega^\star})+\alpha\beta_{\omega^\star}}}$ ($c := \tilde{c} \cdot N_0^{\frac{\beta_{\omega^\star}\gamma_{\omega^\star}(1+\alpha)}{\gamma_{\omega^\star}(2\beta_{\omega^\star}+d_{\omega^\star})+\alpha\beta_{\omega^\star}}}$). 
\end{proof}

\section{Details of the upper bound proofs in Section~\ref{sec:proofs}
\label{sec:UBproofs}}

\subsection{Additional results for the proof of the upper bound in Theorem~\ref{thm:minmax_bounds} 
and proofs of the results in Section~\ref{subsec:proofUB} 
\label{sec:appendixUBproofs}}

Here we provide the remaining details for the proof of the upper bound in Theorem~\ref{thm:minmax_bounds}
.  First, recall that we are treating the missingness indicators $o_1, \ldots, o_n \in \mathcal{O}\subseteq\{0,1\}^d$ as fixed and all probability statements in this section should be interpreted as being conditional on $O_1 = o_1, \ldots, O_n = o_n$. Recall also the definitions of $\mathcal{X}$ and the events $E_1^{\delta}(x)$ and $E_2^{\delta}(x)$ from Section~\ref{subsec:proofUB}
. Our first two lemmas bound the probability of the these events. 

\begin{lemma}\label{lem:E_1}
Fix $\delta \in (0,1)$. Then, for every $x \in \mathcal{X}$, we have $\IP(E_1^\delta(x)^c)\leq\delta$.
\end{lemma}

\begin{proof}[Proof of Lemma~\ref{lem:E_1}]
Fix $x \in \mathcal{X}$ and $\omega \in \mathcal{N}$. If $n_\omega \rho_\omega(x^\omega)/2 \leq \lceil4\log_+(|\mathcal N|/\delta)\rceil$, then $\tilde{k}_{\omega} = 0$ and $\mathbb{P}(A^{\delta}_{\omega}) = 1$.  Now suppose that $\lceil4\log_+(|\mathcal N|/\delta)\rceil < n_\omega \rho_\omega(x^\omega)/2$. 
Define
\begin{align*}
    r:=\biggl(\frac{2\tilde{k}_\omega}{n_\omega\rho_\omega(x^\omega)}\biggr)^{1/d_\omega} < 1.
\end{align*}
Therefore, by the definitions of the lower density and $\rho_\omega$, we have for $o \succeq \omega$ that 
\[
\mu_{\omega|o}(B_r(x^\omega))\geq \rho_{\mu_{\omega|o},d_{\omega}}(x^\omega)r^{d_\omega} \geq  \rho_\omega(x^\omega)r^{d_\omega}=\frac{2\tilde{k}_\omega}{n_\omega}.
\]
Then, using a multiplicative Chernoff bound, we obtain
\begin{align*}
  \IP\left(\lVert X_{(\tilde{k}_\omega)_\omega}^{\omega}(x^\omega)-x^\omega\rVert> r \right) &\leq\IP\biggl(\sum_{i\in 
N_\omega}\mathbbm{1}_{\{X_i^\omega\in B_r(x^{\omega})\}}< \tilde{k}_\omega\biggr)\\
&\leq\IP\biggl(\sum_{i\in N_\omega}\mathbbm{1}_{\{X_i^\omega\in B_r(x^\omega)\}}< \frac{\sum_{i \in N_{\omega}} \mu_{\omega|o_i}(B_r(x^\omega))}{2}\biggr)\\
&\leq e^{- \frac{1}{8}\sum_{i\in N_\omega} \mu_{\omega|o_i}(B_r(x^\omega))} \leq e^{-\tilde{k}_\omega/4} \leq \frac{\delta}{|\mathcal{N}| }.
\end{align*}
Taking a union bound over $\omega\in\mathcal N$ concludes the proof.
\end{proof}

\begin{lemma}\label{lem:E_2_prob}
Fix $\delta \in (0,1)$ and $x\in\IR^d$, then we have $\IP \bigl(E_{2}^\delta(x)^c \bigr)\leq\delta$.
\end{lemma}

\begin{proof}[Proof of Lemma~\ref{lem:E_2_prob}]
Fix $\omega \in \mathcal{N}$, and recall that $N_\omega = \{i \in [n] : \omega \preceq o_i\}$. First note that the labels $Y_{(1)_\omega}(x), \ldots, Y_{(n_\omega)_\omega}(x)$ are conditionally independent given $(X_\ell^{\omega})_{\ell \in N_{\omega}}$. Further, for $i \in [n_{\omega}]$, we have that 
\[
\mathbb{E}\bigl\{Y_{(i)_\omega}(x) \bigm| (X_\ell^{\omega})_{\ell \in N_{\omega}} \bigr\} =\mathbb{E}\bigl\{Y_{(i)_\omega}(x) \bigm| X_{(i)_{\omega}}^\omega(x) \bigr\} = \eta_{\omega}(X_{(i)_{\omega}}(x)).
\]
Then by Hoeffding's inequality, we have that
\begin{align*}
    \IP\Biggl(\Bigl|\frac{1}{k_\omega}\sum_{i=1}^{k_\omega} \bigl\{ Y_{(i)_\omega}(x)-\eta_\omega(X_{(i)_\omega}(x)) \bigr\}\Bigr| > \sqrt{\frac{\log_+(2 |\mathcal{N}|/\delta)}{2k_\omega}} \Biggm|(X_\ell^{\omega})_{\ell \in N_{\omega}}\Biggr) \leq \frac{\delta}{|\mathcal{N}|}.
\end{align*}
Taking expectation over $(X_\ell^{\omega})_{\ell \in N_{\omega}}$ gives $\IP\bigl((B_\omega^\delta(x))^c\bigr)\leq\delta/|\mathcal{N}|$ and the result follows via a union bound.
\end{proof}

Now, working on the events $E_1^{\delta}(x)$ and $E_2^{\delta}(x)$, we can bound the approximation error of $\hat{f}_{\omega}(x)$ which will put us in a position to prove the results in Lemmas~\ref{lem:f_is_not_zero} 
and~\ref{lem:hatC=BayesC}
. These results rely on the following intermediate lemma, which uses the constant $C_{\mathrm{B}}$ given in Proposition~\ref{prop:f_bounded}.  

\begin{lemma}\label{lem:f_Xx_bound}
Fix $\boldsymbol{\beta}=(\beta_1,\ldots,\beta_d) \in (0,1]^d$, $C_{\mathrm{S}} \geq 1$, $\omega \in \mathcal{N}$ and $P \in \mathcal{P}_{\mathrm{S}}(\boldsymbol{\beta},C_{\mathrm{S}})$.  Then, on the event $E_1^\delta(x)$, for any $x\in \mathcal{X}$ and any $\omega'\preceq\omega$, we have 
\begin{align}\label{bound_f_error_on_E1}
    \max_{i\in\{1,\dots,\tilde k_\omega\}}\bigl|f_{\omega'}(X_{(i)_\omega}(x))-f_{\omega'}(x)\bigr| \leq 2C_{\mathrm{B}}  
    C_{\mathrm{S}} \cdot \biggl(\frac{2\tilde k_{\omega}}{n_\omega\rho_\omega(x^{\omega})}\biggr)^{\beta_{\omega}/d_{\omega}}.
\end{align}
\end{lemma}

\begin{proof}[Proof of Lemma~\ref{lem:f_Xx_bound}]
Fix $x\in \mathcal{X}$. If $n_\omega\rho_\omega(x^\omega)/2\leq\lceil 4\log_{+}(|\mathcal{N}|/\delta)\rceil$ then both sides of \eqref{bound_f_error_on_E1} are zero by the definition of $\tilde k_\omega$. If $k_\omega\geq n_\omega\rho_\omega(x^\omega)/2>\lceil 4\log_{+}(|\mathcal{N}|/\delta)\rceil$, then the right hand side of \eqref{bound_f_error_on_E1} is greater than $2C_\mathrm{B}$, since $C_S\geq1$, so the result follows from the bound on $f_{\omega'}$ given in Proposition~\ref{prop:f_bounded}. If $\lceil 4\log_{+}(|\mathcal{N}|/\delta)\rceil\leq k_\omega< n_\omega\rho_\omega(x^\omega)/2$, then $\tilde k_\omega=k_\omega$, and from the definition of $E_1^\delta$ and the smoothness of $f_{\omega'}$ (see~\eqref{def:smoothness_eq}
), we have
\begin{align*}
    \max_{i\in\{1,\dots,k_{\omega}\}}\left|f_{\omega'}(X_{(i)_\omega}(x))\!-\!f_{\omega'}(x)\right| &\leq C_{\mathrm{S}} \max_{i\in\{1,\dots,k_{\omega}\}}\|X_{(i)_\omega}^{\omega}\!-\!x^{\omega}\|^{\beta_{\omega'}} \\
    &\leq C_S\left(\frac{2k_{\omega}}{n_\omega\rho_\omega(x^\omega)}\right)^{\beta_{\omega'}/d_{\omega}}\leq C_S\left(\frac{2k_{\omega}}{n_\omega\rho_\omega(x^\omega)}\right)^{\beta_{\omega}/d_{\omega}}.
\end{align*} 
Finally, if $k_\omega< \lceil 4\log_+(|\mathcal{N}|n_{\omega}/\delta) \rceil < n_{\omega} \rho_{\omega}(x^{\omega})/2$, then similarly we have
\begin{align*}
\max_{i\in\{1,\dots,\tilde k_{\omega}\}}\left|f_{\omega'}(X_{(i)_\omega}(x))-f_{\omega'}(x)\right|
    &=\max_{i\in\{1,\dots,\lceil4\log_+(|\mathcal{N}|/\delta)\rceil\}}\left|f_{\omega'}(X_{(i)_\omega}(x))-f_{\omega'}(x)\right|\\
    &\leq C_S\left(\frac{2\lceil4\log_+(|\mathcal{N}|/\delta)\rceil}{n_\omega\rho_\omega(x^\omega)}\right)^{\beta_{\omega'}/d_{\omega}}\\
    &\leq C_S\left(\frac{2\lceil4\log_+(|\mathcal{N}|/\delta)\rceil}{n_\omega\rho_\omega(x^\omega)}\right)^{\beta_{\omega}/d_{\omega}}.
\end{align*}
The result follows since $C_{\mathrm{B}} \geq 1$.
\end{proof}

\begin{proof}[Proof of Lemma~\ref{lem:f_is_not_zero}
]
First note that for $\omega = \mathbf{0}_d$ and for $x \in \mathcal{X}$, on the event $B_{\mathbf{0}_d}^{\delta}(x)$ we have that
\begin{align*}
\bigl|\hat{f}_{\mathbf{0}_d}(x)-f_{\mathbf{0}_d}(x)\bigr| =\Bigl|\frac{1}{n} \sum_{i=1}^n Y_i -\mathbb{P}_P(Y=1)\Bigr| \leq \sqrt{\frac{\log_+(2 |\mathcal{N}|/\delta)}{2n}} = R_{\mathbf{0}_d,2}.
\end{align*}
Thus the result holds for $\omega = \mathbf{0}_d$.  

Recall from Algorithm~\ref{alg:nonadaptive} 
that for $\omega \in \mathcal{N} \setminus\{\mathbf{0}_d\}$ we have
\begin{align}\label{def:f_tilde}
    \hat f_\omega(x) = \frac{1}{k_\omega} \sum_{i=1}^{k_\omega}Y_{(i)_\omega}(x) -\frac{1}{2} - \sum_{\omega'\prec\omega} \hat{f}_{\omega'}(x).
\end{align}
Then using also the fact that $\eta_\omega(x) = 1/2 + \sum_{\omega' \preceq \omega} f_{\omega'}(x)$, we have that 
\begin{align}
\nonumber
    \bigl|\hat f_\omega(x)\!-\!f_\omega(x)\bigr| & = \Bigl|\hat{f}_\omega(x)\! -\! \frac{1}{k_\omega}\sum_{i=1}^{k_\omega} \eta_\omega(X_{(i)_\omega}(x)) + \frac{1}{2} + \sum_{\omega'\preceq\omega}\frac{1}{k_\omega}\sum_{i=1}^{k_\omega} f_{\omega'}(X_{(i)_\omega}(x))\!-\!f_\omega(x)\Bigr| \\
    \begin{split}\label{ineq:tilde_f_bound}&\leq \Bigl|\frac{1}{k_\omega}\sum_{i=1}^{k_\omega}f_\omega(X_{(i)_\omega}(x))-f_\omega(x)\Bigr| + \Bigl|\frac{1}{k_\omega}\sum_{i=1}^{k_\omega}\bigl\{ Y_{(i)_\omega}(x)-\eta_\omega(X_{(i)_\omega}(x))\bigr\}\Bigr|
    \\ & \hspace{60pt}+\Bigl|\sum_{\omega'\prec\omega}\Bigl\{\frac{1}{k_\omega}\sum_{i=1}^{k_\omega} f_{\omega'}(X_{(i)_\omega}(x))- \hat{f}_{\omega'}(x)  \Bigr\}\Bigr|.\end{split} 
\end{align}     

On the event $E_1^{\delta}(x)$ the first term in \eqref{ineq:tilde_f_bound} is bounded by $R_{\omega, 1}(x)$ by Lemma~\ref{lem:f_Xx_bound}, and on the event $E_2^{\delta}(x)$ the second term in \eqref{ineq:tilde_f_bound} is bounded by $R_{\omega,2}$. For the third term, on $E_1^{\delta}(x) \cap E_2^{\delta}(x)$, first by Lemma~\ref{lem:f_Xx_bound} we have 
\begin{align*}   
\Bigl|\frac{1}{k_\omega}\sum_{i=1}^{k_\omega} f_{\omega'}(X_{(i)_\omega}(x)) - \hat{f}_{\omega'}(x)  \Bigr| & \leq  \Bigl|\frac{1}{k_\omega}\sum_{i=1}^{k_\omega}f_{\omega'}(X_{(i)_\omega}(x)) -f_{\omega'}(x)\Bigr|+ \bigl|f_{\omega'}(x)- \hat f_{\omega'}(x)\bigr|\\
& \leq R_{\omega, 1}(x) + \bigl|f_{\omega'}(x)- \hat f_{\omega'}(x)\bigr|,
\end{align*} 
for $\omega' \prec \omega$. Suppose now for induction, that~\eqref{ineq:f_bound} 
holds for all $\omega' \prec \omega$. Then using the facts that $R_{\omega', 1}(x) \leq R_{\omega, 1}(x)$ and $R_{\omega',2} \leq R_{\omega,2}$ for all $\omega' \prec \omega$, we deduce that  
\begin{align*}
    \bigl|\hat f_\omega(x)-f_\omega(x)\bigr|
    &\leq 2^{d_\omega} \cdot  R_{\omega, 1}(x)  + R_{\omega,2} + \sum_{\omega'\prec\omega}\bigl|f_{\omega'}(x)- \hat f_{\omega'}(x)\bigr|  \\
      &\leq (1+\sum_{\omega'\prec\omega}\kappa_{\omega'}) \cdot\Bigl(2^{d_\omega}\cdot R_{\omega,1}(x) + R_{\omega,2}\Bigr).
\end{align*}
The constant $\kappa_\omega$ can then be calculated using $\kappa_{\textbf{0}} = 1$ and the relation $\kappa_\omega := 1+\sum_{\omega'\prec\omega}\kappa_{\omega'}$, for $d_\omega > 1$. As in Proposition~\ref{prop:f_bounded}, this sequence is related to the ordered Bell numbers via $\kappa_{\omega} = 2B_{d_\omega} := \sum_{j=0}^{\infty} \frac{j^{d_\omega}}{2^j} \leq 2(d_\omega+1)^{d_\omega}$, for $d_\omega> 1$ \citep[see e.g.][]{zou2018log}. 
\end{proof}

\begin{proof}[Proof of Lemma~\ref{lem:hatC=BayesC}
]
First let  $\epsilon:=  C \cdot R_{\Omega}^{\delta}(x)$, so that $x \in \mathcal{X}$ satisfies $|\eta(x)-1/2|\geq\epsilon$.   We claim that on the event $E_1^\delta (x) \cap E_2^\delta(x)$, we have 
\begin{equation}
\label{eq:toprove}
|\hat{f}_{\omega}(x) - f_{\omega}(x)| \leq \frac{\epsilon}{2|\Omega\cup L(\Omega)|}
\end{equation}
for all $\omega\in \Omega\cup L(\Omega)$.  Then, if $\eta(x)-1/2 \geq \epsilon$, using that $\eta(x) = 1/2 + \sum_{\omega \in \Omega_{\star}\cup L(\Omega_\star)} f_{\omega}(x) = 1/2+ \sum_{\omega \in \Omega\cup  L(\Omega)} f_{\omega}(x)$, we have
\begin{align*}
    \hat\eta_{\Omega}(x)-1/2 = \sum_{\omega \in \Omega\cup  L(\Omega)} \bigl\{\hat{f}_{\omega}(x) - f_{\omega}(x)\bigr\} + \eta(x)- 1/2  \geq -\epsilon/2+\epsilon >0.
\end{align*}
By symmetry, if  $1/2 - \eta(x) \geq \epsilon$, we have that $1/2 - \hat\eta_\Omega(x) \geq \epsilon/2 > 0$.

It remains that to show that \eqref{eq:toprove} holds. First note that since $\eta(x) \in [0,1]$, we have that $\epsilon \leq 1/2$. 
Thus, for $\omega \in \Omega \cup L(\Omega)$, we have that 
\[
\lceil 4\log_+(|\mathcal N|/\delta)\rceil  \leq 4\log_+(2|\mathcal N|/\delta) \leq \frac{n_\omega^{\frac{2\beta_{\omega}\gamma_{\omega}}{\gamma_{\omega}(2\beta_{\omega}+d_{\omega}) + \alpha \beta_{\omega}}}}{C^{2}}  \leq k_\omega.
\]
Moreover, for $\omega \in \Omega\cup L(\Omega)$,

\[
\frac{n_{\omega} \rho_\omega(x^\omega)}{2} \geq 2^{d_{\omega}/\beta_{\omega}-1} C^{d_{\omega}/\beta_{\omega}}  n_{\omega}^{\frac{2\beta_{\omega} \gamma_{\omega}}{\gamma_{\omega}(2\beta_{\omega}+d_{\omega}) + \alpha \beta_{\omega}}} \geq 1 + \lfloor n_{\omega}^{\frac{2\beta_{\omega}\gamma_{\omega}}{\gamma_{\omega}(2\beta_{\omega}+d_{\omega}) + \alpha \beta_{\omega}}}\rfloor =  k_{\omega},
\]
since $C > 2$. Therefore $\tilde{k}_{\omega} = k_{\omega}$, for $\omega \in \Omega\cup L(\Omega)$. 
It follows that
\[
\Bigl(\frac{2\tilde{k}_\omega}{n_\omega\rho_\omega(x^\omega)}\Bigr)^{\beta_{\omega}/d_\omega} = \Bigl(\frac{2k_\omega}{n_\omega\rho_\omega(x^\omega)}\Bigr)^{\beta_{\omega}/d_\omega}
    \leq\Bigl(\frac{4 n_\omega^{\frac{2\beta_{\omega}\gamma_{\omega}}{\gamma_{\omega}(2\beta_{\omega}+d_\omega) + \alpha\beta_{\omega}}}}{n_\omega\rho_\omega(x^\omega)}\Bigr)^{\beta_{\omega}/d_\omega}
   \leq \frac{4n_{\omega}^{-\frac{\beta_{\omega}\gamma_{\omega} + \alpha\beta_{\omega}^2/d_{\omega}}{\gamma_{\omega}(2\beta_{\omega}+d_\omega) + \alpha \beta_{\omega}}}}{\rho_\omega^{\beta_{\omega}/d_\omega}(x^\omega)}.
\]
Then, recalling the definitions of $R_{\omega,1}(x)$ and $R_{\omega,2}$ from~\eqref{eq:biasvariance}
,  we have 
\begin{align}
    \label{eq:bound_bias_R} R_{\omega,1}(x) & = 2C_{\mathrm{B}}C_{\mathrm{S}} \cdot \max_{\omega' \preceq \omega} \Bigl(\frac{2\tilde{k}_{\omega'}}{n_{\omega'}\rho_{\omega'}(x^{\omega'})}\Bigr)^{\beta_{\omega'}/d_{\omega'}} \leq 8C_{\mathrm{B}}C_{\mathrm{S}} \cdot R_{\Omega}^\delta(x),
\end{align}
for $\omega \in \Omega\cup L(\Omega)$. 
Further
\begin{equation}
\label{eq:bound_var_R} 
R_{\omega,2}  = \sqrt{\frac{\log_+(2|\mathcal N|/\delta)}{2k_\omega}}
    \leq \frac{\log_+^{1/2}(2|\mathcal N|/\delta)}{2^{1/2}n_\omega^{\frac{\beta_{\omega}\gamma_{\omega}}{\gamma_{\omega}(2\beta_{\omega}+d_\omega) + \alpha\beta_{\omega}}}} 
   \leq \frac{1}{2^{1/2}} R_\omega^\delta(x)\leq \frac{1}{2^{1/2}} R_{\Omega}^\delta(x).
\end{equation}
Finally, by Lemma~\ref{lem:f_is_not_zero}
, along with \eqref{eq:bound_bias_R} and \eqref{eq:bound_var_R}, for $\omega\in\Omega\cup L(\Omega)$, we have that
\begin{align*}
    |\hat{f}_{\omega}(x) - f_{\omega}(x)|&\leq \kappa_{\omega} \cdot (2^{d_{\omega}} R_{\omega,1}(x) + R_{\omega,2} ) \\  
   & \leq \kappa_{\omega} \cdot \Bigl(2^{3 + d_{\omega}} C_{\mathrm{B}} C_{\mathrm{S}} + \frac{1}{2^{1/2}}\Bigr) \cdot R^\delta(x) \\
   &\leq \frac{C}{2|\Omega\cup L(\Omega)|} \cdot R^\delta(x) = \frac{\epsilon}{2|\Omega\cup L(\Omega)|},
\end{align*}
where the last equality holds since $C \geq 2 \kappa_{\omega} |\Omega\cup L(\Omega)|\bigl(2^{3 + d_{\omega}}C_{\mathrm{B}} C_{\mathrm{S}}+\sqrt{1/2}\bigr)$. This establishes the claim in~\eqref{eq:toprove} and completes the proof. 
\end{proof}

Our events $E_1^\delta(x)$ and $E_2^{\delta}(x)$ give us control of the properties of our estimators at each fixed point $x \in\mathcal{X}$.  Ideally we'd have this control simultaneously for all $x \in \mathcal{X}$, but this will typically not be possible. Instead, we only require this control at a large portion of the feature domain in order to control the excess risk.  We then show that the remaining portion of the feature domain is relatively small with high probability.  More precisely, for $\delta \in (0,1)$, $\Omega\subseteq\{0,1\}^d$ and $\alpha\in[0,\infty)$, let $\delta' := \frac{\delta}{2} \max_{\omega \in \Omega}\{1/n_{\omega}^{1+\alpha}\}$, 
and define
\[
\mathcal{X}_{\delta} = \mathcal{X}_{\delta}(D_n) := \Bigl\{x \in \mathcal{X} : E_{1}^{\delta'}(x) \cap E_{2}^{\delta'}(x) \ \text{holds} \Bigr\}.
\]

\begin{lemma}\label{lem:bound_prob_mu_big}
We have $\IP\Bigl\{\mu(\mathcal{X}_\delta^c)\geq \max_{\omega\in \Omega}(1/n_\omega^{1+\alpha})\Bigr\}\leq \delta$.
\end{lemma}
\begin{proof}[Proof of Lemma~\ref{lem:bound_prob_mu_big}]
By Markov's inequality and Lemmas~\ref{lem:E_1} and \ref{lem:E_2_prob}, we have 
\begin{align*}
    \mathbb{P}\Bigl\{\mu(\mathcal{X}_\delta^c)\geq\max_{\omega\in \Omega} (1/n_\omega^{1+\alpha}) \Bigr\} &\leq \frac{1}{\max_{\omega\in \Omega}\{1/n_\omega^{1+\alpha}\}}\cdot \IE\{\mu(\mathcal{X}_\delta^c)\} \\
    & = \frac{1}{\max_{\omega\in \Omega}\{1/n_\omega^{1+\alpha}\}}\cdot \mathbb{E} \Bigl(\int_{\mathbb{R}^d}\mathbbm{1}_{\{x\in \mathcal{X}_\delta^c\}} \, d\mu(x) \Bigr)\\
    &\leq \frac{1}{\max_{\omega\in \Omega}\{1/n_\omega^{1+\alpha}\}} \int_{\mathbb{R}^d} \Bigl\{ \IP(E_1^{\delta'}(x)^c) + \IP(E_2^{\delta'}(x)^c) \Bigr\} \, d\mu(x) \\& \leq \delta.
\end{align*}
\end{proof}

We are now in a position to provide the proofs of Proposition~\ref{prop:rate_if_Omega_correct} 
and subsequently the upper bound in Theorem~\ref{thm:minmax_bounds}
.

\begin{proof}[Proof of Proposition~\ref{prop:rate_if_Omega_correct}
]
Recall the constant $C$ from the statement of Lemma~\ref{lem:hatC=BayesC}
. For $\omega \in \mathcal{N}$, let 
\[
\lambda_{\omega} := \Bigl(\frac{1}{n_\omega}\Bigr)^{\frac{\beta_{\omega}\gamma_{\omega} + \alpha\beta_{\omega}^2/d_{\omega}}{\gamma_{\omega}(2\beta_{\omega}+d_\omega) + \alpha \beta_{\omega}}}, \quad \text{and} \quad 
    t_\omega := \lambda_{\omega}^{\frac{\alpha d_\omega}{\gamma_\omega d_{\omega}+\alpha\beta_{\omega}}}
\]
Define the partition $(\mathcal{X}_{\delta, j})_{j=0}^\infty$ of $\mathcal{X}_\delta$ given by
\begin{align*}
    \mathcal{X}_{\delta,0}&:= \{x\in \mathcal{X}_\delta:~ t_{\omega} \leq \rho_{\omega}(x^{\omega}), \ \text{for all} \ \omega \in \Omega\}\\
    \mathcal{X}_{\delta,j}&:=\Bigl\{x\in \mathcal{X}_\delta:~2^{-j}t_{\omega} \leq \rho_{\omega}(x^\omega), \ \text{for all}  \ \omega \in \Omega \Bigr\} \setminus \bigcup_{j' \in [j-1]\cup\{0\}} \mathcal{X}_{\delta,j'}.
\end{align*}
Then, by Lemma~\ref{lem:hatC=BayesC}
, for any $j\in \mathbb{N} \cup \{0\}$ and any $x\in \mathcal{X}_{\delta,j}$, we have that
\begin{align*}
|2\eta(x)-1|\cdot&\mathbbm{1}_{\{x\in \mathcal{X}_{\delta,j}:\hat C_\Omega(x)\neq C^{\mathrm{Bayes}}(x)\}} \\
& < 2 C \cdot R^{\delta'}_\Omega(x) \cdot \mathbbm{1}_{\{x\in \mathcal{X}_{\delta,j}\}}\\
& = 2 C \max_{\omega \in \Omega} \biggl\{ \biggl(\frac{n_{\omega}^{-\frac{\alpha\beta_{\omega}^2/d_{\omega}}{\gamma_{\omega}(2\beta_{\omega}+d_{\omega}) + \alpha \beta_{\omega}}}}{\rho_{\omega}^{\beta_{\omega}/d_{\omega}}(x^{\omega}) }  + \log^{1/2}_+(2|\mathcal{N}|/\delta') \biggr) n_{\omega}^{-\frac{\beta_{\omega}\gamma_{\omega}}{\gamma_{\omega}(2\beta_{\omega}+d_{\omega}) + \alpha \beta_{\omega}}} \biggl\} \cdot \mathbbm{1}_{\{x\in \mathcal{X}_{\delta,j}\}}\\ 
& \leq 2 C \max_{\omega \in \Omega} \biggl\{ \biggl(2^{j\beta_{\omega}/d_{\omega}} + \log^{1/2}_+(2|\mathcal{N}|/\delta') \Bigr)\cdot\frac{\lambda_{\omega}}{t_{\omega}^{\beta_{\omega}/d_{\omega}}} \biggl\} \cdot \mathbbm{1}_{\{x\in \mathcal{X}_{\delta,j}\}},
\end{align*} 
since $\hat C_{\Omega}(x)=C^{\mathrm{Bayes}}(x)$ on the event $E_1^{\delta'}(x)\cap E_2^{\delta'}(x)$.  Now, looking at the separate elements of the partition in turn, first since $P \in \mathcal{P}_{\mathrm{M}}(\alpha, C_{\mathrm{M}})$, we have that
\begin{align}
\label{eq:X0bit}
  \nonumber\int_{\mathbb{R}^d} & |2\eta(x)-1| \mathbbm{1}_{\{x\in \mathcal{X}_{\delta,0}:\hat C_\Omega(x)\neq C^{\mathrm{Bayes}}(x)\}} \, d\mu(x)\\ 
   \nonumber &\leq  \mu\bigl(\{x\in \mathcal{X}_{\delta,0}:\hat C_\Omega(x)\neq C^{\mathrm{Bayes}}(x)\}\bigr) \cdot  2 C  \max_{\omega \in \Omega} \biggl\{\Bigl(1 + \log_+^{1/2}(2|\mathcal N|/\delta') \Bigr) \frac{\lambda_{\omega}}{t_{\omega}^{\beta_{\omega}/d_{\omega}} } \biggr\}\\  
   \nonumber  &\leq \mu\biggl(\biggl\{x\in \mathbb{R}^d : |\eta(x)-1/2|<  C  \max_{\omega \in \Omega} \Bigl\{\Bigl(1 + \log_+^{1/2}(2|\mathcal N|/\delta') \Bigr) \frac{\lambda_{\omega}}{t_{\omega}^{\beta_{\omega}/d_{\omega}} } \Bigr\}\biggr\}\biggr) \\ 
   \nonumber & \hspace{120pt}   \cdot  2 C \max_{\omega \in \Omega} \biggl\{\Bigl(1 + \log_+^{1/2}(2|\mathcal N|/\delta') \Bigr) \frac{\lambda_{\omega}}{t_{\omega}^{\beta_{\omega}/d_{\omega}} } \biggr\}\\ 
    & \leq 2 C_{\mathrm{M}} C^{1+\alpha} \Bigl(1 + \log_+^{1/2}(2|\mathcal N|/\delta') \Bigr)^{1+\alpha} \cdot \Bigl(\max_{\omega \in \Omega} \frac{\lambda_{\omega}}{t_{\omega}^{\beta_{\omega}/d_{\omega}}}\Bigr)^{1+\alpha} . 
\end{align}
Further, for $j\geq1$, since $Q \in \mathcal{Q}_{\mathrm{L}}(\boldsymbol{\gamma},C_{\mathrm{L}},\mathcal{O})$, we have 
\begin{align}
\label{eq:Xjbit}
\nonumber &\int_{\mathbb{R}^d} |2\eta(x)-1|\mathbbm{1}_{\{x \in  \mathcal{X}_{\delta,j} :\hat C_\Omega(x)\neq C^{\mathrm{Bayes}}(x)\}} \, d\mu(x) \\
 \nonumber& \leq \mu\bigl(\mathcal{X}_{\delta,j}\bigr) \cdot  2 C \max_{\omega \in \Omega} \biggl\{\Bigl(2^{j\beta_{\omega}/d_{\omega}} + \log_+^{1/2}(2|\mathcal N|/\delta') \Bigr) \frac{\lambda_{\omega}}{t_{\omega}^{\beta_{\omega}/d_{\omega}} } \biggr\}\\
  \nonumber & = \mu\Bigl(\bigcup_{\omega \in \Omega} \{\rho_{\omega}(x^\omega) <  2^{-(j-1)} t_{\omega}\}\Bigr) \cdot 2 C  \max_{\omega \in \Omega} \biggl\{\Bigl(2^{j\beta_{\omega}/d_{\omega}} + \log_+^{1/2}(2|\mathcal N|/\delta') \Bigr) \frac{\lambda_{\omega}}{t_{\omega}^{\beta_{\omega}/d_{\omega}} } \biggr\}\\
 \nonumber & \leq \sum_{\omega \in \Omega} \mu_{\omega}\bigl(\{\rho_{\omega}(x^\omega) <  2^{-(j-1)} t_{\omega}\}\bigr) \cdot 2 C  \max_{\omega \in \Omega} \biggl\{\Bigl(2^{j\beta_{\omega}/d_{\omega}} + \log_+^{1/2}(2|\mathcal N|/\delta') \Bigr) \frac{\lambda_{\omega}}{t_{\omega}^{\beta_{\omega}/d_{\omega}} } \biggr\}\\
 \nonumber  & \leq   2 C C_{\mathrm{L}}  |\Omega| \cdot \max_{\omega \in \Omega} \{(2^{-(j-1)} t_{\omega})^{\gamma_{\omega}} \} \cdot  \max_{\omega \in \Omega} \biggl\{\Bigl(2^{j\beta_{\omega}/d_{\omega}} + \log_+^{1/2}(2|\mathcal N|/\delta') \Bigr) \frac{\lambda_{\omega}}{t_{\omega}^{\beta_{\omega}/d_{\omega}} } \biggr\} \\
 & \leq  2 C  C_{\mathrm{L}} |\Omega|  2^{-(j-1)\min_{\omega \in \Omega} \gamma_{\omega}} \cdot \Bigl(2^{j\max_{\omega\in\Omega}\{\beta_{\omega}/d_\omega\}}\! +\! \log_+^{1/2}(2|\mathcal N|/\delta') \Bigr)   \cdot  \Bigl(\max_{\omega \in \Omega} \lambda_{\omega}^{\frac{d_{\omega} \gamma_{\omega}} {d_{\omega}\gamma_{\omega} + \alpha \beta_{\omega}}}  \Bigr)^{1+\alpha}.
\end{align}
Here we have used that $t_{\omega}^{\gamma_{\omega}} = \lambda_{\omega}^{\frac{d_{\omega} \gamma_{\omega} \alpha}{d_{\omega} \gamma_{\omega} +  \alpha \beta_{\omega}}}$. Observe further that
\begin{align}
\label{eq:sumconverges}
\nonumber &\sum_{j=1}^{\infty} 2^{-j \min_{\omega\in \Omega} \gamma_{\omega}} \cdot \Bigl(2^{j\max_{\omega\in\Omega}\{\beta_{\omega}/d_\omega\}} + \log_+^{1/2}(2|\mathcal N|/\delta') \Bigr) \\ 
\nonumber & \hspace{120pt}  = \frac{1}{2^{\min_{\omega\in \Omega} \gamma_{\omega}-\max_{\omega\in\Omega}\{\beta_{\omega}/d_\omega\}}-1}+\frac{\log_+^{1/2}(2|\mathcal N|/\delta')}{2^{\min_{\omega\in \Omega} \gamma_{\omega}}-1} \\
& \hspace{120pt} \leq \frac{2\log_+^{1/2}(2|\mathcal N|/\delta')}{2^{\min_{\omega\in \Omega} \gamma_{\omega}-\max_{\omega\in\Omega}\{\beta_{\omega}/d_\omega\}}-1}
\end{align}
since $\min_{\omega \in \Omega} \gamma_{\omega} > \max_{\omega\in\Omega} \beta_{\omega}/d_{\omega}$. 
It follows from~\eqref{eq:X0bit},~\eqref{eq:Xjbit} and~\eqref{eq:sumconverges} that
\begin{align*}
    \mathcal{E}_{P}(\hat{C}_{\Omega}) &= \int_{\{x\in\mathbb{R}^d:\hat{C}_{\Omega}(x)\neq C^{\mathrm{Bayes}}(x)\}}|2\eta(x)-1|  \, d\mu(x)\\
    &= \mu(\mathcal{X}_\delta^c) + \sum_{j=0}^\infty\int_{\mathbb{R}^d}|2\eta(x)-1| \cdot \mathbbm{1}_{\{x \in  \mathcal{X}_j:\hat C_\Omega(x)\neq C^{\mathrm{Bayes}}(x)\}}\, d\mu(x)\\
    & \leq \mu(\mathcal{X}_\delta^c) 
    \\ & \hspace{-38pt}\! +\! \Bigl\{ 2 C_{\mathrm{M}} C^{1+\alpha} \log_+^{(1+\alpha)/2}(2|\mathcal N|/\delta') \!+\!  \frac{4CC_{\mathrm{L}} |\Omega| \cdot \log_+^{1/2}(2|\mathcal N|/\delta')}{2^{-\max_{\omega\in\Omega}\{\beta_{\omega}/d_\omega\}}-2^{-\min_{\omega \in \Omega} \gamma_{\omega}}} \Bigr\}\Bigl(\max_{\omega \in \Omega} \lambda_{\omega}^{\frac{d_{\omega} \gamma_{\omega}} {d_{\omega} \gamma_{\omega} + \alpha \beta_{\omega}} }\Bigr)^{1+\alpha}.
\end{align*}
Finally observe that  
\[
\lambda_{\omega}^{\frac{d_\omega\gamma_{\omega} (1+\alpha)} {d_\omega\gamma_{\omega} + \alpha \beta_{\omega}}}
= n_\omega^{-\frac{\beta_{\omega} \gamma_{\omega} (1+\alpha)}{\gamma_{\omega}(2\beta_{\omega}+d_\omega) + \alpha \beta_{\omega}}} \geq n_{\omega}^{-(1+\alpha)} .
\]
Thus, by Lemma~\ref{lem:bound_prob_mu_big}, we have
\[
\mathbb{P}\Bigl(\mu(\mathcal{X}_\delta^c) \geq \max_{\omega \in \Omega} \lambda_{\omega}^{\frac{\gamma_{\omega} d_\omega(1+\alpha)} {\gamma_{\omega}d_{\omega} + \alpha \beta_{\omega}}} \Bigr) \leq  \mathbb{P}\Bigl(\mu(\mathcal{X}_\delta^c) \geq \max_{\omega \in \Omega} n_{\omega}^{-(1+\alpha)}\Bigr) \leq \delta.
\]
The conclusion then follows with 
\[
K_{\Omega} :=1 +  2 C_{\mathrm{M}} C^{1+\alpha} +   \frac{4CC_{\mathrm{L}}|\Omega|}{2^{-\max_{\omega \in \Omega}\{\beta_{\omega}/d_{\omega}\}}-2^{-\min_{\omega \in \Omega} \gamma_{\omega}}} .
\]
\end{proof}

\begin{proof}[Proof of the upper bound in Theorem~\ref{thm:minmax_bounds}
]
Fix $\Omega \in \mathcal{I}(\{0,1\}^d)$. Let 
\[
\delta_0 := \max_{\omega \in \Omega} n_\omega^{-\frac{\beta_{\omega}\gamma_\omega(1+\alpha)}{\gamma_\omega(2\beta_{\omega}+d_\omega)+\alpha\beta_{\omega}}} \cdot \max_{\omega \in \Omega}\{n_{\omega}^{-(1+\alpha)}\} \leq \max_{\omega \in \Omega} n_\omega^{-\frac{3\beta_{\omega}\gamma_\omega(1+\alpha)}{\gamma_\omega(2\beta_{\omega}+d_\omega)+\alpha\beta_{\omega}}} 
\]
and let 
\[
t_0 := K_{\Omega} \cdot \log_+^{\frac{1+\alpha}{2}}\Bigl(\frac{2|\mathcal N|}{\delta_0}\Bigr) \cdot \max_{\omega \in \Omega} n_\omega^{-\frac{\beta_{\omega}\gamma_\omega(1+\alpha)}{\gamma_\omega(2\beta_{\omega}+d_\omega)+\alpha\beta_{\omega}}} 
\] 
Note also that $\delta_0 \geq  \max_{\omega \in \Omega} n_\omega^{-3(1+\alpha)/2}$, and thus 
\[
\log_+^{\frac{1+\alpha}{2}}\Bigl(\frac{2|\mathcal N|}{\delta_0}\Bigr) \leq (3(1+\alpha)/2)^{(1+\alpha)/2} \log_+^{\frac{1+\alpha}{2}}\Bigl(2|\mathcal N| \min_{\omega \in \Omega} n_{\omega}\Bigr).
\]
By Proposition~\ref{prop:rate_if_Omega_correct}
, it follows that 
\begin{align*}
    \mathbb{E}_Q&\bigl\{\mathcal{E}_P(\hat C_{\Omega}) \bigm| O_1=o_1,\ldots,O_n=o_n\bigr\} 
    \\ & \leq t_0 + \mathbb{P}_Q\bigl\{\mathcal{E}_P(\hat C_{\Omega}) > t_0 \mid O_1=o_1,\ldots,O_n=o_n\bigr\} 
    \\ & \leq K_{\Omega} \cdot \log_+^{\frac{1+\alpha}{2}}\Bigl(\frac{2|\mathcal N|}{\delta_0}\Bigr) \cdot \max_{\omega \in \Omega} n_\omega^{-\frac{\beta_{\omega}\gamma_\omega(1+\alpha)}{\gamma_\omega(2\beta_{\omega}+d_\omega)+\alpha\beta_{\omega}}} + \max_{\omega \in \Omega} n_\omega^{-\frac{\beta_{\omega}\gamma_\omega(1+\alpha)}{\gamma_\omega(2\beta_{\omega}+d_\omega)+\alpha\beta_{\omega}}}
    \\ & \leq \Bigl\{1 + 2 K_{\Omega} \cdot \Bigl(\frac{3(1+\alpha)}{2}\Bigr)^{\frac{1+\alpha}{2}} \log_+^{\frac{1+\alpha}{2}}(2^{d+1}) \Bigr\} \cdot \log_+^{\frac{1+\alpha}{2}}\Bigl( \min_{\omega \in \Omega} n_{\omega} \Bigr) \cdot \max_{\omega \in \Omega} n_\omega^{-\frac{\beta_{\omega}\gamma_\omega(1+\alpha)}{\gamma_\omega(2\beta_{\omega}+d_\omega)+\alpha\beta_{\omega}}} .
\end{align*}
The upper bound in Theorem~\ref{thm:minmax_bounds} 
follows by taking $\Omega = \Omega_{\star}$ if $\Omega_\star \subseteq \mathcal{N}$, and otherwise using the trivial bound that $\mathbb{E}_Q\{\mathcal{E}_P(\hat C_{\Omega}) \mid O_1=o_1,\ldots,O_n=o_n\} \leq 1$.
\end{proof}

\subsection{Proof of Theorem~\ref{thm:nonadaptiveUBnew}
\label{subsec:proofnonadaptiveUBnew}}
The proof of Theorem~\ref{thm:nonadaptiveUBnew} 
involves controlling the additional error that we incur from choosing which of the $\hat{f}_{\omega}$ to set to zero based on the training data.  Recall the definitions of $\hat{\sigma}_{\omega}^2$,  $\tau_{\omega}$ and the set $\hat{\Omega} \in \mathcal{I}(\{0,1\}^d)$ from Algorithm~\ref{alg:nonadaptive}
. As in Sections~\ref{subsec:proofUB}
, \ref{subsec:proof2} 
and~\ref{sec:appendixUBproofs}, we treat the missingness indications $o_1, \ldots, o_n$ as fixed and all probability statements in this section should also be interpreted as being conditional on $O_1 = o_1, \ldots, O_n = o_n$. 

To prove the key Proposition~\ref{lem:correct_thresholding} 
in Section~\ref{subsec:proofUB}
, we introduce four high-probability events. The first event, $E_3^{\delta}$, gives us control of the average signal at the training points for every observation pattern $\omega\in\mathcal{N}$, i.e. it ensures that $1/n_\omega\cdot \sum f_\omega^2(X_i)$ is sufficiently large, and the second event, $E_4^\delta$, bounds the number of training data points falling into the tail of the marginal distributions. The third and fourth events, $E_5^\delta$ and $E_6^\delta$, respectively, are similar to the events $E_1^\delta(x)$ and $E_2^\delta(x)$ from the Subsection~\ref{subsec:proofUB}
, and allow us to estimate the error between $\hat{f}_\omega$ and $f_\omega$ at each of the training points in the body of the distribution, that is, those which are not in the tails. Lemmas~\ref{lem:E_3_prob} to~\ref{lem:E_6} bound the probability of the complement of $E_\circ^\delta:=E_3^{\delta} \cap E_4^{\delta} \cap E_5^{\delta} \cap E_{6}^{\delta}$.  Then, working on this event, Lemmas~\ref{lem:f_Xx_bound_training_points} and~\ref{lem:f_bound_at_training_points} provide the rate at which we control $|\hat{f}_\omega(X_j)-f_\omega(X_j)|$ uniformly over all training points $X_j$s in the body of the distribution. Finally, the conclusion of these results is that $\hat{\Omega}=\Omega_\star$ on the event $E_\circ^\delta$, which proves Proposition~\ref{lem:correct_thresholding}
.  
 
Recall that $N_{\omega} = \{i \in [n] : \omega \preceq o_i \}$, then for $\omega \in \mathcal{N}\cap\Omega_\star$ let
\begin{align*}
    D^\delta_\omega:=\biggl\{\frac{1}{n_\omega}\sum_{i \in N_{\omega}} f_\omega^2(X_i) > \sigma_\omega^2 - C^2_{\mathrm{B}} \sqrt{\frac{\log_+(|\mathcal{N}|/\delta)}{2n_{\omega}}}\biggr\},
\end{align*}
and let $E_{3}^\delta:= \bigcap_{\omega \in  \mathcal{N} \cap \Omega_{\star} }D_{\omega}^\delta$.

\begin{lemma}\label{lem:E_3_prob}
For $\delta\in(0,1)$, we have $\IP\bigl((E_{3}^\delta)^c \bigr)\leq\delta.$
\end{lemma}
\begin{proof}[Proof of Lemma~\ref{lem:E_3_prob}]
Fix $\omega \in \mathcal{N} \cap \Omega_{\star}$, then $\{f^2_{\omega}(X_i)\}_{i \in N_{\omega}}$ are independent random variables with mean $\mathbb{E}_Q\{f_{\omega}^2(X)| O = o_i\} \geq \sigma_\omega^2$, which take values in $[0, C_{\mathrm{B}}^2]$ by Proposition \ref{prop:f_bounded}. Then by Hoeffding's inequality, we have 
\begin{align*}
   & \IP\Biggl(\frac{1}{n_\omega}\sum_{i \in N_{\omega}} f_\omega^2(X_i)-\sigma_\omega^2 \leq - C^2_{\mathrm{B}} \sqrt{\frac{\log_+(|\mathcal{N}|/\delta)}{2n_{\omega}}}\Biggr)
    \\ &\leq \IP\Biggl(\frac{1}{n_\omega}\sum_{i \in N_{\omega}} f_\omega^2(X_i)-\frac{1}{n_\omega}\sum_{i \in N_{\omega}} \mathbb{E}_Q\{f_\omega^2(X_i) \mid O_i = o_i\} \leq - C^2_{\mathrm{B}} \sqrt{\frac{\log_+(|\mathcal{N}|/\delta)}{2n_{\omega}}}\Biggr)
    \leq \frac{\delta}{|\mathcal{N}|}.
\end{align*}
Taking a union bound gives the desired result.
\end{proof}

We will also require control on the number of training data points that lie in the tail of the marginal feature distribution. As in the previous subsection, for $\omega \in \{0,1\}^d$, we write $\rho_{\omega}(x)$ in place of $\min_{o\in\mathcal{O}:o\succeq\omega}\rho_{\mu_{\omega|o}, d_{\omega}}(x)$. Then, for $\omega \in \{\Omega_{\star} \cup L(\Omega_{\star})\} \cap \mathcal{N} $ and $r_{\omega} \in (0,1)$, let
\begin{align*}
    F^\delta_\omega = F^{\delta}_\omega( r_{\omega}) :=\biggl\{\frac{1}{n_\omega}\sum_{i\in N_\omega}\mathbbm{1}_{\{\rho_{\omega}(X_i^{\omega})<r_\omega\}} < 2C_{\mathrm{L}}r_{\omega}^{\gamma_\omega}\biggr\},
\end{align*}
and let $E_{4}^\delta:= \bigcap_{\omega \in \{\Omega_{\star} \cup L(\Omega_{\star})\} \cap \mathcal{N}} F_{\omega}^\delta$.

\begin{lemma}\label{lem:E_4_prob}
Fix $\mathcal{O} \subseteq \{0,1\}^d$, $\delta \in (0,1)$, $\boldsymbol{\gamma} \in [0,\infty)^d$, $C_{\mathrm{L}} \geq 1$ and $Q \in \mathcal{Q}^+_{\mathrm{L}}(\boldsymbol{\gamma}, C_{\mathrm{L}},\mathcal{O})$. Then if $r_\omega \geq  \Bigl( \frac{4\log_+(|\mathcal{N}|/\delta)}{C_{\mathrm{L}}n_{\omega}}\Bigr)^{1/\gamma_\omega}$ for all $\omega \in  \{\Omega_{\star} \cup L(\Omega_{\star})\} \cap \mathcal{N}$, we have $\IP\bigl((E_{4}^\delta)^c \bigr)\leq\delta.$
\end{lemma}
\begin{proof}[Proof of Lemma~\ref{lem:E_4_prob}]
Fix $\omega \in \mathcal{N}$, then for $i \in N_{\omega}$, we have that $\mathbbm{1}_{\{\rho_{\omega}(X_i^{\omega}) < r_\omega\}}$ are independent Bernoulli random variables with mean $\mu_{\omega | o_i} (\{x \in\mathbb{R}^d : \rho_{\omega}(x) < r_\omega\}) \leq C_{\mathrm{L}} r_\omega^{\gamma_{\omega}}$, since $Q \in \mathcal{Q}^+_{\mathrm{L}}(\boldsymbol{\gamma}, C_{\mathrm{L}},\mathcal{O})$. Then by a multiplicative Chernoff bound \citep[Theorem 9.3]{billingsley2008probability}, we have 
\begin{align*}
    \IP\biggl(\frac{1}{n_\omega}\sum_{i\in N_\omega}\mathbbm{1}_{\{\rho_{\omega}(X_i^{\omega})<r_\omega\}} \geq  2C_{\mathrm{L}}r_\omega^{\gamma_\omega}\biggr) \leq \exp(-n_\omega C_{\mathrm{L}}r_\omega^{\gamma_{\omega}} /4) \leq \frac{\delta}{|\mathcal{N}|},
\end{align*}
since $r_{\omega} > \bigl( \frac{4\log_+(|\mathcal{N}|/\delta)}{C_{\mathrm{L}}n_{\omega}}\bigr)^{1/\gamma_\omega}$. Taking a union bound over $\omega \in  \{\Omega_{\star} \cup L(\Omega_{\star})\} \cap \mathcal{N}$ gives the desired result.
\end{proof}

Our final two events provide control of the variance and bias of our estimates of $f_\omega$, for $\omega \in \mathcal{N}$, evaluated at the appropriate training data points. For $\delta \in (0,1)$ and $\omega \in \mathcal{N}$ let
\begin{align*}
    G^\delta_\omega := \bigcap_{j\in N_\omega} \Biggl\{\biggl|\frac{1}{k_\omega}\sum_{i=1}^{k_\omega} \Bigl\{Y_{(i)_\omega}(X_j) - \eta_\omega(X_{(i)_\omega}(X_j))\Bigr\} \biggr| \leq \sqrt{\frac{\log_+(2 |\mathcal{N}|n_\omega/\delta)}{2k_\omega}}\Biggr\},
\end{align*}
and let $E_{5}^\delta:= \bigcap_{\omega \in \mathcal{N}} G_{\omega}^\delta$. Moreover, for any $\omega\in\{\Omega_{\star} \cup L(\Omega_{\star})\} \cap \mathcal{N}$, $\delta \in (0,1)$, $n_\omega\geq2$, and $r_{\omega} \in (0,1)$, let
\begin{align*}
    H_{\omega}^\delta&:= \bigcap_{j\in N_\omega}\! \biggl\{\lVert X_{({k}_\omega)_\omega}^{\omega}(X_j^\omega)\!-\!X_j^\omega\rVert\! \cdot\! \mathbbm{1}_{\{\rho_{\omega}(X_j^\omega)\geq r_{\omega}\}} \leq \Bigl(\frac{2\max\{\lceil4\log_{+}(|\mathcal{N}|n_\omega/\delta)\rceil, k_\omega \}}{(n_\omega-1)r_{\omega}}\Bigr)^{\frac{1}{d_\omega}}\biggr\}
\end{align*}
and let $E_{6}^\delta:= \bigcap_{\omega \in \{\Omega_{\star} \cup L(\Omega_{\star})\} \cap \mathcal{N}} H_{\omega}^\delta$.

\begin{lemma}\label{lem:E_5_prob}
For $\delta\in(0,1)$, we have $\IP\bigl((E_{5}^\delta)^c \bigr)\leq\delta.$
\end{lemma}
\begin{proof}[Proof of Lemma~\ref{lem:E_5_prob}]
Fix $\omega \in \mathcal{N}$ and $j \in N_{\omega}$. The labels $Y_{(1)_\omega}(X_j), \ldots, Y_{(n_\omega)_\omega}(X_j)$ are conditionally independent given $(X_\ell^{\omega})_{\ell \in N_{\omega}}$. Further, for $i \in [n_{\omega}]$, we have that 
\[
\mathbb{E}\bigl\{Y_{(i)_\omega}(X_j) \big| (X_\ell^{\omega})_{\ell \in N_{\omega}} \bigr\} =\mathbb{E}\bigl\{Y_{(i)_\omega}(X_j) \big| X_{(i)_{\omega}}^\omega(X_j) \bigr\} = \eta_{\omega}(X_{(i)_{\omega}}(X_j)).
\]
Then by Hoeffding's inequality, we have that
\begin{align*}
    \IP\Biggl(\Bigl|\frac{1}{k_\omega}\sum_{i=1}^{k_\omega} \bigl\{ Y_{(i)_\omega}(X_j)-\eta_\omega(X_{(i)_\omega}(X_j)) \bigr\}\Bigr| > \sqrt{\frac{\log_+(2 |\mathcal{N}|n_\omega/\delta)}{2k_\omega}} \Biggm|(X_\ell^{\omega})_{\ell \in N_{\omega}}\Biggr) \leq \frac{\delta}{|\mathcal{N}|n_\omega}.
\end{align*}
Taking expectation over $(X_\ell^{\omega})_{\ell \in N_{\omega}}$ and using a union bound first over $j\in N_\omega$ and then over $\omega\in\mathcal{N}$ gives the result.
\end{proof}

\begin{lemma}\label{lem:E_6}
Fix $\delta \in (0,1)$ and suppose that $n_{\omega} \geq 2$ for every $\omega \in \{\Omega_{\star} \cup L(\Omega_{\star})\} \cap \mathcal{N}$. If $r_{\omega} > 2(n_{\omega}-1)^{-1}\max\{\lceil4\log_+(|\mathcal{N}|n_{\omega}/\delta) \rceil, k_{\omega}\}$ for all $\omega \in \{\Omega_{\star} \cup L(\Omega_{\star})\} \cap \mathcal{N}$, then $\IP((E_6^\delta)^c)\leq\delta$.
\end{lemma}
\begin{proof}[Proof of Lemma~\ref{lem:E_6}]
Fix $\omega \in \{\Omega_{\star} \cup L(\Omega_{\star})\} \cap \mathcal{N}$, $j \in N_{\omega}$ and $x_j \in \mathcal{X}$. Then, if $\rho_{\omega}(x_j^\omega) < r_{\omega}$, we have
\[
\mathbb{P}\Biggl( \bigl\|  X_{({k}_\omega)_\omega}^{\omega}(X_j^\omega)\!-\!X_j^\omega\bigr\|\! \cdot\! \mathbbm{1}_{\{\rho_{\omega}(X_j^\omega)\geq r_{\omega}\}}  > \biggl(\frac{2\{\lceil4\log_{+}(|\mathcal{N}|n_\omega/\delta)\rceil\! \vee\! k_\omega \}}{(n_\omega-1)r_{\omega}}\biggr)^{\frac{1}{d_\omega}} \Biggm| X_j\! =\! x_j\Biggr)\! =\! 0,
\]
since the first term inside the probability is $0$.
On the other hand, if $\rho_{\omega}(x_j^\omega)\geq r_{\omega}$, then let 
\begin{align*}
    r:=\biggl(\frac{2\max\{\lceil4\log_{+}(|\mathcal{N}|n_\omega/\delta)\rceil, k_\omega \}}{(n_\omega-1)r_{\omega}}\biggr)^{1/d_\omega} < 1,
\end{align*}
where the inequality holds by the assumption on $r_{\omega}$. Therefore, by the definition of $\rho_{\omega}$, we have for $i \in N_{\omega}$, 
\begin{align*}
\mu_{\omega|o_i}(B_r(x_j^\omega)) & \geq \rho_{\mu_{\omega|o_i},d_\omega}(x_j^\omega)r^{d_\omega}
\geq \rho_{\omega}(x_j^\omega)r^{d_\omega}\geq r_{\omega}r^{d_\omega} \\
&= \frac{2\max\{\lceil4\log_{+}(|\mathcal{N}|n_\omega/\delta)\rceil, k_\omega \}}{(n_\omega-1)}\geq\frac{8\log_{+}(|\mathcal{N}|n_\omega/\delta)}{(n_\omega-1)}.
\end{align*}
Then, using a multiplicative Chernoff bound, we obtain
\begin{align*}
\IP\Bigl( \bigl\| X_{(k_\omega)_\omega}^{\omega}(X_j^\omega)-X_j^\omega\bigr\| &> r \Bigm| X_j = x_j \Bigr) \leq\IP\Bigl(\sum_{i \in N_{\omega}} \mathbbm{1}_{\{X_i^\omega\in B_r(X_j^{\omega})\}}< k_\omega \Big|X_j = x_j \Bigr)\\
&= \IP\Bigl(\sum_{i \in N_{\omega}\setminus\{j\}} \mathbbm{1}_{\{X_i^\omega\in B_r(x_j^{\omega})\}}< k_\omega-1  \Bigr)\\
&\leq\IP\biggl(\sum_{i \in N_\omega\setminus\{j\}}\mathbbm{1}_{\{X_i^\omega\in B_r(x_j^{\omega})\}}< \max\{\lceil4\log_{+}(|\mathcal{N}|n_\omega/\delta)\rceil, k_\omega \}\biggr)\\
&\leq\IP\biggl(\sum_{i\in N_\omega\setminus\{j\}}\mathbbm{1}_{\{X_i^\omega\in B_r(x_j^\omega)\}}< \frac{\sum_{i \in N_{\omega} \setminus \{j\}} \mu_{\omega | o_i} (B_r(x_j^\omega))}{2} \biggr)\\
&\leq e^{-\sum_{i \in N_{\omega} \setminus \{j\}} \mu_{\omega | o_i} (B_r(x_j^\omega))/8} \leq  \frac{\delta}{|\mathcal{N}|n_\omega }.
\end{align*}
Taking expectations over $X_j$, a union bound over $j\in N_\omega$, and another union bound over $\omega\in\{\Omega_{\star} \cup L(\Omega_{\star})\} \cap \mathcal{N}$ concludes the proof.
\end{proof}

\begin{lemma}\label{lem:f_Xx_bound_training_points}
Fix $\omega \in \{\Omega_{\star} \cup L(\Omega_{\star})\} \cap \mathcal{N}$, $r_{\omega} \in (0,1)$, $\boldsymbol{\beta}\in(0,1]^d$, and $C_{\mathrm{S}}\geq1$. Let $P \in \mathcal{P}_{\mathrm{S}}(\boldsymbol{\beta},C_{\mathrm{S}})$ and suppose that $n_{\omega} \geq 2$. Then on the event $E_6^\delta$, for any $\omega'\preceq\omega$, we have that
\begin{align}\label{bound_f_error_on_E6}
    & \max_{j\in N_{\omega}} \max_{i\in [k_\omega]}\left|f_{\omega'}(X_{(i)_\omega}(X_j))-f_{\omega'}(X_j)\right| \cdot\mathbbm{1}_{\{\rho_\omega(X_j^\omega)\geq r_{\omega}\}} \nonumber
    \\ & \hspace{90pt} \leq 2C_{\mathrm{B}}  
    C_{\mathrm{S}} \cdot \left(\frac{2\max\{\lceil4\log_{+}(|\mathcal{N}|n_\omega/\delta)\rceil, k_\omega \}}{(n_\omega-1)r_{\omega}}\right)^{\beta_{\omega}/d_{\omega}}.
\end{align}
\end{lemma}
\begin{proof}[Proof of Lemma~\ref{lem:f_Xx_bound_training_points}]
Fix $j\in N_\omega$. If $\rho_\omega(X_j^\omega)< r_{\omega}$, then the left hand side of \eqref{bound_f_error_on_E6} is $0$, so the result follows trivially. Suppose then for the remainder of the proof that $\rho_\omega(X_j^\omega)\geq r_{\omega}$. If $(n_\omega-1)r_{\omega} \leq 2\max\{\lceil 4\log_{+}(2|\mathcal{N}|n_\omega/\delta)\rceil, k_\omega\}$, then the right hand side is at least $2C_{\mathrm{B}}$, so the result follows from Proposition~\ref{prop:f_bounded} since $C_S\geq1$. If $\lceil 4\log_{+}(|\mathcal{N}|n_\omega/\delta)\rceil\leq k_\omega< (n_\omega-1)r_{\omega}/2$, then from the definition of $E_6^\delta$ and the smoothness of $f_{\omega'}$ (since $P \in \mathcal{P}_{\mathrm{S}}(\boldsymbol{\beta}, C_{\mathrm{S}})$, see~\eqref{def:smoothness_eq}
), we have
\begin{align*}
    \max_{i\in[k_{\omega}]}\left|f_{\omega'}(X_{(i)_\omega}(X_j))\!-\!f_{\omega'}(X_j)\right| &\leq C_{\mathrm{S}} \max_{i\in[k_{\omega}]}\|X_{(i)_\omega}^{\omega}(X_j)\!-\!X_j^{\omega}\|^{\beta_{\omega'}}  \\
    &\leq C_S\left(\frac{2k_{\omega}}{(n_\omega\!-\!1)r_{\omega}}\right)^{\beta_{\omega'}/d_{\omega}}\leq C_S\left(\frac{2k_{\omega}}{(n_\omega\!-\!1)r_{\omega}}\right)^{\beta_{\omega}/d_{\omega}}.
\end{align*} 
Finally, if $k_\omega< \lceil 4\log_+(|\mathcal{N}|n_{\omega}/\delta) \rceil < (n_{\omega}-1) r_{\omega}/2$, then similarly we have
\begin{align*}
\max_{i\in[k_{\omega}]}\left|f_{\omega'}(X_{(i)_\omega}(X_j))-f_{\omega'}(X_j)\right| &    \leq \max_{i\in[\lceil4\log_+(|\mathcal{N}|n_\omega/\delta)\rceil]}\left|f_{\omega'}(X_{(i)_\omega}(X_j))-f_{\omega'}(X_j)\right| \\
    &\leq  C_S\left(\frac{2\lceil4\log_+(|\mathcal{N}|n_\omega/\delta)\rceil}{(n_\omega-1)r_{\omega}}\right)^{\beta_{\omega'}/d_{\omega}}\\
    &\leq  C_S\left(\frac{2\lceil4\log_+(|\mathcal{N}|n_\omega/\delta)\rceil}{(n_\omega-1)r_{\omega}}\right)^{\beta_{\omega}/d_{\omega}}.
\end{align*}
The result follows since $C_{\mathrm{B}} \geq 1$.
\end{proof}

We are now in a position to control the bias and variance of the estimators $\hat{f}_{\omega}$, for $\omega \in \mathcal{N}$, evaluated at the feature observations in training data set.  We first provide control of $\hat{f}_{\omega}$ at every training point in the body of the distribution (i.e.~those for which $\rho_\omega(X_j^\omega)\geq r_{\omega}$) on the event $E_5^{\delta} \cap E_6^{\delta}$. Let $R_{\mathbf{0},\mathrm{T}} = \sqrt{\frac{\log_+(2 |\mathcal{N}|n/\delta)}{2n}}$, and for $\omega \in \{\Omega_{\star} \cup L(\Omega_\star) \} \cap \mathcal{N} \setminus \{\mathbf{0}\}$, let 
\[
    R_{\omega,\mathrm{T}}:= \biggl[ 20C_{\mathrm{B}}C_{\mathrm{S}}\cdot \max_{0 \prec \omega' \preceq \omega}\Bigl\{  \log_{+}
    (2|\mathcal{N}|n_{\omega'}/\delta) \cdot \Bigl(r_{\omega'}(n_{\omega'}\!-\!1)^{\frac{\gamma_{\omega'}d_{\omega'} + \alpha \beta_{\omega'}}{\gamma_{\omega'}(2\beta_{\omega'}+d_{\omega'})+\alpha \beta_{\omega'}}}\Bigr)^{-\beta_{\omega'}/d_{\omega'}}\Bigr\} \biggr]\! \vee\! R_{\mathbf{0},\mathrm{T}}.
\]

\begin{lemma}\label{lem:f_bound_at_training_points}
Fix $\boldsymbol{\beta} \in (0,1]^d$, $C_{\mathrm{S}}\geq 1$, and suppose that $P \in \mathcal{P}_{\mathrm{S}}(\boldsymbol{\beta},C_{\mathrm{S}})$.  Suppose that $n_\omega \geq 2$ and $ r_{\omega} \geq 2(n_{\omega}-1)^{-1}\max\{\lceil 4\log_+(|\mathcal N|n_{\omega}/\delta)\rceil,k_{\omega}\}$ for all $\omega \in \{\Omega_{\star} \cup L(\Omega_\star) \} \cap \mathcal{N}$, then on the event $E_5^{\delta} \cap E_{6}^{\delta}$, the following statements hold:

(i) For $\omega \in \{\Omega_{\star} \cup L(\Omega_\star) \} \cap \mathcal{N}$, we have that 
\begin{align}\label{ineq:f_bound_training_points}
  \max_{j \in N_{\omega}}  \bigl|\hat{f}_{\omega}(X_j)-f_{\omega}(X_j)\bigr|\cdot\mathbbm{1}_{\{\rho_\omega(X_j^\omega)\geq  r_{\omega}\}} \leq \kappa_{\omega,T} \cdot R_{\omega,T},
\end{align}
where $1 \leq \kappa_{\omega,T} \leq 2^{d_\omega(d_\omega+1)}$ is given explicitly in the proof. 

(ii) For $\omega \in U(\Omega_{\star}) \cap \mathcal{N}$, we have that
\begin{align*}
  \max_{j \in N_{\omega}}  \bigl|\hat{f}_{\omega}(X_j)\bigr| \leq \max_{\stackrel{\omega' \in \{\Omega_{\star} \cup L(\Omega_{\star})\} \cap \mathcal{N}}{\omega' \preceq \omega}} \kappa_{\omega',T} \cdot R_{\omega',T} + \max_{\stackrel{\omega' \in U(\Omega_{\star}) \cap \mathcal{N}}{\omega' \preceq \omega}} \kappa_{\omega',T} \cdot \sqrt{\frac{\log_+(2 |\mathcal{N}|n_{\omega'}/\delta)}{2k_{\omega'}}}.
\end{align*}
\end{lemma} 
\begin{proof}[Proof of Lemma~\ref{lem:f_bound_at_training_points}]
First, for $\omega = \mathbf{0}_d$, on $E_5^{\delta}$ we have that for $j \in [n]$ that
\begin{align*}
\bigl|\hat{f}_{\mathbf{0}_d}(X_j)-f_{\mathbf{0}_d}(X_j)\bigr| =\Bigl|\frac{1}{n} \sum_{i=1}^n Y_i - \mathbb{P}_P(Y = 1)\Bigr| \leq \sqrt{\frac{\log_+(2 |\mathcal{N}|n/\delta)}{2n}} = R_{\mathbf{0},T}.
\end{align*}
Thus the result holds for $\omega = \mathbf{0}_d$.  

Now, for $\omega \in \{\Omega_{\star} \cup L(\Omega_{\star})\} \cap \mathcal{N}\setminus \{\mathbf{0}_d\}$, fix $j\in N_\omega$. If $\rho_\omega(X_j^\omega)<  r_{\omega}$, then the result follows trivially, since the right hand side of \eqref{ineq:f_bound_training_points} is positive. We thus assume for the remainder of the proof that $\rho_\omega(X_j^{\omega})\geq  r_{\omega}$.  Now recall from Algorithm~\ref{alg:nonadaptive} 
that for $\omega \in \mathcal{N} \setminus\{\mathbf{0}_d\}$ we have 
\begin{align*}
    \hat f_\omega(X_j) = \frac{1}{k_\omega} \sum_{i=1}^{k_\omega}Y_{(i)_\omega}(X_j) - \frac{1}{2} - \sum_{\omega'\prec\omega} \hat{f}_{\omega'}(X_j).
\end{align*}
Then, since $\eta_\omega(X_j) = 1/2 + \sum_{\omega' \preceq \omega} f_{\omega'}(X_j)$, we have that 
\begin{align}
\nonumber
    \bigl| &\hat f_\omega(X_j)-f_\omega(X_j)\bigr| 
    \\ & = \Bigl|\hat{f}_\omega(X_j) - \frac{1}{k_\omega}\sum_{i=1}^{k_\omega} \eta_\omega(X_{(i)_\omega}(X_j)) + \frac{1}{2} + \sum_{\omega'\preceq\omega}\frac{1}{k_\omega}\sum_{i=1}^{k_\omega} f_{\omega'}(X_{(i)_\omega}(X_j)) -f_\omega(X_j)\Bigr| \nonumber\\ 
    &\leq \Bigl|\frac{1}{k_\omega}\sum_{i=1}^{k_\omega}f_\omega(X_{(i)_\omega}(X_j))-f_\omega(X_j)\Bigr| + \Bigl|\frac{1}{k_\omega}\sum_{i=1}^{k_\omega}\bigl\{ Y_{(i)_\omega}(X_j)-\eta_\omega(X_{(i)_\omega}(X_j))\bigr\}\Bigr|\nonumber
    \\ & \hspace{60pt}+\Bigl|\sum_{\omega'\prec\omega}\Bigl\{\frac{1}{k_\omega}\sum_{i=1}^{k_\omega} f_{\omega'}(X_{(i)_\omega}(X_j))- \hat{f}_{\omega'}(X_j)  \Bigr\}\Bigr|. \label{ineq:tilde_f_bound_training_points}
\end{align}     

For the first term in \eqref{ineq:tilde_f_bound_training_points}, by Lemma~\ref{lem:f_Xx_bound_training_points}, on the event $E_6^{\delta}$ we have 
\begin{align*}
    \Bigl|\frac{1}{k_\omega}\sum_{i=1}^{k_\omega}f_\omega(X_{(i)_\omega}(X_j))-f_\omega(X_j)\Bigr|
    &\leq 2C_{\mathrm{B}}  C_{\mathrm{S}} \cdot \left(\frac{2\max\{\lceil4\log_{+}(|\mathcal{N}|n_\omega/\delta)\rceil, k_\omega \}}{(n_\omega-1) r_{\omega}}\right)^{\beta_{\omega}/d_{\omega}}\\
    &\leq  4C_{\mathrm{B}}  C_{\mathrm{S}} \cdot  \frac{1}{ r_{\omega}^{\beta_{\omega}/d_{\omega}}} \cdot \left(\frac{\max\{5\log_{+}(|\mathcal{N}|n_\omega/\delta), k_\omega \}}{(n_\omega-1)}\right)^{\beta_{\omega}/d_{\omega}}\\
    &\leq  20C_{\mathrm{B}}    C_{\mathrm{S}} \cdot  \frac{1}{ r_{\omega}^{\beta_{\omega}/d_{\omega}}} \cdot \frac{\log_{+}(2|\mathcal{N}|n_\omega/\delta)}{(n_{\omega}-1)^{\frac{\beta_{\omega}\gamma_{\omega}+ \alpha \beta_{\omega}^2/d_{\omega}}{\gamma_{\omega}(2\beta_{\omega} + d_{\omega}) + \alpha \beta_{\omega}}}} \leq R_{\omega,\mathrm{T}}.
\end{align*}
since $k_{\omega} \leq 4\log_+(2|\mathcal{N}|n_\omega/\delta) (n_{\omega}-1)^{\frac{2\beta_{\omega}\gamma_{\omega}}{\gamma_{\omega}(2\beta_{\omega} + d_{\omega}) + \alpha \beta_{\omega}}}$.
For the second term in \eqref{ineq:tilde_f_bound_training_points}, on the event $E_5^{\delta}$ we have 
\begin{align*}
    \Bigl|\frac{1}{k_\omega}\sum_{i=1}^{k_\omega}\bigl\{ Y_{(i)_\omega}(X_j)-\eta_\omega(X_{(i)_\omega}(X_j))\bigr\}\Bigr|
    &\leq \sqrt{\frac{\log_+(2 |\mathcal{N}|n_\omega/\delta)}{2k_\omega}} \leq \frac{\log_+(2 |\mathcal{N}|n_\omega/\delta)}{(n_\omega-1)^{\frac{\beta_{\omega}\gamma_{\omega}}{\gamma_{\omega}(2\beta_{\omega}+d_{\omega}) + \alpha\beta_{\omega}}}}
    \\ & \leq 20C_{\mathrm{B}}  
    C_{\mathrm{S}} \cdot  \frac{1}{ r_{\omega}^{\beta_{\omega}/d_{\omega}}} \cdot  \frac{\log_+(2 |\mathcal{N}|n_\omega/\delta)}{(n_\omega-1)^{\frac{\beta_{\omega}\gamma_{\omega} + \alpha\beta_{\omega}^2/ d_{\omega}}{\gamma_{\omega}(2\beta_{\omega}+d_{\omega}) + \alpha\beta_{\omega}}}}\leq R_{\omega,\mathrm{T}},
\end{align*}
since $k_\omega \geq (n_{\omega}-1)^{\frac{2\beta_{\omega}\gamma_{\omega}}{\gamma_{\omega}(2\beta_{\omega} + d_{\omega}) + \alpha \beta_{\omega}}}$ and $r_{\omega} \geq 2(n_{\omega}-1)^{-1} k_{\omega}$. For the third term in~\eqref{ineq:tilde_f_bound_training_points}, 
similarly to above, by Lemma~\ref{lem:f_Xx_bound_training_points} on $E_6^{\delta}$ we have 
\begin{align*}   
\Bigl|\frac{1}{k_\omega}\sum_{i=1}^{k_\omega} f_{\omega'}(X_{(i)_\omega}(X_j)) &- \hat{f}_{\omega'}(X_j)  \Bigr|\\
&\leq  \Bigl|\frac{1}{k_\omega}\sum_{i=1}^{k_\omega}f_{\omega'}(X_{(i)_\omega}(X_j)) -f_{\omega'}(X_j)\Bigr| + \bigl|f_{\omega'}(X_j)- \hat f_{\omega'}(X_j)\bigr| \\
&\leq R_{\omega, \mathrm{T}} + \bigl|f_{\omega'}(X_j)- \hat f_{\omega'}(X_j)\bigr|,
\end{align*} 
for $\omega' \prec \omega$, where we have used the lower bound on $ r_{\omega}$.  To finish the proof of part (i), suppose for induction that \eqref{ineq:f_bound_training_points} holds for all $\omega' \prec \omega$. Then using the fact that $R_{\omega', \mathrm{T}} \leq R_{\omega, \mathrm{T}}$  for all $\omega' \prec \omega$, we deduce that  
\begin{align*}
    \bigl|\hat f_\omega(X_j)-f_\omega(X_j)\bigr|
    &\leq (1 + 2^{d_\omega}) R_{\omega, \mathrm{T}}  + \sum_{\omega'\prec\omega}\bigl|f_{\omega'}(X_j)- \hat f_{\omega'}(X_j)\bigr|  \\
      &\leq (1+2^{d_\omega}+\sum_{\omega'\prec\omega}\kappa_{\omega', \mathrm{T}}) \cdot R_{\omega,\mathrm{T}}.
\end{align*}
The constant $\kappa_{\omega,\mathrm{T}}$ can then be calculated using $\kappa_{\textbf{0}, \mathrm{T}} = 1$ and the relation $\kappa_{\omega,\mathrm{T}} := 1 + 2^{d_\omega}+\sum_{\omega'\prec\omega}\kappa_{\omega',\mathrm{T}}$, for $d_\omega > 1$.

Now for part (ii), fix $\omega \in U(\Omega_{\star}) \cap \mathcal{N}$. Then $f_{\omega}(X_j) = 0$ and, on the event $E_{5}^\delta$, we have
\begin{align*}
    \bigl|\hat{f}_\omega &(X_j)\bigr|  = \Bigl|\hat{f}_\omega(X_j) - \frac{1}{k_\omega}\sum_{i=1}^{k_\omega} \eta_\omega(X_{(i)_\omega}(X_j)) + \frac{1}{2} + \sum_{\omega'\preceq\omega}\frac{1}{k_\omega}\sum_{i=1}^{k_\omega} f_{\omega'}(X_{(i)_\omega}(X_j))\Bigr| \\
    &\hspace{-8pt}\leq \Bigl|\frac{1}{k_\omega}\sum_{i=1}^{k_\omega}\bigl\{ Y_{(i)_\omega}(X_j)-\eta_\omega(X_{(i)_\omega}(X_j))\bigr\}\Bigr| + \Bigl|\sum_{\omega'\prec\omega}\Bigl\{\frac{1}{k_\omega}\sum_{i=1}^{k_\omega} f_{\omega'}(X_{(i)_\omega}(X_j))- \hat{f}_{\omega'}(X_j)  \Bigr\}\Bigr|. 
    \\ & \leq \sqrt{\frac{\log_+(2 |\mathcal{N}|n_\omega/\delta)}{2k_\omega}}  + \Bigl|\sum_{\omega'\prec\omega}\Bigl\{\frac{1}{k_\omega}\sum_{i=1}^{k_\omega} f_{\omega'}(X_{(i)_\omega}(X_j))- \hat{f}_{\omega'}(X_j)  \Bigr\}\Bigr|. 
\end{align*} 
The proof is completed again via induction. 
\end{proof}

\begin{proof}[Proof of Proposition~\ref{lem:correct_thresholding}
]
Let $C_{\rho} := (32C^2_{\mathrm{B}}C_{\mathrm{L}})^{-\frac{1}{\min_{\omega\in \mathcal{N}}\gamma_\omega}}$, where $C_{\mathrm{B}}$ is the constant from Proposition~\ref{prop:f_bounded} and recall $\phi_{\omega} = \gamma_{\omega}(2\beta_{\omega} + d_{\omega}) + \alpha\beta_{\omega}$. Then, for $\omega \in \{\Omega_{\star} \cup L(\Omega_\star)\} \cap \mathcal{N}$, let $r_\omega:= C_{\rho} \cdot n_\omega^{-\frac{\beta_{\omega}}{2\phi_{\omega}}}$,
and consider the event $E_{\circ}^{\delta} = E^\delta_3\cap E^\delta_4\cap E^\delta_5\cap E^\delta_6$ with these choices of $r_{\omega}$. 
First note that, since
\begin{align*}
n_{\omega} &\geq \Bigl(\frac{32\log_+(|\mathcal N|n_{\omega}/\delta)}{C_\rho^{1 \vee \gamma_{\omega}}}\Bigr)^{\frac{2\phi_{\omega}}{2\gamma_\omega d_\omega+2\alpha\beta_{\omega}-\beta_{\omega}}} 
\\&\geq \max\Bigl\{\Bigl(\frac{4\log_+(|\mathcal N|/\delta)}{C_\rho^{\gamma_{\omega}} C_{\mathrm{L}}}\Bigr)^{\frac{2\phi_{\omega}}{2\phi_{\omega}-\beta_{\omega}\gamma_\omega}},\Bigl(\frac{32\log_+(|\mathcal N|n_{\omega}/\delta)}{C_\rho}\Bigr)^{\frac{2\phi_{\omega}}{2\phi_{\omega}-\beta_{\omega}}},\Bigl(\frac{8}{C_\rho}\Bigr)^{\frac{2\phi_{\omega}}{2\phi_{\omega} - 4\gamma_{\omega}\beta_{\omega} -\beta_{\omega}}} \Bigr\}
\end{align*}
for $\omega \in \{\Omega_{\star} \cup L(\Omega_\star)\} \cap \mathcal{N}$, we have that
\begin{align*}
r_{\omega} \geq \max\Bigl\{ \Bigl(\frac{ 4\log_+(|\mathcal{N}|/\delta)}{C_{\mathrm{L}}n_{\omega}}\Bigr)^{1/\gamma_\omega}, \frac{2\max\{\lceil 4\log_+(|\mathcal N|n_\omega/\delta)\rceil,k_\omega\}}{n_{\omega}-1} \Bigr\}.
\end{align*}
Therefore, by a union bound and Lemmas~\ref{lem:E_3_prob}, \ref{lem:E_4_prob}, \ref{lem:E_5_prob} and~\ref{lem:E_6}, we have 
\[
\mathbb{P}_Q\bigl\{\bigl(E_\circ^{\delta}\bigr)^c \bigm| O_1 = o_1, \ldots, O_n = o_n\bigr\} \leq  4\delta. 
\]

In the remainder of the proof we will show that, on the event $E_{\circ}^{\delta}$, we have $\hat \sigma^2_\omega \geq \tau_\omega$ for every $\omega\in\Omega_\star$, and $\hat \sigma_\omega^2<\tau_\omega$ for every $\omega\in \mathcal{N} \cap U(\Omega_\star)$, and therefore that $\hat{\Omega} = \Omega_{\star}$.  First consider $\omega \in \Omega_{\star}$. On the event $E^\delta_\circ$, by Lemma~\ref{lem:f_bound_at_training_points} and Proposition~\ref{prop:f_bounded}, we have that
\begin{align*} 
\hat\sigma^2_\omega - \sigma^2_{\omega} & = \frac{1}{n_\omega} \sum_{i \in N_\omega} \hat{f}^2_{\omega}(X_i) - \sigma^2_{\omega}\\
&= \frac{1}{n_\omega}\sum_{i \in N_{\omega}}\mathbbm{1}_{\{\rho_{\omega}(X_i^{\omega})\geq r_{\omega}\}}\{\hat f_\omega^2(X_i)-f_\omega^2(X_i)\} 
\\ & \hspace{30pt} + \frac{1}{n_\omega}\sum_{i \in N_{\omega}}\mathbbm{1}_{\{\rho_{\omega}(X_i^{\omega})< r_{\omega}\}}\{\hat f_\omega^2(X_i)-f_\omega^2(X_i)\}  + \frac{1}{n_\omega}\sum_{i \in N_{\omega}}f_\omega^2(X_i) - \sigma_\omega^2\\
&\geq -2C_{\mathrm{B}} \kappa_{\omega,\mathrm{T}}\cdot R_{\omega,\mathrm{T}} -8C_{\mathrm{B}}^2 C_{\mathrm{L}}  r_{\omega}^{\gamma_\omega} -C_{\mathrm{B}}^2 \sqrt{\frac{\log_+(|\mathcal N|/\delta)}{2n_\omega}}.
\end{align*}
Then, we have $\sigma_\omega^2 \geq c_{\mathrm{E}}$ since $\omega \in \Omega_{\star}$, and using the fact that 
$\tau_\omega=2^{-4}\cdot n_\omega^{-\frac{\beta_{\omega}\gamma_{\omega}}{2\gamma_{\omega}(2\beta_{\omega}+d_\omega) + 2\alpha\beta_{\omega}}} = 2^{-4}\cdot n_{\omega}^{-1/\kappa_{\omega,\circ}}$, we deduce that 
\[
\hat{\sigma}_\omega^2 - \tau_\omega \geq c_{\mathrm{E}} - \tau_\omega  -2C_{\mathrm{B}} \kappa_{\omega,\mathrm{T}}\cdot R_{\omega,\mathrm{T}} -8C_{\mathrm{B}}^2 C_{\mathrm{L}}  r_{\omega}^{\gamma_\omega} -C_{\mathrm{B}}^2 \sqrt{\frac{\log_+(|\mathcal N|/\delta)}{2n_\omega}}. 
\]
Therefore, we have $\hat{\sigma}_\omega^2 \geq \tau_{\omega}$ for all $\omega \in \Omega_{\star}$, if
\begin{equation}
    \label{eq:cond1}
\max\biggl\{2^{-4}\cdot n_\omega^{-\frac{\beta_{\omega}\gamma_{\omega}}{2\gamma_{\omega}(2\beta_{\omega}+d_\omega) + 2\alpha\beta_{\omega}}}
, 2C_{\mathrm{B}} \kappa_{\omega,\mathrm{T}} R_{\omega,\mathrm{T}}, 8C_{\mathrm{B}}^2 C_{\mathrm{L}}  r_{\omega}^{\gamma_\omega},C_{\mathrm{B}}^2 \sqrt{\frac{\log_+(|\mathcal N|/\delta)}{2n_\omega}}\biggr\} \leq c_{\mathrm{E}}/4
\end{equation}
for all $\omega \in \Omega_{\star}$.  To see that \eqref{eq:cond1} holds, first note that, for $\mathbf{0}_d \prec \omega' \preceq \omega$, we have 
\begin{align*}
\Bigl(n_{\omega'}^{-\frac{\beta_{\omega'}}{2\phi_{\omega'}}}(n_{\omega'}-1)^{\frac{\gamma_{\omega'}d_{\omega'} + \alpha \beta_{\omega'}}{\phi_{\omega'}}}\Bigr)^{-\beta_{\omega'}/d_{\omega'}} & \leq \Bigl(n_{\omega'}^{-\frac{\beta_{\omega'}}{2\phi_{\omega'}}}(n_{\omega'}/2)^{\frac{\gamma_{\omega'}d_{\omega'} + \alpha \beta_{\omega'}}{\phi_{\omega'}}}\Bigr)^{-\beta_{\omega'}/d_{\omega'}} 
 \\ & \leq  2^{\frac{\beta_{\omega'}(\gamma_{\omega'}d_{\omega'} + \alpha \beta_{\omega'})}{\phi_{\omega'}}} \cdot n_{\omega'}^{-\frac{\beta_{\omega'}(2\gamma_{\omega'}d_{\omega'} + 2\alpha \beta_{\omega'} - \beta_{\omega'})}{2\phi_{\omega'}d_{\omega'}}} \\
& <  2 \cdot n_{\omega'}^{-\frac{\beta_{\omega'}\gamma_{\omega'}}{2\phi_{\omega'}}} \leq 2 \cdot n_{\omega'}^{-\frac{\beta_{\omega}\gamma_{\omega}}{2\phi_{\omega}}}.
\end{align*} 
Moreover, since $n \geq n_{\omega'} \geq n_{\omega}$, we have 
\[
\max\Bigl\{\sqrt{\frac{\log(2|\mathcal{N}|n/\delta)}{2n}}, 
\log(2|\mathcal{N}|n_{\omega'}/\delta) n_{\omega'}^{-\frac{\beta_{\omega}\gamma_{\omega}}{2\phi_{\omega}}} \Bigr\} \leq  \kappa_\circ \log(2|\mathcal{N}|n_{\omega}/\delta) \cdot n_{\omega}^{-\frac{\beta_{\omega}\gamma_{\omega}}{2\phi_{\omega}}}, 
\]
for $\kappa_{\omega,\circ} = \frac{2\phi_{\omega}}{\beta_{\omega} \gamma_{\omega}}$.  It follows that
\begin{align*}
R_{\omega,\mathrm{T}} &\! =\! \biggl[ 20C_{\mathrm{B}}C_{\mathrm{S}} \max_{0 \prec \omega' \preceq \omega}\Bigl\{  \log_{+}(2|\mathcal{N}|n_{\omega'}/\delta) \!\cdot\! \Bigl(C_{\rho} \!\cdot\! n_{\omega'}^{-\frac{\beta_{\omega'}}{2\phi_{\omega'}}}(n_{\omega'}\!-\!1)^{\frac{\gamma_{\omega'}d_{\omega'} + \alpha \beta_{\omega'}}{\phi_{\omega'}}}\Bigr)^{-\frac{\beta_{\omega'}}{d_{\omega'}}}\Bigr\} \biggr] \!\vee\! R_{\mathbf{0},\mathrm{T}} \nonumber
\\ & < 40 C_{\mathrm{B}}C_{\mathrm{S}} C_{\rho}^{-\beta_{\omega}} \kappa_{\omega,\circ} \cdot  \log_{+}(2|\mathcal{N}|n_{\omega}/\delta) \cdot n_{\omega}^{-\frac{\beta_{\omega}\gamma_{\omega}}{2\phi_{\omega}}}.
\end{align*} 
Therefore, using that
\begin{align*}
n_{\omega} & \geq \Bigl(\frac{32 C_{B}^2 C_{\mathrm{L}}}{ c_{\mathrm{E}}} \Bigr)^{\frac{\phi_{\omega}(2\vee\alpha^{-1})}{\beta_{\omega}\gamma_{\omega}}} \Bigl(\frac{5 C_{\mathrm{S}} \kappa_{\omega,\mathrm{T}} \kappa_{\omega,\circ} \cdot \log_{+}(2|\mathcal{N}|n_{\omega}/\delta)}{C_{\rho}^{\beta_{\omega}}} \Bigr)^{\frac{2\phi_{\omega}}{\beta_{\omega}\gamma_{\omega}}}
\\ & \geq \max\Bigl\{\Bigl(\frac{1}{4c_{\mathrm{E}}}\Bigr)^{\frac{2\phi_{\omega}}{\beta_{\omega}\gamma_{\omega}}},\Bigl(\frac{160 \cdot  C_{B}^2 C_{\mathrm{S}} \kappa_{\omega,\mathrm{T}}\kappa_{\omega,\circ} \cdot \log_{+}(2|\mathcal{N}|n_{\omega}/\delta)}{C_{\rho}^{\beta_{\omega}} c_{\mathrm{E}}} \Bigr)^{\frac{2\phi_{\omega}}{\beta_{\omega}\gamma_{\omega}}}, 
\\ &\hspace{210 pt}  \Bigl(\frac{32C^2_{\mathrm{B}}C_{\mathrm{L}}C_\rho^{\gamma_\omega}}{c_{\mathrm{E}}}\Bigr)^{\frac{\phi_{\omega}}{\alpha\beta_{\omega}\gamma_\omega}}, \frac{8\log_+(|\mathcal N|/\delta)}{c_{\mathrm{E}}^2C_{\mathrm{B}}^4} \Bigr\}
\end{align*} 
we deduce that \eqref{eq:cond1} holds.  Note further that, since 
\[ 
n_{\omega} \geq \Bigl(\frac{320 \cdot  C_{B} C_{\mathrm{S}} \kappa_{\omega,\mathrm{T}}\kappa_{\omega,\circ} \cdot \log_{+}(2|\mathcal{N}|n_{\omega}/\delta)}{C_{\rho}^{\beta_{\omega}}} \Bigr)^{\frac{4\phi_{\omega}}{\beta_{\omega}\gamma_{\omega}}}, 
\]
we have $\kappa^2_{\omega,T} \cdot R^2_{\omega, T} \leq \tau_{\omega}/4$ for all $\omega \in \Omega_{\star}$, which we will use in \eqref{eq:sigmaU_bound} below.

For $\omega \in U(\Omega_\star)\cap\mathcal{N}$, we have $f_{\omega} \equiv 0$.  Thus, on the event $E^\delta_4\cap E^\delta_5\cap E^\delta_6$, by Lemma~\ref{lem:f_bound_at_training_points} we have
\begin{align}
\nonumber \hat{\sigma}_\omega^2 &= \frac{1}{n_\omega}\sum_{i\in N_{\omega}}\hat f^2_\omega(X_i) \\
\nonumber &\leq  \Biggl( \max_{\stackrel{\omega' \in \{\Omega_{\star} \cup L(\Omega_{\star})\} \cap \mathcal{N}}{\omega' \preceq \omega}} \kappa_{\omega',T} \cdot R_{\omega',T} + \!  \max_{\stackrel{\omega' \in U(\Omega_{\star}) \cap \mathcal{N}}{\omega' \preceq \omega}} \kappa_{\omega',T} \cdot \sqrt{\frac{\log_+(2 |\mathcal{N}|n_{\omega'}/\delta)}{2k_{\omega'}}} \Biggr)^2 
\\\label{eq:sigmaU_bound} & \leq  \max_{\omega' \in \Omega_{\star}, \omega' \preceq \omega} \tau_{\omega'}/2 + \max_{\stackrel{\omega' \in U(\Omega_{\star}) \cap \mathcal{N}}{\omega' \preceq \omega}}  \kappa_{\omega',T}^2 \cdot \log_+(2 |\mathcal{N}|n_{\omega'}/\delta) \cdot n_{\omega'}^{-\frac{2}{\kappa_{\omega',\circ}}}
\\\nonumber  & \leq  \frac{\tau_{\omega}}{2} +  \frac{\kappa_{\omega,\circ} \kappa_{\omega,T}^2 }{2} \cdot \log_+(2 |\mathcal{N}|n_{\omega}/\delta) \cdot n_{\omega}^{-\frac{2}{\kappa_{\omega,\circ}}}  ,
\end{align}
where in the last inequality we used that for $\omega'\preceq\omega$ we have $\tau_{\omega'} \leq \tau_{\omega}$, $\kappa_{\omega',\circ} \leq \kappa_{\omega,\circ}$ and $n_{\omega'} \geq n_{\omega}$, which also implies that  
\[
\max_{\stackrel{\omega' \in U(\Omega_{\star}) \cap \mathcal{N}}{\omega' \preceq \omega}}  \frac{\log(2|\mathcal{N}|n_{\omega'}/\delta)}{n_{\omega'}^{2/\kappa_{\omega',\circ}}} 
\leq \max_{\stackrel{\omega' \in U(\Omega_{\star}) \cap \mathcal{N}}{\omega' \preceq \omega}}  \frac{\log(2|\mathcal{N}|n_{\omega'}/\delta)}{n_{\omega'}^{2/\kappa_{\omega,\circ}}} 
\leq  \frac{\kappa_{\omega,\circ}}{2} \cdot\frac{\log(2|\mathcal{N}|n_{\omega}/\delta) }{n_{\omega}^{2/\kappa_{\omega,\circ}} }.
\]
 
Finally using that, for $\omega\in U(\Omega_\star)\cap\mathcal{N}$, 
\[
n_{\omega} \geq \Bigl(16\kappa_{\omega,\circ} \kappa_{\omega,T}^2  \cdot \log_+(2|\mathcal{N}|n_{\omega}/\delta) \Bigr)^{\kappa_{\omega,\circ}},
\]
we see that $\hat{\sigma}_{\omega}^2 \leq \tau_{\omega}$. This completes the proof.
\end{proof}

\bigskip

\begin{proof}[Proof of Theorem~\ref{thm:nonadaptiveUBnew}
]
Set 
\[
\delta = \delta_{n} := \max_{\omega \in \Omega_{\star}} n_\omega^{-\frac{\beta_{\omega}\gamma_\omega(1+\alpha)}{\gamma_\omega(2\beta_{\omega}+d_\omega)+\alpha\beta_{\omega}}} = \Bigl(\min_{\omega \in \Omega_{\star}} n_\omega^{\frac{\beta_{\omega}\gamma_\omega(1+\alpha)}{\gamma_\omega(2\beta_{\omega}+d_\omega)+\alpha\beta_{\omega}}} \Bigr)^{-1} 
\geq \bigl(\min_{\omega \in \Omega_{\star}} n_\omega \bigr)^{-1},
\]
since $\alpha\beta_{\omega} \leq d_{\omega}$, for $\omega \in \Omega_{\star}$.  Let
\[
C_{\mathrm{N},\omega} := \Bigl(\frac{2^7 C_{B}^2 C_{\mathrm{L}}}{ c_{\mathrm{E}}} \Bigr)^{\kappa_{\omega,\circ}(1\vee\frac{1}{2\alpha})} \Bigl(\frac{20 C_{\mathrm{S}} \kappa_{\omega,\mathrm{T}} \kappa_{\omega,\circ} }{C_{\rho}^{1 \vee \gamma_{\omega}}} \Bigr)^{2\kappa_{\omega,\circ}} 
\]
for $\omega \in \Omega_{\star}$ and $C_{\mathrm{N},\omega}:=  (16\kappa_{\omega,\circ} \kappa_{\omega,T}^2  \bigr)^{\kappa_{\omega,\circ}}$ for $\omega \in U(\Omega_{\star})$.

Assume initially that, for $\omega \in \Omega_{\star}$, we have $n_{\omega} \geq n_{\omega,\circ}$, where
\begin{equation}
\label{eq:boundnomega}
n_{\omega,\circ} := \max\biggl\{2^{d+1} , C_{\mathrm{N},\omega}^2 \cdot (8/e)^{4\kappa_{\omega,\circ}} \bigl(\kappa_{\omega,\circ} + 1+\alpha\bigr)^{4\kappa_{\omega,\circ}} , \frac{1}{4}e^{\max_{\omega \in U(\Omega_\star)} \{C_{N,\omega}^{1/\kappa_{\omega,\circ}}(2\kappa_{\omega,\circ}/e + 1)\} } \biggr\}. 
\end{equation}
We then claim that the sample size conditions in Proposition~\ref{lem:correct_thresholding} 
are satisfied. To see this, first fix $\omega\in\Omega_\star$ and note that by Lemma~\ref{lem:log_n} applied with $m = n_\omega$ and $\epsilon=1/(4\kappa_{\omega,\circ})$ we have
\[
\log^{2\kappa_{\omega,\circ}}(n_\omega)\leq \Bigl(\frac{4\kappa_{\omega,\circ}}{e}\Bigr)^{2\kappa_{\omega,\circ}} \cdot n_\omega^{\frac{1}{2}}.
\]
Therefore, using the first and second terms in the maximum in~\eqref{eq:boundnomega}, we have
\begin{align*}
    C_{\mathrm{N},\omega} \cdot \log^{2\kappa_{\omega,\circ}}_{+}(2|\mathcal{N}|n_{\omega}/\delta_n)
    &\leq C_{\mathrm{N},\omega} \cdot \log^{2\kappa_{\omega,\circ}}\Bigl(2|\mathcal{N}|n_{\omega}^{1+\frac{\beta_{\omega}\gamma_\omega(1+\alpha)}{\gamma_\omega(2\beta_{\omega}+d_\omega)+\alpha\beta_{\omega}}}\Bigr)\\
    &= C_{\mathrm{N},\omega} \cdot \biggl\{\log(2|\mathcal{N}|)+ \Bigl(1+\frac{2(1+\alpha)}{\kappa_{\omega,\circ}}\Bigr) \cdot \log(n_\omega) \biggr\}^{2\kappa_{\omega,\circ}} \\ 
    &\leq C_{\mathrm{N},\omega} \cdot \Bigl(2+\frac{2(1+\alpha)}{\kappa_{\omega,\circ}}\Bigr)^{2\kappa_{\omega,\circ}} \cdot \Bigl(\frac{4\kappa_{\omega,\circ}}{e}\Bigr)^{2\kappa_{\omega,\circ}} \cdot n_\omega^{\frac{1}{2}} \leq  n_\omega,
\end{align*}
for all $\omega\in\Omega_\star$.  Now, for $\omega \in U(\Omega_\star)\cap\mathcal{N}$, again by Lemma~\ref{lem:log_n} applied with $m=n_\omega$, but $\epsilon=1/(2\kappa_{\omega,\circ})$, we have
\[
\log(n_\omega)\leq \Bigl(\frac{2\kappa_{\omega,\circ}}{e}\Bigr) \cdot n_\omega^{\frac{1}{2\kappa_{\omega,\circ}}}.
\]
Then, for $\omega\in U(\Omega_\star)\cap\mathcal{N}$, using the condition on $n_{\omega'}$ in the statement and the third term in~\eqref{eq:boundnomega}, we have
\[
n_\omega \geq  \log_+^{2\kappa_{\omega,\circ}} \bigl(2|\mathcal{N}| \min_{\omega' \in \Omega_{\star}} n_{\omega'}\bigr)  \geq  C_{\mathrm{N},\omega}^2 \cdot \Bigl(\frac{2\kappa_{\omega,\circ}}{e} + 1 \Bigr)^{2\kappa_{\omega,\circ}}
.\]
Therefore 
\begin{align*}
    C_{\mathrm{N},\omega} \cdot \log_{+}^{\kappa_{\omega,\circ}}(2|\mathcal{N}|n_{\omega}/\delta_n)
    &\leq  C_{\mathrm{N},\omega} \cdot \Bigr(\log n_{\omega} +  \log\bigl(2|\mathcal{N}| \min_{\omega' \in \Omega_{\star}} n_{\omega'} \bigr)\Bigr)^{\kappa_{\omega,\circ} }\\
    &\leq  C_{\mathrm{N},\omega} \cdot \Bigr(\frac{2\kappa_{\omega,\circ}}{e}\cdot n_\omega^{\frac{1}{2\kappa_{\omega,\circ}}} + n_\omega^{\frac{1}{2\kappa_{\omega,\circ}}} \Bigr)^{\kappa_{\omega,\circ} }\\
    &\leq  C_{\mathrm{N},\omega} \cdot \Bigr(\frac{2\kappa_{\omega,\circ}}{e} +  1 
    \Bigr)^{\kappa_{\omega,\circ} }  n_\omega^{1/2} \leq n_\omega
\end{align*}
for all $\omega\in U(\Omega_\star)\cap \mathcal{N}$. 
It follows that, by the upper bound in Theorem~\ref{thm:minmax_bounds} 
and Proposition~\ref{lem:correct_thresholding}
, if $n_{\omega} \geq n_{\omega,\circ}$ for $\omega \in \Omega_\star$, then we have 
\begin{align*} 
& \mathbb{E}_Q\Bigl\{\mathcal{E}_P(\hat{C}_{\mathrm{HAM}}) \Bigm| O_1 = o_1, \ldots, O_n = o_n\Bigr\} 
\\ & \hspace{30pt} \leq \mathbb{E}_Q\bigl\{\mathcal{E}_P(\hat{C}_{\Omega_{\star}}) \bigm| O_1 = o_1, \ldots, O_n = o_n \bigr\} + \mathbb{P}_{Q}\bigl(\hat{\Omega} \neq \Omega_{\star} \bigm| O_1 = o_1, \ldots, O_n = o_n\bigr)
\\ & \hspace{30pt} \leq C \cdot \log_+^{\frac{1+\alpha}{2}}\Bigl( \min_{\omega \in \Omega_{\star}} n_{\omega} \Bigr) \cdot \max_{\omega \in \Omega_{\star}} n_\omega^{-\frac{\beta_{\omega}\gamma_\omega(1+\alpha)}{\gamma_\omega(2\beta_{\omega}+d_\omega)+\alpha\beta_{\omega}}} + 4\delta_n,
\end{align*} 
where $C :=  1 + 2 K_{\Omega_{\star}} \cdot \Bigl(\frac{3(1+\alpha)}{2}\Bigr)^{\frac{1+\alpha}{2}} \log_+^{\frac{1+\alpha}{2}}(2^{d+1})$. 

Finally, if $n_\omega < n_{\omega,\circ}$ for some $\omega\in\Omega_\star$, then we have that
\begin{align*} 
\delta_{n} &=  \max_{\omega \in \Omega_{\star}} n_\omega^{-\frac{2(1+\alpha)}{\kappa_{\omega,\circ}}} \geq \max_{\omega \in \Omega_{\star}} \Bigl\{\bigl(\frac{n_{\omega,\circ}}{n_{\omega}}\bigr)^{\frac{2(1+\alpha)}{\kappa_{\omega,\circ}}} \Bigr\} \cdot \min_{\omega \in \Omega_{\star}} n_{\omega,\circ}^{-\frac{2(1+\alpha)}{\kappa_{\omega,\circ}}} 
 \geq \min_{\omega \in \Omega_{\star}} \Bigl\{n_{\omega,\circ}^{-\frac{2(1+\alpha)}{\kappa_{\omega,\circ}}}\Bigr\}.
\end{align*} 
Thus
\begin{align*} 
& \mathbb{E}_Q\Bigl\{\mathcal{E}_P(\hat{C}_{\mathrm{HAM}}) \Bigm| O_1 = o_1, \ldots, O_n = o_n\Bigr\} \leq 1 \leq \frac{\delta_n} {\min_{\omega \in \Omega_{\star}} \Bigl\{n_{\omega,\circ}^{-\frac{2(1+\alpha)}{\kappa_{\omega,\circ}}}\Bigr\}} .
\end{align*}
We therefore conclude that the overall result holds with 
\[
C_{\mathrm{U}} := \max\biggl\{ C+4, \max_{\omega \in \Omega_{\star}} \Bigl(n_{\omega,\circ}^{\frac{2(1+\alpha)}{\kappa_{\omega,\circ}}} \Bigr)  \biggr\}.
\]
\end{proof}

\section{Additional numerical results\label{sec:additional_numerical_results}}

\subsection{Cross-validation version of the HAM algorithm\label{subsec:cvHAM}}
If the parameters $\alpha$, $\boldsymbol{\beta}$, and $\boldsymbol{\gamma}$ are unknown, we propose to use a cross-validation approach to estimate these.  The empirical performance of this approach is illustrated in the numerical simulations in Sections~\ref{sec:numericalresults} 
and~\ref{sec:simulationsappendix}.  Our proposed procedure requires the user to choose a set of possible values for these parameters and employs $5$-fold cross-validation. For each of the $5$ folds and each combination of parameter choices, the method calculates the empirical test error on the left out fold using the other $4$ folds as training data.  It returns whichever combination of parameter values minimises the empirical test error averaged over all folds, which can then be used in the HAM algorithm.  The cross-validation procedure is given in Algorithm~\ref{alg:cv}.  The practitioner is left with choosing the range of possible values for the parameters.  In our experiments, we have taken $A=\{0,1/2,1,d\}$ as possible values for $\alpha$, $B=\{1/4,2/4,3/4,1\}\cdot \mathbf{1}_d$ as candidate values for $\boldsymbol{\beta}$, and $G=\{1/2,1,2,\infty\}\cdot \mathbf{1}_d$ as candidate values for $\boldsymbol{\gamma}$, and we recommend these as suitable default values.  

\begin{algorithm}[ht]
\caption{A cross-validation procedure to choose $\alpha$, $\boldsymbol{\beta}$ and $\boldsymbol{\gamma}$.}\label{alg:cv}
\begin{algorithmic}[1]
\Input:~Data $D_n= ((X_1^{o_1}, Y_1, o_1), \ldots, (X_n^{o_n}, Y_n,o_n)) \in (\mathbb{R}^d \times \{0,1\} \times \{0,1\}^d)^n$, and finite sets $A\subseteq [0,\infty)$, $B\subseteq (0,1]^d$ and $G\subseteq  [0,\infty)^d$.
\EndInput 

\State Split $D_n$ into $5$ equally sized subsets uniformly at random, denote these $D_n^{(1)}$, \ldots, $D_n^{(5)}$.
\For{$(a,\mathbf{b},\mathbf{g})\in A\times B\times G$}
\State $E_{(a,\mathbf{b},\mathbf{g})}=0$
\For{$l\in[5]$}
\State $\hat{f}_{\mathbf{0}}(\cdot) := \frac{1}{|D_n|-|D_n^{(l)}|}\sum_{i\in[n]:(X_i^{o_i}, Y_i,o_i) \notin D_n^{(l)}} Y_i -\frac{1}{2}$
\For{$(x^o,y,o)\in D_n^{(l)}$}
\State $x:=x^o$
\For{$d'=1,\ldots,d_o$}
\For{$\omega\preceq o$ such that $d_\omega=d'$}
\State $b_\omega := \min\{b_j : \omega_j = 1\}$,  $g_\omega := \min\{g_j : \omega_j = 1\}$
\State $m_\omega  := \sum_{i\in[n]:(X_i^{o_i}, Y_i,o_i) \notin D_n^{(l)}} \mathbbm{1}_{\{\omega \preceq o_i\}}$
\State $k_{\omega} := 1 + \lfloor m_{\omega}^{\frac{2b_{\omega}g_{\omega}}{g_{\omega}(2b_{\omega}+d_{\omega}) + a b_{\omega}}}\rfloor$; 
$M_{\omega} := \{i \in [n] : \omega \preceq o_i, (X_i^{o_i}, Y_i,o_i) \notin D_n^{(l)}\}$
\State Let $(X_{(1)_{\omega}}^{\omega}(x),Y_{(1)_{\omega}}(x)),\ldots,(X_{(m_{\omega})_{\omega}}^{\omega}(x),Y_{(m_{\omega})_{\omega}}(x))$ be a reordering of the pairs $ \{(X_i^{o_i}, Y_i) : i \in M_{\omega}, (X_i^{o_i}, Y_i,O_i) \notin D_n^{(l)}\}$ such that $\|X_{(1)_{\omega}}^{\omega}(x) - x^{\omega}\| \leq \cdots \leq \|X_{(m_{\omega})_{\omega}}^{\omega}(x) - x^{\omega}\|$
\State $\hat f_\omega(x): = \frac{1}{k_{\omega}}\sum_{j=1}^{k_{\omega}} Y_{(j)_\omega}(x) -\frac{1}{2} - \sum_{\omega' \prec \omega} \hat{f}_{\omega'}(x)$
\EndFor
\EndFor
\State $\hat{\eta}(x):=\frac{1}{2} + \sum_{\omega\preceq o}\hat f_{\omega}(x)$
\State $\hat{Y} := \mathbbm{1}_{\{\hat{\eta}(x) \geq 1/2\}}$  
\State $E_{(a,\mathbf{b},\mathbf{g})}\leftarrow E_{(a,\mathbf{b},\mathbf{g})}+\mathbbm{1}_{\{\hat{Y}\neq y\}}$
\EndFor
\EndFor
\EndFor

\Output~$(\hat{\alpha},\hat{\boldsymbol{\beta}},\hat{\boldsymbol{\gamma}}) := \argmin_{(a,\mathbf{b},\mathbf{g})\in A\times B\times G}\bigl\{E_{(a,\mathbf{b},\mathbf{g})}\bigr\}$.
\EndOutput
\end{algorithmic}
\end{algorithm}

\subsection{Additional numerical experiments\label{sec:simulationsappendix}}
Here we further investigate the performance of the HAM classifier in our framework.  We consider the setting where there is in fact no missing data.  We carried out similar experiments to those in Section~\ref{sec:numericalresults}
, where in each case, in contrast to the earlier experiments, we take our training data set to be $n = 1000$ independent copies of the pair $(X, Y) \sim P \equiv P_Q$.  Since there is no missing data, the Complete Case approach and all imputation methods are equivalent to directly applying the $k$-nearest neighbour ($k$nn) algorithm to the training data. The results are presented in Figure~\ref{fig:results_no_missingness} -- we see that the HAM classifier outperforms a direct application of $k$nn in each setting.

\begin{figure}[htp]
\centering
\includegraphics[width=0.95\linewidth]{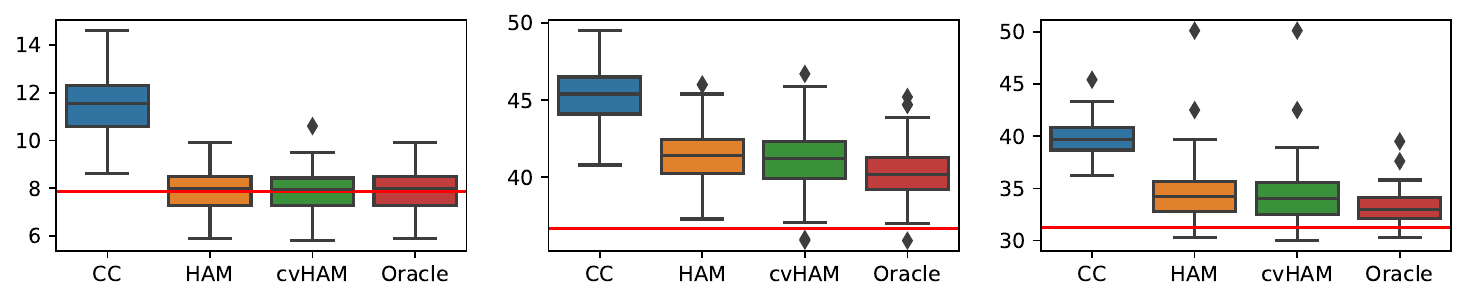}
\caption{Box plots of the empirical percentage test errors for the Complete Case (CC) approach, our HAM classifier with known parameters (HAM), and our classifier with parameters estimated using cross-validation (cvHAM).  For reference, we also include our Oracle HAM (Oracle) classifier, and the Bayes risk is shown as the red horizontal line.  Here there is no missing data, so the imputation methods would be the same as the complete case classifier. We present the results for Setting~1 (left), Setting~2 (middle) and Setting~3 (right) when $n=1000$. 
\label{fig:results_no_missingness}}
\end{figure}

\section{Auxiliary results\label{sec:appendixaux}} 
\begin{lemma}\label{lem:log_n}
For all $m\geq1$ and any $\epsilon>0$, $\log m\leq (m^\epsilon)/(e\cdot\epsilon)$.
\end{lemma}
\begin{proof}
Fix $\epsilon>0$. Let $A_{\epsilon} := \sup_{m\geq1}((\log m)/m^\epsilon)$, then our claim is that $A_{\epsilon} =1/(e\cdot\epsilon)$. To see this, let $g(m)=\log((\log m)/m^\epsilon)=\log\log(m)-\epsilon\cdot\log m$, which is maximised at $m=e^{1/\epsilon}$, and indeed \[\frac{\log(e^{1/\epsilon})}{e^{(1/\epsilon)\cdot\epsilon}}=\frac{1}{e\cdot\epsilon}.\]
\end{proof}

\bibliography{References.bib}

\begin{thebibliography}{}

\bibitem[Ahfock and McLachlan, 2023]{ahfock2022semi}
Ahfock, D. and McLachlan, G.~J. (2023).
\newblock Semi-supervised learning of classifiers from a statistical
  perspective: A brief review.
\newblock {\em Econometrics and Statistics}, 26:124--138.

\bibitem[Audibert and Tsybakov, 2007]{audibert2007fast}
Audibert, J.-Y. and Tsybakov, A.~B. (2007).
\newblock Fast learning rates for plug-in classifiers.
\newblock {\em Ann. Statist.}, 35(2):608--633.

\bibitem[Ayme et~al., 2022]{ayme2022near}
Ayme, A., Boyer, C., Dieuleveut, A., and Scornet, E. (2022).
\newblock Near-optimal rate of consistency for linear models with missing
  values.
\newblock In {\em International Conference on Machine Learning}, pages
  1211--1243. PMLR.

\bibitem[Berrett and Samworth, 2023]{berrett2022optimal}
Berrett, T.~B. and Samworth, R.~J. (2023).
\newblock Optimal nonparametric testing of {M}issing {C}ompletely {A}t
  {R}andom, and its connections to compatibility.
\newblock {\em Ann. Statist.}, 51:2170–2193.

\bibitem[Billingsley, 1995]{billingsley2008probability}
Billingsley, P. (1995).
\newblock {\em Probability and Measure}.
\newblock John Wiley \& Sons.

\bibitem[Bordino and Berrett, 2024]{bordino2024tests}
Bordino, A. and Berrett, T.~B. (2024).
\newblock Tests of {M}issing {C}ompletely {A}t {R}andom based on sample
  covariance matrices.
\newblock {\em arXiv preprint arXiv:2401.05256}.

\bibitem[Boucheron et~al., 2005]{boucheron2005theory}
Boucheron, S., Bousquet, O., and Lugosi, G. (2005).
\newblock Theory of classification: A survey of some recent advances.
\newblock {\em ESAIM: Probability and Statistics}, 9:323--375.

\bibitem[Cai and Wei, 2021]{cai2021transfer}
Cai, T.~T. and Wei, H. (2021).
\newblock Transfer learning for nonparametric classification: Minimax rate and
  adaptive classifier.
\newblock {\em Ann. Statist.}, 49(1):100--128.

\bibitem[Cai and Zhang, 2016]{cai2016minimax}
Cai, T.~T. and Zhang, A. (2016).
\newblock Minimax rate-optimal estimation of high-dimensional covariance
  matrices with incomplete data.
\newblock {\em Journal of Multivariate Analysis}, 150:55--74.

\bibitem[Cai and Zhang, 2019]{cai2019high}
Cai, T.~T. and Zhang, L. (2019).
\newblock High-dimensional linear discriminant analysis: optimality, adaptive
  algorithm and missing data.
\newblock {\em Journal of the Royal Statistical Society: Series B},
  81(4):675--705.

\bibitem[Cannings and Fan, 2022]{cannings2022correlation}
Cannings, T.~I. and Fan, Y. (2022).
\newblock The correlation-assisted missing data estimator.
\newblock {\em Journal of Machine Learning Research}, 23:1--49.

\bibitem[Cannings et~al., 2020]{cannings2020classification}
Cannings, T.~I., Fan, Y., and Samworth, R.~J. (2020).
\newblock Classification with imperfect training labels.
\newblock {\em Biometrika}, 107(2):311--330.

\bibitem[Chandrasekher et~al., 2020]{chandrasekher2020imputation}
Chandrasekher, K.~A., Alaoui, A.~E., and Montanari, A. (2020).
\newblock Imputation for high-dimensional linear regression.
\newblock {\em arXiv preprint arXiv:2001.09180}.

\bibitem[Chapelle et~al., 2006]{chapelle2006semi}
Chapelle, O., Sch{\"o}lkopf, B., and Zien, A. (2006).
\newblock {\em Semi-Supervised Learning}.
\newblock The MIT Press, Cambridge, MA.

\bibitem[Devroye et~al., 1996]{devroye2013probabilistic}
Devroye, L., Gy{\"o}rfi, L., and Lugosi, G. (1996).
\newblock {\em A Probabilistic Theory of Pattern Recognition}.
\newblock Springer, New York.

\bibitem[Efron and Stein, 1981]{efron1981jackknife}
Efron, B. and Stein, C. (1981).
\newblock The jackknife estimate of variance.
\newblock {\em Ann. Statist.}, 9(3):586--596.

\bibitem[Elsener and van~de Geer, 2019]{elsener2019sparse}
Elsener, A. and van~de Geer, S. (2019).
\newblock Sparse spectral estimation with missing and corrupted measurements.
\newblock {\em Stat}, 8(1):e229.

\bibitem[Elter, 2007]{misc_mammographic_mass_161}
Elter, M. (2007).
\newblock {Mammographic Mass}.
\newblock UCI Machine Learning Repository.
\newblock {DOI}: https://doi.org/10.24432/C53K6Z.

\bibitem[Fix and Hodges, 1952]{fix1952discriminatory}
Fix, E. and Hodges, J.~L. (1952).
\newblock Discriminatory analysis-nonparametric discrimination: Small sample
  performance.
\newblock {\em Technical Report number 4, USAF School of Aviation Medicine,
  Randolph Field, Texas}.

\bibitem[Fix and Hodges, 1989]{fix1989discriminatory}
Fix, E. and Hodges, J.~L. (1989).
\newblock Discriminatory analysis-nonparametric discrimination: Small sample
  performance.
\newblock {\em Internat. Statist. Rev.}, 57:238--–247.

\bibitem[Follain et~al., 2022]{follain2022high}
Follain, B., Wang, T., and Samworth, R.~J. (2022).
\newblock High-dimensional changepoint estimation with heterogeneous
  missingness.
\newblock {\em Journal of the Royal Statistical Society Series B: Statistical
  Methodology}, 84(3):1023--1055.

\bibitem[Fr{\'e}nay et~al., 2014]{frenay2014comprehensive}
Fr{\'e}nay, B., Kab{\'a}n, A., et~al. (2014).
\newblock A comprehensive introduction to label noise.
\newblock In {\em ESANN}. Citeseer.

\bibitem[Fr{\'e}nay and Verleysen, 2013]{frenay2013classification}
Fr{\'e}nay, B. and Verleysen, M. (2013).
\newblock Classification in the presence of label noise: a survey.
\newblock {\em IEEE Transactions on Neural Networks and Learning Systems},
  25(5):845--869.

\bibitem[Hastie and Tibshirani, 1986]{hastie1986generalized}
Hastie, T. and Tibshirani, R. (1986).
\newblock Generalized additive models.
\newblock {\em Statistical Science}, 1(3):297--318.

\bibitem[Josse et~al., 2019]{josse2019consistency}
Josse, J., Prost, N., Scornet, E., and Varoquaux, G. (2019).
\newblock On the consistency of supervised learning with missing values.
\newblock {\em arXiv preprint arXiv:1902.06931}.

\bibitem[Josse and Reiter, 2018]{josse2018introductionm}
Josse, J. and Reiter, J.~P. (2018).
\newblock {Introduction to the Special Section on Missing Data}.
\newblock {\em Statistical Science}, 33(2):139--141.

\bibitem[Le~Morvan et~al., 2020]{le2020neumiss}
Le~Morvan, M., Josse, J., Moreau, T., Scornet, E., and Varoquaux, G. (2020).
\newblock Neumiss networks: differentiable programming for supervised learning
  with missing values.
\newblock {\em Advances in Neural Information Processing Systems},
  33:5980--5990.

\bibitem[Le~Morvan et~al., 2021]{le2021sa}
Le~Morvan, M., Josse, J., Scornet, E., and Varoquaux, G. (2021).
\newblock What’s a good imputation to predict with missing values?
\newblock {\em Advances in Neural Information Processing Systems},
  34:11530--11540.

\bibitem[Lee and Barber, 2022]{lee2022binary}
Lee, Y. and Barber, R.~F. (2022).
\newblock Binary classification with corrupted labels.
\newblock {\em Electronic Journal of Statistics}, 16(1):1367--1392.

\bibitem[Little and Rubin, 2019]{little2019statistical}
Little, R.~J. and Rubin, D.~B. (2019).
\newblock {\em Statistical Analysis with Missing Data}, volume 793.
\newblock John Wiley \& Sons.

\bibitem[Loh and Tan, 2018]{loh2018high}
Loh, P.-L. and Tan, X.~L. (2018).
\newblock High-dimensional robust precision matrix estimation: Cellwise
  corruption under $\epsilon$-contamination.
\newblock {\em Electronic Journal of Statistics}, 12:1429--1467.

\bibitem[Loh and Wainwright, 2012]{loh2011high}
Loh, P.-L. and Wainwright, M.~J. (2012).
\newblock High-dimensional regression with noisy and missing data: Provable
  guarantees with nonconvexity.
\newblock {\em Ann. Statist.}, 40(3):1637.

\bibitem[Mammen and Tsybakov, 1999]{mammen1999smooth}
Mammen, E. and Tsybakov, A.~B. (1999).
\newblock Smooth discrimination analysis.
\newblock {\em Ann. Statist.}, 27(6):1808--1829.

\bibitem[McShane, 1934]{mcshane1934extension}
McShane, E.~J. (1934).
\newblock Extension of range of functions.
\newblock {\em Bulletin of the American Mathematical Society}, 40(12):837--842.

\bibitem[Polonik, 1995]{polonik1995measuring}
Polonik, W. (1995).
\newblock Measuring mass concentrations and estimating density contour
  clusters-an excess mass approach.
\newblock {\em Ann. Statist.}, 23(3):855--881.

\bibitem[Reeve et~al., 2021]{reeve2021adaptive}
Reeve, H.~W., Cannings, T.~I., and Samworth, R.~J. (2021).
\newblock Adaptive transfer learning.
\newblock {\em Ann. Statist.}, 49(6):3618--3649.

\bibitem[Sportisse et~al., 2023]{sportisse2023labels}
Sportisse, A., Schmutz, H., Humbert, O., Bouveyron, C., and Mattei, P.-A.
  (2023).
\newblock Are labels informative in semi-supervised learning?--estimating and
  leveraging the missing-data mechanism.
\newblock {\em arXiv preprint arXiv:2302.07540}.

\bibitem[Stekhoven and B{\"u}hlmann, 2012]{stekhoven2012missforest}
Stekhoven, D.~J. and B{\"u}hlmann, P. (2012).
\newblock Missforest—non-parametric missing value imputation for mixed-type
  data.
\newblock {\em Bioinformatics}, 28(1):112--118.

\bibitem[Weiss et~al., 2016]{weiss2016survey}
Weiss, K., Khoshgoftaar, T.~M., and Wang, D. (2016).
\newblock A survey of transfer learning.
\newblock {\em Journal of Big data}, 3(1):1--40.

\bibitem[Zhang et~al., 2018]{zhang2018missing}
Zhang, Q., Yuan, Q., Zeng, C., Li, X., and Wei, Y. (2018).
\newblock Missing data reconstruction in remote sensing image with a unified
  spatial--temporal--spectral deep convolutional neural network.
\newblock {\em IEEE Transactions on Geoscience and Remote Sensing},
  56(8):4274--4288.

\bibitem[Zhu et~al., 2022]{zhu2019high}
Zhu, Z., Wang, T., and Samworth, R.~J. (2022).
\newblock High-dimensional principal component analysis with heterogeneous
  missingness.
\newblock {\em Journal of the Royal Statistical Society, Series B},
  84(5):2000–2031.

\bibitem[Zou, 2018]{zou2018log}
Zou, Q. (2018).
\newblock The log-convexity of the {F}ubini numbers.
\newblock {\em Transactions on Combinatorics}, 7(2):17--23.

\end{thebibliography}
\bibliographystyle{apalike}

\end{document}